\documentclass[reqno,11pt]{amsart}
\usepackage{amssymb,amsmath,amsthm}
\usepackage{amscd}
\usepackage{amsfonts}
\usepackage{amssymb}
\usepackage{latexsym}
\usepackage{color}
\usepackage{esint}
\usepackage{graphicx, subfigure}
\usepackage{caption}
\usepackage{amsthm}
\usepackage{hyperref}
\usepackage{tikz}
\usepackage{todonotes} 
\usepackage{physics, mathtools}
\usepackage{multirow}
\usetikzlibrary{shapes.geometric}
\DeclareMathOperator*{\esssup}{ess\,sup}

 \usepackage[makeroom]{cancel}
 \usetikzlibrary{decorations.pathreplacing}

\let\hide\iffalse

\newtheorem{theorem}{Theorem}

\newtheorem{definition}[theorem]{Definition}
\newtheorem{lemma}[theorem]{Lemma}
\newtheorem{proposition}[theorem]{Proposition}
\theoremstyle{remark}
\newtheorem{remark}[theorem]{Remark}

\newcommand{\Bes}{\begin{eqnarray*}}
\newcommand{\Ees}{\end{eqnarray*}}
\newcommand{\Be}{\begin{equation}}
\newcommand{\Ee}{\end{equation}}

 \numberwithin{equation}{section}
 \numberwithin{theorem}{section}

\def\R{\mathbb{R}}

\def\B{\begin{equation}}
\def\E{\end{equation}}
\def\BN{\begin{eqnarray*}}
\def\EN{\end{eqnarray*}}

\DeclareMathOperator{\supp}{supp}

\usepackage{color}

\definecolor{color1bg}{HTML}{FFB000}
\definecolor{color2bg}{HTML}{648FFF}

\begin{document}

\date{\today}
    
\title{On the Boltzmann-Fermi-Dirac Equation for Hard Potential: Global Existence and Uniqueness, Gaussian Lower Bound, and Moment Estimates}

 \author{Gayoung An}
\address{Department of Mathematics, Yonsei University, South Korea}
\email{gayoungan@yonsei.ac.kr}
\author{Sungbin Park}
\address{Department of Mathematics, Pohang University of Science and Technology, South Korea}
\email{parksb2942@postech.ac.kr}




\begin{abstract}
    In this paper, we study the global existence and uniqueness, Gaussian lower bound, and moment estimates in the spatially homogeneous Boltzmann equation for Fermi-Dirac particles for hard potential ($0\leq \gamma\leq 2$) with angular cutoff $b$. Our results extend classical results to the Boltzmann-Fermi-Dirac setting. In detail, (1) we show existence, uniqueness, and $L^1_2$ stability of global-in-time solutions of the Boltzmann-Fermi-Dirac equation. (2) Assuming the solution is not a saturated equilibrium, we prove creation of a Gaussian lower bound for the solution. (3) We prove creation and propagation of $L^1$ polynomial and exponential moments of the solution under additional assumptions on the angular kernel $b$ and $0<\gamma\leq 2$. (4) Finally, we show propagation of $L^\infty$ Gaussian and polynomial upper bounds when $b$ is constant and $0<\gamma\leq 1$.
\end{abstract}

\maketitle
\tableofcontents

\section{Introduction}
The spatially homogeneous Boltzmann-Fermi-Dirac equation is a quantum modification of the classical Boltzmann equation for Fermi-Dirac particles and is written as
\begin{align}\label{eq:BT}
    \partial_t f  = Q_{FD}(f,f), \quad f(0,v)=f_0(v),
\end{align}
where $v\in\mathbb{R}^3$ and $t\geq0$. The collision operator $Q_{FD}(f,f)(t,v)$ is given by
\begin{align*}
    Q_{FD}(f,f)(t,v)&\coloneqq \int_{\mathbb{R}^3\times \mathbb{S}^2}B(v-v_*, \sigma)\big( f(t,v')f(t,v_*')(1-f(t,v))(1-f(t,v_*)) \\
     &\quad \quad -f(t,v)f(t,v_*)(1-f(t,v'))(1-f(t,v_*')\big) d \sigma dv_*.
\end{align*}
Here, the solution $f(t,v)$ represents the velocity distribution of the particles at time $t$, and $Q_{FD}$ describes the change in $f$ due to particle collisions. The velocities $v',\,v_*'\in\mathbb{R}^3$ represent the post-collision velocities of particles and are expressed in terms of the initial velocities $v,\, v_*\in\mathbb{R}^3$ and $\sigma \in \mathbb{S}^2$ by
\begin{align} \label{eq:def_v'}
    v' = \frac{v + v_*}{2} + \frac{|v - v_*|}{2}\sigma \quad \text{and} \quad v_*' = \frac{v + v_*}{2} - \frac{|v - v_*|}{2}\sigma.
\end{align}
In this paper, we consider the collision kernel $B$ given by 
\begin{align} \label{eq:B_defi}
    B(v-v_*, \sigma) = B(|v-v_*|, \cos\theta) = |v-v_*|^\gamma b(\cos \theta),
\end{align}
where $\cos\theta$ is defined by
\begin{align} \label{eq:theta_sigma}
   \cos \theta \coloneqq  \frac{v-v_*}{|v-v_*|} \cdot \sigma, \quad \theta \in [0, \pi].
\end{align}
We also impose the usual symmetry condition $b(\cos(\pi-\theta)) = b(\cos\theta)$.

For the angular collision kernel $b(\cos\theta)$, we will use various settings depending on the problem.

\textbf{(H1)} Throughout this paper, we consider Grad's cut-off assumption:
\begin{align} \label{eq:C_b:def}
    0<C_b \coloneqq 2\pi\int_0^\pi b(\cos\theta)\sin\theta\,d\theta<\infty.
\end{align}

\textbf{(H2)} For the Gaussian lower bound, we assume
\begin{align}\label{eq:b_lower_naer_pi/2}
    b(\cos\theta)>c_b>0
\end{align}
for some constant $c_b$ for $\theta\in [\pi/4, 3\pi/4]$. We use this condition to give a lower bound of $b(\cos\theta)$ near $\theta = \frac{\pi}{2}$.\\

\textbf{(H3)} For the $L^\infty$ Gaussian upper bound, we assume
\begin{align*}
    b(\cos\theta)\sin^\alpha \theta\leq C
\end{align*}
for some $\alpha<2$ and some constant $C>0$ on $\theta\in (0,\pi)$.\\

\textbf{(H4)} For $L^\infty$ polynomial moments estimates, we assume
\begin{align} \label{eq:b=const}
   b(\cos\theta) = const.
\end{align}
In particular, when $\gamma = 1$ together with \eqref{eq:b=const}, this is called the hard sphere model, describing collisions between two rigid spheres.\\

\textbf{(H5)} To handle some critical cases in the Boltzmann-Fermi-Dirac equation, we make an additional assumption
\begin{align*}
    b(\cos\theta)>0
\end{align*}
on $\theta\in(0,\pi)$. It will only be used in Proposition \ref{prop:entropy_production}. \\

The relation between the assumptions is (H4)$\Rightarrow$(H3)$\Rightarrow$(H1). When we assume (H4) or (H3), therefore, we implicitly assume (H1). Table \ref{table:1} summarizes the assumptions for the corresponding problems.

The assumptions (H1)-(H4) are in fact from the assumptions used in the classical Botlzmannn equation to derive the classical results in Table \ref{table:1}. We also note that the hard sphere model $B(v-v_*,\sigma) = |v-v_*|$ fulfills all the assumptions (H1)-(H5).
\begin{table}[htbp]
\begin{tabular}{|c|c|c|}
\hline
Problems & $\gamma$ & $b(\cos\theta)$ \\ \hline
Gaussian lower bound & $0\leq \gamma\leq 2$ & (H1), (H2), and (H5) \\ \hline
$L^1$ polynomial upper bound & \multirow{2}{*}{$0<\gamma\leq 2$} & (H1)\\ \cline{1-1} \cline{3-3} 
$L^\infty$ Gaussian upper bound  &  & (H3) \\ \hline
$L^\infty$ polynomial upper bound & $0<\gamma\leq 1$  & (H4) \\ \hline
\end{tabular}
\caption{Correspondence between the assumptions on $b(\cos\theta)$ and the problems.}
\label{table:1}
\end{table}

In (H1), we further reserve a constant $C_{b,2}$ by
\begin{align}\label{eq:def_C_b_2}
    C_{b,2}\coloneqq 2\pi\int_0^\pi b(\cos\theta)\sin^3\theta\,d\theta,
\end{align}
and a function $\varphi$ by
\begin{align}\label{eq:phi_def}
    \varphi(\epsilon)\coloneqq \int_{\mathbb{S}^2} b(\cos\theta)\left(\mathbf{1}_{\{0<\theta<\epsilon\}} + \mathbf{1}_{\{\pi-\epsilon<\theta<\pi\}}\right)\,d\sigma
\end{align}
for $0<\epsilon<1$. Since $b(\cos\theta)$ is integrable, it satisfies $0\leq \varphi(\epsilon)\leq C_b$ and $\lim_{\epsilon\rightarrow 0} \varphi(\epsilon) = 0$. These two will be used after Section \ref{sec:L1_estimate}.

As usual, we define the macroscopic quantities of $f(v)$ by
\begin{align}\label{eq:macro_quantity}
    \begin{split}
        \rho = \int_{\mathbb{R}^3}f(v)\,dv,\quad \rho u = \int_{\mathbb{R}^3} vf(v)\,dv,\quad 3\rho T= \int_{\mathbb{R}^3}|v-u|^2 f(v)\,dv.
    \end{split}
\end{align}
Like the classical Boltzmann equation, the collision operator $Q_{FD}$ satisfies
\begin{align*}
    \int_{\mathbb{R}^3}\begin{pmatrix}
        1\\v\\v^2
    \end{pmatrix}Q_{FD}(f,f)\,dv = 0
\end{align*}
for any compactly supported continuous function $f$. Therefore, we can consider a solution of the Boltzmann-Fermi-Dirac equation that conserves mass, momentum, and energy.

Even though there are many structural similarities between the Boltzmann-Fermi-Dirac equation and the classical Boltzmann equation, there are some important distinctions between the two equations inherent from their physical nature. First, the solution $f(t,v)$ of the Boltzmann-Fermi-Dirac equation satisfies
\begin{align}\label{eq:0f1}
    0\leq f(t,v)\leq 1
\end{align}
if $0\leq f_0(v)\leq 1$ due to the Pauli exclusion principle.

Secondly, the entropy functional in the Boltzmann-Fermi-Dirac equation is given by
\begin{align}\label{eq:entropy}
    S(f) = -\int_{\mathbb{R}^3} f\ln f + (1-f)\ln (1-f)\,dv.
\end{align}
Taking the time derivative on both sides and the time integral, we formally get
\begin{align}\label{eq:entropy_identity}
    S(f)(t) = S(f)(0) +\int_0^t \int_{\mathbb{R}^3} D(f)(\tau, v)\,dvd\tau,
\end{align}
where $D(f)(t, v)$ is defined by
\begin{align*}
    D(f)(t,v) = \frac{1}{4}\int_{\mathbb{R}^3\times \mathbb{S}^2} B(v-v_*,\sigma)\Gamma\left(f'f'_*(1-f)(1-f_*), ff_*(1-f')(1-f'_*)\right)\,d\sigma dv_*
\end{align*}
with the notations $f =f(v)$, $f_* =f(v_*)$, $f' = f(v')$, and $f_*' =f(v_*')$. The function $\Gamma(a, b)$ is given by
\begin{align*}
    \Gamma(a, b) = \begin{cases}
        (a-b)\ln\frac{a}{b} & a,b>0,\\
        +\infty &  a>b=0\text{ or } b>a=0,\\
        0 & a=b=0.
    \end{cases}
\end{align*}

Using this definition, we can easily check the H-theorem for the Boltzmann-Fermi-Dirac equation. The equilibrium function for the Boltzmann-Fermi-Dirac equation is given by
\begin{align*}
    f(v) = \frac{1}{e^{a|v-u|^2 + c} + 1}
\end{align*}
for some $a>0$, $c\in\mathbb{R}$, and $u$ given by \eqref{eq:macro_quantity}. Here, $a$ and $c$ are implicitly given by
\begin{align*}
    \frac{\rho}{(3\rho T)^{\frac{N}{N+2}}} = \frac{\int_{\mathbb{R}^3} \frac{1}{e^{|v|^2+c}+1}\,dv}{\left(\int_{\mathbb{R}^3} \frac{|v|^2}{e^{|v|^2+c}+1}\,dv\right)^{\frac{N}{N+2}}}
\end{align*}
and
\begin{align*}
    a = \left(\int_{\mathbb{R}^3} \frac{1}{e^{|v|^2+c}+1}\,dv\right)^{\frac{2}{N}}\rho^{-\frac{2}{N}}.
\end{align*}
We call this equilibrium \textit{Fermi-Dirac equilibrium}. It satisfies $Q_{FD}(f,f) = D(f) = 0$, so it is a time-stationary solution. Note that any Fermi-Dirac equilibrium converges to some Gaussian equilibrium in the classical Boltzmann equation if we can ignore the $1$ in the denominator; for example, $c>>1$ or $|v|\rightarrow \infty$ cases. In fact, such a limit matches the correspondence principle in high quantum numbers. It suggests that the solution of the Boltzmann-Fermi-Dirac equation may share the properties of the solution of the classical Boltzmann equation.

Interestingly, there is an exceptional collection of equilibrium functions, which will be called \textit{saturated Fermi-Dirac distribution}. By taking $a\rightarrow \infty$ and $\frac{c}{a}\rightarrow -r^2$, we have
\begin{align*}
    f(v) = \begin{cases}
        1 & |v-u|<r,\\
        \frac{1}{2} & |v-u|=r,\\
        0 & |v-u|>r.
    \end{cases}
\end{align*}
It satisfies $S(f) = 0$ and has the lowest energy under given $\rho$ and $u$ with the constraint \eqref{eq:0f1}. This distribution can be observed in a very low-temperature, non-interacting Fermi gas. Given $\rho$, the critical temperature $T_F$, which is usually called \textit{Fermi temperature}, and the critical radius $r_F$ for the saturated Fermi-Dirac distribution is given by
\begin{align*}
    T_F = \frac{1}{2}\left(\frac{3\rho}{4\pi}\right)^{2/3},\quad r_F = \left(\frac{3\rho}{4\pi}\right)^{1/3}
\end{align*}
(see \cite{Lu2001353} p. 383 with constant normalization).

Mathematically, the saturated Fermi-Dirac distributions usually make analysis harder due to the discontinuity near the critical radius.

\subsection{History}
The quantum Boltzmann equation is a quantum modification of the Boltzmann equation for the Fermi-Dirac or Bose-Einstein statistics. It was first heuristically formulated by Nordheim \cite{nordhiem1928kinetic} and Uehling and Uhlenbeck \cite{Uehling1933552}. Since then, some progress has been made in its mathematical analysis. Here, we briefly summarize the previous works focused on the Boltzmann-Fermi-Dirac equation. In a mathematical view, it is natural to ask about well-posedness and convergence to the equilibrium of the solution. For the results around this problem, we refer to Dolbeault \cite{Dolbeault1994}, Lions \cite{10.1215/kjm/1250518932}, a series of Lu's papers \cite{Lu2001353, Lu2006, LU20081705}, and Lu-Wennberg \cite{Lu20031}. Those results employ some standard techniques from the classical Boltzmann equation, such as Banach fixed point theorem, $L^1$ convergence theorem with velocity averaging compactness argument, and moments estimate. For the near-equilibrium setting, Ouyang-Wu \cite{Ouyang2022471} summarized the property of the collision operator in the setting. We further list some recent results. Wang-Ren \cite{Wang2023_2} used an $L^1_3$ moment technique to prove the global existence and stability of the classical solution. For the large amplitude problem, we refer to Wang-Xiao-Zhang \cite{Wang2023} and Li \cite{Li2023-rh}. Ziang-Zhou \cite{Jiang2025} dealt with a general collision kernel. Bae-Jang-Yun \cite{Bae2021} studied the relativistic quantum Boltzmann equation. Finally, Brosoni and Lods \cite{borsoni2024Cercignani} analyzed the convergence speed of the solution of the Boltzmann-Fermi-Dirac solution to the Fermi-Dirac equilibrium.

Derivation of the quantum Boltzmann equation from a model describing $N$-body quantum systems is a fundamental problem for validating the equation. We simply refer to Benedetto-Castella-Esposito-Pulvirenti \cite{Benedetto200755}, Colangeli-Pezzotti-Pulvirenti \cite{Colangeli2015}, and references therein. Moreover, we can consider some interesting limits such as $\hbar\rightarrow 0$, which is a limit from quantum mechanics to classical mechanics according to the correspondence principle, and a hydrodynamic limit as in the classical Boltzmann equation. For the readers interested, we refer to Dolbeault \cite{Dolbeault1994} and He-Lu-Pulvirenti \cite{He2024} for classical limit results and Jiang-Xiong-Zhou \cite{Jiang202277}, Ziang-Zhou \cite{Jiang2024}, and Jiang-Wang-Zhou \cite{jiang2025_2} for hydrodynamic limit results.

Although it is not our focus in this paper, the Boltzmann-Bose-Einstein equation has many interesting properties. One can also consider problems such as well-posedness problems or derivation and convergence problems, as in the Boltzmann-Fermi-Dirac equation. Furthermore, one can construct a solution that blows up and formulates a Dirac-delta distribution in finite time, which corresponds to the Bose-Einstein condensation in the equation. We mention \cite{Ouyang2022471, Jiang2025} to refer to the references with explanations therein for those who are interested in this equation.

Now, we turn to some classical Gaussian lower and upper bound results. The Gaussian lower bound problem is a problem asking whether there is an instantaneous vacuum filling and a Gaussian tail in the velocity space for arbitrary initial data. For the spatially homogeneous classical Boltzmann equation with the cutoff setting, Calerman proved an exponential type lower bound $f(t,v)\geq C_1 e^{-C_2 |v|^{2+\epsilon}}$ for arbitrarily small $\epsilon>0$ for $t>0$ in 1933 \cite{Carleman1933}. In 1997, it was improved to be a Gaussian lower bound by Pulvirenti and Wennberg \cite{P1997}; they effectively employed spreading and regularity properties of the gain operator of the Boltzmann equation to get the Gaussian lower bound result. For the spatially inhomogeneous case with a cutoff kernel, Mouhot \cite{Mouhot01052005} constructed a Gaussian lower bound in the torus, and Briant \cite{Briant2015_1, Briant2015_2} extended the result to domains with specular reflection or diffusive boundary conditions. The two authors also constructed an exponential lower bound in a non-cutoff collision kernel, but it was far from the Gaussian function. Finally, using some elliptic PDE arguments, Imbert, Mouhot, and Silvestre \cite{Imbert2020} proved the Gaussian lower bound for the spatially inhomogeneous and non-cutoff Boltzmann equation, assuming the local mass density is bounded both above and away from vacuum, and the local energy and entropy densities are bounded above. This result was later extended by \cite{Henderson2020}, removing the lower bound on the mass density and the upper bound on the entropy density. For another classical model, An and Lee constructed an exponential lower bound in the homogeneous inelastic Boltzmann equation in \cite{AL2023}. 

There are a few Gaussian lower bound results in the quantum Boltzmann equation. In \cite{Nguyen2019}, Nguyen and Tran constructed a Gaussian lower bound for the quantum Boltzmann equation describing the interaction between excited particles and particles in the Bose-Einstein condensation state. Recently, Borsoni \cite{Borsoni2024} constructed a Gaussian lower bound in the Boltzmann-Fermi-Dirac equation under the condition that $\hbar$ is small enough.

Next, we discuss previous works around the upper bound problem. There have been many works about the $L^1$ upper bound problem. We first consider the angular cutoff case. Under this setting, Desvillettes \cite{D1993} established the creation of $L^1$ polynomial moments under the assumption that the initial $(1+|v|^s)f_0(v)\in L^1$ for some $s>2$. After the work, there were several works extending and refining the $L^1$ moments bound; we refer to \cite{W1997, MISCHLER1999467, Lu1999}. Based on the $L^1$ polynomial estimates, Bobylev \cite{B1997} proved the $L^1$ Gaussian moment propagation. Later, Alonso-Ca{\~n}izo-Gamba-Mouhot \cite{ACGM2013} constructed $L^1$ exponential moments propagation and creation using a simple technique. For other works about $L^1$ exponential moments results, we refer to \cite{Mou2006, LU20123305}. There are parallel results in the angular non-cutoff settings; for example, the $L^1$ moments results in \cite{LU20123305} in fact includes some non-cutoff cases. We quote some recent literature \cite{TAGP2018, Fournier2021, CHJ2024} for the readers who are interested in.

There is relatively little literature dealing with the $L^\infty$ upper bound problem. It is mainly because it is hard to employ good techniques like the Povzner inequality. For the $L^\infty$ polynomial moments side, Carleman \cite{T1957} first proved that $L^\infty$ polynomial moments bounds propagate in time for the homogeneous Boltzmann equation with a cut-off collision kernel under the radially symmetric assumption $f = f(t,|v|)$. Arkeryd \cite{L1983} later extended this result to general hard potentials $0<\gamma\leq 1$. In the spatially inhomogeneous non-cutoff case, Imbert-Mouhot-Silvestre \cite{IMS2020} established polynomial moments $L^\infty$ bounds for hard and moderately soft potentials, assuming the local macroscopic quantities are bounded. For more recent works, we cite \cite{Cameron2020, Henderson2025}. There are also some works about the $L^\infty$ exponential moments problem. In 2009, Gamba-Panferov-Villani \cite{GPV2009} first proved a Gaussian upper bound for the solution of the classical Boltzmann equation if it is initially bounded above by some Gaussian function using a comparison technique. Later, it was extended to the pseudo-Maxwell molecule setting in \cite{BG2017} and to the angular non-cutoff setting in \cite{GPN2019}.

In other models like the classical inelastic Boltzmann equation dealing with the $L^1$ or $L^\infty$ moments estimates, we quote \cite{Bobylev2004, Mischler2006, AJL2024}. Furthermore, there are a few upper bound results in the quantum Boltzmann equations. The $L^1$ polynomial moments problem was dealt with in \cite{Lu2001353} for the Fermi-Dirac case and \cite{Lu2004, MR3493188} for the Bosonic case.

\subsection{Main results}
We first define the weighted norms used in this paper. For a measurable function $f(v)$ on $\mathbb{R}^3$, we define 
\begin{align*}
        \|f\|_{p,s} &\coloneqq
        \begin{dcases}
        \left(\int_{\mathbb{R}^3} \left(|f(v)|(1+|v|^2)^{\frac{s}{2}}\right)^p \,dv \right)^\frac{1}{p} & 1\leq p<\infty,\\
        \esssup_{v\in\mathbb{R}^3} |f(v)|(1+|v|^2)^{\frac{s}{2}} & p=\infty
        \end{dcases}
\end{align*}
for $s\geq 0$. The corresponding weighted spaces are defined as 
\begin{align*}
    L_s^p(\mathbb{R}^3) = \left\{f :\text{ $f$ is measurable on} \; \mathbb{R}^3, \; \|f\|_{p,s}<\infty \right\}.
\end{align*}
For $s=0$, we can simplify the notation as $ \|f\|_p \coloneqq \|f\|_{p,0}$ and $ L^p(\mathbb{R}^3) \coloneqq L_0^p(\mathbb{R}^3)$.

We define the solution of the Boltzmann-Fermi-Dirac equation as follows. For $f_0\in L^1_2$ with $0\leq f_0\leq 1$, we call $f$ is a solution of the Boltzmann-Fermi-Dirac equation \eqref{eq:BT} if it satisfies the following (1)-(3):
\begin{enumerate}
    \item It satisfies $f\in C([0,\infty), L^1_2(\mathbb{R}^3))$ and $0\leq f(t,v)\leq 1$ on $[0,\infty)\times\mathbb{R}^3$.
    \item It satisfies the mild version of \eqref{eq:BT}:
    \begin{align*}
        f(t,v) = f_0(v) + \int_0^t Q_{FD}(f,f)(\tau, v)\,d\tau
    \end{align*}
    for $t\in [0,\infty)$ and $v\in\mathbb{R}^3\setminus Z$ for some null set $Z$ independent to $t$.
    \item It is a conservative solution. In other words, it satisfies
        \begin{align*}
            \int_{\mathbb{R}^3}\begin{pmatrix}
                1\\v\\v^2
            \end{pmatrix}f(t,v)\,dv = \int_{\mathbb{R}^3}\begin{pmatrix}
                1\\v\\v^2
            \end{pmatrix}f_0(v)\,dv.
        \end{align*}
\end{enumerate}

Now, we display our main results and remarks. After stating the main theorems, we compare these results with the classical results.
\begin{theorem}\label{thm:exi_uni_thm}
    Assume the collision kernel satisfies $0\leq \gamma\leq 2$ and (H1). For $f_0\in L^1_2$ with $0\leq f_0\leq 1$ on $\mathbb{R}^3$, there exists a unique conservative solution of the Boltzmann-Fermi-Dirac equation. Also, the solution satisfies the entropy identity \eqref{eq:entropy_identity}.
\end{theorem}
\begin{theorem}\label{thm:L1_2_stability}
    Assume the collision kernel satisfies $0\leq \gamma\leq 2$ and (H1). For solutions $f(t,v)$ and $g(t,v)$ of the Boltzmann-Fermi-Dirac equation, there exist constants $C_1$ and $C_2$ and a increasing function $\Phi:\mathbb{R}^+\rightarrow \mathbb{R}^+$ with $\Phi(0) = 0$ such that
    \begin{align*}
        \|f(t,v)-g(t,v)\|_{1,2}\leq C_1 \Phi(\|f_0-g_0\|_{1,2})\exp\left(C_2(t+t^{1/3})\right).
    \end{align*}
    The constants $C_1$ and $C_2$ depend on $\gamma, C_b, C_{b,2}$, $\varphi(\epsilon)$, $\|f_0\|_{1,0}, \|f_0\|_{1,2}$, and $\|g_0\|_{1,2}$, and the function $\Phi$ is given by
    \begin{align*}
        \Phi(r) \coloneqq r + r^{1/3} + r|\ln r| + \|f_0\mathbf{1}_{\{|v|\geq r^{-1/3}\}}\|_{1,2}.
    \end{align*}
\end{theorem}

\begin{theorem}\label{thm:Gaussian_lower_bound}
    We consider the collision kernel \eqref{eq:B_defi} for $0\leq \gamma\leq 2$, (H1), and (H2). Let $f$ be a solution of the Boltzmann-Fermi-Dirac equation, which is not a saturated Fermi-Dirac equilibrium.
    
    If $S(f_0)>0$, then there exist $C_1(t)>0$ and $C_2(t)<\infty$ for $t>0$ depending on $\gamma, C_b, c_b$, $f_0$ such that
    \begin{align*}
        C_1(t) e^{-C_2(t) |v|^2}\leq f(t,v)\leq 1- C_1(t) e^{-C_2(t) |v|^{2\frac{\ln 3}{\ln 2}}}.
    \end{align*}
    Also, $C_1(t)$ and $C_2(t)$ satisfy
    \begin{align*}
        \inf_{T^{-1}\leq t\leq T}C_1(t)>0,\quad \sup_{T^{-1}\leq t\leq T}C_2(t)<\infty
    \end{align*}
    for any $1\leq T<\infty$.

    If $S(f_0) = 0$, we further assume that the collision kernel satisfies (H5). Then, (1) there exists $T_0>0$ depending on $\gamma$, $b(\cos\theta)$, and $f_0$, and (2) there exist $C_1(t)>0$ and $C_2(t)<\infty$ for $t>0$ depending on $\gamma, C_b, c_b, f\left(\frac{1}{2}\min\{t/2,T_0\}, v\right)$, and $T_0$ such that
    \begin{align*}
        C_1(t) e^{-C_2(t) |v|^2}\leq f(t,v)\leq 1- C_1(t) e^{-C_2(t) |v|^{2\frac{\ln 3}{\ln 2}}}.
    \end{align*}
\end{theorem}
\begin{remark}
    In contrast to the classical Gaussian lower bound result (for example, \cite{P1997}), which is uniform if the time $t$ is not near $0$, our choice of $C_1(t)$ and $C_2(t)$ can decay as $t\rightarrow \infty$. It is because the constants $C_1(t)$ and $C_2(t)$ depend not only on the conservative macroscopic quantities but also on the explicit shape of the initial function $f_0$. We conjecture this obstruction is not due to the physical nature but a technical issue.
\end{remark}
\begin{remark}
    When $f_0$ is a saturated equilibrium, it has no Fermi-Dirac lower bound. It makes it hard to consider a function $S(f_0)=0$, but $f_0$ is not a saturated equilibrium. The second part of Theorem \ref{thm:Gaussian_lower_bound} states that we can construct a Gaussian lower bound with worse $C_1(t)$ and $C_2(t)$ than the $S(f_0)>0$ case.

    By Theorem \ref{thm:exi_uni_thm}, $f\left(\frac{1}{2}\min\{t/2,T_0\}, v\right)$ is uniquely determined if $f_0$ is fixed. For more explanation, please refer to Remark \ref{remark:main4_remark_1} and \ref{remark:main4_remark_2}.
\end{remark}

\begin{theorem} \label{thm:L1_bound}
    We consider the collision kernel \eqref{eq:B_defi} for $0 < \gamma \leq 2$, (H1). Let $f(t,v)$ be a solution of the Boltzmann-Fermi-Dirac equation.
    
    \noindent \textbf{(1)} (Creation and propagation of $L^1$ polynomial moments) There exist constants $C_{1,s}$ for all $s\geq 2$ depending on $\|f_0\|_{1,0}, \|f_0\|_{1,2}, \gamma, s, C_b, C_{b,2}$, and $\varphi(\epsilon)$ such that
    \begin{align*}
        \int_{\mathbb{R}^3} f(t,v)|v|^s\,dv\leq C_{1,s} \max\left\{t^{\frac{2-s}{\gamma}}, 1\right\} \text{\quad for \quad} t > 0.
    \end{align*}
    If $\|f_0\|_{1,s}<\infty$ for some $s>2$, then there exists constant $C_{2,s}$ depending on $\|f_0\|_{1,0}, \|f_0\|_{1,2}, \|f_0\|_{1,s}, \gamma, s, C_b, C_{b,2}$, and $\varphi(\epsilon)$ such that
    \begin{align*}
        \int_{\mathbb{R}^3} f(t,v)|v|^s\,dv\leq C_{2,s} \text{\quad for \quad} t \geq 0.
    \end{align*}
    
    \noindent \textbf{(2)} (Creation and propagation of $L^1$ exponential moments) There exist constants $C_1, a>0$ depending on $\|f_0\|_{1,0}, \|f_0\|_{1,2}, \gamma$, and $b(\cos\theta)$ such that
    \begin{align*}
        \int_{\mathbb{R}^3} f(t,v) e^{a \min{\{t,1\}}|v|^\gamma} dv \leq C_1 \text{\quad for \quad} t \geq 0.
    \end{align*}
    If we further assume
    \begin{align*}
        \int_{\mathbb{R}^3} f_0(v) e^{a_0|v|^{s}} dv \leq C_2
    \end{align*}
    for some $s \in [\gamma,2]$ and $C_2>0$, then there exist constants $C_3, a > 0$ depending on $\|f_0\|_{1,0}, \|f_0\|_{1,2}, \gamma, b(\cos\theta), a_0$, and $C_2$ such that 
    \begin{align*}
        \int_{\mathbb{R}^3} f(t,v) e^{a |v|^{s}} dv \leq C_3 \text{\quad for \quad} t \geq 0.
    \end{align*}
    
    \noindent \textbf{(3)} (Propagation of a Gaussian upper bound) Further assume (H3) on collision kernel and let $f_0(v) \leq M_0(v):=e^{-a_0|v|^2+c_0}$ for almost every $v \in \mathbb{R}^3$, where $a_0>0, \; c_0 \in \mathbb{R}$. Then, there exist $a \in (0,a_0)$ and $ \;c \in \mathbb{R}$ depending on $\|f_0\|_{1,0}, \|f_0\|_{1,2}, \gamma, \alpha, C_b, a_0$, and $c_0$ such that
    \begin{align*} 
        f(t,v) \leq M(v)\coloneqq e^{-a|v|^2+c}
    \end{align*}
    for almost every $v \in \mathbb{R}^3$ and every $t \geq 0$.
    
    \noindent \textbf{(4)} (Propagation of $L^\infty$ weighted bound) Assume $0<\gamma\leq 1$ and (H4) on collision kernel. Suppose $f_0\in L^1_2\cap L^\infty_s$ for some $s>2$. If $s\leq 5$, set $s' = s$; otherwise, choose any $s'<s$. Then, there exists a constant $C_4(s')>0$ depending on $\|f_0\|_{\infty,s}, \|f_0\|_{1,0}, \|f_0\|_{1,2}, \gamma, b(\cos\theta)$, $s$, and $s'$ such that 
    \begin{align*}
        \esssup_{v \in \mathbb{R}^3} (1+|v|)^{s'} f(t,v) \leq C_4(s')
        \text{\quad for \quad} t \geq 0.
    \end{align*}
\end{theorem}

\begin{remark}
    For Theorem \ref{thm:exi_uni_thm} and \ref{thm:L1_bound}-(1), we refer to the \cite{Lu2001353} for the same results for the case $0\leq\gamma\leq 1$. Also, we refer to the \cite{Lu20031} for $L^1_2$ stability of the solution for the case $0\leq\gamma\leq 1$. When $\gamma = 0$, it is simpler; one can check it at Proposition \ref{prop:stability}.
\end{remark}

Our main results extend the classical results to the Fermi-Dirac case. In detail, Theorem \ref{thm:exi_uni_thm} corresponds to the existence and uniqueness result in \cite{MISCHLER1999467}, Theorem \ref{thm:L1_2_stability} corresponds to the $L^1_2$ stability result in \cite{LU20123305}, Theorem \ref{thm:Gaussian_lower_bound} corresponds to the Gaussian lower bound result in \cite{P1997}, and Theorem \ref{thm:L1_bound} corresponds to the results in \cite{LU20123305, ACGM2013, GPV2009, L1983}.\\

The paper proceeds as follows. In Section \ref{sec:Preliminary}, we present the preliminaries, introducing the basic properties of the collision operator and the relation between the velocity variables. In Section \ref{sec:Lemmas}, we derive some technical lemmas that will be used in Section \ref{sec:Positivity}. We construct a Gaussian lower bound in Section \ref{sec:Positivity} and \ref{sec:Gaussian_lower}.

After proving the Gaussian lower bound, we turn to the upper bound problem. In Section \ref{sec:L1_estimate}, we prove the creation and propagation of $L^1$ polynomial and exponential moments. Using the creation and propagation of $L^1$ polynomial moments, we prove the existence, uniqueness, and $L^1_2$ stability result in Section \ref{sec:Well-posedness}. We prove $L^\infty$ Gaussian upper bound in Section \ref{sec:L^infty_Gaussian_upper} and $L^\infty$ polynomial moments estimate in Section \ref{sec:L^infty_polynomial_upper}.

\subsection{Notation}
We enumerate some notations used in this paper.
\begin{itemize}
    \item The indicator function of a subset $S$ within a set $X$ is a function $\mathbf{1}_{S} : X \rightarrow \{0,1\}$, defined as
\begin{align*} 
    \mathbf{1}_{S}(x) \coloneqq
    \begin{cases}
        1,\quad x \in S, \\
        0,\quad x \notin S
    \end{cases} 
\end{align*}
for $x \in X$.

\item We introduce the usual notation $x\wedge y = \max\{x, y\}$ and $f^+ \coloneqq \max\{f, 0\}$.

\item In Section \ref{sec:Positivity}, we use some geometric notations. We denote $\mathcal{B}_R(x_0)$ by the compact ball with diameter $R$ with center $x_0$, $\mathcal{Q}_R(x_0)$ by the compact cube with side length $R$ with center $x_0$ having an axis parallel with the Cartesian coordinates, and $S_{x_1,x_2}$ by the sphere shell having two antipodal points $x_1$ and $x_2$. Abusing notation, we will use $|E|$ to denote the standard Borel measure on $\mathbb{R}^3$ for a Borel set $E$. Finally, we call a Borel-measurable $E$ is $(\epsilon, r)$-measurable for $0\leq \epsilon\leq 1$ and $r>0$ if there exists a ball $B_r(x_0)$ such that
\begin{align*}
    |E\cap B_r(x_0)|\geq \epsilon|B_r(x_0)|.
\end{align*}
This notation is borrowed from an article by Tao \cite{TTao2007}.
\end{itemize}

\section{Preliminary}\label{sec:Preliminary}
In this section, we will briefly review the basic properties of the Boltzmann equation. For detailed proof and computations, we refer to the well-known review paper \cite{V2002}.

Before starting, we write the classical collision operator $Q_c(f,f)$ by
\begin{align*}
    Q_c(f_1,f_2)(t,v)&\coloneqq \int_{\mathbb{R}^3  \times \mathbb{S}^2}B(v-v_*, \sigma) \left(f_1(t,v')f_2(t,v_*') -f_1(t,v)f_2(t,v_*)\right)\, d\sigma dv_*.
\end{align*}
If $b(\cos\theta)$ is integrable, then we can split the $Q_c$ operator by gain and loss operators $Q_c^+$ and $Q_c^-$ as follows.
\begin{align*}
    Q_c^+(f_1,f_2)(t,v)&\coloneqq \int_{\mathbb{R}^3  \times \mathbb{S}^2}B(v-v_*, \sigma)f_1(t,v')f_2(t,v_*')\,d\sigma dv_*,\\
    Q_c^-(f_1,f_2)(t,v)&\coloneqq \int_{\mathbb{R}^3  \times \mathbb{S}^2}B(v-v_*, \sigma)f_1(t,v)f_2(t,v_*)\,d\sigma dv_*.
\end{align*}

\subsection{The relationship between the variables \texorpdfstring{$v',v_*',v_*,$}{} and  \texorpdfstring{$v$}{}}
The collision velocities satisfy some special relations thanks to geometric properties of the elastic collision. Since the elastic collision is a time-reversible process, we can reverse the order between pre-collision velocity and post-collision velocity in the collision. As a result, we can obtain the well-known symmetry
\begin{align} \label{eq:bef_aft}
    \begin{split}
        &\int_{\mathbb{R}^3 \times \mathbb{R}^3\times\mathbb{S}^2} B(|v-v_*|, \cos\theta)  F(v,v_*,v',v_*') \,dvdv_* d\sigma\\
        &=\int_{\mathbb{R}^3 \times \mathbb{R}^3\times\mathbb{S}^2} B(|v-v_*|, \cos\theta) F(v',v_*',v,v_*) \,dvdv_* d \sigma
    \end{split}
\end{align}
for any non-negative measurable function $F$. Also, we can interchange $v'$ and $v'_*$ using the mapping $\sigma\mapsto -\sigma$ \eqref{eq:def_v'}, so we have
\begin{align*}
    \int_{\mathbb{S}^2} B(|v-v_*|, \cos \theta) f_1(v')f_2(v_*')\, d\sigma =\int_{\mathbb{S}^2} B(|v-v_*|, \cos (\pi-\theta)) f_1(v_*')f_2(v')\, d\sigma.
\end{align*}
Using this symmetry, we obtain
\begin{align} \label{eq:sym_sigma}
    \begin{split}
        Q_c^{+}(f,f)(v)&=2\int_{\mathbb{R}^3  \times \mathbb{S}^2}
        B(|v-v_*|, \cos \theta)  f(v')f(v_*') \mathbf{1}_{\{0\leq\theta \leq \frac{\pi}{2}\}}\, d\sigma dv_*\\
        &=2\int_{\mathbb{R}^3  \times \mathbb{S}^2}
        B(|v-v_*|, \cos \theta)  f(v')f(v_*') \mathbf{1}_{\{\frac{\pi}{2}\leq\theta \leq \pi\}}\, d\sigma dv_*.
    \end{split}
\end{align}

The collision velocities enjoy more interesting identities. For example, we have
\begin{align}\label{eq:sin}
    \begin{split}
        \frac{v-v_*'}{|v-v_*'|}\cdot \sigma&=\cos \frac{\theta}{2} = \frac{|v'-v_*|}{|v-v_*|}=\frac{|v-v_*'|}{|v-v_*|},\\
        \frac{v_*-v_*'}{|v_*-v_*'|}\cdot \sigma&=\sin \frac{\theta}{2} = \frac{|v_*-v'_*|}{|v-v_*|} =\frac{|v-v'|}{|v-v_*|}.
    \end{split}
\end{align}
One can directly check these identities from the definition \eqref{eq:def_v'} or derive from the geometric relations in Figure \ref{fig:sigma_omega_diagram}.

Using these relations between the variables, we can prove the integral equalities that appear in the proof of the cancellation lemma in \cite{AD2000}. It is written by
\begin{align}\label{eq:cancel}
    \begin{split}
        &\int_{\mathbb{R}^3 \times\mathbb{S}^2}
        B(|v-v_*|, \cos \theta) f(v_*')\mathbf{1}_{\{ 0 \leq \theta \leq \theta_0\}}\, dv_* d\sigma=|\mathbb{S}^1|\int_{\mathbb{R}^3} f(v)\int_0^{\theta_0}
        \frac{\sin \theta}{\cos^3 \frac{\theta}{2}}
        B\left(\frac{|v-v_*|}{\cos \frac{\theta}{2}}, \cos\theta \right) \,d\theta dv,\\
        &\int_{\mathbb{R}^3 \times\mathbb{S}^2}
        B(|v-v_*|, \cos \theta) f(v')\mathbf{1}_{\{\theta_0\leq \theta \leq \pi\}}\, dv_* d\sigma=|\mathbb{S}^1|\int_{\mathbb{R}^3} f(v)\int_{\theta_0}^{\pi}
        \frac{\sin \theta}{\sin^3 \frac{\theta}{2}}
        B\left(\frac{|v-v_*|}{\sin \frac{\theta}{2}}, \cos\theta \right) \,d\theta dv
    \end{split}
\end{align}
for $0<\theta_0<\pi$. For proof, one can refer to Lemma 1 in \cite{AD2000} or Proposition 2.1 in \cite{LU20081705}.

\begin{figure}[htbp]
    \centering
    \begin{tikzpicture}
    \def\R{2.5}
    \def\myangle{60}
        \draw[thick, ->] (-3.5,0) -- (3.5,0);
        \draw[thick](0,0) circle (\R);
        \draw[thick] (\R,0) -- ({\R * cos(\myangle)},{\R * sin(\myangle)}) node[pos = 0.4, left] {$\vec{\omega}$};
        \draw[thick, ->] (\R,0) -- ({0.4*\R * cos(\myangle) + 0.6*\R},{0.4*\R * sin(\myangle)});
        \draw[thick] (-{\R * cos(\myangle)},-{\R * sin(\myangle)}) -- ({\R * cos(\myangle)},{\R * sin(\myangle)}) node[pos=0.7, right] {$\vec{\sigma}$};
        \draw[thick, ->] (0, 0) -- ({0.4*\R * cos(\myangle)},{0.4*\R * sin(\myangle)});
        \draw[] (-\R, 0) node[above left]{$v_*$};
        \draw[] (\R, 0) node[above right]{$v$};
        \draw[] (-{\R * cos(\myangle)},-{\R * sin(\myangle)}) node[left, xshift = -0.1cm]{$v_*'$};
        \draw[] ({\R * cos(\myangle)},{\R * sin(\myangle)}) node[right, yshift = 0.1cm]{$v'$};
        \draw[thick] (0,0) --  (0:0.5) arc(0:\myangle:0.5) -- cycle;
        \draw[thick] (\R,0) --  (\R-0.5, 0) arc(0:{-\myangle}:-0.5) -- cycle;
        \draw[] (0:0.5) node[above, xshift = 0.1cm]{$\theta$};
        \draw[] (\R-0.5, 0) node[above, xshift = -0.25cm, yshift = -0.05cm]{$\theta_\omega$};
    \end{tikzpicture}
    \caption{The collision diagram with pre- and post-collision velocities.}
    \label{fig:sigma_omega_diagram}
\end{figure}
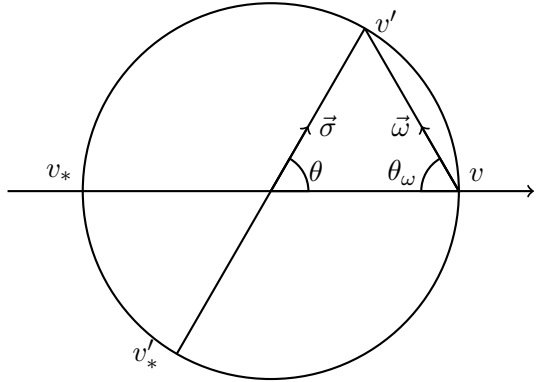

\subsection{\texorpdfstring{$\omega$}{}-representation}
In the definition of the post-collision velocities \eqref{eq:def_v'}, we used the pre-collision velocities and the variable $\sigma$ to represent the post-collision velocities. This is called $\sigma$-representation, and there is another widely used representation: $\omega$-representation. It write $v'$ and $v'_*$ by
\begin{align}\label{eq:post_vel_ome}
    v' = v + ((v_* - v) \cdot \omega)\omega,\quad v_*' = v_* - ((v_* - v) \cdot \omega)\omega
\end{align}
using $\omega \in \mathbb{S}^2_+\coloneqq \{\omega\in\mathbb{S}^2:(v_* - v)\cdot \omega\geq 0\}$.
Likewise \eqref{eq:theta_sigma}, $\theta_{\omega}$ is defined as
\begin{align*}
    \cos \theta_\omega \coloneqq \frac{v_*-v}{|v_*-v|} \cdot \omega \quad \text{for } \quad\theta_{\omega} \in \left[ 0,\frac{\pi}{2} \right].
\end{align*}
The relation between $\sigma$-representation and $\omega$-representation is graphically described in Figure \ref{fig:sigma_omega_diagram}. From this diagram, we can easily see the variables $\theta$ and $\theta_\omega$ satisfy the relation $\theta=\pi-2\theta_{\omega}$. Using the spherical coordinates for $\sigma$ (resp. $\omega)$ with the variable $\theta$ (resp. $\theta_\omega$), we compute the Jacobian between $\sigma$ and $\omega$ as
\begin{align} \label{eq:det_w,n}
    \left|\frac{d\sigma}{d \omega} \right|= 4\cos \theta_\omega=4 \sin\frac{\theta}{2}
\end{align}
We can rewrite $b(\cos\theta)$ by $h(\cos\theta_\omega)$ using $\omega$-representation and the Jacobian:
\begin{align*}
    h(\cos\theta_\omega) = 4\cos \theta_\omega b(\cos(\pi-2\theta_\omega)).
\end{align*}

It is usually convenient to extend the $\omega$ domain from $\mathbb{S}^2_+$ to $\mathbb{S}^2$; from \eqref{eq:post_vel_ome}, we easily see that $v'$ and $v'_*$ are invariant under $\omega\mapsto -\omega$ in \eqref{eq:det_w,n}. It suggests to extend $h$ on $\theta_\omega\in [\pi/2, \pi]$ by
\begin{equation*}
    h(\cos\theta_\omega) = h(\cos(\pi-\theta_\omega)).
\end{equation*}
By this extension, the integral over $\theta_\omega$ on $\mathbb{S}^2$ is doubled compared to $\theta_\sigma$. To make the computation clear, we will divide $h$ by $2$ to compensate for this doubling and redefine it as $h$. Including this compensation, we finally write
\begin{align}\label{eq:h_condi}
    h(\cos\theta_\omega) = 2\cos \theta_\omega b(\cos(\pi-2\theta_\omega)).
\end{align}
One can ignore these constants since it does not essentially change the results.

Under these settings, we rewrite the collision operators $Q_c$ and $Q_{FD}$ in $\omega$-representation by
\begin{align}\label{eq:def_Q_omega}
    \begin{split}
        Q_c(f_1,f_2)(v) &\coloneqq \int_{\mathbb{R}^3\times\mathbb{S}^2} |v-v_*|^\gamma h(\cos\theta_\omega)\left( f_1(v')f_2(v_*')-f_1(v)f_2(v_*)\right)\,d\omega dv_*,\\
        Q_{FD}(f,f)(v)
        &\coloneqq \int_{\mathbb{R}^3 \times \mathbb{S}^2}|v-v_*|^\gamma h(\cos\theta_{\omega})\big( f(v')f(v_*')(1-f(v))(1-f(v_*)) \\
        &\quad \quad -f(v)f(v_*)(1-f(v'))(1-f(v_*')\big)\,d\omega dv_*.
    \end{split} 
\end{align}
Also, we rephrase the assumption \eqref{eq:b_lower_naer_pi/2} for $h$ by
\begin{align}\label{eq:b cut2'}
    h(\cos \theta_\omega)\geq \sqrt{2-\sqrt{2}}c_b
\end{align}
for $\theta_\omega\in [\pi/8, 3\pi/8]\cup [5\pi/8, 7\pi/8]$.

\subsection{The Calreman representation}
One of the benefits of using $\omega$-representation is easy construction of the Calreman representation for the collision operator. In \cite{T1957}, starting from the $\omega$-representation of $Q_c$ in \eqref{eq:def_Q_omega}, Carleman found that the gain term $Q^+_c$ can be rewritten by
\begin{align} \label{eq:def_cal_Q+}
    Q_c^+(f_1,f_2)(v)=\int_{\mathbb{R}^3}f_1(v')\frac{1}{|v' - v|^{2-\gamma}}\int_{v + E_{v'-v}} \frac{h(\cos\theta_\omega)}{\cos^\gamma  \theta_\omega}f_2(v'_*)\,dv'_* dv'.
\end{align}
Here, $v+E_{v'-v}$ is the plane through $v$ and perpendicular to $v'-v$: in equation form, it is given by
\begin{align*}
    E_{v_0}\coloneqq \left\{v\in \mathbb{R}^3 : v \perp v_0 \right\},\quad v_0\in \mathbb{R}^3\setminus \{0\}.
\end{align*}
In contrast to the original gain term operator, it directly integrates the functions by the post-collision velocities, and its special structure makes it easier to use the regularity property of the gain term.

Especially in Section \ref{sec:Positivity}, we will substitute $\tilde{u}_\parallel = v'$ and $\tilde{u}_\perp = v'_*$ to emphasize the geometric structure around the collision velocities and avoid confusion.

\section{Some lemmas for Gaussian lower bound}\label{sec:Lemmas}
In this section, we state and prove some technical lemmas for Section \ref{sec:Positivity} to reduce its complexity of proof.

In contrast to the classical $Q_c$ operator, there are several ways to decompose the positive and negative terms in the operator $Q_{FD}$ due to the complicated structure. Here, we present one decomposition.
\begin{definition}
    Let $v \in \mathbb{R}^3$. We define
    \begin{align*} 
        Q_1(f_1,f_2,f_3)(v) \coloneqq \int_{\mathbb{R}^3 \times\mathbb{S}^2} B(v-v_*,\sigma) f_1(v')f_2(v_*')f_3(v_*)\,d\sigma dv_*.
    \end{align*}
    We also define
    \begin{align*}
        \overline{Q}_1(f,f,f)(v)&\coloneqq Q_1(f,f,1-f)(v)+Q_1(1-f,1-f,f)(v),\\
        \intertext{and}
        G_{t_1}^{t_2}(v)&\coloneqq e^{-\int_{t_1}^{t_2} \overline{Q}_1(f,f,f) (\tau,v)  \,d\tau}  \text{\quad for  \quad } t_2 > t_1 \geq 0
    \end{align*}
    for a function $f(t,v)$.
\end{definition}

Using $Q_1(f_1,f_2,f_3)$, we can rewrite the Boltzmann equation \eqref{eq:BT} as
\begin{align*}
    \partial_t f(t,v) =Q_{FD}(f,f)(t,v)=Q_1(f,f,1-f) (t,v) - f(t,v) \overline{Q}_1(f,f,f)(t,v).
\end{align*}
Since $-Q_{FD}(f,f)(t,v) = Q_{FD}(1-f,1-f)(t,v)$, we also write
\begin{align*}
    \partial_t (1-f(t,v)) =Q_1(1-f,1-f,f) (t,v)-(1-f(t,v))\overline{Q}_1(f,f,f)(t,v).
\end{align*} 
The Duhamel's form of the solutions $f(t,v)$ and $1-f(t,v)$ are written by
\begin{align} \label{eq:duha_f}
    f(t,v) &= f_0(v) G_0^t(v) + \int_0^t G_\tau^t (v) Q_1(f,f,1-f)(\tau,v)\,d\tau,
    \intertext{and}
    1-f(t,v) &= (1-f_0(v)) G_0^t(v) + \int_0^t G_\tau^t(v) Q_1(1-f,1-f,f)(\tau,v)\,d\tau.\label{eq:duha_1-f}
\end{align} 

\begin{lemma} \label{lem:upper_v}
    We consider the collision kernel $B$ satisfying $0\leq \gamma\leq 2$ with (H1) and assume $\|f\|_{1,2} < +\infty$ with $0\leq f\leq 1$. Then, there exists a constant $C>0$ depending on $\|f\|_{1,2}$, $\gamma$, and $C_b$ such that
    \begin{align*}  
          \overline{Q}_1(f,f,f)(v) \leq C(1+|v|^\gamma).
    \end{align*}
\end{lemma}
\begin{proof}
    Since $0 \leq f \leq 1$ and $0\leq \gamma\leq 2$, there exists a constant $C_1>0$ depending on $\|f\|_{1,2}, \gamma$, and $C_b$ such that
    \begin{align}\label{eq:upper_v_1}
        \begin{split}
            Q_1(1-f,1-f,f)(v)&= \int_{\mathbb{R}^3 \times \mathbb{S}^2} |v-v_*|^\gamma b(\cos\theta) (1-f(v'))(1-f(v_*'))f(v_*) \,d\sigma dv_*\\
            &\leq 4\pi\int_0^{\pi} b(\cos\theta)\sin\theta\,d\theta\int_{\mathbb{R}^3}  (|v|^\gamma+|v_*|^\gamma)f(v_*) \,dv_*\\
            &\leq 2C_b\|f\|_{1,2}(1+|v|^\gamma).
        \end{split}
    \end{align}
    In the middle, we used $|v-v_*|^\gamma\leq 2(|v|^\gamma + |v_*|^\gamma)$.
    
    Next, we estimate $Q_1(f,f,1-f)$. Since $0 \leq f \leq 1$, it is bounded by
    \begin{align}
        Q_1(f,f,1-f)(v) 
        &=\int_{\mathbb{R}^3 \times \mathbb{S}^2}B(|v-v_*|, \cos \theta) f(v')f(v_*')(1-f(v_*))\,d\sigma dv_*   \notag  \\
        &\leq \int_{\mathbb{R}^3 \times \mathbb{S}^2}B(|v-v_*|, \cos \theta) f(v')f(v_*')\,d\sigma dv_*. \label{eq:sym v,v_*}  
    \end{align}
    Using \eqref{eq:sym_sigma}, $0 \leq f \leq 1$, and the change of variable in \eqref{eq:cancel} in order, we obtain
    \begin{align}
        \eqref{eq:sym v,v_*} &=2 \int_{\mathbb{R}^3 \times \mathbb{S}^2}  B(|v-v_*|, \cos \theta) f(v')f(v_*') \mathbf{1}_{\left\{0 \leq \theta \leq \frac{\pi}{2}\right\}}\,d\sigma dv_* \notag \\
        &\leq 2 \int_{\mathbb{R}^3 \times \mathbb{S}^2} B(|v-v_*|, \cos \theta) f(v_*') \mathbf{1}_{\left\{0 \leq \theta \leq \frac{\pi}{2}\right\}}\,d\sigma dv_*\notag\\
        &=4\pi\int_{\mathbb{R}^3}f(v_*)\int_0^{\frac{\pi}{2}} \frac{\sin \theta}{\cos^{3+\gamma} \frac{\theta}{2}} |v-v_*|^\gamma b(\cos \theta)\,d\theta dv_*.\label{eq:Q1_temp1}
    \end{align}
    Since $\cos\frac{\theta}{2} \geq \frac{1}{\sqrt{2}}$ for $0\leq \theta\leq \frac{\pi}{2}$, the integral is bounded by
    \begin{align}
        \eqref{eq:Q1_temp1} &\leq 2^{\frac{7+\gamma}{2}}\pi\int_0^{\frac{\pi}{2}} b(\cos \theta)\sin\theta\,d\theta \int_{\mathbb{R}^3}|v-v_*|^\gamma f(v_*)\,dv_*\notag\\
        &\leq 2^{\frac{5+\gamma}{2}}C_b\int_{\mathbb{R}^3}|v-v_*|^\gamma f(v_*)\,dv_*\notag\\
        &\leq 2^{\frac{7+\gamma}{2}}C_b\|f\|_{1,2}(1+|v|^\gamma).\label{eq:Q1_temp2}
    \end{align}
    Combining \eqref{eq:upper_v_1} and \eqref{eq:sym v,v_*}-\eqref{eq:Q1_temp2}, we get the lemma.
\end{proof}

From Lemma \ref{lem:upper_v}, we get
\begin{align}\label{eq:G_upper_bound}
    G^{t_2}_{t_1}(v) \geq \exp\left(-c(t_2-t_1)(1+|v|^\gamma)\right)
\end{align}
for some constant $c>0$, which depends on $\|f_0\|_{1,2}$, $\gamma$, and $C_b$.

\begin{lemma}\label{lem:third_term_bound}
    Assume the collision kernel $B$ satisfies $0\leq \gamma\leq 2$ with (H1), and let $0\leq f_i\leq 1$ be a measurable function for $i=1,2,3$. Then, there exists a constant $C>0$ depending on $\gamma$ such that
    \begin{align*}
        &\int_{\mathbb{R}^6\times\mathbb{S}^2} B(|v-v_*|, \cos \theta)f_1(v)f_2(v_*)f_3(v'_*)\,d\sigma dv_*dv\\
        &\leq C\int_{\mathbb{R}^6\times\mathbb{S}^2} B(|v-v_*|, \cos \theta) \left(f_1(v)f_3(v_*)+f_3(v)f_2(v_*)\right)\,d\sigma dv_*dv.
    \end{align*}
    As a result, we have
    \begin{align*}
        \int_{\mathbb{R}^3} Q_1(f_1,f_2,f_3)(v)\,dv\leq C\int_{\mathbb{R}^3} Q^+_c(f_1,f_3)(v) + Q^+_c(f_3,f_2)(v)\,dv.
    \end{align*}
\end{lemma}
\begin{proof}
    Dividing the range of $\theta$ and using $0 \leq f_i \leq 1$, we have
    \begin{align}
        &\int_{\mathbb{R}^6\times\mathbb{S}^2}  B(|v-v_*|, \cos \theta)f_1(v)f_2(v_*)f_3(v'_*)\,d\sigma dv_*dv \notag\\
        &=\int_{\mathbb{R}^6\times\mathbb{S}^2}  B(|v-v_*|, \cos \theta)f_1(v)f_2(v_*)f_3(v_*')  \left(\mathbf{1}_{\{0\leq \theta \leq \frac{\pi}{2}\}} + \mathbf{1}_{\{\frac{\pi}{2}\leq \theta \leq \pi \}}\right)\,d\sigma dv_*dv \notag\\
        &\leq 
        \int_{\mathbb{R}^3} f_1(v)\int_{\mathbb{R}^3\times\mathbb{S}^2}  B(|v-v_*|, \cos \theta)f_3(v_*') \mathbf{1}_{\left\{0 \leq \theta \leq \frac{\pi}{2}\right\}}\,d\sigma dv_*dv \label{eq:f1f3} \\
        &\quad + 
        \int_{\mathbb{R}^3} f_2(v_*)\int_{\mathbb{R}^3\times\mathbb{S}^2}  B(|v-v_*|, \cos \theta)f_3(v_*') \mathbf{1}_{\{\frac{\pi}{2}\leq \theta \leq \pi\}}\,d\sigma dv_*dv.\label{eq:f2f3}
    \end{align} 
    
    First, we estimate \eqref{eq:f1f3}. As in \eqref{eq:Q1_temp1}, we can bound it by
    \begin{align*}
        \int_{\mathbb{R}^3\times\mathbb{S}^2} B(|v-v_*|, \cos \theta)f_3(v_*') \mathbf{1}_{\{0\leq \theta \leq \frac{\pi}{2} \} }  \,dv_* d\sigma
        \leq C\int_{\mathbb{R}^3\times \mathbb{S}^2} B(|v-v_*|, \cos \theta) f_3(v_*)  \,dv_* d\sigma
    \end{align*}
    for some constant $C$ depending on $\gamma$. Therefore,
    \begin{align}\label{eq:3.3_res1}
        \begin{split}
            \eqref{eq:f1f3} &\leq C\int_{\mathbb{R}^3\times\mathbb{R}^3\times\mathbb{S}^2} B(|v-v_*|,\cos\theta) f_1(v)f_3(v_*) \,dvdv_*d\sigma\\
            &=C\int Q_c^+(f_1,f_3)(v) \,dv.
        \end{split}
    \end{align}

    For \eqref{eq:f2f3}, using the variable interchange $v\leftrightarrow v_*$ with Fubini's theorem, we have
    \begin{align*}
        &\int_{\mathbb{R}^3} f_2(v_*)\int_{\mathbb{R}^3\times\mathbb{S}^2}  B(|v-v_*|, \cos \theta)f_3(v_*') \mathbf{1}_{\{\frac{\pi}{2}\leq \theta \leq \pi\}}\,d\sigma dv_*dv\\
        &=\int_{\mathbb{R}^3} f_2(v_*)\int_{\mathbb{R}^3\times\mathbb{S}^2}  B(|v-v_*|, \cos \theta)f_3(v_*') \mathbf{1}_{\left\{\frac{v-v_*}{|v-v_*|}\cdot \sigma\leq 0\right\}}\,d\sigma dv_*dv \\
        &=\int_{\mathbb{R}^3} f_2(v)\int_{\mathbb{R}^3\times\mathbb{S}^2}  B(|v-v_*|, \cos \theta)f_3(v') \mathbf{1}_{\left\{\frac{v-v_*}{|v-v_*|}\cdot \sigma\geq 0\right\}}\,d\sigma dvdv_*\\
        &=\int_{\mathbb{R}^3} f_2(v)\int_{\mathbb{R}^3\times\mathbb{S}^2}  B(|v-v_*|, \cos \theta)f_3(v') \mathbf{1}_{\{0\leq \theta\leq \frac{\pi}{2}\}}\,d\sigma dv_*dv.
    \end{align*}
    Therefore, we can bound it in the same way as \eqref{eq:f1f3} and get
    \begin{align}\label{eq:3.3_res2}
        \eqref{eq:f2f3}\leq C\int Q_c^+(f_2,f_3)(v) \,dv.
    \end{align}
    Combining \eqref{eq:3.3_res1} and \eqref{eq:3.3_res2}, we get the lemma.
\end{proof}

The next lemma is designed to approximate $Q_1(f_1,Q_1(f_2, f_3, f_4),f_5)$ by a limit of approximate functions $f_{i,\delta}$ which are a continuous approximation of $f_i$ with error $\delta$. One can easily derive this lemma using Lebesgue's dominated convergence theorem. In the next proof, we present an alternative proof giving quantitative convergence speed in $L^1$.
\begin{lemma}\label{lem:Q_1_approx}
    Consider the collision kernel $B$ satisfying $0\leq \gamma\leq 2$ and (H1). Let $f_{i,j}\rightarrow f_i$ in $L^1$ as $j\rightarrow \infty$ for $i=1,2,3,4,5$, $0\leq f_{i,j}\leq 1$ for all $i$ and $j$, and $\cup_{i,j}\supp f_{i,j}$ is a bounded set. Then, there exists a subsequence $f_{i,j_n}$ for all $i$ such that
    \begin{align*}
        \lim_{n\rightarrow \infty} Q_1(f_{1,j_n},Q_1(f_{2,j_n}, f_{3,j_n}, f_{4,j_n}),f_{5,j_n})(v)\rightarrow Q_1(f_1,Q_1(f_2, f_3, f_4),f_5)(v)
    \end{align*}
    a.e. $v$.
\end{lemma}
\begin{proof}
    We will prove
    \begin{align*}
        \lim_{j\rightarrow \infty} \int_{\mathbb{R}^3} \left|Q_1(f_{1,j},Q_1(f_{2,j}, f_{3,j}, f_{4,j}),f_{5,j})(v) - Q_1(f_1,Q_1(f_2, f_3, f_4),f_5)(v)\right|\,dv = 0.
    \end{align*}
    If we show it, we can take a subsequence of $Q_1(f_{1,j},Q_1(f_{2,j}, f_{3,j}, f_{4,j}),f_{5,j})$ such that it converges to $Q_1(f_1,Q_1(f_2, f_3, f_4),f_5)$ a.e. $v$.

    We decompose the difference as follows:
    \begin{align*}
        &\left|Q_1(f_{1,j},Q_1(f_{2,j}, f_{3,j}, f_{4,j}),f_{5,j})(v) - Q_1(f_1,Q_1(f_2, f_3, f_4),f_5)(v)\right|\\
        &\leq \left|Q_1(f_{1,j} - f_1,Q_1(f_{2,j}, f_{3,j}, f_{4,j}), f_{5,j})(v)\right| + \left|Q_1(f_1,Q_1(f_{2,j}-f_2, f_{3,j}, f_{4,j}), f_{5,j})(v)\right| \\
        &\quad + \left|Q_1(f_1,Q_1(f_2, f_{3,j}-f_3, f_{4,j}), f_{5,j})(v)\right| + \left|Q_1(f_1,Q_1(f_2, f_3, f_{4,j}-f_4), f_{5,j})(v)\right|\\
        &\quad + \left|Q_1(f_1,Q_1(f_2, f_3, f_4), f_{5,j}-f_5)(v)\right|\\
        &\leq Q_1(|f_{1,j} - f_1|,Q_1(f_{2,j}, f_{3,j}, f_{4,j}), f_{5,j})(v) + Q_1(f_1,Q_1(|f_{2,j}-f_2|, f_{3,j}, f_{4,j}), f_{5,j})(v)\\
        &\quad + Q_1(f_1,Q_1(f_2, |f_{3,j}-f_3|, f_{4,j}), f_{5,j})(v) + Q_1(f_1,Q_1(f_2, f_3, |f_{4,j}-f_4|), f_{5,j})(v)\\
        &\quad + Q_1(f_1,Q_1(f_2, f_3, f_4), |f_{5,j}-f_5|)(v)\\
        &\eqqcolon I_1+I_2+I_3+I_4+I_5.
    \end{align*}
    We bound each $I_i$.
    
    Since $0\leq f^4_j, f^5_j\leq 1$, we get
    \begin{equation}\label{eq:lem3_I123}
        \begin{alignedat}{3}
            I_1&\leq Q_c^+(|f_{1,j} - f_1|,Q_c^+(f_{2,j}, f_{3,j}))(v),\quad &&I_2\leq Q_c^+(f_1,Q_c^+(|f_{2,j}-f_2|, f_{3,j}))(v),\\
            I_3&\leq Q_c^+(f_1,Q_c^+(f_2, |f_{3,j}-f_3|))(v), &&I_4\leq Q_c^+(f_1,Q_1(f_2, f_3, |f_{4,j}-f_4|))(v),\\
            I_5&\leq Q_1(f_1,Q_c^+(f_2, f_3), |f_{5,j}-f_5|)(v).
        \end{alignedat}
    \end{equation}
    For the $I_4$, we use Lemma \ref{lem:third_term_bound} as follows.
    \begin{align}\label{eq:lem3_I4}
        \begin{split}
            \int_{\mathbb{R}^3} I_4\,dv&=\int_{\mathbb{R}^3} Q_c^+(f_1,Q_1(f_2, f_3, |f_{4,j}-f_4|))(v)\,dv \\
            &=\int_{\mathbb{R}^6\times\mathbb{S}^2} B(v-v_*,\sigma) f_1(v)Q_1(f_2, f_3, |f_{4,j}-f_4|)(v_*)\,dv_* dv d\sigma\\
            &\leq C\int_{\mathbb{R}^6\times\mathbb{S}^2} B(v-v_*,\sigma) f_1(v)\left(Q_c^+(f_2, |f_{4,j}-f_4|) + Q_c^+(|f_{4,j}-f_4|, f_3)\right)(v_*)\,dv_*dvd\sigma\\
            &=C\int_{\mathbb{R}^3} Q_c^+(f_1,Q_c^+(f_2, |f_{4,j}-f_4|))(v) + Q_c^+(f_1,Q_c^+(|f_{4,j}-f_4|, f_3))(v)\,dv.
        \end{split}
    \end{align}
    Similarly for the $I_5$, we get
    \begin{align}\label{eq:lem3_I5}
        \begin{split}
            \int_{\mathbb{R}^3} I_5\,dv&=\int_{\mathbb{R}^3} Q_1(f_1,Q_c^+(f_2, f_3), |f_{5,j}-f_5|)(v)\,dv\\
            &\leq C\int_{\mathbb{R}^3} Q_c^+(f_1, |f_{5,j}-f_5|)(v) + Q_c^+(|f_{5,j}-f_5|, Q_c^+(f_2, f_3))(v)\,dv.
        \end{split}
    \end{align}
    It shows that all the $I_i$ can be decomposed into the iterated $Q_c^+$.
    
    The classical $Q^+_c$ satisfies
    \begin{align*}
        \int_{\mathbb{R}^3} Q_c^+(g_1,Q_c^+(g_2, g_3))(v)\,dv &= \int_{\mathbb{R}^6\times\mathbb{S}^2} B(v-v_*,\sigma_1) g_1(v')Q_c^+(g_2, g_3)(v'_*)\,dvdv_*d\sigma_1\\
        &=\int_{\mathbb{R}^6\times\mathbb{S}^2} B(v-v_*,\sigma_1) g_1(v)Q_c^+(g_2, g_3)(v_*)\,dvdv_*d\sigma_1\\
        &=\int_{\mathbb{R}^9\times\mathbb{S}^2\times\mathbb{S}^2} B(v-w',\sigma_1)B(w-w_*,\sigma_2) g_1(v)g_2(w)g_3(w_*)\,dvdwdw_*d\sigma_1 d\sigma_2.
    \end{align*}
    In the final line, we set $w = v_*$. Since $\cup_{i,j} \supp f_{i,j}$ is a bounded set, $|v-w'|$ and $|w-w_*|$ is bounded in the integral. So, we can bound the collision kernel by
    \begin{align*}
        B(|v-v_*|,\cos\theta)\leq C b(\cos\theta),
    \end{align*}
    and
    \begin{align*}
        &\int_{\mathbb{R}^3} Q_c^+(g_1,Q_c^+(g_2, g_3))(v)\,dv\\
        &\leq C(2\pi)^2\int_{\mathbb{R}^9\times \mathbb{S}^2\times\mathbb{S}^2}\int_0^{\pi}\int_0^{\pi} b(\cos\theta_1) b(\cos\theta_2) g_1(v)g_2(w)g_3(w_*)\sin\theta_1\sin\theta_2\,d\sigma_1 d\sigma_2 dvdwdw_*\\
        &\leq CC_b^2\int_{\mathbb{R}^9}g_1(v)g_2(w)g_3(w_*)dvdwdw_*.
    \end{align*}
    Therefore, we get
    \begin{align*}
        \int_{\mathbb{R}^3} Q^+_c(g_1,Q^+_c(g_2, g_3))(v)\,dv\leq C\prod_{i=1}^3 \|g_i\|_1
    \end{align*}
    for some constant $C$ depending on $\gamma, C_b$, and the diameter of the set $\cup_{i,j}\supp f_{i,j}$. Combining this bound with \eqref{eq:lem3_I123}, \eqref{eq:lem3_I4}, and \eqref{eq:lem3_I5}, we get
    \begin{align*}
        \lim_{j\rightarrow \infty}\int_{\mathbb{R}^3} \sum_{i=1}^5 I_i\,dv &= C\lim_{j\rightarrow \infty}\sum_{i=1}^5 \|f_{i,j}-f_i\|_1 = 0,
    \end{align*}
    where $C$ depends on $\gamma, C_b$, and the set $\cup_{i,j}\supp f_{i,j}$. It proves the lemma.
\end{proof}

\section{Positivity}\label{sec:Positivity}

To prove the positivity of the solution $f(t,v)$, we first compute the lower bound of the iterated $Q_1$ function.
\begin{lemma}\label{lem:Positivity}
    Let $f(t,v)$ be a solution of the Boltzmann-Fermi-Dirac equation. Suppose there exists $\mathcal{B}_R(v_{-1})$ for some $v_{-1}\in\mathbb{R}^3$ and $R>0$ and $\mathfrak{c}>0$ such that
    \begin{align}\label{eq:lem_Posivity_cond}
        \begin{split}
            Q_1(f_0\mathbf{1}_{\mathcal{B}_R(v_{-1})}, Q_1(f_0\mathbf{1}_{\mathcal{B}_R(v_{-1})}, f_0\mathbf{1}_{\mathcal{B}_R(v_{-1})}, (1-f_0)\mathbf{1}_{\mathcal{B}_R(v_{-1})}), (1-f_0)\mathbf{1}_{\mathcal{B}_R(v_{-1})})(v)&>\mathfrak{c},\\
            Q_1((1-f_0)\mathbf{1}_{\mathcal{B}_R(v_{-1})}, Q_1((1-f_0)\mathbf{1}_{\mathcal{B}_R(v_{-1})}, (1-f_0)\mathbf{1}_{\mathcal{B}_R(v_{-1})}, f_0\mathbf{1}_{\mathcal{B}_R(v_{-1})}), f_0\mathbf{1}_{\mathcal{B}_R(v_{-1})})(v)&>\mathfrak{c}
        \end{split}
    \end{align}
    on some set $E\ni v$ such that $E\subset \mathcal{B}_R(v_{-1})$. Then, there exists $\delta>0$ and $T_0>0$ depending on $R$, $\mathfrak{c}$, $C_b$, $\|f_0(v)\|_{1,2}$, and $|v_{-1}|$ such that
    \begin{align*}
        f(t,v)>\delta t^2,\quad (1-f)(t,v)>\delta t^2
    \end{align*}
    for $t\in (0, T_0]$ and $v\in E$.
\end{lemma}
\begin{proof}
    From \eqref{eq:G_upper_bound}, we have
    \begin{align}\label{eq:G_upper_bound_2}
        G_{t_1}^{t_2}(v)\geq \exp(-c(t_2-t_1)(1+(|v_{-1}| + |v-v_{-1}|)^\gamma))
    \end{align}
    for $v\in \mathbb{R}^3$, where $c$ depends on $\|f_0\|_{1,2}$, $\gamma$, and $C_b$.
    
    Since $Q_1$ is a positive function, from \eqref{eq:duha_f} and \eqref{eq:duha_1-f}, we get $g(t,v)\geq g_0(v) G_0^t(v)$ and $1-g(t,v)\geq (1-f_0(v))G_0^t(v)$. Also, using \eqref{eq:G_upper_bound_2} for $|v-v_{-1}|\leq R$, we have $G^{t_2}_{t_1}(v) \geq \exp\left(-c(t_2-t_1)(|v_{-1}|+R)^\gamma\right)$. Inserting these lower bounds to $Q_1$ in \eqref{eq:duha_f}, we get
    \begin{align}\label{eq:f_bound_1}
        \begin{split}
            f(t,v) &= f_0(v)G_0^t(v) + \int_0^t G^t_\tau(v)Q_1(f,f,1-f)(\tau, v)\,d\tau\\
            &\geq f_0(v)G_0^t(v) + \int_0^t G^t_\tau(v) Q_1(f_0G^\tau_0\mathbf{1}_{\mathcal{B}_R(v_{-1})},f_0G^\tau_0\mathbf{1}_{\mathcal{B}_R(v_{-1})},(1-f_0)G^\tau_0\mathbf{1}_{\mathcal{B}_R(v_{-1})})(v)\,d\tau\\
            &\geq f_0(v)e^{-ct(1+(|v_{-1}|+R)^\gamma)} \\
            &\quad + \int_0^t e^{-c(t-\tau)(1+(|v_{-1}|+R)^\gamma)}e^{-3c\tau(1+(|v_{-1}|+R)^\gamma)} Q_1(f_0\mathbf{1}_{\mathcal{B}_R(v_{-1})},f_0\mathbf{1}_{\mathcal{B}_R(v_{-1})},(1-f_0)\mathbf{1}_{\mathcal{B}_R(v_{-1})})(v)\,d\tau\\
            &\geq f_0(v)e^{-ct(1+(|v_{-1}|+R)^\gamma)} \\
            &\quad + e^{-ct(1+(|v_{-1}|+R)^\gamma)}\frac{1-e^{-2ct(1+(|v_{-1}|+R)^\gamma)}}{2c(1+(|v_{-1}|+R)^\gamma)}Q_1(f_0 \mathbf{1}_{\mathcal{B}_R(v_{-1})},f_0 \mathbf{1}_{\mathcal{B}_R(v_{-1})},(1-f_0)\mathbf{1}_{\mathcal{B}_R(v_{-1})})(v)\\
            &\geq f_0(v)e^{-ct(1+(|v_{-1}|+R)^\gamma)} + \frac{t}{2} e^{-ct(1+(|v_{-1}|+R)^\gamma)} Q_1(f_0 \mathbf{1}_{\mathcal{B}_R(v_{-1})},f_0 \mathbf{1}_{\mathcal{B}_R(v_{-1})},(1-f_0)\mathbf{1}_{\mathcal{B}_R(v_{-1})})(v)    
        \end{split}
    \end{align}
    for small enough $t$ making $\frac{1-e^{-2ct(1+(|v_{-1}|+R)^\gamma)}}{2c(1+(|v_{-1}|+R)^\gamma)}\geq \frac{t}{2}$ and $v\in \mathcal{B}_R(v_{-1})$. Repeating the same computation for \eqref{eq:duha_1-f}, we get
    \begin{align}\label{eq:f_bound_2}
        \begin{split}
            1-f(t,v) &= (1-f_0)(v)G_0^t(v) + \int_0^t G_\tau^t(v) Q_1(1-f,1-f,f)(\tau, v)\,d\tau\\
            &\geq (1-f_0)(v)e^{-ct(1+(|v_{-1}|+R)^\gamma)} \\
            &\quad + \frac{t}{2} e^{-ct(1+(|v_{-1}|+R)^\gamma)}Q_1((1-f_0)\mathbf{1}_{\mathcal{B}_R(v_{-1})},(1-f_0)\mathbf{1}_{\mathcal{B}_R(v_{-1})},f_0\mathbf{1}_{\mathcal{B}_R(v_{-1})})(v)
        \end{split}
    \end{align}
    for $v\in \mathcal{B}_R(v_{-1})$ and for small enough $t$. If we replace $R$ by $\sqrt{2}R$, we also get
    \begin{align}\label{eq:f_bound_3}
        \begin{split}
            f(t,v) &\geq f_0(v)e^{-ct(1+(|v_{-1}|+\sqrt{2}R)^\gamma)} \\
            &\quad + \frac{t}{2} e^{-ct(1+(|v_{-1}|+\sqrt{2}R)^\gamma)} Q_1(f_0 \mathbf{1}_{\mathcal{B}_{\sqrt{2}R}(v_{-1})},f_0 \mathbf{1}_{\mathcal{B}_{\sqrt{2}R}(v_{-1})},(1-f_0)\mathbf{1}_{\mathcal{B}_{\sqrt{2}R}(v_{-1})})(v),\\
            1-f(t,v) &\geq (1-f_0)(v)e^{-ct(1+(|v_{-1}|+\sqrt{2}R)^\gamma)}\\
            &\quad + \frac{t}{2} e^{-ct(1+(|v_{-1}|+\sqrt{2}R)^\gamma)}Q_1((1-f_0)\mathbf{1}_{\mathcal{B}_{\sqrt{2}R}(v_{-1})},(1-f_0)\mathbf{1}_{\mathcal{B}_{\sqrt{2}R}(v_{-1})},f_0\mathbf{1}_{\mathcal{B}_{\sqrt{2}R}(v_{-1})})(v)
        \end{split}
    \end{align}
    for small enough $t$ and $v\in \mathcal{B}_{\sqrt{2}R}(v_{-1})$.

    We again apply the lower bounds \eqref{eq:f_bound_1}-\eqref{eq:f_bound_3} to $Q_1$ in \eqref{eq:duha_f} and \eqref{eq:duha_1-f}. Then,
    \begin{align*}
        &f(t,v) \\
        &= f_0(v)G_0^t(v) + \int_0^t G_\tau^t(v) Q_1(f,f,1-f)(\tau, v)\,d\tau\\
        &\geq f_0(v)G_0^t(v) + \int_0^t G_\tau^t(v) Q_1(f\mathbf{1}_{\mathcal{B}_R(v_{-1})},f\mathbf{1}_{\mathcal{B}_{\sqrt{2}R}(v_{-1})},(1-f)\mathbf{1}_{\mathcal{B}_R(v_{-1})})(\tau, v)\,d\tau\\
        &\geq f_0(v)e^{-ct(1+(|v_{-1}|+R)^\gamma)} \\
        &\quad + \frac{1}{2}\int_0^t e^{-c(t-\tau)(1+(|v_{-1}|+R)^\gamma)}e^{-2c\tau(1+(|v_{-1}|+R)^\gamma)}e^{-c\tau(1+(|v_{-1}|+\sqrt{2}R)^\gamma)}\tau \\
        &\qquad \times Q_1\left(f_0\mathbf{1}_{\mathcal{B}_R(v_{-1})},Q_1(f_0\mathbf{1}_{\mathcal{B}_{\sqrt{2}R}(v_{-1})},f_0\mathbf{1}_{\mathcal{B}_{\sqrt{2}R}(v_{-1})},(1-f_0)\mathbf{1}_{\mathcal{B}_{\sqrt{2}R}(v_{-1})})\mathbf{1}_{\mathcal{B}_{\sqrt{2}R}(v_{-1})},(1-f_0)\mathbf{1}_{\mathcal{B}_R(v_{-1})}\right)(v)\,d\tau\\
        &\geq f_0(v)e^{-ct(1+(|v_{-1}|+R)^\gamma)} \\
        &\quad + \frac{t^2}{8}e^{-ct(1+(|v_{-1}|+R)^\gamma)} Q_1\left(f_0\mathbf{1}_{\mathcal{B}_R(v_{-1})},Q_1(f_0\mathbf{1}_{\mathcal{B}_{\sqrt{2}R}(v_{-1})},f_0\mathbf{1}_{\mathcal{B}_{\sqrt{2}R}(v_{-1})},(1-f_0)\mathbf{1}_{\mathcal{B}_{\sqrt{2}R}(v_{-1})})\mathbf{1}_{\mathcal{B}_{\sqrt{2}R}(v_{-1})},\right.\\
        &\quad\qquad\qquad\left.(1-f_0)\mathbf{1}_{\mathcal{B}_R(v_{-1})}\right)(v),
    \end{align*}
    and
    \begin{align*}
        &1-f(t,v)\\
        &=(1-f_0)(v)G_0^t(v) + \int_0^t G_\tau^t(v) Q_1(1-f,1-f,f)(\tau, v)\,d\tau\\
        &\geq (1-f_0)(v)G_0^t(v) + \int_0^t G_\tau^t(v) Q_1((1-f)\mathbf{1}_{\mathcal{B}_R(v_{-1})},(1-f)\mathbf{1}_{\mathcal{B}_{\sqrt{2}R}(v_{-1})},f\mathbf{1}_{\mathcal{B}_R(v_{-1})})(\tau, v)\,d\tau\\
        &\geq (1-f_0)(v)e^{-ct(1+(|v_{-1}|+R)^\gamma)} + \frac{t^2}{8}e^{-ct(1+(|v_{-1}|+R)^\gamma)}\\
        &\qquad \times Q_1\left((1-f_0)\mathbf{1}_{\mathcal{B}_R(v_{-1})},Q_1((1-f_0)\mathbf{1}_{\mathcal{B}_{\sqrt{2}R}(v_{-1})},(1-f_0)\mathbf{1}_{\mathcal{B}_{\sqrt{2}R}(v_{-1})},f_0\mathbf{1}_{\mathcal{B}_{\sqrt{2}R}(v_{-1})})\mathbf{1}_{\mathcal{B}_{\sqrt{2}R}(v_{-1})},\right.\\
        &\quad\qquad\qquad\left.f_0\mathbf{1}_{\mathcal{B}_R(v_{-1})}\right)(v)
    \end{align*}
    for a small enough $t$ making
    \begin{align*}
        \int_0^t \tau e^{-c\tau(1+(|v_{-1}| + \sqrt{2}R)^\gamma)}e^{-c\tau(1+(|v_{-1}| + \sqrt{2}R)^\gamma)}\,d\tau\geq \frac{t^2}{4}
    \end{align*}
    and $v\in \mathcal{B}_R(0)$. Finally, we bound
    \begin{align*}
        &Q_1(f_0\mathbf{1}_{\mathcal{B}_{\sqrt{2}R}(v_{-1})},f_0\mathbf{1}_{\mathcal{B}_{\sqrt{2}R}(v_{-1})},(1-f_0)\mathbf{1}_{\mathcal{B}_{\sqrt{2}R}(v_{-1})})\mathbf{1}_{\mathcal{B}_{\sqrt{2}R}(v_{-1})}\\
        &\geq Q_1(f_0\mathbf{1}_{\mathcal{B}_R(v_{-1})},f_0\mathbf{1}_{\mathcal{B}_R(v_{-1})},(1-f_0)\mathbf{1}_{\mathcal{B}_R(v_{-1})}),\\
        &Q_1((1-f_0)\mathbf{1}_{\mathcal{B}_{\sqrt{2}R}(v_{-1})},(1-f_0)\mathbf{1}_{\mathcal{B}_{\sqrt{2}R}(v_{-1})},f_0\mathbf{1}_{\mathcal{B}_{\sqrt{2}R}(v_{-1})})\mathbf{1}_{\mathcal{B}_{\sqrt{2}R}(v_{-1})}\\
        &\geq Q_1((1-f_0)\mathbf{1}_{\mathcal{B}_R(v_{-1})},(1-f_0)\mathbf{1}_{\mathcal{B}_R(v_{-1})},f_0\mathbf{1}_{\mathcal{B}_R(v_{-1})}).
    \end{align*}
    In the middle, we used
    \begin{align*}
        |v-v_{-1}|^2\leq |v'-v_{-1}|^2+|v'_*-v_{-1}|^2\leq 2R^2,
    \end{align*}
    so we can remove the step function $\mathbf{1}_{\mathcal{B}_{\sqrt{2}R(v_{-1})}}(v)$. From the assumption \eqref{eq:lem_Posivity_cond}, we get
    \begin{align*}
        f(t,v)>\delta t^2,\quad (1-f)(t,v-v_{-1}) = (1-f)(t,v)>\delta t^2\quad t\in (0,T_0]
    \end{align*}
    for some $\delta$ and small enough $T_0$ depending on $R$, $\mathfrak{c}$, $C_b$, $\|f_0\|_{1,2}$, and $|v_{-1}|$ on the set $v\in E$. It proves the lemma.
\end{proof}

The following lemma is a covering lemma for Borel sets. In the proof of Proposition \ref{prop:4.4}, we will approximate a Borel set $E$ by a countable union of closed balls and continue the analysis for each closed ball.
\begin{lemma}[Theorem 5.5.2 of \cite{MR2267655}]\label{lem:Covering_lemma}
    Let $E\subset \mathbb{R}^n$ be a Borel set. Suppose that for every point $x\in E$ and every $\epsilon>0$, we are given a closed ball $\mathcal{B}_{<\epsilon}(x)$ of positive diameter less than $\epsilon$. Then, this family of balls contains at most a countable subfamily of pairwise disjoint balls $\mathcal{B}_k$ such that
    \begin{align*}
        \left|E\setminus \cup_{k=1}^\infty \mathcal{B}_k\right| = 0.
    \end{align*}
\end{lemma}

The next lemma demonstrates that there is a pair of well-separated sub-cubes among the subdivided cubes if we collect sufficiently many sub-cubes.
\begin{lemma}\label{lem:pair_of_subcubes}
    Let $\mathcal{Q}$ be a unit cube in $\mathbb{R}^3$ subdivided into $13^3$ sub-cubes. For any collection of $12^3 + 100$ sub-cubes $E$, there exists a pair of sub-cubes such that the distance between the centers of the two sub-cubes is at least $\frac{10}{13}$.
\end{lemma}
\begin{proof}
    We will prove the lemma by a contradiction argument. Let us assume that there exists a collection $E$ which does not satisfy the lemma. Also, let $E_c$ be a collection of center points of the sub-cubes in $E$. Choose $\mathcal{Q}_{\frac{1}{13}}(v_1)$ and $\mathcal{Q}_{\frac{1}{13}}(v_2)$ in $E$ such that the distance between the centers is maximized. We set $r = |v_1-v_2|$ and draw closed spheres having center $v_1$ (resp. $v_2$) with radius $r$. By the assumption, $E_c$ should be contained in the intersection of the two spheres. Now, let $h(p)$ be the distance between $p\in E$ and the longitudinal bisection plane of the intersection of the sphere passing $v_1$ and $v_2$; we take the minus distance if $p$ is under the bisection plane. We refer to Figure \ref{fig:ball_intersection} for the geometric description. We define $h_1 = \max_{p\in E_c}\{h(p)\}$ and $h_2 = \max_{p\in E_c}\{-h(p)\}$, then $h_1+h_2\leq r$ and $h(p)\in [-h_2, h_1]$ for any $p\in E$ by the definition of $E$. Now, we will maximize the area of the domain bounded by the spheres and the plane having distances $h_1$ and $h_2$ from the bisection plane. The volume of the bounded domain is maximized when $h_1=h_2 = r/2$, and the volume is given by
    \begin{align*}
        V = \frac{11}{12}\left(\frac{2\pi}{3} - \frac{\sqrt{3}}{2}\right)r^3.
    \end{align*}
    We choose $r = \frac{10}{13} + \frac{\sqrt{3}}{13}$, where $\frac{\sqrt{3}}{13}$ is added to cover the sub-cubes having partial intersection with the domain. Since
    \begin{align*}
        (12^3+100)\left(\frac{1}{13}\right)^3 - \frac{11}{12}\left(\frac{2\pi}{3} - \frac{\sqrt{3}}{2}\right)\left(\frac{10}{13} + \frac{\sqrt{3}}{13}\right)^3>0,
    \end{align*}
    it means that there is a cube such that the center is not contained in the domain. It makes a contradiction to the choice of the $E$, and we prove the lemma.

    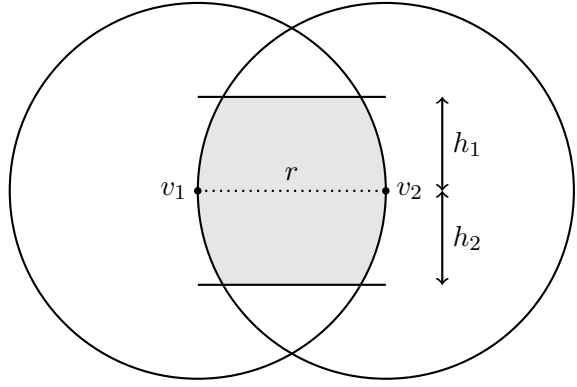
\begin{figure}[htbp]
        \centering
        \begin{tikzpicture}[thick]
            \def\mystep{5};
            \begin{scope}[even odd rule]
                \clip (-\mystep/4, -\mystep/4) rectangle (\mystep/4, \mystep/4);
                \fill[fill=lightgray!40] (0, 0) circle (\mystep) (\mystep/4, 0) circle (\mystep/2) (-\mystep/4, 0) circle (\mystep/2);
            \end{scope}
            \draw (-\mystep/4, 0) circle (\mystep/2);
            \draw (\mystep/4, 0) circle (\mystep/2);
            \draw[dotted] (-\mystep/4, 0) -- (\mystep/4, 0) node[midway, above] {$r$};
            \draw (-\mystep/4, \mystep/4) -- (\mystep/4, \mystep/4);
            \draw (-\mystep/4, -\mystep/4) -- (\mystep/4, -\mystep/4);
            \node[fill=black,circle, inner sep=1] at (-\mystep/4, 0) {};
            \node[fill=black,circle, inner sep=1] at (\mystep/4, 0) {};
            \node[left] at (-\mystep/4, 0) {$v_1$};
            \node[right] at (\mystep/4, 0) {$v_2$};
            \draw[<->] (\mystep/4 + 0.75, -\mystep/4) -- (\mystep/4 + 0.75, 0) node[midway, right] {$h_2$};
            \draw[<->] (\mystep/4 + 0.75, 0) -- (\mystep/4 + 0.75, \mystep/4)node[midway, right] {$h_1$};
        \end{tikzpicture}
        \caption{The intersection of two balls with radius $r$ and center distance $r$. The intersection area is clipped by the height from the bisection plane by $h_1$ and $h_2$. The gray domain contains the domain such that the distance between any two points is smaller than $r$.}
        \label{fig:ball_intersection}
    \end{figure}
\end{proof}

For an initial function $0\leq f_0(v)\leq 1$, suppose that there exist a $\epsilon>0$ and a set $E$ such that $E\subset\{v:\epsilon\leq f(v)\leq 1-\epsilon\}$ has a positive measure. By the Lebesgue density theorem, for any $0<\mathfrak{a}_0<1$, $\{r:\text{$E$ is $(1-\mathfrak{a}_0, r)$-measurable}\}$ is not an empty set. Therefore, we can choose a ball $\mathcal{B}_{4R_0}(v_{-1})$ for some $v_{-1}\in E$ and $R_0>0$ such that
\begin{align}\label{eq:E_density}
    \frac{|\mathcal{B}_{4R_0}(v_{-1})\cap E^c|}{|\mathcal{B}_{4R_0}(v_{-1})|}\leq \mathfrak{a}_0.
\end{align}
Under this setting, we prove that the condition \eqref{eq:lem_Posivity_cond} can be fulfilled if we choose $\mathfrak{a}_0$ small enough.
\begin{proposition}\label{prop:4.4}
    Suppose the collision kernel \eqref{eq:B_defi} satisfies $0\leq \gamma\leq 2$, (H1), and (H2). Also, assume $0\leq f\leq 1$ and that there exists $\epsilon>0$ such that $E\subset \{v:\epsilon\leq f(v)\leq 1-\epsilon\}$ has a positive measure. Then, there exists an explicit small $\mathfrak{a}_0<1$ satisfying the following. For $v_{-1}\in E$ and $R_0>0$ satisfying \eqref{eq:E_density}, we can choose a constant $C$, which depends on $\mathfrak{a}_0$, $c_b$ in \eqref{eq:b_lower_naer_pi/2}, and $\gamma$, such that
    \begin{align*}
        Q_1(f_0\mathbf{1}_{\mathcal{B}_{4R_0}(v_{-1})}, Q_1(f_0\mathbf{1}_{\mathcal{B}_{4R_0}(v_{-1})},f_0\mathbf{1}_{\mathcal{B}_{4R_0}(v_{-1})},(1-f_0)\mathbf{1}_{\mathcal{B}_{4R_0}(v_{-1})}),(1-f_0)\mathbf{1}_{\mathcal{B}_{4R_0}(v_{-1})})(v)&\geq C R_0^{2\gamma+6}\epsilon^3,\\
        Q_1((1-f_0)\mathbf{1}_{\mathcal{B}_{4R_0}(v_{-1})}, Q_1((1-f_0)\mathbf{1}_{\mathcal{B}_{4R_0}(v_{-1})},(1-f_0)\mathbf{1}_{\mathcal{B}_{4R_0}(v_{-1})},f_0\mathbf{1}_{\mathcal{B}_{4R_0}(v_{-1})}),f_0\mathbf{1}_{\mathcal{B}_{4R_0}(v_{-1})})(v)&\geq C R_0^{2\gamma+6}\epsilon^3.
    \end{align*}
\end{proposition}
\begin{proof}
    We will assume the angular collision kernel $b$ is continuous on $\mathbb{S}^2$. We will relax this condition at the end of the proof. Also, we will only prove the first one in the proposition; the second one follows by taking the symmetry $f\leftrightarrow 1-f$ in the proof.
    
    First, we consider the cube $\mathcal{Q}_{R_0}(v_{-1})$ having an axis parallel with the Cartesian coordinates. From \eqref{eq:E_density}, the density of the defect set $\mathcal{Q}_{R_0}(v_{-1})\cap E^c$ is bounded by
    \begin{align*}
        \frac{|\mathcal{Q}_{R_0}(v_{-1})\cap E^c|}{|\mathcal{Q}_{R_0}(v_{-1})|}\leq \frac{\mathfrak{a}_0|\mathcal{B}_{4R_0}(v_{-1})|}{|\mathcal{Q}_{R_0}(v_{-1})|} \leq 280 \mathfrak{a}_0.
    \end{align*}
    We subdivide the cube $\mathcal{Q}_{R_0}(v_{-1})$ by $13^3$ sub-cubes. We claim that there exist at least $12^3+100$ sub-cubes having a density of the defect set smaller than $\left(1-\frac{12^3+100}{13^3}\right)^{-1}(280\mathfrak{a}_0)$. Indeed,
    \begin{align*}
        13^3\left(280\mathfrak{a}_0\right) -(13^3 - (12^3+100))\left(1-\frac{12^3+100}{13^3}\right)^{-1}(280\mathfrak{a}_0)=0.
    \end{align*}
    
    By Lemma \ref{lem:pair_of_subcubes}, there exist at least two cells $\mathcal{Q}_{R_1}(v_1)$ and $\mathcal{Q}_{R_2}(v_2)$ with $R_1 = R_2 = \frac{R_0}{13}$ in the $12^3+100$ sub-cubes satisfying that the distance between $v_1$ and $v_2$ satisfies $\frac{10}{13}R_0\leq |v_1-v_2|\leq \sqrt{3}R_0$. We now draw $\mathcal{B}_{R_1}(v_1)$ and $\mathcal{B}_{R_2}(v_2)$ inside the sub-cubes. Also, let $v_0 = \frac{v_1+v_2}{2}$ and $\mathcal{B}_{R_0}(v_0)$ be a ball having radius $\frac{R_0}{2}$ and center $v_0$. For the detailed geometric picture, we refer to Figure \ref{fig:choice_of_balls}.
    
    By \eqref{eq:E_density}, the density of the defect set in $\mathcal{B}_{R_0}(v_0)$, $\mathcal{B}_{R_1}(v_1)$, and $\mathcal{B}_{R_2}(v_2)$ are bounded by
    \begin{align}\label{eq:density_defect_set}
        \begin{split}
            \frac{|\mathcal{B}_{R_0}(v_0)\cap E^c|}{|\mathcal{B}_{R_0}(v_0)|}&\leq \frac{\mathfrak{a}_0|\mathcal{B}_{4R_0}(v_{-1})|}{|\mathcal{B}_{R_0}(v_0)|}\leq 4^3\mathfrak{a}_0,\\
            \frac{|\mathcal{B}_{R_1}(v_1)\cap E^c|}{|\mathcal{B}_{R_1}(v_1)|}&\leq \frac{(1-(12^3+100)/13^3)^{-1}280\mathfrak{a}_0|\mathcal{Q}_{R_1}(v_1)|}{|\mathcal{B}_{R_1}(v_1)|}\leq 3184\mathfrak{a}_0,\\
            \frac{|\mathcal{B}_{R_2}(v_2)\cap E^c|}{|\mathcal{B}_{R_2}(v_2)|}&\leq \frac{(1-(12^3+100)/13^3)^{-1}280\mathfrak{a}_0|\mathcal{Q}_{R_2}(v_2)|}{|\mathcal{B}_{R_2}(v_2)|}\leq 3184\mathfrak{a}_0.
        \end{split}
    \end{align}
    \begin{figure}[htbp]
        \centering
        \begin{tikzpicture}
            \def\subdiv{13};
            \def\tot{16};
            \def\mystep{0.5};
            \draw[step=\mystep cm,gray, very thin] (-\mystep*\tot/2,-\mystep*\tot/2) grid (\mystep*\tot/2-\mystep,\mystep*\tot/2-\mystep);
            {\foreach \i in {0, ..., 12}
                \foreach \j in {0, ..., 12}
                    \draw[fill = gray!70, opacity=0.4] (\mystep*\i -\mystep*\tot/2+ \mystep,\mystep*\j -\mystep*\tot/2+ \mystep) rectangle (\mystep*\i -\mystep*\tot/2 + 2*\mystep,\mystep*\j-\mystep*\tot/2+ 2*\mystep);
            \foreach \i in {0}
                \foreach \j in {2, 4, 5, 8, 9}
                    \draw[fill = color1bg, opacity=0.4] (\mystep*\i -\mystep*\tot/2+ \mystep,\mystep*\j -\mystep*\tot/2+ \mystep) rectangle (\mystep*\i -\mystep*\tot/2 + 2*\mystep,\mystep*\j-\mystep*\tot/2+ 2*\mystep);
            \foreach \i in {1}
                \foreach \j in {2, 4, 5, 7, 8, 9, 10}
                    \draw[fill = color1bg, opacity=0.4] (\mystep*\i -\mystep*\tot/2+ \mystep,\mystep*\j -\mystep*\tot/2+ \mystep) rectangle (\mystep*\i -\mystep*\tot/2 + 2*\mystep,\mystep*\j-\mystep*\tot/2+ 2*\mystep);
            \foreach \i in {2}
                \foreach \j in {0, 1, 3, 4, 5, 6, 7, 8, 10, 12}
                    \draw[fill = color1bg, opacity=0.4] (\mystep*\i -\mystep*\tot/2+ \mystep,\mystep*\j -\mystep*\tot/2+ \mystep) rectangle (\mystep*\i -\mystep*\tot/2 + 2*\mystep,\mystep*\j-\mystep*\tot/2+ 2*\mystep);
            \foreach \i in {3}
                \foreach \j in {1, 3, 4, 5, 6, 7, 10, 12}
                    \draw[fill = color1bg, opacity=0.4] (\mystep*\i -\mystep*\tot/2+ \mystep,\mystep*\j -\mystep*\tot/2+ \mystep) rectangle (\mystep*\i -\mystep*\tot/2 + 2*\mystep,\mystep*\j-\mystep*\tot/2+ 2*\mystep);
            \foreach \i in {4}
                \foreach \j in {1, 3, 6, 7, 8, 9, 10, 11, 12}
                    \draw[fill = color1bg, opacity=0.4] (\mystep*\i -\mystep*\tot/2+ \mystep,\mystep*\j -\mystep*\tot/2+ \mystep) rectangle (\mystep*\i -\mystep*\tot/2 + 2*\mystep,\mystep*\j-\mystep*\tot/2+ 2*\mystep);
            \foreach \i in {5}
                \foreach \j in {0, 1, 2, 3, 6, 7, 8, 10, 11}
                    \draw[fill = color1bg, opacity=0.4] (\mystep*\i -\mystep*\tot/2+ \mystep,\mystep*\j -\mystep*\tot/2+ \mystep) rectangle (\mystep*\i -\mystep*\tot/2 + 2*\mystep,\mystep*\j-\mystep*\tot/2+ 2*\mystep);
            \foreach \i in {6}
                \foreach \j in {1, 3, 5, 6, 7, 9, 10, 11}
                    \draw[fill = color1bg, opacity=0.4] (\mystep*\i -\mystep*\tot/2+ \mystep,\mystep*\j -\mystep*\tot/2+ \mystep) rectangle (\mystep*\i -\mystep*\tot/2 + 2*\mystep,\mystep*\j-\mystep*\tot/2+ 2*\mystep);
            \foreach \i in {7}
                \foreach \j in {0, 1, 2, 3, 4, 5, 6, 8, 10, 11, 12}
                    \draw[fill = color1bg, opacity=0.4] (\mystep*\i -\mystep*\tot/2+ \mystep,\mystep*\j -\mystep*\tot/2+ \mystep) rectangle (\mystep*\i -\mystep*\tot/2 + 2*\mystep,\mystep*\j-\mystep*\tot/2+ 2*\mystep);
            \foreach \i in {8}
                \foreach \j in {0, 1, 3, 4, 5, 7,  8, 10, 11}
                    \draw[fill = color1bg, opacity=0.4] (\mystep*\i -\mystep*\tot/2+ \mystep,\mystep*\j -\mystep*\tot/2+ \mystep) rectangle (\mystep*\i -\mystep*\tot/2 + 2*\mystep,\mystep*\j-\mystep*\tot/2+ 2*\mystep);
            \foreach \i in {9}
                \foreach \j in {1, 2, 3, 4, 5, 7, 9,  10, 11, 12}
                    \draw[fill = color1bg, opacity=0.4] (\mystep*\i -\mystep*\tot/2+ \mystep,\mystep*\j -\mystep*\tot/2+ \mystep) rectangle (\mystep*\i -\mystep*\tot/2 + 2*\mystep,\mystep*\j-\mystep*\tot/2+ 2*\mystep);
            \foreach \i in {10}
                \foreach \j in {0, 3, 4, 5, 8, 9, 10, 11, 12}
                    \draw[fill = color1bg, opacity=0.4] (\mystep*\i -\mystep*\tot/2+ \mystep,\mystep*\j -\mystep*\tot/2+ \mystep) rectangle (\mystep*\i -\mystep*\tot/2 + 2*\mystep,\mystep*\j-\mystep*\tot/2+ 2*\mystep);
            \foreach \i in {11}
                \foreach \j in {0, 2, 3, 6, 7, 9, 11, 12}
                    \draw[fill = color1bg, opacity=0.4] (\mystep*\i -\mystep*\tot/2+ \mystep,\mystep*\j -\mystep*\tot/2+ \mystep) rectangle (\mystep*\i -\mystep*\tot/2 + 2*\mystep,\mystep*\j-\mystep*\tot/2+ 2*\mystep);
            \foreach \i in {12}
                \foreach \j in {2, 3, 5, 6, 9, 11}
                    \draw[fill = color1bg, opacity=0.4] (\mystep*\i -\mystep*\tot/2+ \mystep,\mystep*\j -\mystep*\tot/2+ \mystep) rectangle (\mystep*\i -\mystep*\tot/2 + 2*\mystep,\mystep*\j-\mystep*\tot/2+ 2*\mystep);
                    }
            \draw[fill = color2bg] (\mystep*9 -\mystep*\tot/2+ \mystep + \mystep/2, \mystep*1 -\mystep*\tot/2+ \mystep + \mystep/2) circle (\mystep/2);
            \draw[fill = color2bg] (\mystep*9 -\mystep*\tot/2+ \mystep + \mystep/2,\mystep*12 -\mystep*\tot/2+ \mystep + \mystep/2) circle (\mystep/2);
            \draw[very thick] (\mystep*9 -\mystep*\tot/2+ 3*\mystep/2,\mystep*1 -\mystep*\tot/2+ 3*\mystep/2) -- (\mystep*9 -\mystep*\tot/2 + 3*\mystep/2,\mystep*12-\mystep*\tot/2+ 3*\mystep/2);
            \node[fill=black,circle, inner sep=1] at (\mystep*9 -\mystep*\tot/2+ 3*\mystep/2,\mystep*1 -\mystep*\tot/2+ 3*\mystep/2) {};
            \node[fill=black,circle, inner sep=1] at (\mystep*9 -\mystep*\tot/2 + 3*\mystep/2,\mystep*12-\mystep*\tot/2+ 3*\mystep/2) {};
            \node[fill=black,circle, inner sep=1] at (\mystep*9 -\mystep*\tot/2 + 3*\mystep/2,\mystep*6.5-\mystep*\tot/2+ 3*\mystep/2) {};
            \node[right, xshift = 0.1cm, yshift = 0.1cm] at (\mystep*9 -\mystep*\tot/2 + 3*\mystep/2,\mystep*6.5-\mystep*\tot/2+ 3*\mystep/2) {$v_0$};
            \node[right, xshift = 0.1cm] at (\mystep*9 -\mystep*\tot/2+ 3*\mystep/2,\mystep*1 -\mystep*\tot/2+ 3*\mystep/2) {$\mathcal{B}_{R_1}(v_1)$};
            \node[right, xshift = 0.1cm] at (\mystep*9 -\mystep*\tot/2 + 3*\mystep/2,\mystep*11-\mystep*\tot/2+ 3*\mystep/2) {$\mathcal{B}_{R_2}(v_2)$};
            \node[right, xshift = 0.1cm] at (\mystep*6 -\mystep*\tot/2 + 11*\mystep/2,\mystep*13-\mystep*\tot/2+ 3*\mystep/2) {$\mathcal{Q}_{R_0}(v_{-1})$};
            \node[fill=black,circle, inner sep=1] at (\mystep*6 -\mystep*\tot/2 + 3*\mystep/2,\mystep*6-\mystep*\tot/2+ 3*\mystep/2) {};
            \node[right, xshift = 0cm] at (\mystep*6 -\mystep*\tot/2 + 3*\mystep/2,\mystep*6-\mystep*\tot/2+ 3*\mystep/2) {$v_{-1}$};
        \end{tikzpicture}
        \caption{The gray boxes represents $\mathcal{Q}_{R_0}(v_{-1})$. The orange boxes represent the boxes having relatively high density. The blue balls are $\mathcal{B}_{R_1}(v_1)$ and $\mathcal{B}_{R_2}(v_2)$ and have distance between the centers at least $\frac{10}{13}R_0$.}
        \label{fig:choice_of_balls}
    \end{figure}
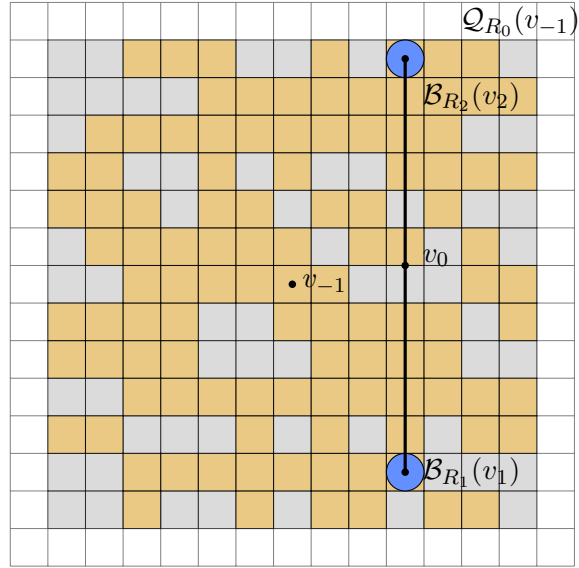
    
    From now on, we will take
    \begin{align*}
        f_1(v) = f_0(v)\mathbf{1}_{\mathcal{B}_{R_0}(v_0)},\quad f_2(v) = f_0(v)\mathbf{1}_{\mathcal{B}_{R_1}(v_1)},\quad f_3(v) = f_0(v)\mathbf{1}_{\mathcal{B}_{R_2}(v_2)}.
    \end{align*}
    
    Since $0\leq f_0\leq 1$ on whole space and $\epsilon\leq f_0\leq 1-\epsilon$ on $E$, we have
    \begin{align*}
        &Q_1(f_0\mathbf{1}_{\mathcal{B}_{4R_0}(v_{-1})}, Q_1(f_0\mathbf{1}_{\mathcal{B}_{4R_0}(v_{-1})},f_0\mathbf{1}_{\mathcal{B}_{4R_0}(v_{-1})},(1-f_0)\mathbf{1}_{\mathcal{B}_{4R_0}(v_{-1})}),(1-f_0)\mathbf{1}_{\mathcal{B}_{4R_0}(v_{-1})})(v)\\
        &\geq Q_1(f_1, Q_1(f_2,f_3,(1-f_0)\mathbf{1}_{\mathcal{B}_{4R_0}(v_{-1})}),(1-f_0)\mathbf{1}_{\mathcal{B}_{4R_0}(v_{-1})})(v)\\
        &\geq Q_1(f_1, Q_1(f_2,f_3,\epsilon\mathbf{1}_{\mathcal{B}_{4R_0}(v_{-1})} - \epsilon\mathbf{1}_{\mathcal{B}_{4R_0}(v_{-1})\cap E^c}),\epsilon\mathbf{1}_{\mathcal{B}_{4R_0}(v_{-1})} - \epsilon\mathbf{1}_{\mathcal{B}_{4R_0}(v_{-1})\cap E^c})(v)\\
        &=\epsilon^2 Q_1(f_1, Q_1(f_2,f_3,\mathbf{1}_{\mathcal{B}_{4R_0}(v_{-1})} - \mathbf{1}_{\mathcal{B}_{4R_0}(v_{-1})\cap E^c}),\mathbf{1}_{\mathcal{B}_{4R_0}(v_{-1})} - \mathbf{1}_{\mathcal{B}_{4R_0}(v_{-1})\cap E^c})(v).
    \end{align*}
    
    Now, for arbitrary small $0<\delta\leq \frac{1}{2}$, using Lusin's theorem, we choose compact sets $G_{i,\delta}$ and open sets $O_{i,\delta}$ as follows:
    \begin{align}\label{eq:G_O_construction_1}
        \begin{split}
            &G_{0,\delta} \subset E\cap \mathcal{B}_{R_0}(v_0)\subset O_{0,\delta}\subset \mathcal{B}_{(1+2\delta)R_0}(v_0),\\
            &G_{1,\delta} \subset E\cap \mathcal{B}_{R_1}(v_1)\subset O_{1,\delta}\subset \mathcal{B}_{(1+2\delta)R_1}(v_1),\\
            &G_{2,\delta} \subset E\cap \mathcal{B}_{R_2}(v_2)\subset O_{2,\delta}\subset \mathcal{B}_{(1+2\delta)R_2}(v_2),\\
            &G_{3,\delta} \subset E\subset O_{3,\delta}\subset \mathcal{B}_{(1+2\delta)4R_0}(v_{-1})
        \end{split}
    \end{align}
    such that
    \begin{align}\label{eq:G_O_construction_2}
        \begin{split}
            &(1+\mathfrak{a}_0\delta)|G_{0,\delta}|\geq |E\cap \mathcal{B}_{R_0}(v_0)|\geq (1-\mathfrak{a}_0\delta)|O_{0,\delta}|,\\
            &(1+\mathfrak{a}_0\delta)|G_{1,\delta}|\geq |E\cap \mathcal{B}_{R_1}(v_1)|\geq (1-\mathfrak{a}_0\delta)|O_{1,\delta}|,\\
            &(1+\mathfrak{a}_0\delta)|G_{2,\delta}|\geq |E\cap \mathcal{B}_{R_2}(v_2)|\geq (1-\mathfrak{a}_0\delta)|O_{2,\delta}|,\\
            &(1+\mathfrak{a}_0\delta)|G_{3,\delta}|\geq |E|\geq (1-\mathfrak{a}_0\delta)|O_{3,\delta}|.
        \end{split}
    \end{align}
    Using Lusin's theorem for the compact sets $G$, the Tietze extension theorem, and Uryshon's lemma if necessary, we choose continuous functions $f_{i,\delta},\varphi_\delta:\mathbb{R}^3\rightarrow \mathbb{R}$ such that $0\leq f_{i,\delta},\varphi_\delta\leq 1$ and
    \begin{align*}
        f_{1,\delta}|G_{0,\delta} &= f_1,\quad \supp f_{1,\delta}\subset \mathcal{B}_{(1+2\delta)R_0}(v_0),\\
        f_{2,\delta}|G_{1,\delta} &= f_2,\quad \supp f_{2,\delta}\subset \mathcal{B}_{(1+2\delta)R_1}(v_1),\\
        f_{3,\delta}|G_{2,\delta} &= f_3,\quad \supp f_{3,\delta}\subset \mathcal{B}_{(1+2\delta)R_2}(v_2),\\
        \varphi_\delta|G_{3, \delta} &= 1,\quad \supp \varphi_\delta\subset \mathcal{B}_{(1+2\delta)4R_0}(v_{-1}).
    \end{align*}
    Now, we will compute the lower bound of 
    \begin{align*}
        Q_1(f_{1, \delta}, Q_1(f_{2, \delta},f_{3, \delta},\varphi_\delta),\varphi_\delta)(v).
    \end{align*}
    We first reserve some variables. We write
    \begin{align*}
            &Q_1(f_{1, \delta}, Q_1(f_{2, \delta},f_{3, \delta},\varphi_\delta),\varphi_\delta)(v)\\
            &=\int_{\mathbb{R}^3}d\tilde{u}_\parallel\,\frac{1}{|\tilde{u}_\parallel - v|^{2-\gamma}}f_{1, \delta}(\tilde{u}_\parallel)\int_{v + E_{\tilde{u}_\parallel - v}}dw_1 \frac{h(\cos\theta_\omega)}{(\cos\theta_\omega)^\gamma}\\
        &\qquad \times \int_{\mathbb{R}^3\times\mathbb{S}^2} dw_2 d\tilde{\omega}\,B(w_1-w_2,\tilde{\omega}) f_{2, \delta}(w_1')f_{3, \delta}(w_2')\varphi_\delta(w_2) \varphi_\delta(\tilde{u}_\parallel + w_1 - v),
    \end{align*}
    using the Carleman representation, where $\tilde{u}_\parallel = v'$, $w_1 = v'_*$, and $w_1',w_2'$ are post-collision velocities generated by $(w_1,w_2, \tilde{\omega})$ corresponding to $Q_1(f_2,f_3,\varphi_\delta)$. The $\cos\theta_\omega$ is explicitly given by
    \begin{align*}
        \cos\theta_\omega = \frac{|\tilde{u}_\parallel-v|}{|\tilde{u}_\parallel+w_1-2v|}.
    \end{align*}
    
    We will denote $\Theta_{v+E_{\tilde{u}_\parallel-v}}$ be a distribution satisfying
    \begin{align*}
        \int_{\mathbb{R}^3} f(x)\Theta_{v+E_{\tilde{u}_\parallel-v}}(x)\,dx = \int_{v+E_{\tilde{u}_\parallel-v}} f(x)\,dx
    \end{align*}
    for any compactly supported continuous function $f$. Also, we define $\Theta_{\epsilon, v+E_{\tilde{u}_\parallel-v}}$ by the characteristic function such that the supporting set is the collection of the points whose distance from the plane $v+E_{\tilde{u}_\parallel-v}$ is not greater than $\epsilon$. If $f$ is a compactly supported continuous function, we have
    \begin{align*}
        \int_{v+E_{\tilde{u}_\parallel-v}} f(x)\,dx = \lim_{\epsilon\rightarrow 0}\frac{1}{2\epsilon}\int_{\mathbb{R}^3} f(x)\Theta_{\epsilon, v+E_{\tilde{u}_\parallel-v}}(x)\,dx.
    \end{align*}
    
    Using these notations, we will take some change of variable for $Q_1(f_{1, \delta}, Q_1(f_{2, \delta},f_{3, \delta},\varphi_\delta),\varphi_\delta)(v)$. Since all the functions in $Q_1$ are continuous, we have
    \begin{align}\label{eq:L_q_comp1}
        \begin{split}
            &Q_1(f_{1, \delta}, Q_1(f_{2, \delta},f_{3, \delta},\varphi_\delta),\varphi_\delta)(v)\\
            &=\int_{\mathbb{R}^3}d\tilde{u}_\parallel\,\frac{1}{|\tilde{u}_\parallel - v|^{2-\gamma}}f_{1, \delta}(\tilde{u}_\parallel)\int_{v + E_{\tilde{u}_\parallel - v}}dw_1\,\frac{h(\cos\theta_\omega)}{|\cos\theta_\omega|^\gamma}\\
            &\qquad \times \left(\int_{\mathbb{R}^3\times\mathbb{S}^2} dw_2 d\tilde{\omega}\,B(w_1-w_2,\tilde{\omega}) f_{2, \delta}(w_1')f_{3, \delta}(w_2')\varphi_\delta(w_2)\right) \varphi_\delta(\tilde{u}_\parallel + w_1 - v)\\
            &=\int_{\mathbb{R}^3}d\tilde{u}_\parallel\,\frac{1}{|\tilde{u}_\parallel - v|^{2-\gamma}}f_{1, \delta}(\tilde{u}_\parallel)\lim_{\epsilon\rightarrow \infty}\frac{1}{2\epsilon}\int_{\mathbb{R}^3}dw_1\,\frac{h(\cos\theta_\omega)}{|\cos\theta_\omega|^\gamma}\Theta_{\epsilon, v + E_{\tilde{u}_\parallel - v}}(w_1) \\
            &\qquad \times \left(\int_{\mathbb{R}^3\times\mathbb{S}^2} dw_2 d\tilde{\omega}\,B(w_1-w_2,\tilde{\omega})f_{2, \delta}(w_1')f_{3, \delta}(w_2')\varphi_\delta(w_2)\right) \varphi_\delta(\tilde{u}_\parallel + w_1 - v)\\
            &=\int_{\mathcal{B}_{(1+2\delta)R_0}(v_0)}d\tilde{u}_\parallel\,\frac{1}{|\tilde{u}_\parallel - v|^{2-\gamma}}f_{1, \delta}(\tilde{u}_\parallel)\int_{\mathcal{B}_{(1+2\delta)R_1}(v_1)}\,dw_1f_{2, \delta}(w_1)\int_{\mathcal{B}_{(1+2\delta)R_2}(v_2)}\,dw_2f_{3, \delta}(w_2)\\
            &\qquad \times \lim_{\epsilon\rightarrow 0}\frac{1}{2\epsilon}\int_{\mathbb{S}^2} d\tilde{\omega}\, \frac{h(\cos\theta_\omega)}{|\cos\theta_\omega|^\gamma} B(w_1-w_2, \tilde{\omega})\varphi_\delta(w_2')\varphi_\delta(\tilde{u}_\parallel + w_1' - v)\Theta_{\epsilon, v + E_{\tilde{u}_\parallel - v}}(w'_1)\\
            &=\int_{\mathcal{B}_{(1+2\delta)R_0}(v_0)}d\tilde{u}_\parallel\,\frac{1}{|\tilde{u}_\parallel - v|^{2-\gamma}}f_{1, \delta}(\tilde{u}_\parallel)\int_{\mathcal{B}_{(1+2\delta)R_1}(v_1)}\,dw_1f_{2, \delta}(w_1)\int_{\mathcal{B}_{(1+2\delta)R_2}(v_2)}\,dw_2f_{3, \delta}(w_2)\\
            &\qquad \times 2\lim_{\epsilon\rightarrow 0}\frac{1}{2\epsilon}\int_0^{2\pi}d\tilde{\phi}\,\int_0^{\pi/2} d\tilde{\theta}_\omega \, \sin\tilde{\theta}_\omega \frac{h(\cos\theta_\omega)}{|\cos\theta_\omega|^\gamma}|w_1-w_2|^\gamma h(\cos\tilde{\theta}_\omega )\varphi_\delta(w_2')\varphi_\delta(\tilde{u}_\parallel + w_1' - v)\Theta_{\epsilon, v + E_{\tilde{u}_\parallel - v}}(w'_1).
        \end{split}
    \end{align}
    Here, $\cos\tilde{\theta}_\omega $ is given by
    \begin{align*}
        \cos\tilde{\theta}_\omega  = \frac{w_2-w_1}{|w_2-w_1|}\cdot \tilde{\omega}.
    \end{align*}
    In the final step, we used spherical coordinates $(\tilde{\theta}_\omega, \tilde{\phi})$ for $\mathbb{S}^2$ and use $\tilde{\theta}_\omega$ symmetry about $\pi-\tilde{\theta}_\omega$ so that we use domain $\tilde{\theta}_\omega\in[0,\pi/2]$.
    
    Now, we use change of variable $\tilde{\theta}_\omega = \frac{\pi-\theta'}{2}$, which corresponds to the change of variable from $\omega$-representation to $\sigma$-representation. Using \eqref{eq:h_condi} to replace $h(\cos\tilde{\theta}_\omega)$ by $b(\cos\theta')$,
    \begin{align*}
        |w_1-w_2|^2 h(\cos\theta_\omega)\sin\tilde{\theta}_\omega d\tilde{\theta}_\omega d\tilde{\phi} &= 2|w_1-w_2|^2 b(\cos(\pi-2\theta_\omega))\sin\tilde{\theta}_\omega \cos\tilde{\theta}_\omega d\tilde{\theta}_\omega d\tilde{\phi}\\
        &=\frac{1}{2}|w_1-w_2|^2b(\cos\theta')\sin\theta' d\theta'd\tilde{\phi}.
    \end{align*}
    Note that $|w_1-w_2|^2 \sin\theta'd\theta'd\tilde{\phi}$ is the Jacobian of the spherical coordinates for $S_{w_1,w_2}$, so we denote this measure by $dw_1'$. We rewrite the integral using $dw_1'$ with $\sigma$-representation and get
    \begin{align}\label{eq:L_q_comp2}
        \begin{split}
            \eqref{eq:L_q_comp1}&=\int_{\mathcal{B}_{(1+2\delta)R_0}(v_0)}\,d\tilde{u}_\parallel\,\frac{1}{|\tilde{u}_\parallel - v|^{2-\gamma}}f_{1, \delta}(\tilde{u}_\parallel)\int_{\mathcal{B}_{(1+2\delta)R_1}(v_1)}\,dw_1f_{2, \delta}(w_1)\int_{\mathcal{B}_{(1+2\delta)R_2}(v_2)}\,dw_2f_{3, \delta}(w_2)\\
            &\qquad \times |w_1-w_2|^{\gamma-2}\lim_{\epsilon\rightarrow 0}\frac{1}{2\epsilon}\int_{S_{w_1,w_2}}dw_1'\, \frac{h(\cos\theta_\omega)}{|\cos\theta_\omega|^\gamma} b(\cos\theta')\varphi_\delta(w_1+w_2-w_1')\\
            &\qquad \qquad \times \varphi_\delta(\tilde{u}_\parallel + w_1' - v)\Theta_{\epsilon, v+E_{\tilde{u}_\parallel-v}}(w_1').
        \end{split}
    \end{align}
    
    At the start of the proof, we assumed that $b$ is a continuous function on $\mathbb{S}^2$. Therefore, the integrand consists of continuous and compactly supported functions, so we can easily take limit $\epsilon\rightarrow 0$.
    \begin{align}\label{eq:L_q_comp3}
        \begin{split}
            \eqref{eq:L_q_comp2}&= \int_{\mathcal{B}_{(1+2\delta)R_0}(v_0)}\,d\tilde{u}_\parallel\,\frac{1}{|\tilde{u}_\parallel - v|^{2-\gamma}}f_{1, \delta}(\tilde{u}_\parallel)\int_{\mathcal{B}_{(1+2\delta)R_1}(v_1)}\,dw_1f_{2, \delta}(w_1)\int_{\mathcal{B}_{(1+2\delta)R_2}(v_2)}\,dw_2f_{3, \delta}(w_2)\\
            &\qquad \times |w_1-w_2|^{\gamma-2}\int_{S_{w_1,w_2}}dw_1'\,\frac{h(\cos\theta_\omega)}{|\cos\theta_\omega|^\gamma}b(\cos\theta')\varphi_\delta(w_1+w_2-w_1')\varphi_\delta(\tilde{u}_\parallel + w_1' - v)\Theta_{v+E_{\tilde{u}_\parallel-v}}(w_1').
        \end{split}
    \end{align}
    
    The $\theta_\omega$ and $\theta'$ term in \eqref{eq:L_q_comp3} are explicitly given by
    \begin{align}\label{eq:theta'}
        |\cos\theta_\omega| = \frac{|\tilde{u}_\parallel-v|}{|\tilde{u}_\parallel+w_1'-2v|},\quad |\cos\frac{\theta'}{2}| = \frac{|w_1'-w_1|}{|w_1-w_2|}.
    \end{align}
    For a more graphical illustration, we refer to Figure \ref{fig:balls}.

    By the assumption $b(\cos\theta)\geq c_b$ on $\theta\in (\pi/4, 3\pi/4)$ and \eqref{eq:b cut2'}, we can write
    \begin{align}\label{eq:L_q_comp4}
        \begin{split}
            \eqref{eq:L_q_comp3}&\geq \int_{\mathcal{B}_{(1+2\delta)R_0}(v_0)}\,d\tilde{u}_\parallel\,\frac{1}{|\tilde{u}_\parallel - v|^{2-\gamma}}f_{1, \delta}(\tilde{u}_\parallel)\int_{\mathcal{B}_{(1+2\delta)R_1}(v_1)}\,dw_1f_{2, \delta}(w_1)\int_{\mathcal{B}_{(1+2\delta)R_2}(v_2)}\,dw_2f_{3, \delta}(w_2)\\
            &\qquad \times \sqrt{2-\sqrt{2}}\left(\frac{2}{\sqrt{2+\sqrt{2}}}\right)^\gamma c_b^2 |w_1-w_2|^{\gamma-2}\\
            &\qquad\qquad \times \int_{S_{w_1,w_2}}dw_1'\,\mathbf{1}_{\{\frac{1}{4}\pi<\theta'<\frac{3}{4}\pi\}}\mathbf{1}_{\{\frac{1}{8}\pi<\theta_\omega<\frac{3}{8}\pi\}}\varphi_\delta(w_1+w_2-w_1')\varphi_\delta(\tilde{u}_\parallel + w_1' - v)\Theta_{v+E_{\tilde{u}_\parallel-v}}(w_1').
        \end{split}
    \end{align}

    For later use, we estimate the distance between variables. For $v_0 = \frac{v_1+v_2}{2}$, $\tilde{u}_\parallel\in \mathcal{B}_{(1+2\delta)R_0}(v_0)$, $(w_1, w_2)\in \mathcal{B}_{(1+2\delta)R_1}(v_1)\times \mathcal{B}_{(1+2\delta)R_2}(v_2)$, $v\in \mathcal{B}_{\frac{R_0}{13}}(v_0)$, and $w_1'\in S_{w_1,w_2}$, $\frac{10}{13}R_0\leq |v_1-v_2|\leq \sqrt{3}R_0$, we have
    \begin{align*}
        &|\tilde{u}_\parallel - v|\leq |\tilde{u}_\parallel - v_0| + |v_0 - v|\leq (1+2\delta)\frac{R_0}{2} + \frac{R_0}{26},\\
        &\frac{10}{13}R_0 - (1+2\delta)\frac{R_0}{13}\leq |v_1-v_2| - |v_1-w_1| - |v_2-w_2|\leq |w_1-w_2|\leq \sqrt{3}R_0 + (1+2\delta)\frac{R_0}{13},\\
        &\begin{dcases}
            \frac{5}{13}R_0 - (1+2\delta)\frac{R_0}{26} - \frac{R_0}{26}\leq |w'_1 - \frac{w_1+w_2}{2}| - |\frac{w_1+w_2}{2} - v_0|\leq |w'_1 - v_0|,\\
            |w'_1 - v_0|\leq |w'_1 - \frac{w_1+w_2}{2}| + |\frac{w_1+w_2}{2} - v_0|\leq \frac{\sqrt{3}}{2}R_0 + (1+2\delta)\frac{R_0}{26} + \frac{R_0}{26},
        \end{dcases}\\
        &\left|\frac{w_1+w_2}{2} - v\right|\leq \left|\frac{w_1+w_2}{2} - v_0\right| + |v-v_0|\leq (1+2\delta)\frac{R_0}{26} + \frac{R_0}{26},\\
        &|w'_1 - v_0| - \frac{R_0}{26}\leq |w'_1 - v_0| - |v_0 - v|\leq |w'_1 - v|\leq |w'_1 - v_0| + |v_0 - v|\leq |w'_1 - v_0| + \frac{R_0}{26}.
    \end{align*}
    We choose
    \begin{align}\label{eq:delta_cond1}
        \delta\leq \frac{1}{104}.
    \end{align}
    Under this small $\delta$, we can choose the lower and upper bounds by
    \begin{align}\label{eq:dist_list}
        \begin{split}
            &|\tilde{u}_\parallel - v|\leq \frac{57}{104}R_0,\\
            &\frac{467}{676}R_0\leq |w_1-w_2|\leq \left(\sqrt{3} + \frac{53}{676}\right)R_0,\\
            &\frac{415}{1352}R_0\leq |w'_1 - v_0|\leq \left(\frac{\sqrt{3}}{2} + \frac{105}{1352}\right)R_0,\\
            &\left|\frac{w_1+w_2}{2} - v\right|\leq \frac{105}{1352}R_0,\\
            &\frac{363}{1352}R_0\leq |w'_1 - v|\leq \left(\frac{\sqrt{3}}{2} + \frac{157}{1352}\right)R_0.
        \end{split}
    \end{align}
    
    We also prove that $|v_0-v_{-1}|\leq 0.8R_0$ by a simple argument. Consider the triangles consisting of $\{v_{-1}, v_0, v_1\}$ and $\{v_{-1}, v_0, v_2\}$. One of the angles $\angle v_{-1}v_0v_1$ and $\angle v_{-1}v_0v_2$ is an obtuse angle, so we assume $\angle v_{-1}v_0v_1$ is without loss of generality. It means that
    \begin{align*}
        |v_1-v_{-1}|^2\geq |v_0-v_{-1}|^2 + |v_0-v_1|^2 = |v_0-v_{-1}|^2 + \left|\frac{v_2-v_1}{2}\right|^2\geq |v_0-v_{-1}|^2 + (\frac{5}{13}R_0)^2.
    \end{align*}
    Since $v_1\in \mathcal{Q}_{R_0}(v_{-1})$, $|v_1-v_{-1}|\leq \frac{\sqrt{3}}{2}R_0$, so we finally get
    \begin{align}\label{eq:dist_list2}
        |v_0-v_{-1}|\leq \sqrt{\frac{3}{4}R^2_0 - (\frac{5}{13}R_0)^2}\leq 0.8R_0.
    \end{align}
    
    By the construction of $\varphi_\delta$, it satisfies
    \begin{align*}
        \varphi_\delta\geq \mathbf{1}_{\mathcal{B}_{4R_0}(v_{-1})}-\mathbf{1}_{G_{3,\delta}^c\cap \mathcal{B}_{4R_0}(v_{-1})}.
    \end{align*}
    
    Using Lemma \ref{lem:Covering_lemma}, \eqref{eq:E_density}, \eqref{eq:G_O_construction_1}, and \eqref{eq:G_O_construction_2}, we cover the set $G_{3,\delta}^c\cap \mathcal{B}_{4R_0}(v_{-1})$ by a countable collection of closed balls $\{\mathcal{B}_i\}_{i=1}^\infty$ such that $G_{3,\delta}^c\cap \mathcal{B}_{4R_0}(v_{-1})\subset \cup_i \mathcal{B}_i$ and
    \begin{align}\label{eq:B_i_control}
        \sum_{i=1}^\infty |\mathcal{B}_i|\leq (1+\delta)|G_{3,\delta}^c\cap \mathcal{B}_{4R_0}(v_{-1})|&= (1+\delta)\left(|E^c\cap \mathcal{B}_{4R_0}(v_{-1})| + |E\setminus G_{3,\delta}|\right)\leq (1+\delta)^2\mathfrak{a}_0|\mathcal{B}_{4R_0}(v_{-1})|.
    \end{align}
    Using this covering, we have
    \begin{align}\label{eq:final_form}
        \begin{split}
            \eqref{eq:L_q_comp4}&=\int_{\mathcal{B}_{(1+2\delta)R_0}(v_0)}\,d\tilde{u}_\parallel\,\frac{1}{|\tilde{u}_\parallel - v|^{2-\gamma}}f_{1, \delta}(\tilde{u}_\parallel)\int_{\mathcal{B}_{(1+2\delta)R_1}(v_1)}\,dw_1f_{2, \delta}(w_1)\int_{\mathcal{B}_{(1+2\delta)R_2}(v_2)}\,dw_2f_{3, \delta}(w_2)\\
            &\qquad \times \sqrt{2-\sqrt{2}}\left(\frac{2}{\sqrt{2+\sqrt{2}}}\right)^\gamma c_b^2 |w_1-w_2|^{\gamma-2}\\
            &\qquad\qquad \times \int_{S_{w_1,w_2}}dw_1'\,\mathbf{1}_{\{\frac{1}{4}\pi<\theta'<\frac{3}{4}\pi\}}\mathbf{1}_{\{\frac{1}{8}\pi<\theta_\omega<\frac{3}{8}\pi\}}\varphi_\delta(w_1+w_2-w_1')\varphi_\delta(\tilde{u}_\parallel + w_1' - v)\Theta_{v+E_{\tilde{u}_\parallel-v}}(w_1')\\
            &\geq\int_{\mathcal{B}_{(1+2\delta)R_0}(v_0)}\,d\tilde{u}_\parallel\,\frac{1}{|\tilde{u}_\parallel - v|^{2-\gamma}}f_{1, \delta}(\tilde{u}_\parallel)\int_{\mathcal{B}_{(1+2\delta)R_1}(v_1)}\,dw_1f_{2, \delta}(w_1)\int_{\mathcal{B}_{(1+2\delta)R_2}(v_2)}\,dw_2f_{3, \delta}(w_2)|w_1-w_2|^{\gamma-2}\\
            &\qquad \times \sqrt{2-\sqrt{2}}\left(\frac{2}{\sqrt{2+\sqrt{2}}}\right)^\gamma c_b^2\int_{S_{w_1,w_2}}dw_1'\,\mathbf{1}_{\{\frac{1}{4}\pi<\theta'<\frac{3}{4}\pi\}}\mathbf{1}_{\{\frac{1}{8}\pi<\theta_\omega<\frac{3}{8}\pi\}}\Theta_{v+E_{\tilde{u}_\parallel-v}}(w_1')\\
            &\qquad\qquad \times \left(\mathbf{1}_{\mathcal{B}_{4R_0}(v_{-1})}-\mathbf{1}_{\cup_j \mathcal{B}_j}\right)(w_1+w_2-w_1')\left(\mathbf{1}_{\mathcal{B}_{4R_0}(v_{-1})}-\mathbf{1}_{\cup_j \mathcal{B}_j}\right)(\tilde{u}_\parallel + w_1' - v)\\
            &=\int_{\mathcal{B}_{(1+2\delta)R_0}(v_0)}\,d\tilde{u}_\parallel\,\frac{1}{|\tilde{u}_\parallel - v|^{2-\gamma}}f_{1, \delta}(\tilde{u}_\parallel)\int_{\mathcal{B}_{(1+2\delta)R_1}(v_1)}\,dw_1f_{2, \delta}(w_1)\int_{\mathcal{B}_{(1+2\delta)R_2}(v_2)}\,dw_2f_{3, \delta}(w_2)|w_1-w_2|^{\gamma-2}\\
            &\qquad \times \sqrt{2-\sqrt{2}}\left(\frac{2}{\sqrt{2+\sqrt{2}}}\right)^\gamma c_b^2\int_{S_{w_1,w_2}}dw_1'\,\mathbf{1}_{\{\frac{1}{4}\pi<\theta'<\frac{3}{4}\pi\}}\mathbf{1}_{\{\frac{1}{8}\pi<\theta_\omega<\frac{3}{8}\pi\}}\Theta_{v+E_{\tilde{u}_\parallel-v}}(w_1')\\
            &\qquad\qquad \times \mathbf{1}_{\mathcal{B}_{4R_0}(v_{-1})}(w_1+w_2-w_1')\mathbf{1}_{\mathcal{B}_{4R_0}(v_{-1})}(\tilde{u}_\parallel + w_1' - v)\\
            &\quad -\int_{\mathcal{B}_{(1+2\delta)R_0}(v_0)}\,d\tilde{u}_\parallel\,\frac{1}{|\tilde{u}_\parallel - v|^{2-\gamma}}f_{1, \delta}(\tilde{u}_\parallel)\int_{\mathcal{B}_{(1+2\delta)R_1}(v_1)}\,dw_1f_{2,\delta}(w_1)\int_{\mathcal{B}_{(1+2\delta)R_2}(v_2)}\,dw_2f_{3, \delta}(w_2)|w_1-w_2|^{\gamma-2}\\
            &\qquad \times \sqrt{2-\sqrt{2}}\left(\frac{2}{\sqrt{2+\sqrt{2}}}\right)^\gamma c_b^2\int_{S_{w_1,w_2}}\,dw_1' \mathbf{1}_{\{\frac{1}{4}\pi<\theta'<\frac{3}{4}\pi\}}\mathbf{1}_{\{\frac{1}{8}\pi<\theta_\omega<\frac{3}{8}\pi\}}\Theta_{v+E_{\tilde{u}_\parallel-v}}(w_1')\\
            &\qquad \qquad \times \left(\mathbf{1}_{\cup_j \mathcal{B}_j}(w_1+w_2-w_1')+\mathbf{1}_{\cup_j \mathcal{B}_j}(\tilde{u}_\parallel + w_1' - v) - \mathbf{1}_{\cup_j \mathcal{B}_j}(w_1+w_2-w_1')\mathbf{1}_{\cup_j \mathcal{B}_j}(\tilde{u}_\parallel + w_1' - v)\right)
        \end{split}
    \end{align}
    If $\tilde{u}_\parallel\in \mathcal{B}_{(1+2\delta)R_0}(v_0)$ and $(w_1,w_2)\in \mathcal{B}_{(1+2\delta)R_1}(v_1)\times \mathcal{B}_{(1+2\delta)R_2}(v_2)$, by \eqref{eq:dist_list} and \eqref{eq:dist_list2},
    \begin{align*}
        |w_1+w_2-w_1'-v_{-1}|&\leq |w_1-v_1|+|w_2-v_2| + |v_1+v_2 - w_1'-v_0| + |v_0-v_{-1}|\\
        &=|w_1-v_1|+|w_2-v_2| + |- w_1' + v_0| + |v_0-v_{-1}|\\
        &\leq (1+2\delta)R_1 +  \left(\frac{\sqrt{3}}{2} + \frac{105}{1352}\right)R_0 + 0.8R_0< 2R_0
    \end{align*}
    for $\delta$ satisfying \eqref{eq:delta_cond1}. Since $\tilde{u}_\parallel - v$ and $w_1' - v$ are perpendicular to each other, using \eqref{eq:dist_list} and \eqref{eq:dist_list2},
    \begin{align*}
        |\tilde{u}_\parallel + w_1' - v - v_{-1}|&\leq |(\tilde{u}_\parallel - v) + (w_1' - v)| + |v - v_0| + |v_0 - v_{-1}|\\
        &=\sqrt{(\tilde{u}_\parallel - v)^2 + (w_1' - v)^2} + |v - v_0| + |v_0 - v_{-1}|\\
        &\leq \sqrt{(|\tilde{u}_\parallel - v_0| + |v_0 - v|)^2 + (|w_1' - v_0| + |v_0 - v|)^2} + |v - v_0| + |v_0 - v_{-1}|\\
        &\leq \sqrt{\left(\frac{57}{104}R_0\right)^2 + \left(\frac{\sqrt{3}}{2}R_0 + \frac{157}{1352}R_0\right)^2} + \frac{1}{26}R_0 + 0.8R_0<2R_0.
    \end{align*}
    
    Therefore, we obtain $w_1+w_2-w_1', \tilde{u}_\parallel + w_1' - v\in \mathcal{B}_{4R_0}(v_{-1})$ for $(\tilde{u}_\parallel, w_1, w_2)\in \mathcal{B}_{(1+2\delta)R_0}(v_0)\times \mathcal{B}_{(1+2\delta)R_1}(v_1)\times \mathcal{B}_{(1+2\delta)R_2}(v_2)$ under \eqref{eq:delta_cond1}. So, we can write the first integral in \eqref{eq:final_form} by (ignoring the constant)
    \begin{align}\label{eq:final_form_1}
        \begin{split}
            &\int_{\mathcal{B}_{(1+2\delta)R_0}(v_0)}\,d\tilde{u}_\parallel\,\frac{1}{|\tilde{u}_\parallel - v|^{2-\gamma}}f_{1, \delta}(\tilde{u}_\parallel)\int_{\mathcal{B}_{(1+2\delta)R_1}(v_1)}\,dw_1f_{2, \delta}(w_1)\int_{\mathcal{B}_{(1+2\delta)R_2}(v_2)}\,dw_2f_{3, \delta}(w_2)|w_1-w_2|^{\gamma-2}\\
            &\qquad \times \int_{S_{w_1,w_2}}dw_1'\,\mathbf{1}_{\{\frac{1}{4}\pi<\theta'<\frac{3}{4}\pi\}}\mathbf{1}_{\{\frac{1}{8}\pi<\theta_\omega<\frac{3}{8}\pi\}}\Theta_{v+E_{\tilde{u}_\parallel-v}}(w_1')\mathbf{1}_{\mathcal{B}_{4R_0}(v_{-1})}(w_1+w_2-w_1')\mathbf{1}_{\mathcal{B}_{4R_0}(v_{-1})}(\tilde{u}_\parallel + w_1' - v)\\
            &=\int_{\mathcal{B}_{(1+2\delta)R_0}(v_0)}\,d\tilde{u}_\parallel\,\frac{1}{|\tilde{u}_\parallel - v|^{2-\gamma}}f_{1, \delta}(\tilde{u}_\parallel)\int_{\mathcal{B}_{(1+2\delta)R_1}(v_1)}\,dw_1f_{2, \delta}(w_1)\int_{\mathcal{B}_{(1+2\delta)R_2}(v_2)}\,dw_2f_{3, \delta}(w_2)|w_1-w_2|^{\gamma-2}\\
            &\qquad \times \int_{S_{w_1,w_2}}dw_1'\,\mathbf{1}_{\{\frac{1}{4}\pi<\theta'<\frac{3}{4}\pi\}}\mathbf{1}_{\{\frac{1}{8}\pi<\theta_\omega<\frac{3}{8}\pi\}}\Theta_{v+E_{\tilde{u}_\parallel-v}}(w_1').
        \end{split}
    \end{align}
    
    Now, we need to bound the second integral in \eqref{eq:final_form}: (again, ignoring the constant)
    \begin{align}\label{eq:final_form_2}
        \begin{split}
            &\int_{\mathcal{B}_{(1+2\delta)R_0}(v_0)}\,d\tilde{u}_\parallel\,\frac{1}{|\tilde{u}_\parallel - v|^{2-\gamma}}f_{1, \delta}(\tilde{u}_\parallel)\int_{\mathcal{B}_{(1+2\delta)R_1}(v_1)}\,dw_1f_{2,\delta}(w_1)\int_{\mathcal{B}_{(1+2\delta)R_2}(v_2)}\,dw_2f_{3, \delta}(w_2)|w_1-w_2|^{\gamma-2}\\
            &\qquad \times \int_{S_{w_1,w_2}}\,dw_1' \mathbf{1}_{\{\frac{1}{4}\pi<\theta'<\frac{3}{4}\pi\}}\mathbf{1}_{\{\frac{1}{8}\pi<\theta_\omega<\frac{3}{8}\pi\}}\Theta_{v+E_{\tilde{u}_\parallel-v}}(w_1')\\
            &\qquad \qquad \times \left(\mathbf{1}_{\cup_j \mathcal{B}_j}(w_1+w_2-w_1')+\mathbf{1}_{\cup_j \mathcal{B}_j}(\tilde{u}_\parallel + w_1' - v) - \mathbf{1}_{\cup_j \mathcal{B}_j}(w_1+w_2-w_1')\mathbf{1}_{\cup_j \mathcal{B}_j}(\tilde{u}_\parallel + w_1' - v)\right).
        \end{split}
    \end{align}
    In other words, we will bound the set such that $w_1+w_2-w_1'\in \cup_j \mathcal{B}_j$ or $\tilde{u}_\parallel + w_1' - v\in \cup_j \mathcal{B}_j$.\\

    Suppose $\mathcal{B}_{R_3}(v_3)\in \{\mathcal{B}_i\}$. By \eqref{eq:B_i_control}, we have
    \begin{align*}
        |\mathcal{B}_{R_3}(v_3)|\leq (1+\delta)^2\mathfrak{a}_0|\mathcal{B}_{4R_0}(v_{-1})| = \frac{4}{3}\pi(1+\delta)^2\mathfrak{a}_0 (2R_0)^3,
    \end{align*}
    so
    \begin{align}\label{eq:R_3}
        R_3 = 2\left(\frac{3}{4\pi}|\mathcal{B}_{R_3}(v_3)|\right)^{1/3}\leq 4(1+\delta)^{2/3}\mathfrak{a}_0^{1/3}R_0.
    \end{align}
    
    To bound \eqref{eq:final_form_2}, we need to use the geometric properties of the Calerman representation. From now on, we state and prove some technical geometric lemmas to bound the integral.

    The next lemma is a slight modification of Lemma 5.3 in \cite{LPY2024}.
    \begin{lemma}\label{lem:intersection_lemma1}
        Let $a\in\mathbb{R}^3$, and $r\geq 0$ satisfies $r\leq |a|/4$. Then, there exists a set $C_{a, r}$ which satisfies
        \begin{align*}
            \{x:\{y:y\perp x\}\cap \{y:|a+y|\leq r\}\neq \emptyset\} \subset C_{a,r}
        \end{align*}
        and for any isotropic function $f(x) = f(|x|)$,
        \begin{align*}
            \int_{\{|x|\leq R\}\cap C_{a, r}} f(x)\,dx \leq \frac{5}{2}\pi\frac{r}{|a|}\int_{\{|x|\leq R\}}f(x)\,d|x|
        \end{align*}
        for any $R>0$.
    \end{lemma}
    \begin{proof}
        We will slightly modify the proof in \cite{LPY2024}. By rotation, it is enough to assume $a = (0,0,|a|)$. $|a+y|\leq r$ is a closed ball centered at $-a$ with radius $r$ about $y$. When the distance between $-a$ and the plane $\{y:y\perp x\}$ is smaller than $r$, the plane intersects with the ball. For fixed $x$, the distance is given by
        \begin{align*}
            \left|a\cdot \hat{x}\right| = \left|a\cdot \hat{x}\right|\leq r.
        \end{align*}
        It shows that the sufficient condition to make an intersection is 
        \begin{align*}
            -r\leq a\cdot \hat{x}\leq r.
        \end{align*}
        We consider the spherical coordinates with $\theta$ to be the angle about the $z$-axis. Let $\theta_0 = \frac{\pi}{2}$ and $\delta\theta= \frac{5}{4}\frac{r}{|a|}$. Since $\delta\theta\leq \frac{5}{16}$,
        \begin{align*}
            \cos(\theta_0+\delta\theta)& = - \sin \delta\theta\leq -\frac{r}{|a|},\\
            \cos(\theta_0-\delta\theta)&= \sin \delta\theta \geq \frac{r}{|a|}.
        \end{align*}
        Therefore, we have
        \begin{align*}
            \{x:-r\leq a\cdot \hat{x}\leq r\}\subset \{x:\theta\in [\theta_0-\delta\theta, \theta_0+\delta\theta]\}\eqqcolon C_{a, r}.
        \end{align*}
    
        Finally, for any isotropic function $f(x) = f(|x|)$ and $R\geq 0$,
        \begin{align*}
            \int_{\{|x|\leq R\}\cap C_{a, r}} f(x)\,dx &\leq 2\pi\int_0^R f(|x|)|x|^2\,d|x|\int_{\theta_0-\delta\theta}^{\theta_0+\delta\theta}\sin\theta\,d\theta\\
            &\leq 4\pi\delta\theta\int_0^{R}f(|x|)|x|^2\,d|x|\\
            &=2\pi\delta\theta\int_0^{R}\int_0^\pi f(|x|)|x|^2\sin\theta\,d\theta d|x|\\
            &= \frac{5}{2}\pi\frac{r}{|a|}\int_{|x|\leq R}f(x)\,dx.
        \end{align*}
    \end{proof}
    
    The next lemma bounds the size of the set $\{(w_1,w_2)\in \mathcal{Q}_R(v_1)\times \mathcal{Q}_R(v_2)\}$ making $S_{w_1,w_2}\cap (\cup_j \mathcal{B}_j)\neq \emptyset$.
    \begin{lemma}\label{lem:intersection_lemma2}
    Let $|v_1-v_2|\geq \frac{10}{13}R_0$, $\frac{R_0}{13}\leq R\leq \frac{2R_0}{13}$, $R_3\leq \frac{1}{104}R_0$, and $v_3$ be an arbitrary point in $\mathbb{R}^3$. Then,
        \begin{align*}
            \frac{|\{(w_1,w_2)\in \mathcal{Q}_R(v_1)\times \mathcal{Q}_R(v_2):S_{w_1,w_2}\cap \mathcal{B}_{R_3}(v_3)\neq \emptyset\}|}{|\mathcal{Q}_{R}(v_1)||\mathcal{Q}_{R}(v_2)|}\leq C\frac{R_3}{R_0}
        \end{align*}
        for some constant $C$.
    \end{lemma}
    \begin{proof}
        Let
        \begin{align*}
            A \coloneqq \{(w_1,w_2)\in \mathcal{Q}_R(v_1)\times \mathcal{Q}_R(v_2):S_{w_1,w_2}\cap \mathcal{B}_{R_3}(v_3)\neq \emptyset\}.
        \end{align*}
        We choose new coordinates by
        \begin{align*}
            U = w_1 + w_2,\quad V = w_1-w_2.
        \end{align*}
        Using these coordinates, we can rephrase the condition $S_{w_1,w_2}\cap \mathcal{B}_{R_3}(v_3)\neq \emptyset$ by
        \begin{align*}
            \left|\left|\frac{U}{2} - v_3\right|-\frac{|V|}{2}\right|\leq \frac{R_3}{2}
        \end{align*}
        as $|V|>R_3$. For fixed $U$, it is equivalent to say that
        \begin{align*}
            \left|\frac{U}{2} - v_3\right| - \frac{R_3}{2}\leq \frac{|V|}{2}\leq \left|\frac{U}{2} - v_3\right| + \frac{R_3}{2}.
        \end{align*}
    
        We write $\mathcal{Q}_{R}(v_1) = [a_{x,1}, a_{x,2}]\times [a_{y,1}, a_{y,2}]\times [a_{z,1}, a_{z,2}]$ and $\mathcal{Q}_{R}(v_2) = [b_{x,1}, b_{x,2}]\times [b_{y,1}, b_{y,2}]\times [b_{z,1}, b_{z,2}]$. Then, the domain of $U$ is given by
        \begin{align*}
            [a_{x,1} + b_{x,1}, a_{x,2} + b_{x,2}]\times [a_{y,1} + b_{y,1}, a_{y,2} + b_{y,2}]\times [a_{z,1} + b_{z,1}, a_{z,2} + b_{z,2}].
        \end{align*}
        For each fixed $U$ in the domain, the $V$ domain is given by
        \begin{align*}
            V\in X\times Y\times Z,
        \end{align*}
        where
        \begin{align*}
            X&=\begin{cases}
                [2a_{x,1} - U_x, U_x - 2b_{x,1}] & U_x\leq a_{x,1} + b_{x,2},\\
                [U_x - 2b_{x,2}, 2a_{x,2} - U_x] & U_x\geq a_{x,1} + b_{x,2},
            \end{cases}\\
            Y&=\begin{cases}
                [2a_{y,1} - U_y, U_y - 2b_{y,1}] & U_y\leq a_{y,1} + b_{y,2},\\
                [U_x - 2b_{y,2}, 2a_{y,2} - U_y] & U_y\geq a_{y,1} + b_{y,2},
            \end{cases}\\
            Z&=\begin{cases}
                [2a_{z,1} - U_z, U_z - 2b_{z,1}] & U_z\leq a_{z,1} + b_{z,2},\\
                [U_z - 2b_{z,2}, 2a_{z,2} - U_z] & U_z\geq a_{z,1} + b_{z,2}.
            \end{cases}
        \end{align*}
        $X\times Y\times Z$ a cuboid centered at $(a_{x,1}-b_{x,1}, a_{y,1}-b_{y,1}, a_{z,1}-b_{z,1})$ having side length at most $2(a_{x,2}-a_{x,1}) = 2R$.
    
        Now, we will compute the volume of the set
        \begin{align*}
            A_2\coloneqq \left\{V:\left|\frac{U}{2} - v_3\right| - \frac{R_3}{2}\leq \frac{|V|}{2}\leq \left|\frac{U}{2} - v_3\right| + \frac{R_3}{2}\right\}\cap \{V:V\in X\times Y\times Z\}.
        \end{align*}
        By the assumption of the lemma,
        \begin{align*}
            \frac{10-2\sqrt{3}}{13}R_0\leq \frac{10}{13}R_0 - \sqrt{3}R\leq |V|\leq \frac{10}{13}R_0 + \sqrt{3}R\leq \frac{10+2\sqrt{3}}{13}R_0.
        \end{align*}
        It implies that
        \begin{align*}
            2\left(\left|\frac{U}{2} - v_3\right|-\frac{R_3}{2}\right)&\geq 2\left(\frac{|V|}{2}-R_3\right)\geq \frac{39-8\sqrt{3}}{52}R_0\\
            2\left(\left|\frac{U}{2} - v_3\right|+\frac{R_3}{2}\right)&\leq 2\left(\frac{|V|}{2}+R_3\right)\leq \frac{41+8\sqrt{3}}{52}R_0
        \end{align*}
        in $A_2$. We compute the upper bound of the area of the intersection between the spherical shell
        \begin{align*}
            \left\{V:\max\left\{2\left(\left|\frac{U}{2} - v_3\right| - \frac{R_3}{2}\right), \frac{39-8\sqrt{3}}{52}R_0\right\}\leq |V|\leq 2\left(\left|\frac{U}{2} - v_3\right| + \frac{R_3}{2}\right)\right\}
        \end{align*}
        and $X\times Y\times Z$. To make the analysis easier, we cover $X\times Y\times Z$ by a sphere centered at $(a_{x,1}-b_{x,1}, a_{y,1}-b_{y,1}, a_{z,1}-b_{z,1})$ with diameter $2\sqrt{3}R$. Figure \ref{fig:spherical_shell_intersection} illustrates the intersection between a spherical shell with radius $|V|$ and a ball with diameter $2\sqrt{3}R$. Let $2\theta$ be the angle of the intersection arc in the section passing through the centers of the spheres.
        
        To bound the area of the spherical cap in Figure \ref{fig:spherical_shell_intersection}, we first bound the $\theta$. The triangle, which consists of two intersection points and the center of the sphere, has two side lengths $|V|$ and one side length bonded by $2\sqrt{3}R$. Therefore, the $\theta$ satisfies
        \begin{align*}
            \cos(2\theta)\geq \frac{2|V|^2 - (2\sqrt{3}R)^2}{2|V|^2}\geq 1-\frac{1}{2}\left(\frac{2\sqrt{3}R}{\frac{39-8\sqrt{3}}{52}R_0}\right)^2 \geq 1-\frac{1}{2}\left(\frac{16\sqrt{3}}{39-8\sqrt{3}}\right)^2.
        \end{align*}
        Therefore, the area of the spherical cap is bounded by
        \begin{align*}
            2\pi r^2(1-\cos\theta)\leq \frac{1}{3}\pi \left(\frac{41+8\sqrt{3}}{52}R_0\right)^2.
        \end{align*}
        The volume of $A_2$ is bounded by the area of the spherical cap times $2R_3$, which is the interval length of the possible $|V|$. Thus, we get
        \begin{align*}
            |A_2|\leq \frac{1}{3}\pi \left(\frac{41+8\sqrt{3}}{52}R_0\right)^2(2R_3).
        \end{align*}
        
        The total volume of the set $A$ is bounded by the volume of the domain of $U$ times the upper bound of the volume of the set $|A_2|$, so
        \begin{align*}
            \frac{|A|}{|\mathcal{Q}_R(v_1)||\mathcal{Q}_R(v_2)|}\leq \frac{(2R)^3 \frac{1}{3}\pi \left(\frac{41+8\sqrt{3}}{52}R_0\right)^2(2R_3)}{|\mathcal{Q}_R(v_1)||\mathcal{Q}_R(v_2)|}\leq \frac{\frac{2^4}{3}\pi \left(\frac{41+8\sqrt{3}}{52}R_0\right)^2 R_3}{R^3} \leq C\frac{R_3}{R_0}
        \end{align*}
        for some constant $C$.
        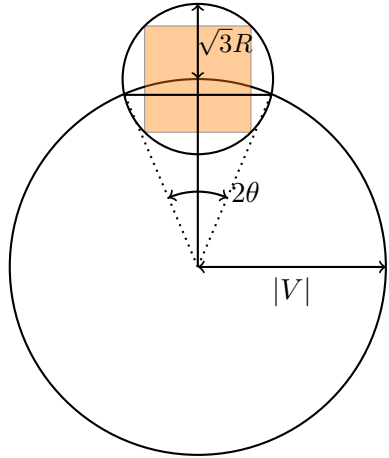
\begin{figure}[htbp]
            \centering
            \begin{tikzpicture}
                \draw[thick](0,0) circle (2.5);
                \draw[fill = orange, opacity=0.4] (-1/1.41421356237, 2.5-1/1.41421356237) rectangle (1/1.41421356237, 2.5+1/1.41421356237);
                \draw[thick](0,3-0.5) circle (1);
                \draw[thick] (0,0) -- (0,2.5);
                \draw[dotted, thick, rotate=-23] (0,0) -- (0,2.5);
                \draw[dotted, thick, rotate=23] (0,0) -- (0,2.5);
                \draw[<->, thick]  (23+90:1) arc(23+90:-23+90:1) node[midway,right, xshift = 0.3cm]{$2\theta$};
                \draw[<->, thick] (0,0) -- (2.5,0) node[midway, below]{$|V|$};
                \draw[thick] (-0.98,2.29) -- (.98,2.29);
                \draw[<->, thick, rotate around={90:(0,2.5)}] (0,2.5) -- (1,2.5) node[midway, right, xshift = -0.15cm]{\small $\sqrt{3}R$};
            \end{tikzpicture}
            \caption{Intersection of spherical shell with radius $\|V\|$ and a ball with diameter $2\sqrt{3}R$.}
            \label{fig:spherical_shell_intersection}
        \end{figure}
    \end{proof}
    
    Let $C_{w_1,w_2}$ be a circle having antipodals $w_1$ and $w_2$ in $\mathbb{R}^3$. Note that it is not unique in $\mathbb{R}^3$. The next lemma is to bound the size of the set $C_{w_1,w_2}\cap \mathcal{B}_{R_3}(v_3)$.
    \begin{lemma}\label{lem:intersection_lemma3}
        Suppose $|v_1-v_2|\geq \frac{10}{13}R_0$, $(w_1,w_2)\in \mathcal{B}_R(v_1)\times \mathcal{B}_R(v_2)$, $\frac{R_0}{13}\leq R\leq \frac{2R_0}{13}$, and $R_3<\frac{R_0}{104}$. Let $v_3$ be an arbitrary point in $\mathbb{R}^3$. Then, $|C_{w_1,w_2}\cap \mathcal{B}_{R_3}(v_3)|\leq \frac{3}{\sqrt{2}}R_3$.
    \end{lemma}
    \begin{proof}
        The intersection $|C_{w_1,w_2}\cap \mathcal{B}_{R_3}(v_3)|$ is maximized when the circle and $v_3$ is on the same plane, so we assume $v_3$ is on the plane. Suppose $C_{w_1,w_2}\cap \mathcal{B}_{R_3}(v_3)\neq \emptyset$, so $A$ and $B$ be the two intersection points between $C_{w_1,w_2}$ and $\partial \mathcal{B}_{R_3}(v_3)$. Let $2\theta = \angle A w_0 B$, where $w_0 = \frac{w_1+w_2}{2}$. Since $\frac{8}{13}R_0\leq |w_1-w_2|\leq \frac{12}{13}R_0$, we have
        \begin{align*}
            1-\theta^2\geq \cos 2\theta\geq 1-\frac{1}{2}\left(\frac{R_3}{4R_0/13}\right)^2,
        \end{align*}
        so $\theta\leq \frac{1}{\sqrt{2}}\frac{R_3}{4R_0/13}$. Therefore,
        \begin{align*}
            l = 2r\theta \leq \frac{1}{\sqrt{2}}\frac{12R_0}{13}\frac{R_3}{4R_0/13}\leq \frac{3}{\sqrt{2}}R_3.
        \end{align*}
    \end{proof}

    Now, we state the main lemma.
    \begin{lemma}\label{lem:final_form_bound}
        We have
        \begin{align}\label{eq:final_form_bound}
            \begin{split}
                &\iint_{\mathcal{B}_{(1+2\delta)R_0}(v_0)\times \mathcal{B}_{(1+2\delta)R_1}(v_1)\times \mathcal{B}_{(1+2\delta)R_2}(v_2)}d\tilde{u}_\parallel dw_1dw_2\,\int_{S_{w_1,w_2}}dw_1'\,\Theta_{v+E_{\tilde{u}_\parallel+v}}(w_1')\\
                &\qquad\times \left(\mathbf{1}_{\cup_j \mathcal{B}_j}(w_1+w_2-w_1')+\mathbf{1}_{\cup_j \mathcal{B}_j}(\tilde{u}_\parallel + w_1' - v) -\mathbf{1}_{\cup_j \mathcal{B}_j}(w_1+w_2-w_1')\mathbf{1}_{\cup_j \mathcal{B}_j}(\tilde{u}_\parallel + w_1' - v)\right)\\
                &\leq CR_0^7\sum_{j=1}^\infty |\mathcal{B}_j|
            \end{split}
        \end{align}
        for $\delta$ satisfying \eqref{eq:delta_cond1}, 
        \begin{align}\label{eq:alpha0_cond1}
            \mathfrak{a}_0\leq \frac{1}{4^3(1+1/104)^2}\left(\frac{1}{104}\right)^3,
        \end{align}
        and an uniform constant $C$ about $\delta$ and $\mathfrak{a}_0$ in the region.
    \end{lemma}
    
\begin{figure}[htbp]
    \centering
    \begin{tikzpicture}
        \def\R{3.5}
        \def\myangle{40}
        \draw[thick, fill=black!20](0,\R) circle (\R/10) node[above right, xshift = 0.2cm, yshift = -0.1cm]{$w_1\in B_{R_1}(v_1)$};
        \draw[thick, fill=black!20](0,-\R) circle (\R/10) node[below right, xshift = 0.2cm, yshift = 0.1cm]{$w_2\in B_{R_2}(v_2)$};
        \draw[thick](0,0) circle (\R);
        \draw[thick, fill=black!20](0,0) circle (\R/10) node[below left, xshift = -0.2cm, yshift = -0.1cm]{$v\in B_{R_1}(v_0)$};
        \draw[thick] (-5,-0.3+2.5) -- (5,-0.3-2.5);
        \draw[] (5,-2.6) node[right]{$v+E_{\tilde{u}_\parallel-v}$};
        \draw[thick, ->](0,-0.3) -- (2.5/3,5/3) node[right]{$\tilde{u}_\parallel$};
        \draw[thick, fill=black!20](4.5,0.7) circle (\R/20) node[above right, xshift = 0.1cm, yshift = -0.2cm]{$B_{R_3}(v_3)$};
        \draw[] ({sqrt(6089)/25 - 3/25}, {-6/25 - sqrt(6089)/50}) node[right, xshift = 0.1cm, yshift = 0.15cm]{$w_1'\in S_{w_1,w_2}$};
        \draw[thick]({sqrt(6089)/25 - 3/25}, {-6/25 - sqrt(6089)/50}) -- (2.5/3,5/3) node[right]{$\tilde{u}_\parallel$};
        \draw[thick, dotted](0, \R) -- (0, -\R);
        \draw[thick](0, -\R) -- ({sqrt(6089)/25 - 3/25}, {-6/25 - sqrt(6089)/50});
        \draw[thick, ->](0,-0.3) -- (2.5/3,5/3) node[right]{$\tilde{u}_\parallel$};
        \draw[thick] (2.5/3,5/3) -- (2.5/3,5/3) ++(-114:0.6) arc(-116:-58:0.6) node[midway, xshift = 0.2cm, yshift = -0.25cm]{$\theta_\omega$};
        \draw[thick] (0, -\R) -- (0, -\R) ++(90:0.6) arc(90:30:0.6) node[midway, xshift = 0.15cm, yshift = 0.25cm]{$\tilde{\theta}_\omega$};
        \draw[thick, dotted] ({sqrt(6089)/25 - 3/25}, {-6/25 - sqrt(6089)/50}) -- (0, 0);
        \draw[thick] (0, 0) -- (0, 0) ++(-90:0.6) arc(-90:-30:0.6) node[midway, xshift = 0.05cm, yshift = -0.2cm]{$\theta'$};
    \end{tikzpicture}
    \caption{Geometry of the balls used in the proof.}
    \label{fig:balls}
\end{figure}
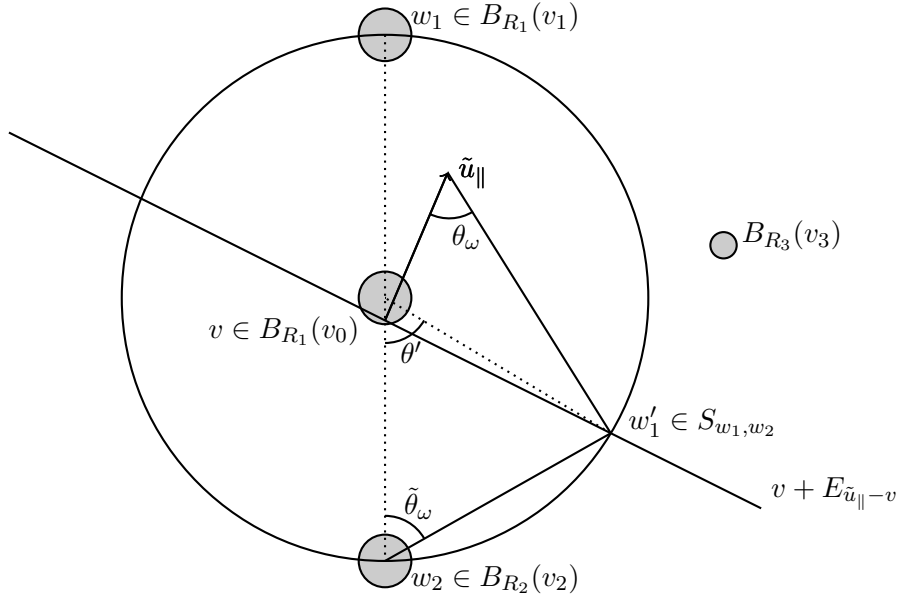

\begin{proof}
    Figure \ref{fig:balls} illustrates the position and roles of the balls in the integration and the proof of this lemma. We first note that for any $B_{R_3}(v_3)\in \cup_j \mathcal{B}_j$, by \eqref{eq:R_3} and \eqref{eq:alpha0_cond1}, we have
    \begin{align}\label{eq:R_3_cond1}
        R_3\leq \frac{1}{104}R_0.
    \end{align}
    
    We divide the cases (i) $w_1+w_2-w_1'\in \cup_j \mathcal{B}_j$ and (ii) $\tilde{u}_\parallel + w_1' - v\in \cup_j \mathcal{B}_j$.\\
    
    \noindent (i) By Fubini's theorem, for the first term in \eqref{eq:final_form_bound}, we have
    \begin{align*}
        &\iint_{\mathcal{B}_{(1+2\delta)R_0}(v_0)\times \mathcal{B}_{(1+2\delta)R_1}(v_1)\times \mathcal{B}_{(1+2\delta)R_2}(v_2)}d\tilde{u}_\parallel dw_1dw_2\,\int_{S_{w_1,w_2}}\,dw_1' \mathbf{1}_{\cup_j \mathcal{B}_j}(w_1+w_2-w_1')\Theta_{v+E_{\tilde{u}_\parallel-v}}(w_1')\\
        &=\iint_{\mathcal{B}_{(1+2\delta)R_1}(v_1)\times \mathcal{B}_{(1+2\delta)R_2}(v_2)}\,dw_1dw_2 \int_{\mathcal{B}_{(1+2\delta)R_0}(v_0)}d\tilde{u}_\parallel\,\int_{S_{w_1,w_2}}dw_1'\, \mathbf{1}_{\cup_j \mathcal{B}_j}(w_1+w_2-w_1')\Theta_{v+E_{\tilde{u}_\parallel-v}}(w_1').
    \end{align*}
    
    We choose a ball $\mathcal{B}_{R_3}(v_3) = \mathcal{B}_j \in \cup_j\mathcal{B}_j$. Define
    \begin{align*}
        \mathcal{B}_{1, j,1}&\coloneqq \{(w_1,w_2)\in \mathcal{B}_{(1+2\delta)R_1}(v_1)\times \mathcal{B}_{(1+2\delta)R_2}(v_2):S_{w_1,w_2}\cap (w_1+w_2-\mathcal{B}_j)\neq \emptyset\},\\
        \mathcal{B}_{1, j,2}(w_1,w_2) &\coloneqq \{\tilde{u}_\parallel\in \mathcal{B}_{(1+2\delta)R_0}(v_0):(v+E_{\tilde{u}_\parallel-v})\cap (w_1+w_2-\mathcal{B}_j)\neq \emptyset\}.
    \end{align*}
    Note that the first $\mathcal{B}_{1, j,1}$ can be defined independent to $\tilde{u}_\parallel$.
    
    For any $\tilde{u}_\parallel$, there exists a intersection between $S_{w_1,w_2}$ and $w_1+w_2-\mathcal{B}_j$ if and only if there is a intersection between $S_{w_1,w_2}$ and $\mathcal{B}_j$ considering point symmetry for $\frac{w_1+w_2}{2}$. By Lemma \ref{lem:intersection_lemma2}, therefore, we obtain
    \begin{equation}\label{eq:1_intersection_1}
        \begin{split}
            |\mathcal{B}_{1, j,1}| &= |\{(w_1,w_2)\in \mathcal{B}_{(1+2\delta)R_1}(v_1)\times \mathcal{B}_{(1+2\delta)R_2}(v_2):S_{w_1,w_2}\cap \mathcal{B}_j\neq \emptyset\}|\\
            &\leq |\{(w_1,w_2)\in \mathcal{Q}_{(1+2\delta)R_1}(v_1)\times \mathcal{Q}_{(1+2\delta)R_2}(v_2):S_{w_1,w_2}\cap \mathcal{B}_j\neq \emptyset\}|\\
            &\leq C_1(1+2\delta)^6 R_0^5R_3
        \end{split}
    \end{equation}
    for some constant $C_1$.
    
    For fixed $w_1$ and $w_2$ satisfying $S_{w_1,w_2}\cap \mathcal{B}_{R_3}(v_3) = S_{w_1,w_2}\cap \mathcal{B}_j\neq\emptyset$, the distance between $v_3$ and $v\in \mathcal{B}_{\frac{R_0}{13}}(v_0)$ is larger than $\frac{3}{13}R_0$. Indeed, let $w\in S_{w_1,w_2}\cap \mathcal{B}_j$. By \eqref{eq:dist_list},
    \begin{align*}
        |w-v|&\geq \frac{363}{1352}R_0
    \end{align*}
    for $\delta$ satisfying \eqref{eq:delta_cond1}. Therefore, the distance between $v_3$ and $v$ is lower-bounded by
    \begin{align*}
        |v_3-v|\geq |w-v|-|v_3-w| \geq \frac{363}{1352}R_0 - \frac{R_3}{2}\geq \frac{3}{13}R_0.
    \end{align*}
    Combining it with \eqref{eq:delta_cond1}, we also get
    \begin{align*}
        |w_1+w_2 - v_3-v|&\geq |v-v_3| - 2\left|\frac{w_1+w_2}{2} - v\right|\geq \frac{363}{1352}R_0 - \frac{R_3}{2} - \frac{210}{1352}R_0\\
        &\geq \frac{1}{13}R_0
    \end{align*}
    for $R_3$ satisfying \eqref{eq:R_3_cond1}.

    We enlarge $\mathcal{B}_{(1+2\delta)R_0}(v_0)$ by $\mathcal{B}_{\frac{15}{13}R_0}(v)$ as $|v-v_0|\leq \frac{R_0}{26}$ and \eqref{eq:delta_cond1} and define
    \begin{align*}
        \mathcal{B}_{1, j,2}'(w_1,w_2) &\coloneqq \{\tilde{u}_\parallel\in \mathcal{B}_{\frac{15}{13}R_0}(v):(v+E_{\tilde{u}_\parallel-v})\cap (w_1+w_2-\mathcal{B}_j)\neq \emptyset\}.
    \end{align*}
    Since the Borel measure is translation-invariant, we can write
    \begin{align*}
        |\mathcal{B}_{1, j,2}'(w_1,w_2)| = \{\tilde{u}_\parallel\in \mathcal{B}_{\frac{15}{13}R_0}(0):E_{\tilde{u}_\parallel}\cap (w_1+w_2-\mathcal{B}_j - v)\neq \emptyset\}.
    \end{align*}
    Here, we can apply Lemma \ref{lem:intersection_lemma1} for $a = w_1+w_2 - v_3 - v$ and $r = \frac{R_3}{2}$, we get
    \begin{equation}\label{eq:1_intersection_2}
        \begin{split}
            |\mathcal{B}_{1, j,2}(w_1,w_2)|\leq |\mathcal{B}_{1, j,2}'(w_1,w_2)|&\leq \frac{5\pi}{2}\frac{R_3/2}{\frac{1}{13}R_0}\int_{|x|\leq \frac{15}{13}R_0}\,dx = \frac{65\pi}{4}\frac{R_3}{R_0}\left|\mathcal{B}_{\frac{15}{13}R_0}(v)\right|\\
            &=C_2 R_0^2 R_3
        \end{split}
    \end{equation}
    for some constant $C_2$ uniform about $w_1$ and $w_2$.
    
    Finally, for fixed $v\in \mathcal{B}_{\frac{R_0}{13}}(v_0)$,  $(w_1,w_2)\in \mathcal{B}_{1,j,1}$, and $\tilde{u}_\parallel\in \mathcal{B}_{1,j,2}(w_1,w_2)$, the intersection portion between $S_{w_1,w_2}\cap (v+E_{\tilde{u}_\parallel - v})$ and $w_1+w_2-\mathcal{B}_j$ is bounded by $C_3R_3$ by Lemma \ref{lem:intersection_lemma3}. Combining \eqref{eq:1_intersection_1} and \eqref{eq:1_intersection_2} with this bound, we have
    \begin{align*}
        B_{1,j} &\coloneqq \iint_{\mathcal{B}_{(1+2\delta)R_0}(v_0)\times \mathcal{B}_{(1+2\delta)R_1}(v_1)\times \mathcal{B}_{(1+2\delta)R_2}(v_2)}d\tilde{u}_\parallel dw_1dw_2\,\int_{S_{w_1,w_2}}dw_1'\,\mathbf{1}_{\mathcal{B}_j}(w_1+w_2-w_1')\Theta_{v+E_{\tilde{u}_\parallel+v}}(w_1')\\
        &\leq |\mathcal{B}_{1,j,1}||\mathcal{B}_{1,j,2}|(C_3R_3)\\
        &\leq 3C_1(1+2\delta)^6 C_2 (R_0^2 R_3)(R_0^5 R_3)C_3R_3\leq CR_0^7 |\mathcal{B}_j|
    \end{align*}
    for some uniform constant $C$ about $\delta$ and $\mathfrak{a}_0$. Now, we get
    \begin{align}\label{eq:final_form_bound_1}
        \begin{split}
            &\iint_{\mathcal{B}_{(1+2\delta)R_0}(v_0)\times \mathcal{B}_{(1+2\delta)R_1}(v_1)\times \mathcal{B}_{(1+2\delta)R_2}(v_2)}d\tilde{u}_\parallel dw_1dw_2\,\int_{S_{w_1,w_2}}dw_1'\,\mathbf{1}_{\cup_j \mathcal{B}_j}(w_1+w_2-w_1')\Theta_{v+E_{\tilde{u}_\parallel+v}}(w_1')\\
            &\leq \sum_{j=1}^\infty B_{1,j} \leq CR_0^7 \sum_{j=1}^\infty |\mathcal{B}_j|.
        \end{split}
    \end{align}\\
    
    \noindent (ii) Using Fubini's theorem again, we bound the second term in \eqref{eq:final_form_bound}:
    \begin{align*}
        \int_{\mathcal{B}_{(1+2\delta)R_0}(v_0)}d\tilde{u}_\parallel\,\iint_{\mathcal{B}_{(1+2\delta)R_1}(v_1)\times \mathcal{B}_{(1+2\delta)R_2}(v_2)}dw_1dw_2\, f_3(w_2) \int_{S_{w_1,w_2}}\,dw_1'\,\mathbf{1}_{\cup_j \mathcal{B}_j}(w_1' + \tilde{u}_\parallel - v)\Theta_{v+E_{\tilde{u}_\parallel-v}}(w_1').
    \end{align*}

    We again choose a ball $\mathcal{B}_{R_3}(v_3)  = \mathcal{B}_j \in \cup_j \mathcal{B}_j$ and define
    \begin{align*}
        \mathcal{B}_{2, j,1}(\tilde{u}_\parallel)&\coloneqq \{(w_1,w_2)\in \mathcal{B}_{(1+2\delta)R_1}(v_1)\times \mathcal{B}_{(1+2\delta)R_2}(v_2):S_{w_1,w_2}\cap (\mathcal{B}_j + \tilde{u}_\parallel - v)\neq \emptyset\},\\
        \mathcal{B}_{2, j,2} &\coloneqq \{\tilde{u}_\parallel\in \mathcal{B}_{(1+2\delta)R_0}(v_0):(v+E_{\tilde{u}_\parallel-v})\cap (\mathcal{B}_j+ \tilde{u}_\parallel - v) \neq \emptyset\}.
    \end{align*}
    Here, $\mathcal{B}_{2, j,2}$ can be defined independent to $w_1$ and $w_2$.
    
    For fixed $\tilde{u}_\parallel$ and $v$, we apply Lemma \ref{lem:intersection_lemma2} to  $\mathcal{B}_j+v-\tilde{u}_\parallel$. Therefore, we get
    \begin{equation}\label{eq:2_intersection_1}
        \begin{split}
            |\mathcal{B}_{2, j,1}| &\leq |\{(w_1,w_2)\in \mathcal{Q}_{(1+2\delta)R_1}(v_1)\times \mathcal{Q}_{(1+2\delta)R_2}(v_2):S_{w_1,w_2}\cap (\mathcal{B}_j + \tilde{u}_\parallel - v)\neq \emptyset\}|\\
            &\leq C_1(1+2\delta)^6 R_0^5R_3.
        \end{split}
    \end{equation}

    The second bound requires a different approach from (i). We choose a slightly larger ball $\mathcal{B}_{\frac{15}{13}R_0}(v)$ containing $\mathcal{B}_{(1+2\delta)R_0}(v_0)$ since $|v-v_0|\leq \frac{R_0}{26}$. We choose a spherical coordinates of $\tilde{u}_\parallel\in \mathcal{B}_{\frac{15}{13}R_0}(v)$. We write $\tilde{u}_\parallel - v = r\omega$. Note that the plane $v+E_{\tilde{u}_\parallel-v}$ is invariant if we do not change the direction of $\omega$. Therefore, for any fixed direction of $\omega$, if we change the length of the vector, $\mathcal{B}_j+v-\tilde{u}_\parallel$ orthogonally passes through the plane. Therefore, there exists $0\leq r_1(\omega)\leq r_2(\omega)\leq \frac{15}{13}R_0$ such that $|r_2(\omega)-r_1(\omega)|\leq 2R_3$ and $(v+E_{\tilde{u}_\parallel - v})\cap (\mathcal{B}_j + \tilde{u}_\parallel - v) = \emptyset$ if $r\not\in [r_1, r_2]$. Using this analysis, we get
    \begin{equation}\label{eq:2_intersection_2}
        \begin{split}
            |\mathcal{B}_{2, j,2}| &\leq |\{\tilde{u}_\parallel\in \mathcal{B}_{\frac{15}{13}R_0}(v):(v+E_{\tilde{u}_\parallel-v})\cap (\mathcal{B}_j+ \tilde{u}_\parallel - v) \neq \emptyset\}|\\
            &\leq \int d\omega\,\int_{r_1(\omega)}^{r_2(\omega)}dr\,r^2\leq C_4 R_0^2 R_3
        \end{split}
    \end{equation}
    for some constant $C_4$.
    
    Finally, for fixed $v\in \mathcal{B}_{\frac{R_0}{13}}(v_0)$, $\tilde{u}_\parallel\in \mathcal{B}_{2,j,1}$, and $(w_1,w_2)\in \mathcal{B}_{2,j,2}(\tilde{u}_\parallel)$, the intersection portion between $S_{w_1,w_2}\cap (v+E_{\tilde{u}_\parallel - v})$ and $\mathcal{B}_j+ \tilde{u}_\parallel - v$ is again bounded by $C_3R_3$ by Lemma \ref{lem:intersection_lemma3}. Combining \eqref{eq:2_intersection_1} and \eqref{eq:2_intersection_2} with this bound, we have
    \begin{align*}
        B_{2, j} &\coloneqq \iint_{\mathcal{B}_{(1+2\delta)R_0}(v_0)\times \mathcal{B}_{(1+2\delta)R_1}(v_1)\times \mathcal{B}_{(1+2\delta)R_2}(v_2)}d\tilde{u}_\parallel dw_1dw_2\,\int_{S_{w_1,w_2}}dw_1'\,\mathbf{1}_{\mathcal{B}_j}(w_1' + \tilde{u}_\parallel - v)\Theta_{v+E_{\tilde{u}_\parallel+v}}(w_1')\\
        &\leq |\mathcal{B}_{2,j,1}||\mathcal{B}_{2,j,2}|(C_3R_3)\\
        &\leq 3C_1C_4 (R_0^2 R_3)(R_0^5 R_3)C_3R_3\leq CR_0^7 |\mathcal{B}_j|.
    \end{align*}
    for some uniform constant $C$ about $\delta$ and $\mathfrak{a}_0$. Therefore, we get
    \begin{align}\label{eq:final_form_bound_2}
        \begin{split}
            &\iint_{\mathcal{B}_{(1+2\delta)R_0}(v_0)\times \mathcal{B}_{(1+2\delta)R_1}(v_1)\times \mathcal{B}_{(1+2\delta)R_2}(v_2)}d\tilde{u}_\parallel dw_1dw_2\,\int_{S_{w_1,w_2}}dw_1'\,\mathbf{1}_{\cup_j \mathcal{B}_j}(w_1' + \tilde{u}_\parallel - v)\Theta_{v+E_{\tilde{u}_\parallel+v}}(w_1')\\
            &\leq \sum_{j=1}^\infty B_{2,j}\leq CR_0^7 \sum_{j=1}^\infty |\mathcal{B}_j|.
        \end{split}
    \end{align}
    Combining \eqref{eq:final_form_bound_1} and \eqref{eq:final_form_bound_2}, we get the lemma.
    \end{proof}
    
    In \eqref{eq:final_form}, we restricted the domain of the angular variables to $\theta_\omega\in (\pi/8, 3\pi/8)$ and $\theta'\in (\pi/4, 3\pi/4)$. To employ these conditions, we again use Figure \ref{fig:balls}. We first consider the domain of $w'_1$ satisfying $\theta'\in (\pi/4, 3\pi/4)$ and $w'_1\in v+E_{\tilde{u}_\parallel+v}$. It corresponds to the intersection between a curved cylinder near the great circle at $\theta'= \pi/2$ and a plane through $v$. The smallest circle given by $S_{w_1,w_2}\cap (v+E_{\tilde{u}_\parallel - v})$ has radius greater than $\sqrt{\left(\frac{1}{2}\frac{467}{676}\right)^2 - \left(\frac{1}{26}\right)^2} = \frac{\sqrt{215358}}{1352}R_0$ by \eqref{eq:dist_list}. Also, the intersection length between the curved cylinder $\theta'\in (\pi/4, 3\pi/4)$ and the circle is minimized when the circle perpendicularly meets the great circle given by $\theta' = \pi/2$, and the length is greater than $\pi\frac{\sqrt{215358}}{1352}R_0$. Therefore, we get
    \begin{align}\label{eq:theta'_cond}
        \left|\left\{w_1'\in S_{w_1,w_2}\cap (v+E_{\tilde{u}_\parallel - v}): \theta'\in (\pi/4, 3\pi/4)\right\}\right|\geq \pi\frac{\sqrt{215358}}{1352}R_0.
    \end{align}
    
    Next, we consider $\theta_\omega$. From \eqref{eq:theta'}, we need to restrict
    \begin{align*}
        \cos \frac{3\pi}{8}<\frac{|\tilde{u}_\parallel - v|}{|\tilde{u}_\parallel + w_1' - 2v|}<\cos \frac{\pi}{8}.
    \end{align*}
    By \eqref{eq:dist_list}, we have
    \begin{align*}
        \frac{|\tilde{u}_\parallel-v|}{\sqrt{|\tilde{u}_\parallel-v|^2 + \left(\frac{\sqrt{3}}{2} + \frac{105}{1352}\right)^2R^2_0}}\leq \frac{|\tilde{u}_\parallel-v|}{|\tilde{u}_\parallel+w_1'-2v|} = \frac{|\tilde{u}_\parallel-v|}{\sqrt{|\tilde{u}_\parallel-v|^2 + |w_1'-v|^2}}\leq \frac{|\tilde{u}_\parallel-v|}{\sqrt{|\tilde{u}_\parallel-v|^2 + \left(\frac{415}{1352}\right)^2R_0^2}}.
    \end{align*}
    The left and right terms correspond to $\cos\frac{3\pi}{8}$ and $\cos\frac{\pi}{8}$. Therefore, $|\tilde{u}_\parallel - v|$ should satisfy
    \begin{align}\label{eq:E_1_condition}
        \left(\frac{\sqrt{3}}{2} + \frac{105}{1352}\right)\sqrt{\frac{2-\sqrt{2}}{2+\sqrt{2}}}R_0<|\tilde{u}_\parallel-v|< \frac{415}{1352}\sqrt{\frac{2+\sqrt{2}}{2-\sqrt{2}}}R_0.
    \end{align}
    Let the collection $\tilde{u}_\parallel$ satisfying \eqref{eq:E_1_condition} by $D_{\tilde{u}_\parallel}(v)$.
    
    Now, let
    \begin{align*}
        C_{1,\delta}(v)\coloneqq \frac{|D_{\tilde{u}_\parallel}(v)|}{|\mathcal{B}_{(1+2\delta)R_0}(v_0)|}.
    \end{align*}
    Since $C_{1,\delta}(v)$ is a uniformly lower bounded constant about $0<\delta\leq \frac{1}{104}$ and $v\in \mathcal{B}_{\frac{R_0}{13}}(v_0)$, we can take lower bound $0<C_1 = \inf_{\delta, v}C_{1,\delta}(v)$.
    
    Combining with \eqref{eq:theta'_cond}, we get
    \begin{align}\label{eq:E_1_upper_bound1}
        \begin{split}
            &\iint_{\mathcal{B}_{(1+2\delta)R_0}(v_0)\times \mathcal{B}_{(1+2\delta)R_1}(v_1)\times \mathcal{B}_{(1+2\delta)R_2}(v_2)}d\tilde{u}_\parallel dw_1dw_2\, \mathbf{1}_{D_{\tilde{u}_\parallel}(v)}\\
            &\qquad \times \int_{S_{w_1,w_2}}dw_1'\,\mathbf{1}_{\{\frac{1}{4}\pi<\theta'<\frac{3}{4}\pi\}}\Theta_{v+E_{\tilde{u}_\parallel+v}}(w_1')\\
            &\geq \pi\frac{\sqrt{215358}}{1352}C_{1,\delta}(v)|\mathcal{B}_{(1+2\delta)R_0}(v_0)||\mathcal{B}_{(1+2\delta)R_1}(v_1)||\mathcal{B}_{(1+2\delta)R_2}(v_2)|R_0.
        \end{split}
    \end{align}

    For notational convenience, we denote
    \begin{align*}
        &\chi(v, \tilde{u}_\parallel, w_1, w_2, w_1') \\
        &=\mathbf{1}_{\cup_j \mathcal{B}_j}(w_1+w_2-w_1')+\mathbf{1}_{\cup_j \mathcal{B}_j}(\tilde{u}_\parallel + w_1' - v) -\mathbf{1}_{\cup_j \mathcal{B}_j}(w_1+w_2-w_1')\mathbf{1}_{\cup_j \mathcal{B}_j}(\tilde{u}_\parallel + w_1' - v).
    \end{align*}
    By \eqref{eq:B_i_control} and Lemma \ref{lem:final_form_bound}, there exists a constant $C_2$, which is uniformly bound about $\delta$, such that
    \begin{align}\label{eq:E_1_lower_bound1}
        \begin{split}
            &\iint_{\mathcal{B}_{(1+2\delta)R_0}(v_0)\times \mathcal{B}_{(1+2\delta)R_1}(v_1)\times \mathcal{B}_{(1+2\delta)R_2}(v_2)}d\tilde{u}_\parallel dw_1dw_2\,\int_{S_{w_1,w_2}}dw_1'\,\chi\Theta_{v+E_{\tilde{u}_\parallel+v}}(w_1')\\
            &\leq C_2|\mathcal{B}_{(1+2\delta)R_0}(v_0)||\mathcal{B}_{(1+2\delta)R_1}(v_1)||\mathcal{B}_{(1+2\delta)R_2}(v_2)| R_0 \mathfrak{a}_0.
        \end{split}
    \end{align}    
    Adding \eqref{eq:E_1_upper_bound1} and \eqref{eq:E_1_lower_bound1} (from \eqref{eq:final_form_1} and \eqref{eq:final_form_2}), we have
    \begin{align*}
        &\iint_{\mathcal{B}_{(1+2\delta)R_0}(v_0)\times \mathcal{B}_{(1+2\delta)R_1}(v_1)\times \mathcal{B}_{(1+2\delta)R_2}(v_2)}d\tilde{u}_\parallel dw_1dw_2\,\mathbf{1}_{D_{\tilde{u}_\parallel}(v)}\\
        &\qquad \times \int_{S_{w_1,w_2}}dw_1'\,(1-\chi)\mathbf{1}_{\{\frac{1}{4}\pi<\theta'<\frac{3}{4}\pi\}}\Theta_{v+E_{\tilde{u}_\parallel+v}}(w_1')\\
        &\geq \iint_{\mathcal{B}_{(1+2\delta)R_0}(v_0)\times \mathcal{B}_{(1+2\delta)R_1}(v_1)\times \mathcal{B}_{(1+2\delta)R_2}(v_2)}d\tilde{u}_\parallel dw_1dw_2\,\mathbf{1}_{D_{\tilde{u}_\parallel}(v)}\\
        &\qquad \times \int_{S_{w_1,w_2}}dw_1'\,\mathbf{1}_{\{\frac{1}{4}\pi<\theta'<\frac{3}{4}\pi\}}\Theta_{v+E_{\tilde{u}_\parallel+v}}(w_1')\\
        &\quad -\iint_{\mathcal{B}_{(1+2\delta)R_0}(v_0)\times \mathcal{B}_{(1+2\delta)R_1}(v_1)\times \mathcal{B}_{(1+2\delta)R_2}(v_2)}d\tilde{u}_\parallel dw_1dw_2\,\\
        &\qquad \times \int_{S_{w_1,w_2}}dw_1'\,\chi\Theta_{v+E_{\tilde{u}_\parallel+v}}(w_1')\\
        &\geq |\mathcal{B}_{(1+2\delta)R_0}(v_0)||\mathcal{B}_{(1+2\delta)R_1}(v_1)||\mathcal{B}_{(1+2\delta)R_2}(v_2)| R_0\left(\pi\frac{\sqrt{215358}}{1352}C_{1,\delta}(v) - C_2\mathfrak{a}_0\right).
    \end{align*}
    Now, we choose 
    \begin{align}\label{eq:alpha0_cond2}
        \mathfrak{a}_0\leq \min\left\{\frac{C_1}{C_2}\left(\pi\frac{\sqrt{215358}}{1352} - \frac{1}{2}\right), \frac{1}{4^3(1+1/104)^2}\left(\frac{1}{104}\right)^3\right\}
    \end{align}
    so that
    \begin{align*}
        &\iint_{D_{\tilde{u}_\parallel}(v)\times \mathcal{B}_{(1+2\delta)R_1}(v_1)\times \mathcal{B}_{(1+2\delta)R_2}(v_2)}d\tilde{u}_\parallel dw_1dw_2\,\\
        &\qquad \times \int_{S_{w_1,w_2}}dw_1'\,(1-\chi)\mathbf{1}_{\{\frac{1}{4}\pi<\theta'<\frac{3}{4}\pi\}}\Theta_{v+E_{\tilde{u}_\parallel+v}}(w_1')\\
        &\geq \frac{C_{1,\delta}}{2}(v)|\mathcal{B}_{(1+2\delta)R_0}(v_0)||\mathcal{B}_{(1+2\delta)R_1}(v_1)||\mathcal{B}_{(1+2\delta)R_2}(v_2)| R_0
    \end{align*}
    for $\delta$ and $\alpha_0$ satisfying \eqref{eq:delta_cond1} and \eqref{eq:alpha0_cond2}. At the last, we used $\pi\frac{\sqrt{215358}}{1352} - \frac{1}{2}>\frac{1}{2}$. In \eqref{eq:alpha0_cond2}, the constant $C_2$ from Lemma \ref{lem:final_form_bound} is a uniform constant about $\alpha_0$ under $\alpha_0\leq \frac{1}{4^3(1+1/104)^2}\left(\frac{1}{104}\right)^3$, so it is not a circular logic. As a result, we computed a lower bound of the size of the set $\{(\tilde{u}_\parallel, w_1,w_2,w_1')\}$ which makes $(1-\chi)=1$ and $\frac{1}{4}\pi<\theta'<\frac{3}{4}\pi$.
    
    Now, let
    \begin{align*}
        D_{1,\delta}(v)&\coloneqq \left\{(\tilde{u}_\parallel,w_1,w_2)\in D_{\tilde{u}_\parallel}(v)\times \prod_{i=1}^2\mathcal{B}_{(1+2\delta)R_i}(v_i): \int_{S_{w_1,w_2}} dw_1'\,(1-\chi)\mathbf{1}_{\{\frac{1}{4}\pi<\theta'<\frac{3}{4}\pi\}}\Theta_{v+E_{\tilde{u}_\parallel + v}}(w_1')> \frac{R_0}{4}\right\},\\
        D_{2,\delta}(v)&\coloneqq \left\{(\tilde{u}_\parallel,w_1,w_2)\in D_{\tilde{u}_\parallel}(v)\times \prod_{i=1}^2\mathcal{B}_{(1+2\delta)R_i}(v_i): \int_{S_{w_1,w_2}} dw_1'\,(1-\chi)\mathbf{1}_{\{\frac{1}{4}\pi<\theta'<\frac{3}{4}\pi\}}\Theta_{v+E_{\tilde{u}_\parallel + v}}(w_1')\leq\frac{R_0}{4}\right\}.
    \end{align*}
    By the definition of $D_{1,\delta}(v)$ and $D_{2,\delta}(v)$, we have
    \begin{align*}
        &\iint_{D_{1,\delta}(v)}d\tilde{u}_\parallel dw_1dw_2\,\int_{S_{w_1,w_2}}dw_1'\,(1-\chi)\mathbf{1}_{\{\frac{1}{4}\pi<\theta'<\frac{3}{4}\pi\}}\Theta_{v+E_{\tilde{u}_\parallel+v}}(w_1')\\
        &\geq \iint_{D_{\tilde{u}_\parallel}(v)\times \mathcal{B}_{(1+2\delta)R_1}(v_1)\times \mathcal{B}_{(1+2\delta)R_2}(v_2)}d\tilde{u}_\parallel dw_1dw_2\,\int_{S_{w_1,w_2}}dw_1'\,(1-\chi)\Theta_{v+E_{\tilde{u}_\parallel+v}}(w_1')\\
        &\quad -\iint_{D_{2,\delta}(v)}d\tilde{u}_\parallel dw_1dw_2\,\int_{S_{w_1,w_2}}dw_1'\,(1-\chi)\mathbf{1}_{\{\frac{1}{4}\pi<\theta'<\frac{3}{4}\pi\}}\Theta_{v+E_{\tilde{u}_\parallel+v}}(w_1')\\
        &\geq \frac{C_{1,\delta}(v)}{2}|\mathcal{B}_{(1+2\delta)R_0}(v_0)||\mathcal{B}_{(1+2\delta)R_1}(v_1)||\mathcal{B}_{(1+2\delta)R_2}(v_2)|R_0 \\
        &\quad - C_{1,\delta}(v)|\mathcal{B}_{(1+2\delta)R_0}(v_0)||\mathcal{B}_{(1+2\delta)R_1}(v_1)||\mathcal{B}_{(1+2\delta)R_1}(v_2)|\frac{R_0}{4}\\
        &=\frac{C_{1,\delta}(v)}{4}|\mathcal{B}_{(1+2\delta)R_0}(v_0)||\mathcal{B}_{(1+2\delta)R_1}(v_1)||\mathcal{B}_{(1+2\delta)R_2}(v_2)|R_0.
    \end{align*}
    As a consequence, we get
    \begin{align}\label{eq:E_1_lower_bound2}
        \begin{split}
            |D_{1,\delta}(v)|&\geq \frac{\frac{C_{1,\delta}}{4} R_0\prod_{i=1}^3 |\mathcal{B}_{(1+2\delta)R_i}(v_i)|}{\int_{S_{w_1,w_2}} dw_1'\,\Theta_{v+E_{\tilde{u}_\parallel + v}}(w_1')}\\
            &\geq \frac{\frac{C_{1,\delta}}{4} R_0\prod_{i=1}^3 |\mathcal{B}_{(1+2\delta)R_i}(v_i)|}{\pi|w_1-w_2|}\\
            &\geq \frac{C_1}{8\pi}|\prod_{i=1}^3 |\mathcal{B}_{R_i}(v_i)|.
        \end{split}
    \end{align}
    In the middle, we used \eqref{eq:dist_list}.
    
    We collect \eqref{eq:dist_list}, \eqref{eq:final_form}, \eqref{eq:E_1_condition}, and the definition of $D_{1,\delta}(v)$ to compute the lower bound:
    \begin{align}\label{eq:3_final_1}
        \begin{split}
            &\int_{\mathcal{B}_{(1+2\delta)R_0}(v_0)}\,d\tilde{u}_\parallel\,\frac{1}{|\tilde{u}_\parallel - v|^{2-\gamma}}f_{1, \delta}(\tilde{u}_\parallel)\int_{\mathcal{B}_{(1+2\delta)R_1}(v_1)}\,dw_1f_{2,\delta}(w_1)\int_{\mathcal{B}_{(1+2\delta)R_2}(v_2)}\,dw_2f_{3, \delta}(w_2)\\
            &\qquad \times |w_1-w_2|^{\gamma-2}\int_{S_{w_1,w_2}}\,dw_1' \frac{h(\cos\theta_\omega)}{|\cos\theta_\omega|^\gamma}b(\cos\theta')(1-\chi)\Theta_{v+E_{\tilde{u}_\parallel-v}}(w_1')\\
            &\geq \sqrt{2-\sqrt{2}}\left(\frac{2}{\sqrt{2+\sqrt{2}}}\right)^\gamma c_b^2 \iint_{D_{1,\delta}(v)}d\tilde{u}_\parallel dw_1dw_2\,\left(\frac{57}{104}R_0\right)^{\gamma-2}\\
            &\qquad \times f_{1, \delta}(\tilde{u}_\parallel)f_{2,\delta}(w_1)f_{3, \delta}(w_2) \left(\left(\sqrt{3}+\frac{53}{676}\right)R_0\right)^{\gamma-2}\frac{R_0}{4}\\
            &\geq C R_0^{2\gamma - 3}\iint_{D_{1,\delta}(v)}d\tilde{u}_\parallel dw_1dw_2\,f_{1, \delta}(\tilde{u}_\parallel)f_{2,\delta}(w_1)f_{3, \delta}(w_2),
        \end{split}
    \end{align}
    where the last constant $C$ only depends on $c_b$ and $\gamma$ under $\delta$ and $\alpha$ conditions \eqref{eq:delta_cond1} and \eqref{eq:alpha0_cond2}.
    
    Now, we need to compute the last integral. From \eqref{eq:density_defect_set} and \eqref{eq:G_O_construction_2}, we have
    \begin{align*}
        (1+\mathfrak{a}_0\delta)|G_{0,\delta}|&\geq |E\cap \mathcal{B}_{R_0}(v_0)|\geq (1-4^3\mathfrak{a}_0)|\mathcal{B}_{R_0}(v_0)|,\\
        (1+\mathfrak{a}_0\delta)|G_{1,\delta}|&\geq |E\cap \mathcal{B}_{R_1}(v_1)|\geq (1-3184\mathfrak{a}_0)|\mathcal{B}_{R_1}(v_1)|,\\
        (1+\mathfrak{a}_0\delta)|G_{2,\delta}|&\geq |E\cap \mathcal{B}_{R_2}(v_2)|\geq (1-3184\mathfrak{a}_0)|\mathcal{B}_{R_2}(v_2)|.
    \end{align*}
    Therefore, using \eqref{eq:E_1_lower_bound2}, we conclude
    \begin{align*}
        &|D_{1,\delta}(v)\cap \prod_i G_{i,\delta}|\\
        &\geq |D_{1,\delta}(v)\cap \prod_i B_{(1+2\delta)R_i}(v_i)| - |\prod_i \mathcal{B}_{(1+2\delta)R_i}(v_i)\setminus \prod_i G_{i,\delta}|\\
        &\geq \frac{C_1}{8\pi}|\mathcal{B}_{R_0}(v_0)||\mathcal{B}_{R_1}(v_1)||\mathcal{B}_{R_2}(v_2)| - \left(|\prod_i \mathcal{B}_{(1+2\delta)R_i}(v_i)\setminus \prod_i \mathcal{B}_{R_i}(v_i)| + |\prod_i \mathcal{B}_{R_i}(v_i)\setminus \prod_i G_{i,\delta}|\right)\\
        &\geq \frac{C_1}{8\pi}|\mathcal{B}_{R_0}(v_0)||\mathcal{B}_{R_1}(v_1)||\mathcal{B}_{R_2}(v_2)| - C_3\left(\delta + \mathfrak{a}_0\right)|\mathcal{B}_{R_0}(v_0)||\mathcal{B}_{R_1}(v_1)||\mathcal{B}_{R_2}(v_2)|
    \end{align*}
    for some fixed constant $C_3>0$. We impose the final condition for $\delta$ and $\mathfrak{a}_0$ to meet
    \begin{align}\label{eq:alpha0_cond3}
        \frac{C_1}{16\pi}\geq C_3\max\{\delta, \mathfrak{a}_0\}.
    \end{align}
    
    Under this choice of $\delta$ and $\mathfrak{a}_0$, we get
    \begin{align}\label{eq:3_final_2}
        \begin{split}
            &\iint_{D_{1,\delta}}d\tilde{u}_\parallel dw_1dw_2\,f_{1, \delta}(\tilde{u}_\parallel)f_{2,\delta}(w_1)f_{3, \delta}(w_2)\\
            &\geq\iint_{D_{1,\delta}\cap \prod_i G_{i,\delta}} d\tilde{u}_\parallel dw_1dw_2\, \epsilon^3\\
            &\geq \epsilon^3\frac{C_1}{8\pi}|\mathcal{B}_{R_0}(v_0)||\mathcal{B}_{R_1}(v_1)||\mathcal{B}_{R_2}(v_2)|.
        \end{split}
    \end{align}
    Note that $C_1$ and $C_3$ are fixed constants under $\delta$ and $\mathfrak{a}_0$ conditions \eqref{eq:delta_cond1}, \eqref{eq:alpha0_cond1}, and \eqref{eq:alpha0_cond2}.
    
    Combining \eqref{eq:3_final_1} and \eqref{eq:3_final_2}, we get
    \begin{align*}
        &Q_1(f_{1, \delta}, Q_1(f_{2,\delta},f_{3,\delta},\varphi_\delta),\varphi_\delta)(v)\\
        &\geq CR_0^{2\gamma+6}\epsilon^3
    \end{align*}
    for some constant $C$ depending on $c_b$ and $\gamma$ and for $\mathfrak{a}_0$ and $\delta$ satisfying \eqref{eq:delta_cond1}, \eqref{eq:alpha0_cond2}, and \eqref{eq:alpha0_cond3}. We let $\delta\rightarrow 0$ and use Lemma \ref{lem:Q_1_approx} to conclude
    \begin{align*}
        Q_1(f_1, Q_1(f_2,f_3,\chi),\chi)(v)\geq CR_0^{2\gamma+6}\epsilon^3
    \end{align*}
    for a.e. $v\in \mathcal{B}_{R_0/13}(v_0)$. Using a similar limiting argument, we can replace the continuity condition on the angular collision kernel $b$ by the measurability condition. It ends the proof.
\end{proof}

If $S(f_0)>0$, which is defined in \eqref{eq:entropy}, $f_0$ is strictly above $0$ and below $1$ in some set. In this case, we can directly apply  Proposition \ref{prop:4.4}. However, we can also consider some initial functions such that $S(f_0)=0$, but $f_0$ is not a saturated Fermi-Dirac equilibrium. One such example is $f(t,v) = 1$ on $1\leq |v|\leq 2$ and $0$ otherwise. Since $f_0$ is not an equilibrium, it should collapse to an intermediate distribution and eventually may converge to the equilibrium with the corresponding macroscopic quantities. Therefore, we can guess that it also has a Fermi-Dirac lower bound and satisfies the results in this section for positive time $t>0$. The next lemma proves it.

\begin{proposition}\label{prop:entropy_production}
    Suppose the collision kernel \eqref{eq:B_defi} satisfies (H1), (H5), and $0\leq \gamma\leq 2$. Also, suppose $f$ is a solution of the Boltzmann-Fermi-Dirac equation that satisfies the entropy identity \eqref{eq:entropy_identity}. If $f_0(v)$ only has values $0$ or $1$ but is not a saturated equilibrium, then there exists $t_0>0$ depending on $f_0$, $\gamma$, and $b(\cos\theta)$ such that a solution $f(t,v)$ having initial data $f_0(v)$ meets
    \begin{align*}
        S(f)(t)\geq t
    \end{align*}
    for $0\leq t\leq t_0$.
\end{proposition}
\begin{proof}
    Since $f(t,v)$ satisfies the entropy identity \eqref{eq:entropy_identity}, $S(f)$ is given by
    \begin{align*}
        S(f)(t) = \int_0^t D(f)(\tau)\,d\tau.
    \end{align*}
    We first claim that $D(f_0) = \infty$. In \cite{Lu2001353} and \cite{lu2022}, Lu proved that any function $f_0$ satisfying $f_0'f'_{0,*}(1-f_0)(1-f_{0,*}) = f_0f_{0,*}(1-f'_0)(1-f'_{0,*})$ for all $v$, $v_*$, and $\omega$ is an equilibrium for the Boltzmann-Fermi-Dirac equation. Since $f_0$ is not a saturated equilibrium, it means that $f_0'f'_{0,*}(1-f_0)(1-f_{0,*}) \neq f_0f_{0,*}(1-f'_0)(1-f'_{0,*})$ on positive measure set $E\subset \mathbb{R}^6\times\mathbb{S}^2$. However, as $f_0=0$ or $1$, $\Gamma(f_0'f'_{0,*}(1-f_0)(1-f_{0,*}),f_0f_{0,*}(1-f'_0)(1-f'_{0,*})) = \infty$ on the set. By the assumption (H5), we get $D(f_0) = \infty$.

    There are two possibilities in $E$: $f_0(v)f_0(v_*) = 1$ with $f_0(v') + f_0(v'_*) = 0$ or $f_0(v')f_0(v'_*) = 1$ with $f_0(v) + f_0(v_*) = 0$. As
    \begin{align*}
        &\int_{\mathbb{R}^6\times\mathbb{S}^2}\mathbf{1}_{\{f_0(v')f_0(v'_*) = 1\}}\mathbf{1}_{\{f_0(v) + f_0(v_*) = 0\}}\,dvdv_*d\omega\\
        &=\int_{\mathbb{R}^6\times\mathbb{S}^2}\mathbf{1}_{\{f_0(v)f_0(v_*) = 1\}}\mathbf{1}_{\{f_0(v') + f_0(v'_*) = 0\}}\,dvdv_*d\omega,
    \end{align*}
    the set
    \begin{align*}
        E_1\coloneqq \{(v,v_*,\omega):f_0(v)f_0(v_*) = 1\text{ and }f_0(v') + f_0(v'_*) = 0\}
    \end{align*}
    satisfies $|E_1| = \frac{|E|}{2}$. Also, $E_1' = E_1\cap \{(v,v_*):|v|^2+|v_*|^2\leq R\}\times\mathbb{S}^2$ for $R = \left(\frac{17\pi\|f_0\|^2_{1,2}}{|E_1|}\right)^{1/2}$ has measure greater than $\frac{|E_1|}{2}$. Indeed, assume its measure is not greater than $\frac{|E_1|}{2}$. Then,
    \begin{align*}
        8\pi\|f_0\|_{1,2}^2 &\geq \int_{\mathbb{R}^6\times\mathbb{S}^2} (|v|^2+|v_*|^2)f(v)f(v_*)\,dv dv_*d\omega\\
        &\geq \int_{E_1\cap (\{|v|^2+|v_*^2|\leq R^2\}^c\times\mathbb{S}^2)} (|v|^2+|v_*|^2)f(v)f(v_*)\,dv dv_*d\omega\\
        &\geq R^2\int_{E_1\cap (\{|v|^2+|v_*^2|\leq R^2\}^c\times\mathbb{S}^2)}\,dv dv_*d\omega\geq \frac{R^2}{2}|E_1|\geq \frac{17}{2}\pi\|f_0\|^2_{1,2},
    \end{align*}
    which is a contradiction.

    Now, let $(v,v_*,\omega)\in E_1'$. From \eqref{eq:duha_f}, we get
    \begin{align*}
        f(t,v)&\geq f_0(v) G_0^t(v)\geq e^{-ct(1+|v|^\gamma)}\geq e^{-ct(1+R^\gamma)},\\
        f(t,v_*)&\geq f_0(v_*) G_0^t(v_*)\geq e^{-ct(1+|v_*|^\gamma)}\geq e^{-ct(1+R^\gamma)}.
    \end{align*}
    Since $|v'|^2+|v'_*|^2 = |v|^2+|v_*|^2$, we have $|v'|^2+|v'_*|^2\leq R^2$. Applying $\eqref{eq:duha_1-f}$ to $f(v')$ and $f(v'_*)$, we also have
    \begin{align*}
        f(t,v')&\leq 1-e^{-ct(1+R^\gamma)},\\
        f(t,v'_*)&\leq 1-e^{-ct(1+R^\gamma)}.
    \end{align*}
    Choose $t_1>0$ such that $e^{-ct(1+R^\gamma)}\geq \frac{1}{2}$ for $0\leq t\leq t_1$. For $0<t\leq t_1$, we get
    \begin{align*}
        \Gamma(f'f'_*(1-f)(1-f_*),ff_*(1-f')(1-f'_*)) \geq \frac{1}{2^4}\left[\left(\frac{1-e^{-ct(1+R^\gamma)}}{e^{-ct(1+R^\gamma)}}\right)^4 - 1\right]\ln\left(\frac{(1-e^{-ct(1+R^\gamma)})}{e^{-ct(1+R^\gamma)}}\right)
    \end{align*}
    for $(v,v_*,\omega)\in E_1'$. Therefore, there exists $t_0\leq t_1$ such that
    \begin{align*}
        D(f)(t)\geq 1
    \end{align*}
    for $0\leq t\leq t_0$, so $S(f)(t)\geq t$. Since the construction of $E_1'$ depends on $f_0$, and the integral $D(f)$ depends on $\gamma$ and $b(\cos\theta)$, the $t_0$ depends on $f_0$, $\gamma$, and $b(\cos\theta)$.
\end{proof}
Now, we can assume that $S(f)(t)>0$ for some $t\geq 0$ if $f$ is not a saturated equilibrium.

The next lemma proves that we can find some $\epsilon>0$ and $v_{-1}\in\mathbb{R}^3$ to fulfill the conditions in Proposition \ref{prop:4.4} using $S(f)>0$ and $\|f_0\|_{1,2}$.
\begin{lemma}\label{lem:construct_epsilon_v-1}
    Let $f(v)\in L^1_2$ satisfies $0\leq f\leq 1$ and $S(f)>0$. Then, there exists $\epsilon_0$ depending on $S(f)$ and $\|f\|_{1,2}$ such that $|\{v:\epsilon_0\leq f(v)\leq 1-\epsilon_0\}|\geq \frac{S(f)}{2\ln 2}$. Furthermore, we can choose $R>0$ depending on $\epsilon_0$, $S(f)$, and $\|f\|_{1,2}$ such that
    \begin{align*}
        \left|\{v:\epsilon_0\leq f(v)\leq 1-\epsilon_0\}\cap B_R(0)\right|>\frac{|\{v:\epsilon_0\leq f(v)\leq 1-\epsilon_0\}|}{2}.
    \end{align*}
\end{lemma}
\begin{proof}
    Let $0<\epsilon<\frac{1}{4}$, it will be chosen later. Define
    \begin{align*}
        E_1 = \{f> 1-\epsilon\},\quad E_2 = \{\epsilon\leq f\leq 1-\epsilon\},\text{ and }E_{3, n} = \{\frac{1}{2^n}\epsilon\leq f<\frac{1}{2^{n-1}}\epsilon\}
    \end{align*}
    for $n\geq 1$. Since
    \begin{align*}
        \|f\|_{1,0}\geq \int_{E_1} f\,dv\geq (1-\epsilon)|E_1|,
    \end{align*}
    we have $|E_1|\leq \frac{\|f\|_{1,0}}{1-\epsilon}$. Also,
    \begin{align*}
        \frac{1}{2^n}\epsilon \frac{4\pi}{5}\left(\frac{3|E_{3,n}|}{4\pi}\right)^{5/3}\leq \frac{1}{2^n}\epsilon\int_{E_{3, n}} |v|^2\,dv\leq \int_{E_{3, n}} |v|^2 f(v)\,dv\leq \|f\|_{1,2},
    \end{align*}
    so
    \begin{align*}
        |E_{3,n}|\leq C 2^{\frac{3}{5}n}\epsilon^{-\frac{3}{5}}
    \end{align*}
    for some constant $C$ depending on $\|f\|_{1,2}$. In the middle, we used the Hardy-Littlewood inequality so that the integral is minimized when $E_{3,n}$ is a ball centered at $0$ with radius $r = \left(\frac{3|E_{3,n}|}{4\pi}\right)^{1/3}$. Therefore,
    \begin{align*}
        &-\int_{E_2} \left(f\ln f + (1-f)\ln(1-f)\right)\,dv\\
        &= S(f) + \int_{E_1} \left(f\ln f + (1-f)\ln(1-f)\right)\,dv + \sum_{n=1}^\infty \int_{E_{3,n}} \left(f\ln f + (1-f)\ln(1-f)\right)\,dv\\
        &\geq S(f) + (\epsilon\ln \epsilon + (1-\epsilon)\ln (1-\epsilon)) |E_1| + \sum_{n=1}^\infty \left(\frac{1}{2^{n-1}}\epsilon \ln\frac{1}{2^{n-1}}\epsilon + \left(1-\frac{1}{2^{n-1}}\epsilon\right)\ln\left(1-\frac{1}{2^{n-1}}\epsilon\right)\right)|E_{3,n}|\\
        &\geq S(f) + (\ln \epsilon - 2)\epsilon\|f\|_{1,0} + \sum_{n=1}^\infty \left(2\ln\frac{1}{2^{n-1}}\epsilon - 4\right)\frac{\epsilon}{2^n}|E_{3,n}|\\
        &\geq S(f) + (\epsilon^{1/2}\ln \epsilon - 2\epsilon^{1/2})\epsilon^{1/2}\|f\|_{1,0} + C\sum_{n=1}^\infty \left(2^{4/5}\left(\frac{\epsilon}{2^{n-1}}\right)^{\frac{1}{5}}\ln\frac{\epsilon}{2^{n-1}}- 4\left(\frac{\epsilon}{2^n}\right)^{\frac{1}{5}}\right)\left(\frac{\epsilon}{2^n}\right)^{\frac{1}{5}}.
    \end{align*}
    In the computation, we used
    \begin{align*}
        (x\ln x + (1-x)\ln(1-x))\frac{1}{(1-x)}\geq x(\ln x -2)
    \end{align*}
    and $\ln(1-x)\geq -2x$ for $0\leq x\leq \frac{1}{4}$. 
    
    Since $|x^{1/2}\ln x|, |x^{\frac{1}{5}}\ln (2x)|\leq C$ for some constant $C$ for $0<x\leq \frac{1}{4}$, we obtain
    \begin{align*}
        -\int_{E_2} \left(f\ln f + (1-f)\ln(1-f)\right)\,dv&\geq S(f) - C\left(\epsilon^{\frac{1}{2}} + \sum_{n=1}^\infty \left(\frac{\epsilon}{2^n}\right)^{\frac{1}{5}}\right)\\
        &\geq S(f) - C\left(\epsilon^{\frac{1}{2}} + \epsilon^{\frac{1}{5}}\right).
    \end{align*}
    for some constants $C$ depending on $\|f\|_{1,2}$ for $0<\epsilon\leq \frac{1}{4}$. As
    \begin{align*}
        -\int_{E_2} \left(f\ln f + (1-f)\ln(1-f)\right)\,dv\leq (\ln 2)|E_2|,
    \end{align*}
    if we choose small enough $\epsilon_0\leq \frac{1}{4}$ by
    \begin{align*}
        C\left(\epsilon_0^{\frac{1}{2}} + \epsilon_0^{\frac{1}{5}}\right) = (\ln 2)|E_2|,
    \end{align*}
    we can make
    \begin{align*}
        |E_2|\geq \frac{S(f)}{2\ln 2}.
    \end{align*}
    Finally, if we choose $R\geq \left(\frac{3\|f\|_{1,2}}{\epsilon_0|E_2|}\right)^{1/2}$, but $|E_2\cap B_R(0)|\leq \frac{|E_2|}{2}$, then
    \begin{align*}
        \int_{\mathbb{R}^3} |v|^2f(v)\,dv\geq \int_{B_R(0)^c} |v|^2f(v)\,dv\geq R^2\epsilon_0\frac{|E_2|}{2}\geq \frac{3}{2}\|f\|_{1,2},
    \end{align*}
    which is a contradiction. Therefore, $|E_2\cap B_R(0)|\geq \frac{|E_2|}{2}$ for such $R$.
\end{proof}

Finally, we prove the main theorem of this section.
\begin{theorem}\label{thm:main4}
    We consider the collision kernel \eqref{eq:B_defi} for $0\leq \gamma\leq 2$, (H1), and (H2). Let $f$ be a solution of the Boltzmann-Fermi-Dirac equation with $S(f_0)>0$. Then, there exist $C>0$, $r>0$, $v_0$, and $T_0>0$ depending on $\gamma$, $C_b$, $c_b$, and $f_0$ such that 
    \begin{align*}
        Ct^2\leq f(t,v),\quad Ct^2\leq 1-f(t,v)
    \end{align*}
    on $v\in B_r(v_0)$. Furthermore, we can control $R(t) = |v_0|$ using $S(f_0)$ and $\|f_0\|_{1,2}$.
     
    If $S(f_0) = 0$, but it is not a saturated equilibrium, we further assume that the collision kernel satisfies (H5) and $f$ satisfies the entropy identity \eqref{eq:entropy_identity}.
    Then, there exist (1) $T_0>0$ depending on $\gamma$, $b(\cos\theta)$, and $f_0$ and (2) $C(t)>0$, $r(t)>0$, and $v_0(t)$ depending on $\gamma$, $C_b$, $c_b$, and $f(t/2,v)$ for $0<t\leq T_0$ such that 
    \begin{align*}
        C(t)\leq f(t,v),\quad C(t)\leq 1-f(t,v)
    \end{align*}
    on $v\in B_{r(t)}(v_0(t))$ for each $0<t\leq T_0$. Furthermore, we can control $R = |v_0(t)|$ using $t$ and $\|f_0\|_{1,2}$.
\end{theorem}
\begin{proof}
    If $S(f_0)>0$, we apply Lemma \ref{lem:construct_epsilon_v-1}, Proposition \ref{prop:4.4}, and Lemma \ref{lem:Positivity} in sequence to get the theorem.
    
    If $S(f_0) = 0$, but $f_0$ is not a saturated equilibrium, we use Proposition \ref{prop:entropy_production} to get $S(f)(\frac{t}{2})\geq\frac{t}{2}$ for $0\leq t\leq T_0$, there $T_0$ depends on $\gamma$, $b(\cos\theta)$, and $f_0$. Taking $f(t/2, v)$ as initial data, we use the same proof for $S(f_0)>0$ and get the theorem.
\end{proof}
\begin{remark}\label{remark:main4_remark_1}
    The dependency on the shape of the initial function $f_0$ or $f(t/2, v)$ is necessary as the proof depends on the Lebesgue density theorem. To remove the dependency, we need to develop another technique which do not rely on the Lebesgue density theorem.
\end{remark}
\begin{remark}\label{remark:main4_remark_2}
    Suppose $S(f_0)=0$, but $f_0$ is not a saturated equilibrium. As Proposition \ref{prop:entropy_production} only tells us $S(f)(\frac{t}{2})\geq \frac{t}{2}$ and does not give any information about the shape of the set $\{v:\epsilon\leq f(\frac{t}{2}, v)\leq 1-\epsilon\}$, we can only guarantee that there exists a $R_0>0$ in \eqref{eq:E_density} at the time $t/2$ by the Lebesgue density theorem. In consequence, the dependency on not only the initial data but also $f(t/2, v)$ is indispensable. As pointed out in the remark below Theorem \ref{thm:Gaussian_lower_bound}, $f(t/2, v)$ is uniquely chosen when $f_0$ is fixed.
\end{remark}

\section{Creation of Gaussian lower bound}\label{sec:Gaussian_lower}

In this section, we establish a Gaussian lower bound for a solution of the Boltzmann-Fermi-Dirac equation. We first construct a spreading lemma for the $Q_1$ operator starting from the classical spreading lemma in \cite{P1997}, and then prove the main result. The next lemma consists of two parts: one assumes $f\leq 1-\epsilon$, and the other does not. The one assuming $f\leq 1-\epsilon$ is to construct an exponential lower bound for the solution $f$ and $1-f$, and the other one refines the exponential lower bound to a Gaussian lower bound for $f$ in the proof of the main theorem. We first cite the classical result.
\begin{lemma}[Lemma 3.2 of \cite{P1997}]\label{lem:spread_classical}
    We consider the collision kernel \eqref{eq:B_defi} satisfying $0 \leq \gamma \leq 2$ and (H2). Assume that there exists $\epsilon>0$ such that
    \begin{align*}
        f(v) \geq \epsilon, \text{\quad where \quad} |v-\bar{v}| \leq \delta
    \end{align*}
    for some $\bar{v} \in \mathbb{R}^3$ and $\delta>0$. Then there exists a constant $C$ depending on $\gamma, c_b$ such that
    \begin{align*}
        Q^+_c(f,f)(v) \geq  C\delta^{3+\gamma} \eta^{\frac{5}{2}}\epsilon^2,
    \end{align*}
    where $|v-\bar{v}|\leq \sqrt{2}\delta(1-\eta)$ for $0<\eta <1$.
\end{lemma}
We extend this lemma to the Fermi-Dirac case.
\begin{lemma} \label{lem:spread}
    We consider the collision kernel \eqref{eq:B_defi} for $0 \leq \gamma \leq 2$, (H1), and (H2). Assume $0\leq f\leq 1$ on $\mathbb{R}^3$ and that there exists $0< \epsilon<1$ such that
    \begin{align} \label{eq:f>epsilon}
        f(v) \geq \epsilon, \text{\quad where \quad} |v-\bar{v}| \leq \delta
    \end{align}
    for some $\bar{v} \in \mathbb{R}^3$ and $\delta>0$. Then there exist constants $C_1>0$ and $C_2>0$ depending on $\gamma, C_b,$ and $c_b$ such that
    \begin{align} \label{eq:epsilon_2}
        Q_1(f,f,1-f)(v) \geq  \delta^{3+\gamma} \epsilon^2 \left(C_1\eta^{\frac{5}{2}}-C_2 \min\left\{\delta^{-3}\|f\|_{1,2}^{\frac{3}{5}}, 1\right\} \right),
    \end{align}
    where $\delta<|v-\bar{v}|\leq \sqrt{2}\delta(1-\eta)$ for $0< \eta <1-\frac{1}{\sqrt{2}}$. If we further assume $f(v) \leq 1-\epsilon$ for $|v-\bar{v}| \leq \delta$, then there exists a constant $C_3>0$ depending only on $\gamma$ and $c_b$ such that
    \begin{align}\label{eq:epsilon_3_f}
        Q_1(f,f,1-f)(v) \geq C_3 \delta^{3+\gamma} \eta^{\frac{5}{2}} \epsilon^3,\quad Q_1(1-f,1-f,f)(v)\geq C_3 \delta^{3+\gamma} \eta^{\frac{5}{2}} \epsilon^3,
    \end{align}
    where $\delta<|v-\bar{v}|\leq \sqrt{2}\delta(1-\eta)$ for $0< \eta <1-\frac{1}{\sqrt{2}}$.
\end{lemma}
\begin{proof}
    We start with a set estimate; we define
    \begin{align*}
        E_1 \coloneqq \left\{ v \in \mathbb{R}^3 : |v-\bar{v}|\leq \delta \text{\; and \;} f(v) \leq \frac{1}{2} \right\}\quad \text{and} \quad  E_2 \coloneqq\left\{ v \in \mathbb{R}^3 : |v-\bar{v}|\leq \delta \text{\; and \;} f(v) > \frac{1}{2} \right\}.
    \end{align*}
    By the construction of $E_2$, we have
    \begin{align*}
        \|f\|_{1,2}\geq \int_{E_2} \frac{1}{2}|v|^2 \,dv.
    \end{align*}
    By Hardy-Littlewood inequality, for fixed $|E_2|$, the integral has its minimum when $\bar{v} = 0$ and $E_2 = \{v:|v|\leq c\}$ for some $c$. Therefore,
    \begin{align*}
        \int_{E_2} \frac{1}{2}|v|^2 \,dv \geq \int_{\{|v|\leq c\}}
        \frac{1}{2}|v|^2 \,dv = \frac{2}{5}\pi c^5.
    \end{align*}
    Combining the two inequalities, we obtain
    \begin{align} \label{eq:enengy}
        |E_2|\leq \min\left\{\frac{4\pi}{3}c^3, \frac{4\pi}{3}\delta^3\right\}= \min\left\{\frac{4\pi}{3}\left(\frac{5}{2\pi}\|f\|_{1,2}\right)^{\frac{3}{5}}, \frac{4\pi}{3}\delta^3\right\}.
    \end{align} 

    Now, we estimate $Q_1(f,f,1-f)$. Using the assumption \eqref{eq:f>epsilon}, we have
    \begin{align}\label{eq:spread_Q1}
        Q_1(f,f,1-f) (v)\geq  \epsilon^2 \int_{\mathbb{R}^3 \times \mathbb{S}^2} B(v-v_*, \omega) \mathbf{1}_{\{|v'-\bar{v}|\leq \delta\}}\mathbf{1}_{\{|v_*'-\bar{v}|\leq \delta \}} (1-f(v_*)) \,d\omega dv_*.
    \end{align}
    For given $\delta<|v-\bar{v}|\leq \sqrt{2}\delta(1-\eta)$, $|v'-\bar{v}|$ and $|v_*'-\bar{v}| \leq \delta$ imply $|v_*-\bar{v}|\leq\delta$. Therefore, the $v_*$ integral domain is confined in the set $|v_*-\bar{v}|\leq\delta$ and further can be split into the domains by $E_1$ and $E_2$. Now, we obtain
    \begin{align}\label{eq:E0}
        \begin{split}
            \eqref{eq:spread_Q1} &\geq \frac{1}{2} \epsilon^2 \int_{\mathbb{R}^3 \times \mathbb{S}^2} B(v-v_*, \omega) \mathbf{1}_{\{ |v'-\bar{v}|\leq \delta\}}\mathbf{1}_{\{ |v_*'-\bar{v}|\leq \delta \}} \mathbf{1}_{E_1}(v_*)\,d\omega dv_*\\
            &\geq \frac{1}{2} \epsilon^2 \int_{\mathbb{R}^3 \times \mathbb{S}^2} B(v-v_*, \omega) \mathbf{1}_{\{|v'-\bar{v}|\leq \delta\}}\mathbf{1}_{\{|v_*'-\bar{v}|\leq \delta \}} \,d\omega dv_*\\
            &\quad - \frac{1}{2} \epsilon^2 \int_{\mathbb{R}^3 \times \mathbb{S}^2}  B(v-v_*, \omega) \mathbf{1}_{\{|v'-\bar{v}|\leq \delta\}}\mathbf{1}_{\{|v_*'-\bar{v}|\leq\delta \}}\mathbf{1}_{E_2}(v_*) \,d\omega dv_*.
        \end{split}
    \end{align} 
    For the first term, by Lemma \ref{lem:spread_classical}, we obtain
    \begin{align}\label{eq:E1}
        \begin{split}
            &\frac{1}{2} \epsilon^2 \int_{\mathbb{R}^3 \times \mathbb{S}^2}  B(v-v_*, \omega) \mathbf{1}_{\{|v'-\bar{v}|\leq \delta\}}\mathbf{1}_{\{|v_*'-\bar{v}|\leq \delta \}} \,d\omega dv_*\\
            &=\frac{1}{2}Q_c^+\left( \epsilon\mathbf{1}_{\{|v-\bar{v}|\leq \delta\}}, \epsilon\mathbf{1}_{\{|v-\bar{v}|\leq \delta\}}\right)(v) \geq C_1\delta^{3+\gamma}\epsilon^2\eta^{\frac{5}{2}},
        \end{split}
    \end{align}
    where $C_1$ depends on $\gamma$ and $c_b$.
    
    Using $|v-v_*|=|v'-v_*'|\leq 2\delta$, \eqref{eq:C_b:def}, and \eqref{eq:enengy}, the second term is bounded by
    \begin{align}\label{eq:E2}
        \begin{split}
            &\frac{1}{2} \epsilon^2 \int_{\mathbb{R}^3 \times \mathbb{S}^2}  |v-v_*|^\gamma h(\cos\theta_{\omega})\mathbf{1}_{\{ |v'-\bar{v}| < \delta\}}\mathbf{1}_{\{ |v_*'-\bar{v}| < \delta \}}\mathbf{1}_{E_2}(v_*) \,d\omega dv_*\\
            &\leq 2^{\gamma-1}\delta^{\gamma}\epsilon^2
            \int_{\mathbb{R}^3 \times \mathbb{S}^2} h(\cos \theta_{\omega})\mathbf{1}_{E_2}(v_*) \; \,d\omega dv_* \\
            &\leq 2^{\gamma-1}C_b\delta^{\gamma}\epsilon^2 |E_2| \leq C\delta^{\gamma}\epsilon^2 \min\left\{\left(\frac{5}{2\pi}\|f\|_{1,2}\right)^{\frac{3}{5}}, \delta^3\right\}
        \end{split}
    \end{align}
    for some constant $C$ depending on $\gamma$ and $C_b$.
    
    Applying \eqref{eq:E1} and \eqref{eq:E2} to \eqref{eq:E0}, we conclude that there exist a constant $C_2>0$ depending on $C_b$ and $\gamma$ such that
    \begin{align*}
        Q_1(f,f,1-f) (v) \geq \delta^{3+\gamma} \epsilon^2 \left(C_1\eta^{\frac{5}{2}}-C_2 \min\left\{\delta^{-3}\|f\|_{1,2}^{\frac{3}{5}}, 1\right\}\right)
    \end{align*}
    for $\delta<|v-\bar{v}|<\sqrt{2}\delta(1-\eta)$.

    Next, we prove \eqref{eq:epsilon_3_f} under the condition $\epsilon\leq f\leq 1-\epsilon$ on $|v-\bar{v}|\leq \delta$. For $\delta<|v-\bar{v}|<\sqrt{2}\delta(1-\eta)$, we obtain
    \begin{align*}
        Q_1 \left(f\mathbf{1}_{\{ |v-\bar{v}| < \delta\}}, f\mathbf{1}_{\{ |v-\bar{v}| < \delta\}},(1-f)\mathbf{1}_{\{ |v-\bar{v}| < \delta\}} \right)(v)
        \geq\epsilon \cdot Q_c^+ \left(\epsilon\mathbf{1}_{\{ |v-\bar{v}| < \delta\}}, \epsilon\mathbf{1}_{\{ |v-\bar{v}| < \delta\}}\right)(v)
    \end{align*}
    for the same reason before. Directly applying Lemma \ref{lem:spread_classical} for $Q^+_c$, we get \eqref{eq:epsilon_3_f}. Using the symmetry $f\mapsto 1-f$, we get the corresponding result for $Q_1(1-f,1-f,f)$.
\end{proof}

Now, we are ready to prove the main theorem.

\begin{theorem}[Creation of a Gaussian lower bound] \label{thm:lower_maxi}
    We consider the collision kernel \eqref{eq:B_defi} satisfying $0 \leq \gamma \leq 2$, (H1), and (H2). Let $f$ be a solution of the Boltzmann-Fermi-Dirac equation. If there exists $0<\epsilon<1$ such that
    \begin{align} \label{eq:f>epsilon_0}
        \epsilon \leq f_0(v) \leq 1-\epsilon, \text{\quad where \quad} |v-\bar{v}| \leq \delta, \quad |\bar{v}|< r_0
    \end{align}
    for some $\bar{v}\in \mathbb{R}^3, r_0>0$ and $\delta>0$, then there exist constants $C_1(t)>0$ and $C_2(t)>0$ such that
    \begin{align*}
        C_1(t) e^{-C_2(t)|v|^2}\leq f(t,v) \leq 1-C_1(t) e^{-C_2(t)|v|^p}, \text{\quad where \quad} p=2\frac{\ln 3}{\ln 2} \approx 3.17
    \end{align*}
    for $t>0$ and $v \in \mathbb{R}^3$. These constants $C_1(t)$ and $C_2(t)$ depend on $\|f_0\|_{1,2},\gamma, C_b, c_b, \delta, \epsilon$, and $r_0$. Also, it satisfies
    \begin{align*}
        \inf_{T^{-1}\leq t\leq T} C_1(t)>0,\quad \sup_{T^{-1}\leq t\leq T} C_2(t)<\infty
    \end{align*}
    for any $1\leq T<\infty$.
\end{theorem}

\begin{proof}
    We will iteratively apply the Lemma \ref{lem:spread} for each small time length $t_i>0$ and small $\eta_i$ for each $i$ to get a Gaussian lower bound. Let us first consider the time interval $[0, t_1]$. For $|v-\bar{v}|\leq \sqrt{2}\delta(1-\eta_1)$, from \eqref{eq:G_upper_bound_2}, we have
    \begin{align*}
        G_{t_1}^{t_2}(v)\geq e^{-ct_1(1+2(r_0^\gamma + \sqrt{2}^\gamma\delta^\gamma))}.
    \end{align*}
    By the computation in \eqref{eq:f_bound_1}, we obtain    
    \begin{align} \label{eq:f_t0t1}
        f(t_1,v)\geq\left(f_0(v) + \frac{t_1}{2}Q_1 \left(f,f,1-f\right)(0,v)\right)e^{-ct_1(1+2(r_0^\gamma + \sqrt{2}^\gamma\delta^\gamma))}
    \end{align}
    for a small enough $t_1$ satisfying
    \begin{align*}
        \frac{1-e^{-2ct_1(1+2(r_0^\gamma + \sqrt{2}^\gamma\delta^\gamma))}}{2c(1+2(r_0^\gamma + \sqrt{2}^\gamma\delta^\gamma))}\geq \frac{t_1}{2}.
    \end{align*}
    It is satisfied when
    \begin{align*}
        t_1\leq \frac{3}{2}\frac{1}{2c(1+2(r_0^\gamma + \sqrt{2}^\gamma\delta^\gamma))}.
    \end{align*}
    For later analysis, we further impose the condition $t_1\leq \frac{1}{2}$. Using \eqref{eq:f>epsilon_0} and \eqref{eq:epsilon_3_f}, we have a lower bounds
    \begin{align*}
        f(t_1,v)\geq\begin{dcases}
            \epsilon e^{-ct_1(1+2(r_0^\gamma + \sqrt{2}^\gamma\delta^\gamma))} & |v-\bar{v}|\leq \delta\\
            C_1 \delta^{3+\gamma}\epsilon^3\eta_1^{\frac{5}{2}} t_1 e^{-ct_1(1+2(r_0^\gamma + \sqrt{2}^\gamma\delta^\gamma))} & \delta<|v-\bar{v}|\leq \sqrt{2}(1-\eta_1)\delta
        \end{dcases}
    \end{align*}
    
    To make the analysis clear, we will assume $C_1\delta^{3+\gamma}\leq 1$, so we always choose the second one to take a lower bound for $|v-\bar{v}|\leq \sqrt{2}(1-\eta_1)\delta$. The condition is achieved by taking $C_1$ or $\delta$ appropriately small. Also, we define
    \begin{align*}
        C_1':=C_1 e^{-ct_1(1+2r_0^\gamma)},\quad c':=2c
    \end{align*}
    to write
    \begin{align}\label{eq:step1_f}
        f(t_1, v)\geq C'_1 \delta^{3+\gamma}\epsilon^3\eta_1^{\frac{5}{2}} t_1 e^{-c' \sqrt{2}^\gamma\delta^\gamma t_1}.
    \end{align}
    
    Similarly, we can construct \eqref{eq:f_t0t1} for $1-f$ using \eqref{eq:f_bound_2} to get
    \begin{align*}
        (1-f)(t_1,v)\geq\left((1-f)(0, v) + \frac{t_1}{2}Q_1 \left(1-f,1-f,f\right)(0,v)\right)e^{-ct_1(1+2(r_0^\gamma + \sqrt{2}^\gamma\delta^\gamma))},
    \end{align*}
    so
    \begin{align}\label{eq:step1_1-f}
        1-f(t_1, v)\geq C'_1 \delta^{3+\gamma}\epsilon^3\eta_1^{\frac{5}{2}} t_1 e^{-c'\sqrt{2}^\gamma\delta^\gamma t_1}
    \end{align}
    for $|v-\bar{v}| \leq \sqrt{2}(1-\eta_1)\delta$.

    Next, we treat $f(t_1,v)$ as an initial function with lower bounds \eqref{eq:step1_f} and \eqref{eq:step1_1-f} and proceed the time by $t_2$. Applying the same step, we get
    \begin{align*}
        &f(t_1+t_2,v)\,(resp.\, 1-f(t_1+t_2,v))\\
        &\geq C'_1\left(\sqrt{2}\delta(1-\eta_1)\right)^{3+\gamma}\left(C'_1 \delta^{3+\gamma}\epsilon^3\eta_1^{\frac{5}{2}} t_1 e^{-c' \sqrt{2}^\gamma \delta^\gamma t_1}\right)^3\eta_2^{\frac{5}{2}}t_2 e^{-c \sqrt{2}^{2\gamma}\delta^\gamma t_2}\\
        &\geq C'_1\delta^{3+\gamma} (C'_1\delta^{3+\gamma})^3\epsilon^{3^2}\left(\sqrt{2}(1-\eta_1)\right)^{3+\gamma}(\delta^{3+\gamma})^3\left(\eta_1^{\frac{5}{2}} t_1 \right)^3\eta_2^{\frac{5}{2}}t_2 e^{-c' 3\sqrt{2}^\gamma \delta^\gamma t_1} e^{-c' \sqrt{2}^{2\gamma}\delta^\gamma t_2}
    \end{align*}
    for $|v-\bar{v}| \leq \sqrt{2}^2(1-\eta_1)(1-\eta_2)\delta$ and $t_2$ satisfying
    \begin{align*}
        t_2< \min\left\{\frac{3}{2}\frac{1}{2c(1+2(r_0^\gamma + \sqrt{2}^{2\gamma}\delta^\gamma))}, \frac{1}{2}\right\}.
    \end{align*}
    
    We further repeat this process for each $t_i$ for $i\geq 3$ and obtain the general formula
    \begin{align}\label{eq:lower_bound1}
        \begin{split}
            &f(\sum_{k=1}^n t_k, v)\, (resp.\, 1-f(\sum_{k=1}^n t_k, v))\\
            &\geq \epsilon^{3^n}(C'_1\delta^{3+\gamma})^{\sum_{k=1}^{n-1} 3^k}\left(\prod_{k=1}^{n-1}\left(\sqrt{2}^k\prod_{i=1}^k (1-\eta_i)\right)^{(3+\gamma)3^{n-1-k}}\right)\\
            &\qquad \times \left(\prod_{k=1}^{n}(\eta_k^{\frac{5}{2}}t_k)^{3^{n-k}}\right) \exp\left(-c' \delta^\gamma \sum_{k=1}^n 3^{n-k} \sqrt{2}^{\gamma k} t_k\right),
        \end{split}
    \end{align}
    where $|v-\bar{v}| \leq \sqrt{2}^2\delta\prod_{k=1}^n (1-\eta_k)$ and $t_k$ given by
    \begin{align}\label{eq:t_cond}
        t_k< \min\left\{\frac{3}{2}\frac{1}{2c(1+2(r_0^\gamma + \sqrt{2}^{k\gamma}\delta^\gamma))}, \frac{1}{2}\right\}.
    \end{align}
    From now on, we will only deal with $f$ and treat the $1-f$ case as a corollary of the $f$ case.
    
    Next, we plug $t_k \coloneqq  t_0^k$ and $\eta_k \coloneqq  \eta_0^k$ for $k = 1, 2, \ldots$ with
    \begin{align}\label{eq:t_0eta_0_cond}
        0<t_0\leq \min\left\{\frac{3}{2}\frac{1}{2c(1+2(r_0^\gamma + \sqrt{2}^\gamma\delta^\gamma))}, \frac{1}{2}\right\},\quad 0<\eta_0< 1-\frac{1}{\sqrt{2}}.
    \end{align}
    Note that it satisfies \eqref{eq:t_cond} since
    \begin{align*}
        t_k= t_0^k\leq \frac{1}{2^{k-1}}\frac{3}{2}\frac{1}{2c(1+2(r_0^\gamma + \sqrt{2}^{\gamma}\delta^\gamma))}\leq \frac{3}{2}\frac{1}{2c(1+2(r_0^\gamma + \sqrt{2}^{k\gamma}\delta^\gamma))}
    \end{align*}
    for $0\leq \gamma\leq 2$.
    
    Let us denote $D_{\eta, l}\coloneqq \prod_{k=1}^l (1-\eta_i)$ with $D_{\eta, 0} = 1$. It is a monotonic decreasing sequence with
    \begin{align*}
        D_\eta \coloneqq \lim_{l\rightarrow \infty} D_{\eta, l} = \lim_{l\rightarrow \infty} e^{\sum_{j=1}^l \ln(1-\eta_0^j)} \geq e^{-2\sum_{j=1}^\infty \eta_0^j} = e^{-\frac{2\eta_0}{1-\eta_0}}
    \end{align*}
    since $\ln(1-x) \geq -2x$ for $0\leq x\leq 1-\frac{1}{\sqrt{2}}$. Also, $\sqrt{2}^l D_{\eta,l}$ is a strictly increasing sequence by the choice of $\eta_0$, so there is only one $l$ for each $v$ such that $\sqrt{2}^{l-1} D_{\eta,l-1}\delta<|v-\bar{v}|\leq \sqrt{2}^l C_{\eta,l}\delta$ or $|v-\bar{v}|\leq \delta$.
    
    We bound each equation in the parenthesis in \eqref{eq:lower_bound1}. First,
    \begin{align*}
        \begin{split}
            \prod_{k=1}^{n-1}\left(\sqrt{2}^k\prod_{i=1}^k (1-\eta_i)\right)^{(3+\gamma)3^{n-1-k}}&\geq \prod_{k=1}^{n-1}\left(\sqrt{2}^k e^{-\frac{2\eta_0}{1-\eta_0}}\right)^{(3+\gamma)3^{n-1-k}}= \left(e^{-\frac{2\eta_0}{1-\eta_0}}\sqrt{2}\right)^{(3+\gamma)\sum_{k=1}^{n-1} k 3^{n-1-k}}\\
            &= \left(e^{-\frac{2\eta_0}{1-\eta_0}}\sqrt{2}\right)^{\frac{3+\gamma}{4}(3^n - 2n - 1)}.
        \end{split}
    \end{align*}
    Also, it holds that
    \begin{align*}
        \prod_{k=1}^n (t_k \eta_k^{\frac{5}{2}})^{3^{n-k}} = (t_0 \eta_0^{\frac{5}{2}})^{\sum_{k=1}^n k \cdot 3^{n-k}} = (t_0 \eta_0^{\frac{5}{2}})^{\frac{1}{4}(3^{n+1}-2n-3)},
    \end{align*}
    and
    \begin{align*}
        \sum_{k=1}^n 3^{n-k} \sqrt{2}^{\gamma k} t_k = 3^n\sum_{k=1}^n \left(\frac{\sqrt{2}^\gamma t_0}{3}\right)^k = 3^n\frac{\sqrt{2}^\gamma t_0}{3}\frac{1-\left(\frac{\sqrt{2}^\gamma t_0}{3}\right)^n}{1-\frac{\sqrt{2}^\gamma t_0}{3}}\leq 3^n \frac{\sqrt{2}^\gamma t_0}{3-\sqrt{2}^\gamma t_0}
    \end{align*}
    for $0 \leq \gamma \leq 2$. Therefore, we obtain
    \begin{align*}
        f(\sum_{k=1}^n t_k, v)&\geq \epsilon^{3^n}(C'_1\delta^{3+\gamma})^{\frac{3^n-1}{2}} \left(e^{-\frac{2\eta_0}{1-\eta_0}}\sqrt{2}\right)^{\frac{3+\gamma}{4}(3^n - 2n - 1)}(t_0 \eta_0^{\frac{5}{2}})^{\frac{1}{4}(3^{n+1}-2n-3)} \exp\left(-c'\frac{\sqrt{2}^\gamma t_0}{3-\sqrt{2}^\gamma t_0} \delta^\gamma 3^n\right).
    \end{align*}
    We take the logarithm function on both sides and get
    \begin{align*}
        \ln f(\sum_{k=1}^n t_k, v) &\geq 3^n\left[\ln \left(\epsilon (C'_1 \delta^{3+\gamma})^{\frac{1}{2}}\left(e^{-\frac{2\eta_0}{1-\eta_0}}\sqrt{2}\right)^{\frac{3+\gamma}{4}}(t_0 \eta_0^{\frac{5}{2}})^{\frac{3}{4}}\right) - c'\frac{\sqrt{2}^\gamma t_0/3}{1-\sqrt{2}^\gamma t_0/3} \delta^\gamma\right]\\
        &\quad - \frac{n}{2}\ln \left(\left(e^{-\frac{2\eta_0}{1-\eta_0}}\sqrt{2}\right)^{3+\gamma}(t_0 \eta_0^{\frac{5}{2}})\right) - C\\
        &\geq 3^n\left[\ln \left(\epsilon (C'_1 \delta^{3+\gamma})^{\frac{1}{2}}\left(e^{-\frac{2\eta_0}{1-\eta_0}}\sqrt{2}\right)^{\frac{3+\gamma}{4}}(t_0 \eta_0^{\frac{5}{2}})^{\frac{3}{4}}\right) - c'\frac{\sqrt{2}^\gamma t_0/3}{1-\sqrt{2}^\gamma t_0/3} \delta^\gamma\right]- C
    \end{align*}
    for some constant $C\geq 0$ and for $|v-\bar{v}| \leq \sqrt{2}^nD_{\eta, n}\delta$.

    There is an ambiguity in choosing $n$ since $n$ is dependent on both the final time $\sum_{k=1}^n t_i$ and $v$. For given $v$ with $|v-\bar{v}|>\delta$, we choose the largest $n$ such that $\sqrt{2}^{n-1} D_{\eta, n-1}\delta< |v-\bar{v}|$. Then,
    \begin{align*}
        3^n\leq \left(\frac{\sqrt{2}}{D_{\eta, n-1}}\frac{|v-\bar{v}|}{\delta} \right)^{2\frac{\ln3}{\ln2}},
    \end{align*}
    so
    \begin{align*}
        \begin{split}
            &\ln f(\sum_{k=1}^n t_k, v)\\
            &\geq \left(\frac{\sqrt{2}}{D_{\eta, n-1}}\frac{|v-\bar{v}|}{\delta} \right)^{2\frac{\ln3}{\ln2}}\Bigg[\ln \left(\epsilon (C'_1 \delta^{3+\gamma})^{\frac{1}{2}}\left(e^{-\frac{2\eta_0}{1-\eta_0}}\sqrt{2}\right)^{\frac{3+\gamma}{4}}(t_0 \eta_0^{\frac{5}{2}})^{\frac{3}{4}}\right)- c'\frac{\sqrt{2}^\gamma t_0/3}{1-\sqrt{2}^\gamma t_0/3} \delta^\gamma\Bigg] - C\\
            &\geq \left(\frac{\sqrt{2}}{D_\eta}\frac{|v-\bar{v}|}{\delta} \right)^{2\frac{\ln3}{\ln2}}\Bigg[\ln \left(\epsilon (C'_1 \delta^{3+\gamma})^{\frac{1}{2}}\left(e^{-\frac{2\eta_0}{1-\eta_0}}\sqrt{2}\right)^{\frac{3+\gamma}{4}}(t_0 \eta_0^{\frac{5}{2}})^{\frac{3}{4}}\right)- c'\frac{\sqrt{2}^\gamma t_0/3}{1-\sqrt{2}^\gamma t_0/3} \delta^\gamma\Bigg] - C.
        \end{split}
    \end{align*}
    
    Next, we choose $t_0$ satisfying
    \begin{align*}
        \sum_{k=1}^n t_0^k = T,
    \end{align*}
    where $T$ is the desired time. If such $t_0$ exists, we just choose it. In this case, $\frac{T}{1+T}\leq t_0\leq T$. Therefore, we have
    \begin{align*}
        &\ln f(T, v)\\
        &\geq \inf_{\frac{T}{1+T}\leq t_0\leq \frac{1}{2}}\left(\frac{\sqrt{2}}{D_\eta}\frac{|v-\bar{v}|}{\delta} \right)^{2\frac{\ln3}{\ln2}}\Bigg[\ln \left(\epsilon (C'_1 \delta^{3+\gamma})^{\frac{1}{2}}\left(e^{-\frac{2\eta_0}{1-\eta_0}}\sqrt{2}\right)^{\frac{3+\gamma}{4}}(t_0 \eta_0^{\frac{5}{2}})^{\frac{3}{4}}\right)- c'\frac{\sqrt{2}^\gamma t_0/3}{1-\sqrt{2}^\gamma t_0/3} \delta^\gamma\Bigg] - C.
    \end{align*}
    Taking the exponential function on both sides, we have
    \begin{align*}
        f(T, v)\geq K_1(T) e^{-K_2(T) |v-\bar{v}|^{2\frac{\ln 3}{\ln 2}}}
    \end{align*}
    for some $K_1(T)$ and $K_2(T)$.
    
    If $t_0$ satisfying $\sum_{k=1}^n t_0^k = T$ is greater than $\min\left\{\frac{3}{2}\frac{1}{2c(1+2(r_0^\gamma + \sqrt{2}^\gamma\delta^\gamma))}, \frac{1}{2}\right\}$, we just choose $t_0 = \min\left\{\frac{3}{2}\frac{1}{2c(1+2(r_0^\gamma + \sqrt{2}^\gamma\delta^\gamma))}, \frac{1}{2}\right\}$. To fill the time gap between $\sum_{k=1}^n t_k$ and $T$, we just put
    \begin{align*}
        f(T,v)\geq f(\sum_{k=1}^n t_k, v) e^{-c(T-\sum_{k=1}^n t_k)(1+|v|^\gamma)}
    \end{align*}
    from \eqref{eq:f_bound_1}. Since $\gamma\leq 2\leq 2\frac{\ln 3}{\ln 2}$, we again get
    \begin{align*}
        f(T, v)\geq K_1(T) e^{-K_2(T) |v-\bar{v}|^{2\frac{\ln 3}{\ln 2}}}
    \end{align*}
    for some $K_1(T)$ and $K_2(T)$. Combining two lower bounds, we finally get the exponential lower bound
    \begin{align}\label{eq:exponential_lower_bound_f}
        f(t,v)\geq K_1(t) e^{-K_2(t) |v|^{2\frac{\ln 3}{\ln 2}}}
    \end{align}
    with constants depending on $\|f_0\|_{1,2},\gamma, C_b, \delta, \epsilon$, and $r_0$. Here, we used $|v-\bar{v}|^{2\frac{\ln 3}{\ln 2}}\leq 5\left(|v|^{2\frac{\ln 3}{\ln 2}} + |\bar{v}|^{2\frac{\ln 3}{\ln 2}}\right)$ and absorbed $|\bar{v}|$ to the constants.
    
    Finally, replacing $f\mapsto 1-f$ in the previous proof, we also have
    \begin{align}\label{eq:exponential_lower_bound_1-f}
        1-f(t,v)\geq K_1(t) e^{-K_2(t) |v|^{2\frac{\ln 3}{\ln 2}}}
    \end{align}
    for the same $K_1(t)$ and $K_2(t)$. By the construction, we have
    \begin{align*}
        \inf_{T^{-1}\leq t\leq T} K_1(t)>0,\quad \sup_{T^{-1}\leq t\leq T} K_2(t)<\infty
    \end{align*}
    for any $1\leq T<\infty$.
    
    \vspace{5mm}
    We are now ready to establish a Gaussian lower bound. From \eqref{eq:exponential_lower_bound_f}, for any given $\bar{\delta}>0$ and $t_0>0$, we can find $0<\bar{\epsilon}<1$ depending on $\|f_0\|_{1,2},\gamma, C_b, c_b, \bar{\delta}, \epsilon,r_0$, and $t_0$ such that 
    \begin{align*}
        f(t_0,v) \geq \bar{\epsilon}, \text{\quad on \quad} |v| \leq \bar{\delta}.
    \end{align*}
    
    By \eqref{eq:epsilon_2}, there exist constants $\overline{C}_1, \overline{C}_2>0$ depending on $\|f_0\|_{1,2}$, $\gamma$, $C_b$, and $c_b$, such that
    \begin{align*}
        Q_1(f,f,1-f)(t_0,v) \geq  {\bar{\delta}}^{3+\gamma} \bar{\epsilon}^2 \left(\overline{C}_1\eta_0^{\frac{5}{2}}- \overline{C}_2\frac{1}{\bar{\delta}^3} \right),
    \end{align*}
    where $\bar{\delta}<|v| \leq \sqrt{2}(1-\eta_0)\bar{\delta}$ for $0<\eta_0<1-\frac{1}{\sqrt{2}}$. We choose
    \begin{align*}
        \eta_0 = \frac{1}{2}\left(1-\frac{1}{\sqrt{2}}\right),\quad \bar{\delta} = \left(2\frac{\overline{C_2}}{\overline{C_1}\eta_0^{5/2}}\right)^{1/3},
    \end{align*}
    Then, we get
    \begin{align}\label{eq:Q1_gaussian_bound1}
        Q_1(f,f,1-f)(t_0,v) \geq  \frac{1}{2}{\bar{\delta}}^{3+\gamma} \bar{\epsilon}^2 \overline{C}_1\eta_0^{\frac{5}{2}}.
    \end{align}
    
    Next, we define
    \begin{align*}
        t_k = t_1^k,\quad  \eta_k = \frac{1}{2^{\frac{k}{4}}}\eta_0,\quad \bar{\delta}_k=\sqrt{2}^{k}\left(\prod_{l=1}^{k} (1-\eta_l)\right)\bar{\delta}
    \end{align*}
    for $k\geq 1$, where $t_1$ satisfies
    \begin{align*}
        0<t_1\leq \min\left\{\frac{3}{2}\frac{1}{2c(1+2(r_0^\gamma + \sqrt{2}^\gamma\bar{\delta}^\gamma))}, \frac{1}{2}\right\}.
    \end{align*}
    Note that $t_1$ and $\eta_0$ satisfies \eqref{eq:t_0eta_0_cond}. By the choice of $\eta_k$ and $\bar{\delta}_k$, we get
    \begin{align*}
        \overline{C}_1\eta_k^{\frac{5}{2}}- \overline{C}_2\frac{1}{\bar{\delta}_k^3} &= \frac{1}{2^{\frac{5}{2}\frac{k}{4}}}\overline{C}_1\eta_0^{\frac{5}{2}}- \frac{1}{\sqrt{2}^{3k}\prod_{l=1}^k(1-\eta_l)^3}\overline{C}_2\frac{1}{\bar{\delta}^3}\\
        &=\frac{1}{2^{\frac{5}{2}\frac{k}{4}}}\left(\overline{C}_1\eta_0^{\frac{5}{2}}- \frac{1}{2^{\frac{7}{8}k}\prod_{l=1}^k (1-\eta_l)^3}\overline{C}_2\frac{1}{\bar{\delta}^3}\right)\\
        &\geq \frac{1}{2^{\frac{5}{2}\frac{k}{4}}}\left(\overline{C}_1\eta_0^{\frac{5}{2}}- \overline{C}_2\frac{1}{\bar{\delta}^3}\right)\\
        &=\frac{1}{2}\frac{1}{2^{\frac{5}{2}\frac{k}{4}}}\overline{C}_1\eta_0^{\frac{5}{2}} = \frac{1}{2}\overline{C}_1\eta_k^{\frac{5}{2}}.
    \end{align*}
    In the middle, we used $2^{\frac{7}{8}k}\prod_{l=1}^k(1-\eta_l)^3\geq 1$ for all $k\geq 1$.

    Under these settings, we can again apply the previous proof for the Gaussian case. The power of $\epsilon$ is changed from $3$ to $2$ in \eqref{eq:Q1_gaussian_bound1}, so it gives the Gaussian lower bound if we follow the previous proof lines. In fact, the only difference in variables is the choice of $\eta_k$, but it satisfies $\prod_{k=1}^l (1-\eta_i)>0$ and
    \begin{align*}
        \prod_{k=1}^n (1-\eta_k)&\geq\lim_{n\rightarrow \infty} \exp\left(\sum_{k=1}^n \ln\left(1-\frac{1}{2^{\frac{k}{4}}}\eta_0\right)\right) \geq \exp\left(-2\sum_{k=1}^\infty \frac{1}{2^{\frac{k}{4}}}\eta_0\right) = \exp\left(-\frac{2}{2^{\frac{1}{4}}-1}\eta_0\right),\\
        \prod_{k=1}^n (t_k \eta_k^{\frac{5}{2}})^{2^{n-k}} &= \left(\frac{t_1}{2^{5/8}}\right)^{\sum_{k=1}^n k \cdot 2^{n-k}}\left(\eta_0^{\frac{5}{2}}\right)^{\sum_{k=1}^n 2^{n-k}} = \left(\frac{t_1}{2^{5/8}}\right)^{2^{n+1}-n-2}\left(\eta_0^{\frac{5}{2}}\right)^{2^n-1}\\
        &\geq\left(\frac{t_1\eta_0^{\frac{5}{4}}}{2^{5/8}}\right)^{2^{n+1}-2}.
    \end{align*}
    Therefore, it modifies just some constants in the Gaussian function.
    
    As a result, we obtain
    \begin{align}\label{eq:Gau_lower}
        f(t,v) \geq K_3(t) e^{-K_4(t)|v|^2}
    \end{align}
    for constants $K_3(t), K_4(t)>0$, which depend on $\|f_0\|_{1,2}, \gamma, C_b,c_b, \bar{\epsilon},\delta,$ and $t\geq t_0$. Since $t_0>0$ is arbitrary, we finally get the Gaussian lower bound for $t>0$. $K_3$ and $K_4$ also fulfill
    \begin{align*}
        \inf_{T^{-1}\leq t\leq T} K_3(t)>0,\quad \sup_{T^{-1}\leq t\leq T} K_4(t)<\infty
    \end{align*}
    for any $1\leq T<\infty$.
    
    We combine \eqref{eq:exponential_lower_bound_1-f} and \eqref{eq:Gau_lower} to complete the proof.
\end{proof}

We write the proof of Theorem \ref{thm:Gaussian_lower_bound} here.

\begin{proof}[proof of Theorem \ref{thm:Gaussian_lower_bound}]
    Fix an arbitrary $t>0$. We first consider the $S(f_0)>0$ case. Using Theorem \ref{thm:main4}, we can find $C_0, r, v_0, r_0$, and $T_0$ depending on $\gamma, C_b, c_b$, and $f_0$ such that 
    \begin{align*}
        C_0t^2\leq f(\min\{t/2, T_0\},v),\quad C_0t^2\leq 1-f(\min\{t/2, T_0\},v)
    \end{align*}
    on $v\in B_r(v_0)$ with $|v_0|\leq r_0$. Using these positivity results at time $\min\left\{t/2, T_0\right\}$, we apply Theorem \ref{thm:lower_maxi} and construct a Gaussian lower bound and an exponential upper bound
    \begin{align*}
        C_1(t) e^{-C_2(t) |v|^2}\leq f(t,v)\leq 1-C_1(t) e^{-C_2(t) |v|^{2\frac{\ln 3}{\ln 2}}}
    \end{align*}
    Collecting all the dependencies, the constants $C_1(t)$ and $C_2(t)$ depend on $\gamma, C_b, c_b$, and $f_0$. It also satisfies
    \begin{align*}
        \inf_{T^{-1}\leq t\leq T}C_1(t)>0,\quad \sup_{T^{-1}\leq t\leq T}C_2(t)<\infty
    \end{align*}
    for any $1\leq T<\infty$ in this case.

    If $S(f_0) = 0$, but $f_0$ is not a saturated equilibrium, using again Theorem \ref{thm:main4}, we again find (1) $T_0$ depending on $\gamma, b(\cos\theta)$, and $f_0$ and (2) $C_0, r, v_0$, and $r_0$ at time $\min\left\{t/2, T_0\right\}$, depending on $\gamma, C_b, c_b, f\left(\frac{1}{2}\min\left\{t/2, T_0\right\}, v\right)$ such that
    \begin{align*}
        C_0\leq f(\min\{t/2, T_0\},v),\quad C_0\leq 1-f(\min\{t/2, T_0\},v)
    \end{align*}
    on $v\in B_r(v_0)$ with $|v_0|\leq r_0$. Taking $f(\min\{t/2, T_0\},v)$ as an initial function, we employ Theorem \ref{thm:lower_maxi} and get the lower and upper bound results.
\end{proof}

\begin{remark}
    This theorem does not guarantee a uniform Gaussian lower bound for a long time. In fact, the constants are improved when $t$ increases from $t=0$ and worsen as $t\rightarrow \infty$. It is because we can not repeatedly apply the theorem for time intervals; the theorem depends on $\delta$, which heavily relies on the shape of the set $\{\epsilon\leq f\leq 1-\epsilon\}$ and the Lebesgue density theorem.
\end{remark}

\section{Creation and propagation of \texorpdfstring{$L^1$}{} polynomial and exponential moments}\label{sec:L1_estimate}
In Section \ref{sec:L1_estimate}, we study polynomial and exponential weighted $L^1$ estimates for creation and propagation. In the first half, we prove polynomial $L^1$ estimates, adapting the classical inequalities in \cite{LU20123305} to the Fermi-Dirac case. In the remaining parts, we show exponential $L^1$ bounds following the classical estimates in \cite{ACGM2013}.

\begin{lemma}[Lemma 3.7 of \cite{LU20123305}]\label{lem:Lu_Mouhot_L1_p_bound}
Assume the collision kernel satisfies $0<\gamma\leq 2$ and (H1), and let $f\in L^1_q$ for all $q\geq 2$. Then, for $s\geq 6$,
\begin{align*}
    \int_{\mathbb{R}^3} Q_c(f,f)(1+|v|^2)^{s/2}\,dv\leq 2^{s+1} C_{b,2} \|f\|_{1,2}\|f\|_{1,s} - \frac{C_{b,2}}{4}\|f\|_{1,0}\|f\|_{1,s+\gamma},
\end{align*}
where $C_{b,2}$ is defined in \eqref{eq:def_C_b_2}.
\end{lemma}

From the definition of $Q_{FD}(f,f)$ and $Q_c(f,f)$, we obtain
\begin{align*}
    Q_{FD}(f,f)(v) &\leq Q_c(f,f)(v) + f(v)\int_{\mathbb{R}^3 \times \mathbb{S}^2} B(|v-v_*|,\cos \theta)  f(v_*) (f(v')+f(v'_*))\,d\sigma dv_*.
\end{align*}
The $s$th moment of the classical term $Q_c(f,f)$ can be bounded using Lemma \ref{lem:Lu_Mouhot_L1_p_bound}, so we need to control the second term. The next lemma is designed for this task.
\begin{lemma} \label{lem:epsv}
    We consider the collision kernel \eqref{eq:B_defi} for $0 < \gamma \leq 2$ and (H1). Assume $f \in L_2^1(\mathbb{R}^3)$ and $0\leq f\leq 1$. Then, there exist constants $C_1>0$ and $C_2>0$ depending on $\gamma$ and $C_b$ such that
    \begin{align*}
        &\int_{\mathbb{R}^3 \times \mathbb{S}^2} B(|v-v_*|,\cos \theta)  f(v_*) (f(v')+f(v'_*))\,d\sigma dv_* \leq \frac{C_1}{\epsilon^3}\|f\|_{1,2}+ C_2 \varphi(\epsilon)(\|f\|_{1,2}+|v|^\gamma\|f\|_{1,0})
    \end{align*}
    for every $0<\epsilon <1$ and $v\in\mathbb{R}^3$. Here, $\varphi(\epsilon)$ is defined in \eqref{eq:phi_def}.
\end{lemma}

\begin{proof}
    By performing the change of variable $\sigma \rightarrow -\sigma$, we have
    \begin{align*}
           \int_{\mathbb{R}^3 \times \mathbb{S}^2} B(|v-v_*|,\cos \theta)  f(v_*) (f(v')+f(v'_*))\,d\sigma dv_*=2\int_{\mathbb{R}^3 \times \mathbb{S}^2}  B(|v-v_*|,\cos \theta)f(v_*) f(v_*')\,d\sigma dv_*.
    \end{align*}
    Next, we divide the interval of $\theta$ into $\epsilon\leq \theta\leq \pi-\epsilon$ and the remainder part for $0<\epsilon<1$ in the $\sigma$-integral.\\
    
    \noindent (1) First, we consider the set $\epsilon\leq \theta\leq \pi-\epsilon$. From \eqref{eq:sin} and $0\leq f\leq 1$, we have
    \begin{align}
        &\int_{\mathbb{R}^3 \times \mathbb{S}^2}
        |v-v_*|^\gamma  b(\cos\theta)f(v_*) f(v'_*) \mathbf{1}_{\{\epsilon\leq \theta\leq \pi-\epsilon\}} \,d\sigma dv_*  \notag \\
        &=\int_{\mathbb{R}^3 \times \mathbb{S}^2} \left(\frac{|v_*'-v_*|}{\sin \frac{\theta}{2}}\right)^\gamma  b(\cos \theta)
        f(v_*) f(v'_*) \mathbf{1}_{\{\epsilon\leq \theta\leq \pi-\epsilon\}} \,d\sigma dv_* \notag \\
        &\leq 2\int_{\mathbb{R}^3 \times \mathbb{S}^2} \frac{1}{\sin^\gamma  \frac{\theta}{2}} b(\cos \theta)  |v'_*|^\gamma  f(v'_*) \mathbf{1}_{\{\epsilon\leq \theta\leq \pi-\epsilon\}} \,d\sigma dv_*  \label{eq:L1_temp1}\\
        &\quad + 2\int_{\mathbb{R}^3 \times \mathbb{S}^2} \frac{1}{\sin^\gamma  \frac{\theta}{2}} b(\cos \theta) |v_*|^\gamma  f(v_*) \mathbf{1}_{\{\epsilon\leq \theta\leq \pi-\epsilon\}} \,d\sigma dv_*.  \label{eq:L1_temp2}
    \end{align} 

    Using \eqref{eq:cancel}, we can bound \eqref{eq:L1_temp1} as follows:
    \begin{align*}
        \begin{split}
            \eqref{eq:L1_temp1}&= 4\pi\int_{\mathbb{R}^3}|v_*|^\gamma f(v_*) \,dv_*\int_\epsilon^{\pi-\epsilon}\frac{1}{\sin^\gamma  \frac{\theta}{2}}\frac{1}{\cos^3 \frac{\theta}{2}} b(\cos \theta)\sin\theta\,d\theta.
        \end{split}
    \end{align*}
    Since $b(\cos \theta)\sin\theta$ is integrable from (H1),
    \begin{align*}
        \int_\epsilon^{\pi-\epsilon}\frac{1}{\sin^\gamma \frac{\theta}{2}}\frac{1}{\cos^3 \frac{\theta}{2}} b(\cos \theta)\sin\theta \,d\theta &\leq \frac{C_b}{2\pi}\max\left\{\frac{1}{\sin^\gamma  \frac{\epsilon}{2}}\frac{1}{\cos^3 \frac{\epsilon}{2}}, \frac{1}{\cos^\gamma  \frac{\epsilon}{2}}\frac{1}{\sin^3 \frac{\epsilon}{2}}\right\}\leq \frac{C}{\epsilon^3}
    \end{align*}
    for some constant $C$. Since $0<\gamma\leq 2$, we have $|v_*|^\gamma\leq 1+|v_*|^2$, so $\int_{\mathbb{R}^3}|v_*|^\gamma f(v_*) \,dv_*\leq \|f\|_{1,2}$. Combining these two, we obtain
    \begin{align*}
        \eqref{eq:L1_temp1} \leq \frac{C}{\epsilon^3}\|f\|_{1,2}
    \end{align*}
    for some $C$ depending on $C_b$ and $\gamma$. Bounding \eqref{eq:L1_temp2} is more simple: we have
    \begin{align*}
        \eqref{eq:L1_temp2}&\leq 4\pi\int_{\mathbb{R}^3} |v_*|^\gamma  f(v_*)\,dv_* \int_\epsilon^{\pi-\epsilon} \frac{1}{\sin^\gamma  \frac{\theta}{2}} b(\cos \theta)\sin\theta\,d\theta\\
        &\leq \frac{C}{\epsilon^\gamma} \|f\|_{1,2}.
    \end{align*}
    We add these two results and obtain an upper bound for the set $\epsilon\leq \theta\leq \pi-\epsilon$ by
    \begin{align*}
        \int_{\mathbb{R}^3 \times \mathbb{S}^2}
        |v-v_*|^\gamma  b(\cos\theta)f(v_*) f(v'_*) \mathbf{1}_{\{\epsilon\leq \theta\leq \pi-\epsilon\}} \,d\sigma dv_*\leq \frac{2C}{\epsilon^3}\|f_0\|_{1,2}.
    \end{align*}\\

    \noindent (2) Now, we consider the remainder. Since $b(\cos\theta)\sin\theta\,d\theta$ is integrable, we can define $\varphi(\epsilon)$ by \eqref{eq:phi_def}. Then,
    \begin{align*}
        &\int_{\mathbb{R}^3 \times \mathbb{S}^2} |v-v_*|^\gamma  b(\cos \theta)f(v_*) f(v'_*) \left(\mathbf{1}_{\{0<\theta<\epsilon\}}(\theta) + \mathbf{1}_{\{\pi-\epsilon<\theta<\pi\}}(\theta)\right)\,d\sigma dv_*\\
        &\leq 2 \int_{\mathbb{R}^3} (|v|^\gamma +|v_*|^\gamma )f(v_*) \,dv_* \int_{\mathbb{S}^2}\left(\mathbf{1}_{\{0<\theta<\epsilon\}}(\theta) + \mathbf{1}_{\{\pi-\epsilon<\theta<\pi\}}(\theta)\right)b(\cos\theta)\,d\sigma\\
        &\leq C\varphi(\epsilon)(\|f\|_{1,2}+|v|^\gamma\|f\|_{1,0})
    \end{align*}
    for some constant $C$ depending on $\gamma$.
    
    From the two inequalities, we get the lemma.
\end{proof}

Using this lemma, we extend Lemma \ref{lem:Lu_Mouhot_L1_p_bound} to the Fermi-Dirac case.
\begin{lemma}\label{lem:Quantum_L1_p_bound}
Assume the collision kernel satisfies $0<\gamma\leq 2$ and (H1), and let $f\in L^1_q$ for all $q\geq 2$ with $0\leq f\leq 1$. Then, for $s\geq 6$,
\begin{align*}
    \int_{\mathbb{R}^3} Q_{FD}(f,f)(1+|v|^2)^{s/2}\,dv\leq C \|f\|_{1,2}\|f\|_{1,s} - \frac{C_{b,2}}{8}\|f\|_{1,0}\|f\|_{1,s+\gamma},
\end{align*}
where $C$ is a constant depending on $s$, $\gamma$, $C_b$, $C_{b,2}$, and $\varphi(\epsilon)$.
\end{lemma}
\begin{proof}
    By the splitting of $Q_{FD}$, we write
    \begin{align}
        &\int_{\mathbb{R}^3} Q_{FD}(f,f)(1+|v|^2)^{s/2}\,dv\notag\\
        &\leq\int_{\mathbb{R}^3 } Q_c(f,f)(t,v)(1+|v|^2)^{s/2} \,d\sigma dv_* dv \label{eq:sp_1_1} \\
        &\quad + \int_{\mathbb{R}^6 \times \mathbb{S}^2} B(|v-v_*|, \cos \theta)
        f(t,v)f(t,v_*) (f(t,v')+f(t,v_*')) (1+|v|^2)^{s/2}\,d\sigma dv_* dv \label{eq:sp_1_2}.
    \end{align}
    From Lemma \ref{lem:Lu_Mouhot_L1_p_bound} and \ref{lem:epsv} for $s\geq 6$,
    \begin{align*}
        \eqref{eq:sp_1_1}&\leq 2^{s+1} C_{b,2} \|f\|_{1,2}\|f\|_{1,s} - \frac{C_{b,2}}{4}\|f\|_{1,0}\|f\|_{1,s+\gamma},\\
        \eqref{eq:sp_1_2}&\leq \left(C_1\frac{1}{\epsilon^3}+C_2\varphi(\epsilon)\right)\|f\|_{1,2}\|f\|_{1,s} + C_2\varphi(\epsilon)\|f\|_{1,0}\|f\|_{1,s+\gamma}.
    \end{align*}
    We choose $\epsilon_*$ such that $C_2 \varphi(\epsilon_*) = \frac{C_{b,2}}{8}$, then
    \begin{align*}
        \int_{\mathbb{R}^3} Q_{FD}(f,f)(1+|v|^2)^{s/2}\,dv\leq \left(2^{s+1} C_{b,2} +C_1\frac{1}{\epsilon_*^3}+C_2\frac{C_{b,2}}{8}\right)\|f\|_{1,2}\|f\|_{1,s} - \frac{C_{b,2}}{8}\|f\|_{1,0}\|f\|_{1,s+\gamma}.
    \end{align*}
\end{proof}

The next theorem states and proves Theorem \ref{thm:L1_bound}-(1) assuming that a solution of the Boltzmann-Fermi-Dirac equation $f(t,v)$ satisfies $\|f\|_{1,s}(t)\in C^1((0,\infty))$. In Section \ref{sec:Well-posedness}, we will prove the existence and uniqueness of the solution of the Boltzmann-Fermi-Dirac equation using these \textit{a priori} estimates. In consequence, we discard the \textit{a priori} assumptions and get Theorem \ref{thm:L1_bound}-(1).

\begin{theorem}\label{lem:poly_L1}
    We consider the collision kernel \eqref{eq:B_defi} satisfying $0 < \gamma \leq 2$ and (H1). For a solution $f$ of the Boltzmann-Fermi-Dirac equation, assume $\|f\|_{1,s}(t)\in C^1((0,\infty))$ for all $s\geq 2$. Then, there exists a constant $C_{s,1} \geq 0$ for $s\geq 2$ depending on $s$, $\gamma$, $C_b$, $\varphi(\epsilon)$, $C_{b,2}$, $\|f_0\|_{1,0}^{-1}$, and $\|f_0\|_{1,2}$ such that
    \begin{align}  \label{eq:poly_cre}
        \|f\|_{1,s}(t) \leq C_{s,1} \max\left\{t^{-\frac{s-2}{\gamma}}, 1\right\}
    \end{align}
    for $t>0$ and $s>2$. Furthermore, if $\|f\|_{1,s}(0)$ is finite, $f\in C([0,\infty), L^1_s)$, and $\|f\|_{1,s}(t)\in C^1([0,\infty))$, then there exists a constant $C_{s,2} \geq 0$ depending on $s$, $\gamma$, $C_b$, $C_{b,2}$, $\varphi(\epsilon)$, $\|f_0\|_{1,0}^{-1}, \|f_0\|_{1,2}$, and $\|f_0\|_{1,s}$ such that
    \begin{align} \label{eq:poly_pro}
        \|f\|_{1,s}(t) \leq C_{s,2}
    \end{align}
    for $t\geq 0$ and $s\geq 2$.
\end{theorem}
\begin{proof}
    From Lemma \ref{lem:Quantum_L1_p_bound},
    \begin{align*}
        \dv{t}\|f\|_{1,s}(t) = \int_{\mathbb{R}^3}Q_{FD}(f,f)(1+|v|^2)^{s/2}\,dv\leq C(s)\|f\|_{1,2}\|f\|_{1,s} - \frac{C_{b,2}}{8}\|f\|_{1,0}\|f\|_{1,s+\gamma}
    \end{align*}
    for $s\geq 6$ for some constant $C$ depending on $s$, $\gamma$, $C_b$, $\varphi(\epsilon)$, and $C_{b,2}$. Also, by H{\"o}lder's inequalty,
    \begin{align*}
        \|f\|_{1,2}^{-\frac{\gamma}{s-2}}\|f\|_{1,s}^{\frac{s - 2 + \gamma}{s-2}}\leq \|f\|_{1,s+\gamma}.
    \end{align*}
    Therefore,
    \begin{align}\label{eq:d/dt msp_2}
        \dv{t}\|f\|_{1,s}(t)\leq C(s)\|f\|_{1,2}\|f\|_{1,s} - \frac{C_{b,2}}{8}\|f\|_{1,0}\|f\|_{1,2}^{-\frac{\gamma}{s-2}}\|f\|_{1,s}^{\frac{s - 2 + \gamma}{s-2}}
    \end{align}
    for $s\geq 6$.

    By the differential inequality \eqref{eq:d/dt msp_2}, we can deduce
    \begin{align*}
        \|f\|_{1,s}(t)&\leq \left(\frac{C(s)\|f\|_{1,2}}{\frac{C_{b,2}}{8}\|f\|_{1,0}\|f\|_{1,2}^{-\frac{\gamma}{s-2}}}\frac{1}{1-\exp\left(-\frac{\gamma}{s-2}C(s)\|f\|_{1,2}t\right)}\right)^{\frac{s-2}{\gamma}}\\
        &=\|f\|_{1,2}\left(\frac{C(s)\|f\|_{1,2}}{\frac{C_{b,2}}{8}\|f\|_{1,0}}\frac{1}{1-\exp\left(-\frac{\gamma}{s-2}C(s)\|f\|_{1,2}t\right)}\right)^{\frac{s -2}{\gamma}}
    \end{align*}
    for $s\geq 6$.

    Now, let $2\leq s< 6$. By the interpolation,
    \begin{align*}
        \|f\|_{1,s}&\leq \|f\|_{1,2}^{\frac{6-s}{4}}\|f\|_{1,6}^{\frac{s-2}{4}}\\
        &\leq \|f\|_{1,2}^{\frac{6-s}{4}}\left(\|f\|_{1,2}\left(\frac{C(6)\|f\|_{1,2}}{\frac{C_{b,2}}{8}\|f\|_{1,0}}\frac{1}{1-\exp\left(-\frac{\gamma}{4}C(6)\|f\|_{1,2}t\right)}\right)^{\frac{4}{\gamma}}\right)^{\frac{s-2}{4}}\\
        &=\|f\|_{1,2}\left(\frac{C(6)\|f\|_{1,2}}{\frac{C_{b,2}}{8}\|f\|_{1,0}}\frac{1}{1-\exp\left(-\frac{\gamma}{4}C(6)\|f\|_{1,2}t\right)}\right)^{\frac{s-2}{\gamma}}.
    \end{align*}
    Finally, as
    \begin{align*}
        \frac{1}{1-e^{-\frac{\gamma}{4}C(6)\|f\|_{1,2}t}}\leq 1+\frac{1}{\frac{\gamma}{4}C(6)\|f\|_{1,2}t},
    \end{align*}
    we finally get
    \begin{align*}
        \|f\|_{1,s}\leq \|f\|_{1,2}\left(\frac{1}{\frac{C_{b,2}}{8}\|f\|_{1,0}}\left(C\|f\|_{1,2}+\frac{1}{\frac{\gamma}{4}t}\right)\right)^{\frac{s-2}{\gamma}}
    \end{align*}
    for a constant $C$ depending on $s$, $\gamma$, $C_b$, $\varphi(\epsilon)$, and $C_{b,2}$. It proves \eqref{eq:poly_cre}.

    If $\|f\|_{1,s}$ is finite, using a maximum principle argument to \eqref{eq:d/dt msp_2}, we get \eqref{eq:poly_pro}.
\end{proof}

Now, we turn to the exponential $L^1$ estimates. As noted in the first paragraph in the beginning of this section, the main stream of the proof follows \cite{ACGM2013}. We first write a kind of Povzner inequality.
\begin{lemma}[Lemma 3 of \cite{ACGM2013}]\label{lem:bobylev_2}
    Assume the collision kernel satisfies $0<\gamma\leq 2$ and (H1). Then, there exists a constant $\varpi_p>0$ for each $p\geq 1$, depending on $b(\cos\theta)$, such that
\begin{align*}
    \int_{\mathbb{S}^2}(|v'|^{2p} + |v'_*|^{2p})b(\cos\theta)\,d\sigma\leq C_b\varpi_p(|v|^2 + |v_*|^2)^p.
\end{align*}
Also, it satisfies $\varpi_1=1$, $p \rightarrow \varpi_p$ is strictly decreasing, and $\lim_{p\rightarrow \infty} \varpi_p = 0$.
\end{lemma}

Next, we define a $p$th moment function of $f(t,v)$ and a combination function $S_{s,p}$.
\begin{definition}
    For $p\geq 0$ and $t>0$, we define
    \begin{align*}
        m_p = m_p(t) \coloneqq  \int_{\mathbb{R}^3} f(t,v) |v|^p \,dv.
    \end{align*}
    For $s>0$, $t>0$, and integers $p\geq 2$, we define 
    \begin{align*}
        S_{s,p} =  S_{s,p}(t) \coloneqq \sum_{k=1}^{k_p} 
            \binom{p}{k}
            \left(m_{sk+\gamma}m_{s(p-k)}+m_{sk}m_{s(p-k)+\gamma} \right),
    \end{align*}
    where $k_p$ is the integer part of $(p+1)/2$. Here, $0<\gamma \leq 2$ is the power of the velocity part of the collision kernel in \eqref{eq:B_defi}.
\end{definition}

We refer to a classical differential inequality for $m_p(t)$ in the next lemma. For technical reason, we temporarily replace $p\mapsto sp$ for some $s\in (0,2]$ and $p\geq \frac{2}{s}$. We will later choose $s = \gamma$, which was defined in $B(v-v_*,\sigma) = |v-v_*|^\gamma b(\cos\theta)$.

\begin{lemma}[Lemma 6 of \cite{ACGM2013}]\label{lem:6} 
Assume the collision kernel satisfies $0<\gamma\leq 2$ and (H1), and let $f\in L^1_q$ for all $q\geq 2$. Then, for $s \in (0,2]$ and integers $p > 2/s$,
    \begin{align} \label{eq:msp/t, ACGM}
        \int_{\mathbb{R}^3}Q_{c}(f,f)|v|^{sp}\,dv
        \leq C_b\left(2\varpi_{sp/2}S_{s,p} - K_1 m_{sp+\gamma} + K_2 m_{sp}\right),
    \end{align}
    where
    \begin{align*}
        K_1 \coloneqq 2^{2-\gamma}(1-\varpi_{sp/2}) m_0,\quad K_2\coloneqq 2m_\gamma
    \end{align*}
    for $t>0$.
\end{lemma}
If $f(t,v)$ is a solution of the classical Boltzmann equation with appropriate assumptions, it directly implies
\begin{align*}
    \dv{t}m_{sp}(t)\leq C_b\left(2\varpi_{sp/2}S_{s,p} - K_1 m_{sp+\gamma} + K_2 m_{sp}\right).
\end{align*}

NOw, we extend Lemma \ref{lem:6} to our Fermi-Dirac case. The main idea is to bound the extra terms using Lemma \ref{lem:epsv}.

\begin{lemma} \label{lem:d/dt_m_sp}
    We consider the collision kernel \eqref{eq:B_defi} for $0<\gamma \leq 2$ and (H1), and let $f$ be a solution of the Boltzmann-Fermi-Dirac equation with $f\in C([0,\infty), L^1_q)$ and $m_q(t)\in C^1([0,\infty))$ for all $q\geq 2$. For $s \in (0,2]$ and integers $p\geq p_0 > 2/s$, following the constants $K_1$ and $K_2$ in Lemma \ref{lem:6}, we obtain
    \begin{align} \label{eq:d/dt msp}
        \frac{d}{dt}m_{sp} 
        \leq C_b\left(2\varpi_{sp/2}S_{s,p}
        - \frac{K_1}{2} m_{sp+\gamma} + 
        K_2'm_{sp}\right)
    \end{align}
    for $t>0$, where $K'_2$ is a large enough constant depending on $K_1, K_2, \|f_0\|_{1,2}, \gamma$, $C_b$, and $\varphi(\epsilon)$.
\end{lemma}
\begin{proof}
    Since $f\in C([0,\infty), L^1_q)$ and $m_q(t)\in C^1([0,\infty))$ for all $q\geq 2$, all the quantities in \eqref{eq:d/dt msp} are all well-defeind. Since $f$ is the solution of the Boltzmann-Fermi-Dirac equation, we have
    \begin{align}
        &\frac{d}{dt} \int_{\mathbb{R}^3} f(t,v) |v|^{sp} \,dv \notag \\
        &\leq \int_{\mathbb{R}^6\times\mathbb{S}^2} Q_c(f,f)(t,v)|v|^{sp} \,d\sigma dv_* dv \label{eq:sp_1} \\
        &\quad + \int_{\mathbb{R}^6 \times \mathbb{S}^2} B(|v-v_*|, \cos \theta)
        f(t,v)f(t,v_*) (f(t,v')+f(t,v_*')) |v|^{sp}\,d\sigma dv_* dv. \label{eq:sp_3}
    \end{align}
    
    Using Lemma \ref{lem:6}, the first classical term is bounded by
    \begin{align}\label{eq:est_sp_1}
        \begin{split}
            (\ref{eq:sp_1}) &\leq C_b\left(2\varpi_{sp/2}S_{s,p} - K_1 m_{sp+\gamma} + K_2 m_{sp}\right).
        \end{split}
    \end{align}
    
    Next, we consider \eqref{eq:sp_3}. By Lemma \ref{lem:epsv}, we can find a constant $C>0$ that depends on $\|f_0\|_{1,2}, \gamma$, and $C_b$, such that 
    \begin{align} 
        \eqref{eq:sp_3} \leq \left(\frac{C_1}{\epsilon^3}\|f\|_{1,2}m_{sp}+C_2\varphi(\epsilon) (\|f\|_{1,2}m_{sp} + \|f\|_{1,0}m_{sp+\gamma})\right).\label{eq:est_sp_3}
    \end{align} 

    We choose $\epsilon_*>0$ such that
    \begin{align*}
        C_2\varphi(\epsilon_*) = \frac{C_b}{2}K_1.
    \end{align*}
    For $\epsilon = \epsilon_*$, combining \eqref{eq:est_sp_1} and \eqref{eq:est_sp_3}, we obtain
    \begin{align*}
        \frac{d}{dt}m_{sp} 
        &\leq C_b\left(2\varpi_{sp/2}S_{s,p} - \frac{K_1}{2} m_{sp+\gamma} + \left(K_2 + \left(\frac{C_1}{\epsilon_*^3 C_b} + \frac{C_2}{C_b}\varphi(\epsilon_*)\right)\|f\|_{1,2}\right)m_{sp}\right).
    \end{align*} 
    It proves the lemma.
\end{proof}

\begin{remark}
    In fact, we can deduce the $L^1$ polynomial bound from \eqref{eq:d/dt msp}. However, it gives inferior estimate when $sp$ is near $2$ compared to Lemma \ref{lem:Quantum_L1_p_bound} as $\varpi_1 = 0$. Since it gives strong estimates when $sp\rightarrow \infty$, we can detour this problem by estimating the lower moment as an interpolation between $m_2$ and a higher moment.
\end{remark}

Now, we give a proof of \ref{thm:L1_bound}-(2) assuming that $\|f\|_{1,s}(t)\in C^1((0,\infty))$ for all $s\geq 2$ as in Theorem \ref{lem:poly_L1}. Key idea is to replace the classical Lemma \ref{lem:6} in \cite{ACGM2013} by Lemma \ref{lem:d/dt_m_sp}.

\begin{proof}[Proof of Theorem \ref{thm:L1_bound}-(2)]
    We follow the proof of Theorem 1 and Theorem 2 of \cite{ACGM2013} starting from \eqref{eq:d/dt msp} instead of the classical inequality \eqref{eq:msp/t, ACGM}. The only difference between \eqref{eq:msp/t, ACGM} and \eqref{eq:d/dt msp} is in the coefficients in front of $m_{sp}$ and $m_{sp+\gamma}$, so we can use the same arguments and obtain the same results.
\end{proof}

\section{Well-posedness of the solution of the Boltzmann-Fermi-Dirac equation}\label{sec:Well-posedness}
In this section, we prove the well-posedness of the solution of the Boltzmann-Fermi-Dirac equation. We first start with the simple equality.
\begin{lemma}\label{lem:fg+}
Let $f$ and $g$ be solutions of the Boltzmann-Fermi-Dirac equation. Then, it satisfies
    \begin{align}\label{eq:f-g+}
        (f(b, v) - g(b, v))^+ = (f(a, v) - g(a, v))^+ + \int_a^b (Q_{FD}(f,f) - Q_{FD}(g,g))(\tau, v)\mathbf{1}_{\{f(\tau, v)\geq g(\tau, v)\}}\,d\tau
    \end{align}
    for all $0\leq a\leq b$ and a.e. $v$.
\end{lemma}
\begin{proof}
    Since $f$ and $g$ are absolutely continuous about $t$ for a.e. fixed $v$, and $\phi(x) = \max\{x, 0\}$ is Lipschitz continuous, we have
    \begin{align*}
        \dv{t}\phi(f(t,v) - g(t,v)) &= \phi'(f(t,v) - g(t,v))(f(t,v) - g(t,v))' \\
        &= \left(Q_{FD}(f,f) - Q_{FD}(g,g)\right)(t,v)\mathbf{1}_{\{f(t,v)>g(t,v)\}}
    \end{align*}
    a.e. $t$ for a.e. fixed $v$. Integrating both sides about $t\in [a,b]$, we get \eqref{eq:f-g+}.
\end{proof}

The next lemma is an integral inequality used in this section.
\begin{lemma}[Lemma 2 of \cite{Lu2001353}]\label{lem:temp}
    Let $s\geq 0$ and the collision kernel satisfies (H1) and $0\leq \gamma\leq 2$. For $\|f\|_{1, \max\{s+\gamma,2\}}<\infty$ with $0\leq f\leq 1$, we have
    \begin{align*}
        &\int_{\mathbb{R}^3\times\mathbb{S}^2} B(v-v_*, \sigma) f'f'_*(1+|v_*|^2)^{s/2}\,d\sigma dv_*\\
        &\leq C_1\|f\|_{1,s+\gamma}(1+|v|^2)^{\gamma/2} + C_2\|f\|_{1,0}(1+|v|^2)^{(s+\gamma)/2},
    \end{align*}
    where the constants $C_1$ and $C_2$ depend on $s$, $\gamma$, and $C_b$.
\end{lemma}

The next lemma is a sharp version of the previous lemma. It is the crucial inequality in showing the $L^1_2$ stability.
\begin{lemma}\label{lem:f'f'_*_k}
    Let $0\leq k\leq 3$ and the collision kernel satisfies (H1) and $0\leq \gamma\leq 2$. For $\|f\|_{1, \max\{k+\gamma,2\}}<\infty$ with $0\leq f\leq 1$, we have
    \begin{align}\label{eq:f'f'_*_k}
        \begin{split}
            &\int_{\mathbb{R}^3\times \mathbb{S}^2} B(v-v_*,\sigma) f'f'_*(1+|v_*|^2)^{k/2}\,d\sigma dv_*\\
            &\leq C\left((1+|v|^2)^{\frac{\gamma}{2}} \|f\|_{1,k} + \|f\|_{1,k+\gamma} + \left((1+|v|^2)^{\frac{k+\gamma}{2}}\|f\|_{1,0} + \|f\|_{1, k+\gamma}\right)^{\frac{3-k}{3}}\|f\|_{1, k+\gamma}^{\frac{k}{3}}\right),
        \end{split}
    \end{align}
    where $C$ depends on $k$, $\gamma$, and $C_b$.
\end{lemma}

\begin{proof}
    It is a slightly refined version of Lemma 3 of \cite{Lu20031}. In fact, we will follow the proof in \cite{Lu20031}; the only difference is  we use $|v-v_*|^\gamma\leq 2\left(|v|^{\gamma} + |v_*|^{\gamma}\right)$. For completeness, we present the proof here.
    
    Using $(1+|v_*|^2)^{k/2}\leq 2^{k/2}(1+|v_*|^k)$, we divide the integral by $2^{k/2}(I(v)+J(v))$, where
    \begin{align*}
        I(v) &\coloneqq \int_{\mathbb{R}^3\times \mathbb{S}^2} B(v-v_*,\sigma) f'f'_*\,d\sigma dv_*,\\
        J(v) &\coloneqq \int_{\mathbb{R}^3\times \mathbb{S}^2} B(v-v_*,\sigma) f'f'_*|v_*|^k\,d\sigma dv_*.
    \end{align*}
    From Lemma \ref{lem:temp}, we get
    \begin{align*}
        I(v)\leq C(1+|v|^\gamma),
    \end{align*}
    where $C$ depends on $C_b$, $\gamma$, and $\|f\|_{1,2}$. For $J(v)$, we further decompose it by
    \begin{align*}
        J_1(v) &\coloneqq \int_{\mathbb{R}^3\times \mathbb{S}^2} B(v-v_*,\sigma) f'f'_*|v_*|^k \mathbf{1}_{\{0\leq \theta\leq \frac{\pi}{2}\}}\mathbf{1}_{\{|v_*|\leq 2|v_*'|\}}\,d\sigma dv_*,\\
        J_2(v) &\coloneqq \int_{\mathbb{R}^3\times \mathbb{S}^2} B(v-v_*,\sigma) f'f'_*|v_*|^k \mathbf{1}_{\{0\leq \theta\leq \frac{\pi}{2}\}}\mathbf{1}_{\{|v_*|\geq 2|v_*'|\}}\,d\sigma dv_*,\\
        J_3(v) &\coloneqq \int_{\mathbb{R}^3\times \mathbb{S}^2} B(v-v_*,\sigma) f'f'_*|v_*|^k \mathbf{1}_{\{\frac{\pi}{2}< \theta\leq \pi\}}\mathbf{1}_{\{|v_*|\leq 2|v'|\}}\,d\sigma dv_*,\\
        J_4(v) &\coloneqq \int_{\mathbb{R}^3\times \mathbb{S}^2} B(v-v_*,\sigma) f'f'_*|v_*|^k \mathbf{1}_{\{\frac{\pi}{2}< \theta\leq \pi\}}\mathbf{1}_{\{|v_*|\geq 2|v'|\}}\,d\sigma dv_*.
    \end{align*}
    For $J_1$ and $J_3$, from \eqref{eq:cancel}, we have
    \begin{align}\label{eq:J_1,J_3_bound}
        \begin{split}
            J_1(v) + J_3(v) &\leq 2^k\int_{\mathbb{R}^3\times \mathbb{S}^2} B(v-v_*,\sigma) f'f'_*\left(|v'_*|^k \mathbf{1}_{\{0\leq \theta\leq \frac{\pi}{2}\}} + |v'|^k \mathbf{1}_{\{\frac{\pi}{2}<\theta\leq \pi\}}\right)\,d\sigma dv_*\\
            &\leq 2^{k+2}\pi\int_{\mathbb{R}^3}\int_0^{\frac{\pi}{2}} \frac{\sin\theta}{\cos^3\frac{\theta}{2}}B\left(\frac{|v-v_*|}{\cos\frac{\theta}{2}},\cos\theta\right) f(v_*)|v_*|^k\,d\theta dv_*\\
            &= 2^{k+2}\pi\int_{\mathbb{R}^3}\int_0^{\frac{\pi}{2}} \frac{\sin\theta}{\cos^{3+\gamma}\frac{\theta}{2}}|v-v_*|^\gamma b(\cos\theta) f(v_*)|v_*|^k \,d\theta dv_*\\
            &\leq 2^{k+\frac{5+3\gamma}{2}}C_b\int_{\mathbb{R}^3}\left(|v|^\gamma + |v_*|^\gamma\right) f(v_*)|v_*|^k \,d\theta dv_*\\
            &\leq 2^{k+\frac{5+3\gamma}{2}}C_b\left((1+|v|^2)^{\frac{\gamma}{2}} \|f\|_{1,k} + \|f\|_{1,k+\gamma}\right).
        \end{split}
    \end{align}
    For $J_2$, since $|v_*|\geq 2|v'_*|$,
    \begin{align*}
        \frac{|v_*|}{2}\leq |v_*-v'_*| = |v-v_*|\sin\frac{\theta}{2}.
    \end{align*}
    Therefore,
    \begin{align*}
        J_2(v)\leq 2^k\int_{\mathbb{R}^3\times \mathbb{S}^2} |v-v_*|^{k+\gamma}b(\cos\theta)\left(\sin \frac{\theta}{2}\right)^k f'f'_* \mathbf{1}_{\{0\leq \theta\leq \frac{\pi}{2}\}}\,d\sigma dv_*.
    \end{align*}
    We further divide it by
    \begin{align*}
        J_{21}(v) \coloneqq 2^k\int_{\mathbb{R}^3\times \mathbb{S}^2} |v-v_*|^{k+\gamma}b(\cos\theta)\left(\sin \frac{\theta}{2}\right)^k f'f'_* \mathbf{1}_{\{0\leq \theta\leq \frac{\pi}{2}\}}\mathbf{1}_{\{|v'|\leq |v_*'|\}}\,d\sigma dv_*,\\
        J_{22}(v) \coloneqq 2^k\int_{\mathbb{R}^3\times \mathbb{S}^2} |v-v_*|^{k+\gamma}b(\cos\theta)\left(\sin \frac{\theta}{2}\right)^k f'f'_* \mathbf{1}_{\{0\leq \theta\leq \frac{\pi}{2}\}}\mathbf{1}_{\{|v'|> |v_*'|\}}\,d\sigma dv_*.
    \end{align*}
    For $J_{21}$, using $|v-v_*|\leq |v'|+|v'_*|\leq 2|v'_*|$ and \eqref{eq:cancel} again, we get
    \begin{align*}
        J_{21}(v) &\leq 2^{2k+\gamma}\int_{\mathbb{R}^3\times \mathbb{S}^2} |v'_*|^{k+\gamma}b(\cos\theta)\left(\sin \frac{\theta}{2}\right)^k f'_* \mathbf{1}_{\{0\leq \theta\leq \frac{\pi}{2}\}}\,d\sigma dv_*\\
        &\leq 2^{2k+\gamma+1}\pi\int_{\mathbb{R}^3}f(v_*)|v_*|^{k+\gamma}\int_0^{\frac{\pi}{2}} \frac{b(\cos\theta)}{\left(\cos\frac{\theta}{2}\right)^3}\left(\sin \frac{\theta}{2}\right)^k\sin\theta \,d\theta dv_*\\
        &\leq 2^{\frac{5}{2}k+\gamma+\frac{5}{2}}\pi\int_{\mathbb{R}^3}f(v_*)|v_*|^{k+\gamma}\int_0^{\frac{\pi}{2}} b(\cos\theta)\sin\theta \,d\theta dv_*\\
        &\leq 2^{\frac{5}{2}k+\gamma+\frac{3}{2}} C_b\|f\|_{1, k+\gamma}.
    \end{align*}
    For $J_{22}$, we first use H{\"o}lder's inequality and get
    \begin{align}\label{eq:J_22}
        \begin{split}
            J_{22}(v) &\leq 2^k\left(\int_{\mathbb{R}^3\times \mathbb{S}^2} |v-v_*|^{k+\gamma}b(\cos\theta)(f'_*)^p \mathbf{1}_{\{0\leq \theta\leq \frac{\pi}{2}\}}\,d\sigma dv_*\right)^{1/p}\\
            &\qquad \times \left(\int_{\mathbb{R}^3\times \mathbb{S}^2} |v-v_*|^{k+\gamma}b(\cos\theta)\left(\sin \frac{\theta}{2}\right)^{kq}(f')^q \mathbf{1}_{\{0\leq \theta\leq \frac{\pi}{2}\}}\mathbf{1}_{\{|v'|> |v_*'|\}}\,d\sigma dv_*\right)^{1/q}
        \end{split}
    \end{align}
    for some $p\geq 1$ and $\frac{1}{p} + \frac{1}{q} = 1$. Next, we use $f\leq 1$ and \eqref{eq:cancel} for the first term to get
    \begin{align*}
        \int_{\mathbb{R}^3\times \mathbb{S}^2} |v-v_*|^{k+\gamma}b(\cos\theta)(f'_*)^p \mathbf{1}_{\{0\leq \theta\leq \frac{\pi}{2}\}}\,d\sigma dv_*&\leq 2^{\frac{k+\gamma+3}{2}}C_b\int_{\mathbb{R}^3}|v-v_*|^{k+\gamma}f(v_*)dv_*\\
        &\leq 2^{\frac{k+\gamma+3}{2}} C_b((1+|v|^2)^{\frac{k+\gamma}{2}}\|f\|_{1,0} + \|f\|_{1, k+\gamma}).
    \end{align*}
    For the second one, we choose $\frac{1}{p} = 1-\frac{k}{3}$ to make $kq = 3$. Then,
    \begin{align*}
        &\int_{\mathbb{R}^3\times \mathbb{S}^2} |v-v_*|^{k+\gamma}b(\cos\theta)\left(\sin \frac{\theta}{2}\right)^{kq}(f')^q \mathbf{1}_{\{0\leq \theta\leq \frac{\pi}{2}\}}\mathbf{1}_{|v'|> |v_*'|}\,d\sigma dv_*\\
        &\leq 2^{k+\gamma}\int_{\mathbb{R}^3\times \mathbb{S}^2} b(\cos\theta)\left(\sin \frac{\theta}{2}\right)^3|v'|^{k+\gamma}f' \mathbf{1}_{\{0\leq \theta\leq \frac{\pi}{2}\}}\,d\sigma dv_*\\
        &\leq 2^{k+\gamma+1}\pi\int_{\mathbb{R}^3}|v_*|^{k+\gamma}f(v_*)\int_0^{\frac{\pi}{2}} b(\cos\theta)\sin\theta\,d\theta dv_*\\
        &\leq 2^{k+\gamma}C_b\|f\|_{1, k+\gamma}.
    \end{align*}
    Combining these two bounds to \eqref{eq:J_22}, we have
    \begin{align*}
        J_{22}(v)\leq C((1+|v|^2)^{\frac{k+\gamma}{2}}\|f\|_{1,0} + \|f\|_{1, k+\gamma})^{1-\frac{k}{3}}\|f\|_{1, k+\gamma}^{\frac{k}{3}}
    \end{align*}
    for some constant depending on $k$, $\gamma$, and $C_b$. Therefore, we get
    \begin{align}\label{eq:J_2_bound}
        J_2(v)\leq C\left(\|f\|_{1, k+\gamma} + ((1+|v|^2)^{\frac{k+\gamma}{2}}\|f\|_{1,0} + \|f\|_{1, k+\gamma})^{1-\frac{k}{3}}\|f\|_{1, k+\gamma}^{\frac{k}{3}}\right).
    \end{align}

    The $J_4$ can be bounded almost similarly to the $J_2$ with the same bounding quantity. Combining \eqref{eq:J_1,J_3_bound} and \eqref{eq:J_2_bound}, we get \eqref{eq:f'f'_*_k}.
\end{proof}

Using the above lemma, we construct an $L^1_2$ stability result in the Boltzmann-Fermi-Dirac equation. The next lemma is to compute the difference $(Q_{FD}(f,f) - Q_{FD}(g,g))\phi(v)\mathbf{1}_{\{f>g\}}$ for $\phi(v) = 1$ or $|v|^2$.
\begin{lemma}[Lemma 1 of \cite{Lu20031}]\label{lem:fg_diff_sym}
    Let $f(v)$ and $g(v)$ are real valued functions with $0\leq f,g\leq 1$. For $\phi(v) = 1$ or $|v|^2$, it satisfies
    \begin{align*}
        &ff_*(1-f')(1-f'_*) - gg_*(1-g')(1-g'_*)(\phi'\mathbf{1}_{\{f'>g'\}} + \phi'_*\mathbf{1}_{\{f'_*>g'_*\}} - \phi \mathbf{1}_{\{f>g\}} - \phi_*\mathbf{1}_{\{f_*>g_*\}})\\
        &\leq (f\phi)|f_*-g_*| + (f\phi)_*|f-g| + ff_*(|f'-g'|\phi'_* + |f'_*-g'_*|\phi').
    \end{align*}
\end{lemma}

Using this difference lemma, we can control the following weighted difference integral.
\begin{lemma}\label{lem:Q_FD_diff_int}
    Let $f(v)\in L^1_s$ for all $s\geq 2$ and $g(v)\in L^1_2$ with $0\leq f,g\leq 1$. Also, let the collision kernel $B$ satisfy $0\leq \gamma\leq 2$ and (H1). Then, for $k=0$ or $2$,
    \begin{align}\label{eq:Q_FD_diff_int}
        \begin{split}
            &\int_{\mathbb{R}^3} (1+|v|^2)^{\frac{k}{2}}(Q_{FD}((f, f) - Q_{FD}(g, g))(\tau, v))\mathbf{1}_{\{f(v)\geq g(v)\}}\,dv\\
            &\leq C\left(\|f\|_{1,k}\|f-g\|_{1,\gamma} + \|f\|_{1,k+\gamma}\|f-g\|_{1,0}  + \|f\|^{\frac{3-k}{3}}_{1,0}\|f\|^{\frac{k}{3}}_{1, k+\gamma}\|f-g\|_{1,\frac{(k+\gamma)(3-k)}{3}}\right)
        \end{split}
    \end{align}
    for some constant depending on $\gamma$, $k$, and $C_b$.
\end{lemma}
\begin{proof}
    Let $\phi(v) = 1$ or $1+|v|^2$. Temporarily, we assume $g\in L^1_s$ for all $s\geq 2$. We will relax this condition at the end of the proof. Since $f,g\in L^1_s$ for all $s\geq 2$, the integral in \eqref{eq:Q_FD_diff_int} is well-defined.
    
    From Lemma \ref{lem:fg_diff_sym}, applying the symmetrization, we have
    \begin{align}\label{eq:Q_FD_diff_int_2}
        \begin{split}
            &\int_{\mathbb{R}^3} \phi(v)(Q_{FD}((f, f) - Q_{FD}(g, g))(v))\mathbf{1}_{\{f(v)\geq g(v)\}}\,dv\\
            &=\frac{1}{2}\int_{\mathbb{R}^6\times \mathbb{S}^2} B(v-v_*,\sigma)\left(ff_*(1-f')(1-f'_*) - gg_*(1-g')(1-g'_*)\right)\\
            &\qquad \times (\phi'\mathbf{1}_{\{f'>g'\}} + \phi'_*\mathbf{1}_{\{f'_*>g'_*\}} - \phi\mathbf{1}_{\{f>g\}} - \phi_*\mathbf{1}_{\{f_*>g_*\}})\,d\sigma dv dv_*\\
            &\leq \frac{1}{2}\int_{\mathbb{R}^6\times \mathbb{S}^2} B(v-v_*,\sigma)\left((f\phi)|f_*-g_*| + (f\phi)_*|f-g| + ff_*(|f'-g'|\phi'_* + |f'_*-g'_*|\phi')\right)\,d\sigma dv dv_*.
        \end{split}
    \end{align}
    The first two terms are bounded by
    \begin{align*}
        &\frac{1}{2}\int_{\mathbb{R}^6\times \mathbb{S}^2} B(v-v_*,\sigma)\left(f\phi|f_*-g_*| + (f\phi)_*|f-g|\right)\,d\sigma dv dv_*\\
        &\leq \frac{C_b}{2}\int_{\mathbb{R}^6}((1+|v|^2)^{\frac{\gamma}{2}} + (1+|v_*|^2)^{\frac{\gamma}{2}}\left(f\phi|f_*-g_*| + (f\phi)_*|f-g|\right) dv dv_*\\
        &\leq \frac{C_b}{2}(\|f\|_{1,k + \gamma}\|f-g\|_{1,0} + \|f\|_{1,k}\|f-g\|_{1,\gamma}).
    \end{align*}
    Using Lemma \ref{lem:f'f'_*_k}, the last term is bounded by
    \begin{align*}
        &\frac{1}{2}\int_{\mathbb{R}^6\times \mathbb{S}^2} B(v-v_*,\sigma)ff_*(|f'-g'|\phi'_* + |f'_*-g'_*|\phi')\,d\sigma dv dv_*\\
        &=\int_{\mathbb{R}^3} |f-g|\int_{\mathbb{R}^3\times \mathbb{S}^2}B(v-v_*,\sigma)f'f'_*\phi_*\,d\sigma dv_* dv\\
        &\leq C\int_{\mathbb{R}^3}\left((1+|v|^2)^{\frac{\gamma}{2}} \|f\|_{1,k} + \|f\|_{1,k+\gamma} + \left((1+|v|^2)^{\frac{k+\gamma}{2}} \|f\|_{1,0} + \|f\|_{1, k+\gamma}\right)^{\frac{3-k}{3}}\|f\|_{1, k+\gamma}^{\frac{k}{3}}\right)|f-g|\,dv\\
        &\leq C\int_{\mathbb{R}^3}\left((1+|v|^2)^{\frac{\gamma}{2}}\|f\|_{1,k} + \|f\|_{1,k+\gamma} + \left((1+|v|^2)^{\frac{k+\gamma}{2}} \|f\|_{1,0} + \|f\|_{1, k+\gamma}\right)^{\frac{3-k}{3}}\|f\|_{1, k+\gamma}^{\frac{k}{3}}\right)|f-g|\,dv\\
        &\leq C\int_{\mathbb{R}^3}\left((1+|v|^2)^{\frac{\gamma}{2}}\|f\|_{1,k} + \|f\|_{1,k+\gamma} + (1+|v|^2)^{\frac{(k+\gamma)(3-k)}{6}} \|f\|^{\frac{3-k}{3}}_{1,0}\|f\|^{\frac{k}{3}}_{1, k+\gamma} + \|f\|_{1, k+\gamma}\right)|f-g|\,dv\\
        &\leq C\left(\|f\|_{1,k}\|f-g\|_{1,\gamma} + \|f\|_{1,k+\gamma}\|f-g\|_{1,0}  + \|f\|^{\frac{3-k}{3}}_{1,0}\|f\|^{\frac{k}{3}}_{1, k+\gamma}\|f-g\|_{1,\frac{(k+\gamma)(3-k)}{3}}\right).
    \end{align*}
    Combining the two bounds, we get the lemma when $g\in L^1_s$ for all $s\geq 2$.
    
    Now, we just assume $g\in L^1_2$. All we need to check is whether the symmetrization can be applied in \eqref{eq:Q_FD_diff_int_2}; it is enough to show that
    \begin{align}\label{eq:Q_FD_g_sym}
        \begin{split}
            &\int_{\mathbb{R}^3}\phi(v)Q_{FD}(g,g)(v)\mathbf{1}_{\{f(v)\geq g(v)\}}\,dv \\
            &= \frac{1}{2}\int_{\mathbb{R}^6\times\mathbb{S}^2}gg_*(1-g')(1-g'_*)\left(\phi'\mathbf{1}_{\{f'>g'\}} + \phi'_*\mathbf{1}_{\{f'_*>g'_*\}} - \phi\mathbf{1}_{\{f>g\}} - \phi_*\mathbf{1}_{\{f_*>g_*\}}\right)\,dvdv_*d\sigma.
        \end{split}
    \end{align}
    
    If we first consider the loss part of the $Q_{FD}(g,g)$, then we have
    \begin{align*}
        &\int_{\mathbb{R}^6\times\mathbb{S}^2} (1+|v|^2)B(v-v_*,\sigma)gg_*(1-g')(1-g'_*)\mathbf{1}_{\{f(v)\geq g(v)\}}\,dvdv_*d\sigma\\
        &\int_{\mathbb{R}^6\times\mathbb{S}^2} (1+|v|^2)B(v-v_*,\sigma)g(v)g(v_*)\mathbf{1}_{\{f(v)\geq g(v)\}}\,dvdv_*d\sigma\\
        &\leq C_b\int_{\mathbb{R}^6} ((1+|v|^2)^{\gamma/2} + (1+|v_*|^2)^{\gamma/2})(1+|v|^2)f(v)g(v_*)\,dvdv_*\\
        &\leq C_b(\|f\|_{1,2+\gamma}\|g\|_{1,0} + \|f\|_{1,2}\|g\|_{1,\gamma}).
    \end{align*}
    Therefore, it is integrable, and we can safely decompose the integrand by
    \begin{align*}
        \int_{\mathbb{R}^3}\phi(v)Q_{FD}(g,g)(v)\mathbf{1}_{\{f(v)\geq g(v)\}}\,dv &= \int_{\mathbb{R}^6\times\mathbb{S}^2}\phi(v)B(v-v_*,\sigma)g'g'_*(1-g)(1-g_*)\mathbf{1}_{\{f(v)\geq g(v)\}}\,dvdv_*d\sigma\\
        &\quad -\int_{\mathbb{R}^6\times\mathbb{S}^2}\phi(v)B(v-v_*,\sigma)gg_*(1-g')(1-g'_*)\mathbf{1}_{\{f(v)\geq g(v)\}}\,dvdv_*d\sigma.
    \end{align*}
    Now, we apply Tonelli's theorem and change of variable to each integral to get \eqref{eq:Q_FD_g_sym}; it is again safe since the integrand is composed of non-negative functions.
    
    Following the remaining steps, we finally get \eqref{eq:Q_FD_diff_int_2}.
\end{proof}

In the proof, we can not directly use the symmetrization to $(1+|v|^2)Q_{FD}(g,g)$ as it is not generally integrable when $g$ is just in $L^1_2$. To overcome this problem, we proved that the loss part of $Q_{FD}(g,g)$ is integrable and used the symmetrization for positive functions.

The next lemma states and proves the $L^1_2$ stability result for the solution of the Boltzmann-Fermi-Dirac equation.
\begin{proposition}\label{prop:stability}
    Let $f(t,v)$ and $g(t,v)$ be solutions of the Boltzmann-Fermi-Dirac equation. Also, assume $f$ satisfies \eqref{eq:poly_cre} for all $s>2$. Then, if $\gamma>0$,
    \begin{align*}
        \|f(t,v)-g(t,v)\|_{1,2}\leq C_1\Phi(\|f_0 - g_0\|_{1,2})\exp\left(C_2(t + t^{1/3})\right),
    \end{align*}
    where the function $\Phi$ is a function defined by
    \begin{align*}
        \Phi(r) \coloneqq r + \left(1+\left(\frac{3}{2}\|f_0\|^2_{1,2} + \|g_0\|^2_{1,2}\right)\right)r^{1/3} + \|f_0\mathbf{1}_{\{|v|\geq r^{-1/3}\}}\|_{1,2} + r|\ln r|,
    \end{align*}
    and $C_1$ and $C_2$ are constants depending on $\gamma$, $C_b$, $C_{b,2}$, $\varphi(\epsilon)$, $\|f_0\|_{1,0}$, and $\|f_0\|_{1,2}$. If $r=0$, we set $\Phi(0) = 0$.
    
    If $\gamma = 0$, then
    \begin{align*}
        \|f(t,v)-g(t,v)\|_{1,2}\leq C_3\|f_0 - g_0\|_{1,2}\exp\left(C_4t\right)
    \end{align*}
    for some constants $C_3$ and $C_4$ depending on $C_b$ and $\|f\|_{1,2}$.
\end{proposition}
\begin{proof}
    In the proof, $C$ will denote appropriate constants depending on each line. Also, we assume $\gamma>0$; the $\gamma=0$ case is the easy case, and we will briefly prove it at the end of the proof.
    
    Let $\|f_0-g_0\|_{1,2} = r$. To prove the lemma, it is enough to assume $r\leq 1$. For $t\geq 0$ and $0\leq k\leq 2$,
    \begin{align}\label{eq:f-g_1,k}
        \begin{split}
            \|f(t,v)-g(t,v)\|_{1,k}&= \|(f(t,v)-g(t,v))^+\|_{1,k} + \|(g(t,v)-f(t,v))^+\|_{1,k} \\
            &= \|g(t,v)\|_{1,k} - \|f(t,v)\|_{1,k} + 2\|(f(t,v)-g(t,v))^+\|_{1,k}\\
            &\leq r + 2\|(f(t,v)-g(t,v))^+\|_{1,k}.
        \end{split}
    \end{align}
    For any $R\geq 1$, we have
    \begin{align}\label{eq:7_6_0}
        \begin{split}
            2\|(f(t,v)-g(t,v))^+\|_{1,2}&\leq 4R^2\|(f(t,v)-g(t,v))^+\|_{1,0} + 2\int_{|v|\geq R}(1+|v|^2)(f(t,v)-g(t,v))\,dv\\
            &\leq 4R^2\|f(t,v)-g(t,v)\|_{1,0} + 2\int_{|v|\geq R}(1+|v|^2)f(t,v)\,dv.
        \end{split}
    \end{align}
    
    First, we choose $t\in [0,r]$. We can bound each term as follows:
    \begin{align}\label{eq:7_6_1}
        \begin{split}
            \|f(t,v)-g(t,v)\|_{1,0} &\leq \|f_0-g_0\|_{1,0} + \int_0^t \|Q_{FD}(f,f) - Q_{FD}(g,g)\|_{1,0}(\tau)\,d\tau\\
            &\leq \|f_0-g_0\|_{1,0} + 2^{\gamma+1}C_b\int_0^t (\|f\|^2_{1,2} + \|g\|^2_{1,2})(\tau)\,d\tau\\
            &\leq r + 2^{\gamma+1}C_b(\|f_0\|^2_{1,2} + \|g_0\|^2_{1,2})r,
        \end{split}
    \end{align}
    and
    \begin{align}\label{eq:7_6_2}
        \begin{split}
            \int_{|v|\geq R}(1+|v|^2)f(t,v)\,dv &=\int_{\mathbb{R}^3}(1+|v|^2)f(t,v)\,dv -\int_{|v|< R}(1+|v|^2)f(t,v)\,dv\\
            &=\|f_0\|_{1,2} - \int_{|v|< R}(1+|v|^2)\left(f_0 + \int_0^t Q_{FD}(f,f)(\tau)\,d\tau\right)\,dv\\
            &\leq\|f_0\mathbf{1}_{\{|v|\geq R\}}\|_{1,2} + \int_{|v|< R}\int_0^t  (1+|v|^2)Q^-_{FD}(f,f)(\tau,v)\,d\tau dv\\
            &\leq\|f_0\mathbf{1}_{\{|v|\geq R\}}\|_{1,2} + \int_{|v|< R}\int_0^t  (1+|v|^2)Q_c^-(f,f)(\tau,v)\,d\tau dv\\
            &\leq \|f_0\mathbf{1}_{\{|v|\geq R\}}\|_{1,2} + 2R^2\int_0^r \int_{\mathbb{R}^3} Q_c^-(f,f)(\tau,v)\,dv d\tau\\
            &\leq \|f_0\mathbf{1}_{\{|v|\geq R\}}\|_{1,2} + 2R^2 C_b \|f_0\|_{1,2}^2 r.
        \end{split}
    \end{align}
    
    Applying \eqref{eq:7_6_1} and \eqref{eq:7_6_2} to \eqref{eq:7_6_0} and \eqref{eq:f-g_1,k}, for $t\in [0,r]$, we get
    \begin{align*}
        \|f-g\|_{1,2}(t)&\leq r + 4R^2\left(1+C_b(\|f_0\|^2_{1,2} + \|g_0\|^2_{1,2})\right)r + \|f_0\mathbf{1}_{\{|v|\geq R\}}\|_{1,2} + 2R^2 C_b \|f_0\|_{1,2}^2 r\\
        &\leq r + 4R^2\left(1+C_b\left(\frac{3}{2}\|f_0\|^2_{1,2} + \|g_0\|^2_{1,2}\right)\right)r + \|f_0\mathbf{1}_{\{|v|\geq R\}}\|_{1,2}.
    \end{align*}
    If we choose $R = r^{-1/3}$, then
    \begin{align*}
        \|f-g\|_{1,2}(t)\leq r + 4\left(1+C_b\left(\frac{3}{2}\|f_0\|^2_{1,2} + \|g_0\|^2_{1,2}\right)\right)r^{1/3} + \|f_0\mathbf{1}_{\{|v|\geq r^{-1/3}\}}\|_{1,2}
    \end{align*}
    for $0\leq t\leq r$. We define the right-hand side $U(r)$ with $U(0) = 0$.
    
    Now, we move to $t\in [r, 1]$. By Lemma \ref{lem:fg+}, we start from
    \begin{align}\label{eq:temp1}
        (f(t,v) - g(t,v))^+ = (f(r,v) - g(r,v))^+ + \int_r^t (Q_{FD}(f, f) - Q_{FD}(g, g))(\tau, v)\mathbf{1}_{\{f(\tau, v)\geq g(\tau, v)\}}\,d\tau.
    \end{align}
    
    For $\|(f-g)\|_{1,0}(t)$, using \eqref{eq:f-g_1,k}, Lemma \ref{lem:fg+}, and Lemma \ref{lem:Q_FD_diff_int}, we have
    \begin{align*}
        \|f(t,v)-g(t,v)\|_{1,0} &= \|g(t,v)\|_{1,0} - \|f(t,v)\|_{1,0} + 2\|(f(t,v)-g(t,v))^+\|_{1,0}\\
        &\leq r + 2\left(\|(f_0-g_0)^+\|_{1,0} + \int_0^t \int_{\mathbb{R}^3}(Q_{FD}(f, f) - Q_{FD}(g, g))(\tau, v)\mathbf{1}_{\{f(\tau, v)\geq g(\tau, v)\}}\,dv d\tau\right)\\
        &\leq 3r + 2C\int_0^t\left(\|f\|_{1,0}\|f-g\|_{1,\gamma} + \|f\|_{1,\gamma}\|f-g\|_{1,0}  + \|f\|_{1,0}\|f-g\|_{1,\gamma}\right)(\tau)\,d\tau.
    \end{align*}
    Since $\gamma\leq 2$, we can absorb $\|f\|_{1,0}$ and $\|f\|_{1,\gamma}$ into the constant $C$ and bound $\|f-g\|_{1,\gamma}\leq \|f-g\|_{1,2}$. Therefore, we have
    \begin{align}\label{eq:7_6_4}
        \|f-g\|_{1,0}(t)\leq 3r + C\int_0^t \|f-g\|_{1,2}(\tau)\,d\tau.
    \end{align}
    
    $\|(f-g)\|_{1,2}(t)$ can be computed similarly:
    \begin{align}\label{eq:7_6_5}
        \begin{split}
            &\|f(t,v)-g(t,v)\|_{1,2}\\
            &=\|g(t,v)\|_{1,2} - \|f(t,v)\|_{1,2} + 2\|(f(t,v)-g(t,v))^+\|_{1,2}\\
            &\leq r + 2\left(\|(f(r,v) - g(r,v))^+\|_{1,2} + \int_r^t \int_{\mathbb{R}^3} (1+|v|^2)(Q_{FD}((f, f) - Q_{FD}(g, g))(\tau, v))\mathbf{1}_{\{f(\tau, v)\geq g(\tau, v)\}}\,dv d\tau\right)\\
            &\leq r + 2\left(U(r) + \int_r^t \int_{\mathbb{R}^3} (1+|v|^2)(Q_{FD}((f, f) - Q_{FD}(g, g))(\tau, v))\mathbf{1}_{\{f(\tau, v)\geq g(\tau, v)\}}\,dv d\tau\right).
        \end{split}
    \end{align}
    
    We use Lemma \ref{lem:Q_FD_diff_int} to bound the final integral by
    \begin{align}\label{eq:7_6_3}
        \begin{split}
            &\int_{\mathbb{R}^3} (1+|v|^2)(Q_{FD}((f, f) - Q_{FD}(g, g))(\tau, v))\mathbf{1}_{\{f(\tau, v)\geq g(\tau, v)\}}\,dv\\
            &\leq C\left(\|f\|_{1,2}\|f-g\|_{1,\gamma} + \|f\|_{1,2+\gamma}\|f-g\|_{1,0}  + \|f\|^{\frac{1}{3}}_{1,0}\|f\|^{\frac{2}{3}}_{1, 2+\gamma}\|f-g\|_{1,\frac{2+\gamma}{3}}\right)(\tau).
        \end{split}
    \end{align}
    Combining \eqref{eq:7_6_5} and \eqref{eq:7_6_3}, we write
    \begin{align}\label{eq:f-g_1,2_mid}
        \begin{split}
            &\|f-g\|_{1,2}(t)\\
            &\leq r +2U(r) + 2\int_r^t C\left(\|f\|_{1,2}\|f-g\|_{1,\gamma} + \|f\|_{1,2+\gamma}\|f-g\|_{1,0}  + \|f\|^{\frac{1}{3}}_{1,0}\|f\|^{\frac{2}{3}}_{1, 2+\gamma}\|f-g\|_{1,\frac{2+\gamma}{3}}\right)(\tau) d\tau.
        \end{split}
    \end{align}
    Since $\frac{2+\gamma}{3}\leq 2$ for any $0\leq \gamma\leq 2$, we get
    \begin{align*}
        \|f\|_{1,2}\|f-g\|_{1,\gamma} + \|f\|^{\frac{1}{3}}_{1,0}\|f\|^{\frac{2}{3}}_{1, 2+\gamma}\|f-g\|_{1,\frac{2+\gamma}{3}}\leq (\|f\|_{1,2} + \|f\|^{\frac{1}{3}}_{1,0}\|f\|^{\frac{2}{3}}_{1, 2+\gamma})\|f-g\|_{1,2}
    \end{align*}
    in \eqref{eq:f-g_1,2_mid}.
    
    By absorbing $\|f\|_{1,0}$ and $\|f\|_{1,2}$ into constant $C$ and using \eqref{eq:7_6_4}, we obtain
    \begin{align*}
        &\|f-g\|_{1,2}(t)\\
        &\leq r + 2U(r) + C\int_r^t \left(\left(1+\|f\|^{\frac{2}{3}}_{1, 2+\gamma}(\tau)\right)\|f-g\|_{1,2}(\tau) + \|f\|_{1,2+\gamma}(\tau)\left(r + \int_0^\tau \|f-g\|_{1,2}(s)\,ds\right)\right)\,d\tau.
    \end{align*}
    
    Now, we use the \textit{a priori} estimate \eqref{eq:poly_cre} to get
    \begin{align*}
        &\|f-g\|_{1,2}\\
        &\leq r + 2U(r) + C\int_r^t \left(\left(1+\frac{C_{2+\gamma, 1}}{\tau^{2/3}}\right)\|f-g\|_{1,2}(\tau) + \frac{C_{2+\gamma, 1}}{\tau}\left(r + \int_0^\tau \|f-g\|_{1,2}(s)\,ds\right)\right)\,d\tau\\
        &\leq r + 2U(r) + Cr|\ln r| + C\int_r^t \left(\left(1+\frac{1}{\tau^{2/3}}\right)\|f-g\|_{1,2}(\tau) + \frac{1}{\tau}\int_0^\tau \|f-g\|_{1,2}(s)\,ds\right)\,d\tau
    \end{align*}
    for some constants $C$. For the inner integral, we use Fubini's theorem to obtain
    \begin{align*}
        \int_r^t \frac{1}{\tau}\int_0^\tau \|f-g\|_{1,2}(s)\,ds \,d\tau &=\int_0^r \|f-g\|_{1,2}(s)\int_r^t \frac{1}{\tau}\,d\tau \,ds + \int_r^t \|f-g\|_{1,2}(s)\int_s^t \frac{1}{\tau}\,d\tau \,ds\\
        &\leq \int_0^t \|f-g\|_{1,2}(s)\int_s^t \frac{1}{\tau}\,d\tau \,ds\\
        &\leq\int_0^t |\ln s|\|f-g\|_{1,2}(s)\,ds.
    \end{align*}
    Finally, we get
    \begin{align*}
        \|f-g\|_{1,2}(t)\leq r + 2U(r) + Cr|\ln r| + C\int_0^t \left(1+\frac{1}{\tau^{2/3}} + |\ln \tau|\right)\|f-g\|_{1,2}(\tau)\,d\tau.
    \end{align*}
    At the end, we use the Gr{\"o}nwall inequality and obtain
    \begin{align*}
        \|f-g\|_{1,2}(t)&\leq \left(r + 2U(r) + Cr|\ln r|\right)\exp\left(C\int_0^t 1+\frac{1}{\tau^{2/3}} + |\ln \tau|\,d\tau\right)\\
        &\leq \left(r + 2U(r) + Cr|\ln r|\right)\exp\left(C(t + t^{1/3})\right)
    \end{align*}
    for $t\in [r,1]$ for some constants $C$.
    
    If $t>1$, then we choose the start time at the time $1$ in \eqref{eq:7_6_5} and write
    \begin{align*}
        &\|f(t,v)-g(t,v)\|_{1,2}\\
        &=\|g(t,v)\|_{1,2} - \|f(t,v)\|_{1,2} + 2\|(f(t,v)-g(t,v))^+\|_{1,2}\\
        &\leq r + 2\left(\|(f(1) - g(1))^+\|_{1,2} + \int_1^t \int_{\mathbb{R}^3} (1+|v|^2)(Q_{FD}((f, f) - Q_{FD}(g, g))(\tau, v))\mathbf{1}_{\{f(\tau, v)\geq g(\tau, v)\}}\,dv d\tau\right)\\
        &\leq r + 2C\left(r + 2r|\ln r| + U(r)\right) + 2\int_1^t \int_{\mathbb{R}^3} (1+|v|^2)(Q_{FD}((f, f) - Q_{FD}(g, g))(\tau, v))\mathbf{1}_{\{f(\tau, v)\geq g(\tau, v)\}}\,dv d\tau.
    \end{align*}
    We can bound the final integral exactly with the same method; the only difference is that we change the \textit{a priori} estimate to $\|f\|_{1,s}\leq C_{s,1}$ for $s>2$ in \eqref{eq:poly_cre} since $t\geq 1$. Absorbing $\|f\|_{1,0}$, $\|f\|_{1,2}$ and $C_{2+\gamma, 1}$ to constant $C$, we have
    \begin{align*}
        &\|f-g\|_{1,2}(t)\\
        &\leq 3C\left(r + 2r|\ln r| + U(r)\right) \\
        &\quad +C\int_1^t \left(\|f\|_{1,2}\|f-g\|_{1,\gamma} + \|f\|_{1,2+\gamma}\|f-g\|_{1,0}  + \|f\|^{\frac{1}{3}}_{1,0}\|f\|^{\frac{2}{3}}_{1, 2+\gamma}\|f-g\|_{1,\frac{2+\gamma}{3}}\right)(\tau)\, d\tau\\
        &\leq 3C\left(r + 2r|\ln r| + U(r)\right) + C\int_1^t \left(\|f\|_{1,2}\|f-g\|_{1,2} + C_{2+\gamma, 1}\|f-g\|_{1,2} + \|f\|_{1,0}^{\frac{1}{3}}C^{\frac{2}{3}}_{2+\gamma, 1}\|f-g\|_{1,2}\right)(\tau)\,d\tau\\
        &= 3C\left(r + 2r|\ln r| + U(r)\right) + C\int_1^t \|f-g\|_{1,2}(\tau)\, d\tau,
    \end{align*}
    so
    \begin{align*}
        \|f-g\|_{1,2}(t)\leq 3C\left(r + 2r|\ln r| + U(r)\right)e^{Ct}
    \end{align*}
    for $t\geq 1$ by Gr{\"o}nwall's inequality. It ends the proof for $0<\gamma\leq 2$.
    
    For $\gamma = 0$ case, we use \eqref{eq:f-g_1,k}, \eqref{eq:temp1}, and \eqref{eq:7_6_3} in sequence: for $t\geq 0$,
    \begin{align*}
        &\|f(t,v)-g(t,v)\|_{1,2}\\
        &\leq r + 2\|(f(t,v) - g(t,v))^+\|_{1,2}\\
        &\leq r + 2\left(\|(f(0) - g(0))^+\|_{1,2} + \int_r^t (1+|v|^2)(Q_{FD}(f, f) - Q_{FD}(g, g))(\tau, v)\mathbf{1}_{\{f(\tau, v)\geq g(\tau, v)\}}\,d\tau\right)\\
        &\leq 3r + 2C\int_0^t \left(\|f\|_{1,2}\|f-g\|_{1,0} + (\|f\|_{1,2}\|f-g\|_{1,0}  + \|f\|^{\frac{1}{3}}_{1,0}\|f\|^{\frac{2}{3}}_{1, 2}\|f-g\|_{1,\frac{2}{3}}\right)(\tau)\,d\tau\\
        &\leq 3r + 2C\int_0^t \left(2\|f\|_{1,2} + \|f\|^{\frac{1}{3}}_{1,0}\|f\|^{\frac{2}{3}}_{1, 2}\right)\|f-g\|_{1,2}(\tau)\,d\tau
    \end{align*}
    for some constant $C$ depending on $\gamma$ and $C_b$. Directly applying Gr{\"o}nwall's inequality, we get the lemma for $\gamma = 0$.
\end{proof}

The next lemma is a technical lemma to verify the condition of Lemma \ref{lem:d/dt_m_sp}. It will be used in the proof of the existence and uniqueness of the solution in the Boltzmann-Fermi-Dirac equation.
\begin{lemma}\label{lem:m_s_diff}
    Assume the collision kernel satisfies (H1) and $0\leq \gamma\leq 2$. Let $f$ be a solution of the Boltzmann-Fermi-Dirac equation. If $f\in L^\infty_{loc}([0,\infty), L^1_s)$ for all $s\geq 2$, then in fact $f\in C([0,\infty), L^1_s)$ and $m_s(t)\in C^1([0,\infty))$ for all $s\geq 2$.
\end{lemma}
\begin{proof}
    We use a bootstrap argument. We first show $f\in C([0,\infty), L^1_s)$. For $0\leq t_2\leq t_1$ and a fixed $s\geq 2$, since $|v'|^2\leq |v|^2+|v_*|^2$,
    \begin{align*}
        &\int_{\mathbb{R}^3} |f(t_1,v) - f(t_2,v)|(1+|v|^2)^{s/2}\,dv\\
        &\leq \int_{t_2}^{t_1} \int_{\mathbb{R}^3}\left|Q_{FD}(f,f)(\tau, v)\right|(1+|v|^2)^{s/2}\,dv\,d\tau\\
        &\leq \int_{t_2}^{t_1} \int_{\mathbb{R}^6\times\mathbb{S}^2}|v-v_*|^\gamma b(\cos\theta)f(\tau,v)f(\tau,v_*)((1+|v'|^2)^{s/2} + (1+|v|^2)^{s/2})\,dvdv_*d\sigma d\tau\\
        &\leq C_b \int_{t_2}^{t_1} \int_{\mathbb{R}^6}f(\tau,v)f(\tau,v_*)|v-v_*|^\gamma((1+|v|^2 + |v_*|^2)^{s/2} + (1+|v|^2)^{s/2})\,dvdv_* d\tau\\
        &\leq 2C_b \int_{t_2}^{t_1} \int_{\mathbb{R}^6}f(\tau,v)f(\tau,v_*)(|v|^\gamma + |v_*|^\gamma)\left(2^{s/2}\left((1+|v|^2)^{s/2} + (1+|v_*|^2)^{s/2}\right) + (1+|v|^2)^{s/2}\right)\,dvdv_* d\tau\\
        &\leq 2C_b (2^{s/2} + 1 + s^{s/2})(t_1-t_2)\esssup_{\tau\in [t_2, t_1]}\left(\|f\|_{1,s+\gamma}\|f\|_{1,0} + \|f\|_{1,s}\|f\|_{1,\gamma}\right)(\tau).
    \end{align*}
    Since $f\in L^\infty_{loc}([0,\infty). L^1_s)$ for all $s\geq 2$, it shows that $f\in C([0,\infty), L^1_s)$. This argument can be applied for all $s\geq 2$, so we get $f\in C([0,\infty), L^1_s)$ for all $s\geq 2$.

    Secondly, we prove $Q_{FD}(f,f)\in C([0,\infty), L^1_s)$. Indeed, we first decompose the difference by
    \begin{align*}
        &\int_{\mathbb{R}^3}\left|Q_{FD}(f,f)(t_1, v) - Q_{FD}(f,f)(t_2, v)\right|(1+|v|^2)^{s/2}\,dv\\
        &\leq \int_{\mathbb{R}^6\times\mathbb{S}^2} B(v-v_*,\sigma)|f(t_1, v)-f(t_2, v)|f(t_1,v_*)|((1+|v'|^2)^{s/2} + (1+|v|^2)^{s/2})\,dvdv_*d\sigma\\
        &\quad + \int_{\mathbb{R}^6\times\mathbb{S}^2} B(v-v_*,\sigma)f(t_2, v)|f(t_1,v_*) - f(t_2,v_*)|((1+|v'|^2)^{s/2} + (1+|v|^2)^{s/2})\,dvdv_*d\sigma\\
        &\quad + \int_{\mathbb{R}^6\times\mathbb{S}^2} B(v-v_*,\sigma)f(t_2, v)f(t_2,v_*)|f(t_1,v') - f(t_2,v')|((1+|v'|^2)^{s/2} + (1+|v|^2)^{s/2})\,dvdv_*d\sigma\\
        &\quad + \int_{\mathbb{R}^6\times\mathbb{S}^2} B(v-v_*,\sigma)f(t_2, v)f(t_2,v_*)|f(t_1,v'_*) - f(t_2,v'_*)|((1+|v'|^2)^{s/2} + (1+|v|^2)^{s/2})\,dvdv_*d\sigma.
    \end{align*}
    The first integral is bounded by
    \begin{align*}
        &\int_{\mathbb{R}^6\times\mathbb{S}^2} B(v-v_*,\sigma)|f(t_1, v)-f(t_2, v)|f(t_1,v_*)((1+|v'|^2)^{s/2} + (1+|v|^2)^{s/2})\,dvdv_*d\sigma\\
        &\leq \int_{\mathbb{R}^6\times\mathbb{S}^2} B(v-v_*,\sigma)|f(t_1, v)-f(t_2, v)|f(t_1,v_*)\\
        &\qquad \times\left(2^{s/2}\left((1+|v|^2)^{s/2} + (1+|v_*|^2)^{s/2}\right) + (1+|v|^2)^{s/2}\right)\,dvdv_*d\sigma\\
        &\leq 2C_b\int_{\mathbb{R}^6} |f(t_1, v)-f(t_2, v)|f(t_1,v_*)(|v|^\gamma + |v_*|^\gamma)\\
        &\qquad \times \left(2^{s/2}\left((1+|v|^2)^{s/2} + (1+|v_*|^2)^{s/2}\right) + (1+|v|^2)^{s/2}\right)\,dvdv_*\\
        &\leq 2C_b\left((2^{s/2}+1)\|f(t_1, v)-f(t_2, v)\|_{1,s+\gamma}\|f(t_1,v)\|_{1,0} + 2^{s/2}\|f(t_1, v)-f(t_2, v)\|_{1,s}\|f(t_1,v)\|_{1,\gamma}\right)\\
        &\quad + 2C_b\left((2^{s/2}+1)\|f(t_1, v)-f(t_2, v)\|_{1,\gamma}\|f(t_1,v)\|_{1,s} + 2^{s/2}\|f(t_1, v)-f(t_2, v)\|_{1,0}\|f(t_1,v)\|_{1,s+\gamma}\right).
    \end{align*}
    The second integral can be bounded similarly. For the third integral, using Lemma \ref{lem:temp}, we get
    \begin{align*}
        &\int_{\mathbb{R}^6\times\mathbb{S}^2} B(v-v_*,\sigma)f(t_2, v)f(t_2,v_*)|f(t_1,v'_*) - f(t_2,v'_*)|((1+|v'|^2)^{s/2} + (1+|v|^2)^{s/2})\,dvdv_*d\sigma\\
        &=\int_{\mathbb{R}^6\times\mathbb{S}^2} B(v-v_*,\sigma)f(t_2, v')f(t_2,v'_*)|f(t_1,v) - f(t_2,v)|((1+|v_*|^2)^{s/2} + (1+|v_*'|^2)^{s/2})\,dvdv_*d\sigma\\
        &\leq \int_{\mathbb{R}^6\times\mathbb{S}^2} B(v-v_*,\sigma)f(t_2, v')f(t_2,v'_*)|f(t_1,v) - f(t_2,v)|\\
        &\qquad \times((1+|v_*|^2)^{s/2} + 2^{s/2}((1+|v|^2)^{s/2} + (1+|v_*|^2)^{s/2}))\,dvdv_*d\sigma\\
        &=(2^{s/2}+1)\int_{\mathbb{R}^3}|f(t_1,v) - f(t_2,v)|\int_{\mathbb{R}^3\times\mathbb{S}^2} B(v-v_*,\sigma)f(t_2, v')f(t_2,v'_*)(1+|v_*|^2)^{s/2}\,dv_*d\sigma dv\\
        &\quad + 2^{s/2}\int_{\mathbb{R}^3}|f(t_1,v) - f(t_2,v)|(1+|v|^2)^{s/2}\int_{\mathbb{R}^3\times\mathbb{S}^2} B(v-v_*,\sigma)f(t_2, v')f(t_2,v'_*)\,dv_*d\sigma dv\\
        &\leq C\left(\|f(t_1,v) - f(t_2,v)\|_{1,\gamma}\|f(t_2, v)\|_{1,s+\gamma} + \|f(t_1,v) - f(t_2,v)\|_{1,s+\gamma}\|f(t_2, v)\|_{1,0}\right.\\
        &\quad + \left.\|f(t_1,v) - f(t_2,v)\|_{1,s+\gamma}(\|f(t_2, v)\|_{1,\gamma} + \|f(t_2, v)\|_{1,0})\right)
    \end{align*}
    for some constant $C$. The fourth integral can be controlled in a similar manner.
    
    Combining these four estimates, we have
    \begin{align*}
        &\int_{\mathbb{R}^3}\left|Q_{FD}(f,f)(t_1, v) - Q_{FD}(f,f)(t_2, v)\right|(1+|v|^2)^{s/2}\,dv\\
        &\leq C\|f(t_1,v) - f(t_2,v)\|_{1,s+\gamma}(\|f(t_2, v)\|_{1,\gamma} + \|f(t_2, v)\|_{1,0})\\
        &\quad + C\|f(t_1,v) - f(t_2,v)\|_{1,s}\|f(t_2, v)\|_{1,\gamma} \\
        &\quad + C\|f(t_1,v) - f(t_2,v)\|_{1,\gamma}(\|f(t_2, v)\|_{1,s}+\|f(t_2, v)\|_{1,s+\gamma})\\
        &\quad + C\|f(t_1,v) - f(t_2,v)\|_{1,0}\|f(t_2, v)\|_{1,s+\gamma}
    \end{align*}
    for some constant $C$. Since $f\in C([0,\infty), L^1_s)$ for all $s\geq 2$, it proves $Q_{FD}(f,f)\in C([0,\infty), L^1_s)$ for all $s\geq 2$.

    Finally, we recall
    \begin{align*}
        m_s(t)&\coloneqq \int_{\mathbb{R}^3} f(t,v)|v|^s\,dv\\
        &=\int_{\mathbb{R}^3} f_0(v)|v|^s\,dv + \int_0^t \int_{\mathbb{R}^3}Q_{FD}(f,f)(\tau, v)|v|^s\,dv d\tau.
    \end{align*}
    Since $Q_{FD}(f,f)\in C([0,\infty), L^1_s)$ and
    \begin{align*}
        &\left|\int_{\mathbb{R}^3}Q_{FD}(f,f)(t_1, v)|v|^s\,dv - \int_{\mathbb{R}^3}Q_{FD}(f,f)(t_2, v)|v|^s\,dv\right|\\
        &\leq \int_{\mathbb{R}^3}\left|Q_{FD}(f,f)(t_1, v) - Q_{FD}(f,f)(t_2, v)\right||v|^s\,dv\\
        &\leq \int_{\mathbb{R}^3}\left|Q_{FD}(f,f)(t_1, v) - Q_{FD}(f,f)(t_2, v)\right|(1+|v|^2)^{s/2}\,dv,
    \end{align*}
    we finally obtain $m_s(t)\in C^1([0,\infty))$ by the Fundamental theorem of calculus.
\end{proof}

From now on, we establish the existence and uniqueness of the solution of the Boltzmann-Fermi-Dirac equation. First, we construct a unique solution under the assumption $f_0\in L^1_s$ for all $s\geq 2$. After proving it, we will mitigate this condition to $f_0\in L^1_2$ in Theorem \ref{thm:exi_uni_thm_2}.
\begin{proposition}\label{prop:exi_uni_prop}
    Assume the collision kernel satisfies (H1) and $0\leq \gamma\leq 2$. If $f_0\in L^1_s$ for all $s\geq 2$ and $0\leq f_0\leq 1$, then there exists a unique solution of the Boltzmann-Fermi-Dirac equation. Furthermore, if $0<\gamma\leq 2$, then it satisfies \eqref{eq:poly_cre} and \eqref{eq:poly_pro}.
\end{proposition}

\begin{proof}
    If $\gamma = 0$, the existence and uniqueness are proved in \cite{Lu2001353}. Therefore, we consider the $\gamma>0$ case.
    
    Let $B_n(v-v_*,\sigma) = (|v-v_*|^\gamma \wedge n) b(\cos\theta)$ and
    \begin{align*}
        Q_{FD, n}(f,f) = \int_{\mathbb{R}^3\times \mathbb{S}^2} B_n(v-v_*,\sigma) (f'f'_*(1-f)(1-f_*) - ff_*(1-f')(1-f'_*))\,dv_*d\sigma.
    \end{align*}
    For $s>2$, let $\phi(v) = (1+|v|^2)^{s/2}$ and $\phi_m(v) = \phi(v)\wedge m$. For $f_0\in L^1_s$, by some contraction mapping argument, we can prove that there exists a unique solution of the Boltzmann-Fermi-Dirac equation in $L^\infty([0,\infty), L^1_2(\mathbb{R}^3))$ satisfying
    \begin{align*}
        f_n(t,v) = f_0(v) + \int_0^t Q_{FD, n}(f_n, f_n)(\tau, v)\,d\tau
    \end{align*}
    for all $t$ and a.e. $v$; one can refer to the first paragraph of Section 3 in \cite{Lu2001353}. As $\|B_n\|_{L^1(\mathbb{S}^2)}\leq C_b n^\gamma$, we have
    \begin{align*}
        \|f_n \phi_m\|_{1,0}(t)&\leq \|f_0 \phi_m\|_{1,0} +  \int_0^t\int_{\mathbb{R}^6\times \mathbb{S}^2} B_n(v-v_*,\sigma) f_nf_{n,*}(1-f_n')(1-f'_{n,*})(\phi_m(v') + \phi_m(v))\,dv_* dv d\sigma d\tau\\
        &\leq \|f_0 \phi_m\|_{1,0} + C_b n^\gamma\int_0^t \left(\int_{\mathbb{R}^6} f_nf_{n,*}2^{s/2}(\phi_m(v) + \phi_m(v_*))\,dv dv_* + \|f_n \phi_m\|_{1,0}(\tau)\|f_n\|_{1,0}(\tau)\right)\,d\tau\\
        &\leq \|f_0 \phi_m\|_{1,0} + (2^{s/2+1}+1) C_b n^\gamma\|f_0\|_{1,0}\int_0^t \|f_n \phi_m\|_{1,0}(\tau)\,d\tau.
    \end{align*}
    By Gronwall's inequality and letting $m\rightarrow \infty$ with Fatou's lemma, we get
    \begin{align*}
        \|f_n\|_{1,s}(t)\leq \|f_0\|_{1,s}e^{(2^{s/2+1}+1) C_b n^\gamma \|f_0\|_{1,0} t}.
    \end{align*}
    By Lemma \ref{lem:m_s_diff}, we have $f_n\in C([0,\infty), L^1_s)$ and $m_{n,s}(t)\in C^1([0,\infty))$ for all $s\geq 2$, which is defined by
    \begin{align*}
        m_{n,s}(t)\coloneqq \int_{\mathbb{R}^3} f_n(t,v)|v|^s\,dv.
    \end{align*}
    
    Next, we choose $s = \gamma p$ for any integer $p> 2/\gamma$. By Lemma \ref{lem:bobylev_2} and an elementary inequalities
    \begin{align*}
        &(x^2+y^2)^{\frac{\gamma}{2}}\leq (x^\gamma + y^\gamma),\\
        &(x+y)^p - x^p - y^p\leq \sum_{k=1}^{p-1} \binom{p}{k} x^ky^{p-k},
    \end{align*}
    for $0<\gamma\leq 2$ and $x,y\geq 0$, we obtain
    \begin{align*}
        \dv{t} m_{n,\gamma p}&= \int_{\mathbb{R}^6\times\mathbb{S}^2} f_nf_{n,*}(1-f_n')(1-f'_{n,*})\left((|v'|^{\gamma p}+|v'_*|^{\gamma p}) - |v|^{\gamma p} - |v_*|^{\gamma p}\right) \left(|v-v_*|^\gamma \wedge n\right)b(\cos\theta)\,dvdv_* d\sigma\\
        &\leq C_b \varpi_{\gamma p/2}\int_{\mathbb{R}^6} f_nf_{n,*}(1-f_n')(1-f'_{n,*})\left((|v|^2+|v_*|^2)^{\frac{\gamma p}{2}} - |v|^{\gamma p} - |v_*|^{\gamma p}\right) \left(|v-v_*|^\gamma \wedge n\right)\,dvdv_*\\
        &\leq C_b \int_{\mathbb{R}^6} f_nf_{n,*}(1-f_n')(1-f'_{n,*})\left(\sum_{k=1}^{p-1}\binom{p}{k}|v|^{\gamma k}|v_*|^{\gamma (p-k)}\right) \left(|v-v_*|^\gamma \wedge n\right)\,dvdv_*\\
        &\leq C_b \int_{\mathbb{R}^6} f_nf_{n,*}\left(\sum_{k=1}^{p-1}\binom{p}{k}|v|^{\gamma k}|v_*|^{\gamma (p-k)}\right) |v-v_*|^\gamma\,dvdv_*\\
        &\leq C m_{n,\gamma p}m_{n,\gamma}
    \end{align*}
    for some constant $C$ depending on $\gamma$, $p$, and $C_b$. Since $\|f_n\|_{1,2}$ is conservative, we get
    \begin{align*}
        m_{n,\gamma p}(t)\leq m_{n,\gamma p}(0) e^{Ct},\quad \|f_n\|_{1,\gamma p}(t)\leq \|f_0\|_{1,\gamma p} e^{C t}
    \end{align*}
    for any integer $p>2/\gamma$. Taking interpolation between $m_2$ or $\|f_n\|_{1,2}$ with the above inequalities, we also get $m_{n,s}(t)\leq C_s e^{C_s t}$ and $\|f_n\|_{1,s}(t)\leq C_s e^{C_s t}$ for some $C_s$ and all $s\geq 2$ not depending on $n$.
    
    Now, we will show that $f_n$ is a Cauchy sequence in $C([0, T], L^1_2(\mathbb{R}^3))$ for arbitrary $T<\infty$. Indeed, we consider $f_n-f_m$ for $m\geq n$. Then,
    \begin{align*}
        \|f_n-f_m\|_{1,2}(t) = 2\|(f_n-f_m)^+\|_{1,2}(t),
    \end{align*}
    and
    \begin{align*}
        \|(f_n-f_m)^+\|_{1,2}(t)&\leq \int_0^t \|\left(Q_{FD, n}(f_n,f_n) - Q_{FD, m}(f_n,f_n)\right)\mathbf{1}_{\{f_n\geq f_m\}}\|_{1,2}(\tau)\,d\tau\\
        &\quad + \int_0^t \|\left(Q_{FD, m}(f_n,f_n) - Q_{FD, m}(f_m,f_m)\right)\mathbf{1}_{\{f_n\geq f_m\}}\|_{1,2}(\tau)\,d\tau.
    \end{align*}
    For the second term, by Lemma \ref{lem:Q_FD_diff_int}, we have
    \begin{align*}
        \|\left(Q_{FD, m}(f_n,f_n) - Q_{FD, m}(f_m,f_m)\right)\mathbf{1}_{\{f_n\geq f_m\}}\|_{1,2}(t)&\leq C\|f_n\|_{1,4}(t)\|f_n-f_m\|_{1,2}(t)
    \end{align*}
    for some constant $C$. By the Gronwall inequality, we have
    \begin{align*}
        \sup_{t\in [0, T]}\|f_n-f_m\|_{1,2}(t)&\leq \left(2\int_0^T \|\left(Q_{FD, m}(f_n,f_n) - Q_{FD, m}(f_m,f_m)\right)\mathbf{1}_{\{f_n\geq f_m\}}\|_{1,2}(\tau)\,d\tau\right)\\
        &\qquad \times \exp\left(C\left(\sup_{\tau\in [0,T]}\|f_n(\tau)\|_{1,4}\right)T\right).
    \end{align*}
    
    Now, as $\sup_n \sup_{t\in [0,T]}\|f_n\|_{1,5}(t)<\infty$, we have
    \begin{align*}
        &\lim_{\substack{n\rightarrow \infty \\ m\geq n}}\int_0^T \|\left(Q_{FD, n}(f_n,f_n) - Q_{FD, m}(f_n,f_n)\right)\mathbf{1}_{\{f_n\geq f_m\}}\|_{1,2}(\tau)\,d\tau\\
        &\leq\lim_{\substack{n\rightarrow \infty \\ m\geq n}}\int_0^T \int_{\mathbb{R}^6\times\mathbb{S}^2}(B_m(v-v_*,\sigma) - B_n(v-v_*,\sigma))f_nf_{n,*}((1+|v'|^2) + (1+|v|^2))\,dvdv_* d\sigma d\tau\\
        &\leq\lim_{n\rightarrow \infty}\int_0^T \int_{\mathbb{R}^6\times\mathbb{S}^2}|v-v_*|^\gamma b(\cos\theta)f_nf_{n,*}((1+|v|^2 + |v_*|^2) + (1+|v|^2))\mathbf{1}_{\{|v-v_*|> n\}}\,dvdv_* d\sigma d\tau\\
        &\leq 2C_b\lim_{n\rightarrow \infty}\int_0^T \int_{\mathbb{R}^6}f_nf_{n,*}(|v|^\gamma + |v_*|^\gamma)((1+|v|^2 + |v_*|^2) + (1+|v|^2))\mathbf{1}_{\{|v|> n/2\}}\,dvdv_* d\tau\\
        &\quad +2C_b\lim_{n\rightarrow \infty}\int_0^T \int_{\mathbb{R}^6}f_nf_{n,*}(|v|^\gamma + |v_*|^\gamma)((1+|v|^2 + |v_*|^2) + (1+|v|^2))\mathbf{1}_{\{|v_*|> n/2\}}\,dvdv_* d\tau\\
        &\leq C\lim_{n\rightarrow \infty}\int_0^T \|f_n\mathbf{1}_{\{|v|\geq n/2\}}\|_{1,4}(\tau)\|f_n\|_{1,4}(\tau)\,d\tau\\
        &\leq C\lim_{n\rightarrow \infty}\frac{2}{n}\int_0^T \|f_n\|_{1,5}(\tau)\|f_n\|_{1,4}(\tau)\,d\tau = 0
    \end{align*}
    for any fixed $T<\infty$. Therefore, we get
    \begin{align*}
        \lim_{\substack{n\rightarrow \infty \\ m\geq n}}\sup_{t\in [0, T]}\|f_n-f_m\|_{1,2}(t) = 0.
    \end{align*}  
    It shows that $f_n$ is a Cauchy sequence in $C([0, T], L^1_2)$ for any fixed $T<\infty$, so we choose a unique limit point $f\in C([0, T], L^1_2)$. The convergence first implies that $0\leq f(t,v)\leq 1$ a.e. $t$ and $v$ as all the $f_n$ satisfy $0\leq f_n\leq 1$. As $f_n$ are all conservative solutions, $f$ also enjoys mass, momentum, and energy conservation. Since $f_n\rightarrow f$ in $C([0, T], L^1_2)$, we get
    \begin{align*}
        \esssup_{0\leq t\leq T}\|Q_{FD}(f,f) - Q_{FD,n}(f_n, f_n)\|_{1,0}\rightarrow 0,
    \end{align*}
    and it means that there exists a subsequence of $Q_{FD,n}(f_n,f_n)$ converges to $Q_{FD}(f,f)$ a.e. $t$ and $v$. Therefore, $f(t,v)$ satisfies
    \begin{align*}
        f(t,v) = f_0(v) + \int_0^t Q_{FD}(f,f)(\tau, v)\,d\tau
    \end{align*}
    a.e. $t$ and $v$. Defining
    \begin{align*}
        g(t,v) = f_0(v) + \int_0^t Q_{FD}(|f|\wedge 1,|f|\wedge 1)(\tau, v)\,d\tau,
    \end{align*}
    we have $f=g$ a.e. $t$ and $v$, and in fact one can show that
    \begin{align*}
        Q_{FD}(|f|\wedge 1,|f|\wedge 1)(t,v) = Q_{FD}(|g|\wedge 1,|g|\wedge 1)(t,v)
    \end{align*}
    for a.e. $t$ and $v$. So,
    \begin{align*}
        \int_0^T |Q_{FD}(|f|\wedge 1,|f|\wedge 1)-Q_{FD}(|g|\wedge 1,|g|\wedge 1)|(\tau,v)\,d\tau=0
    \end{align*}
    a.e. $v$. Replacing $Q_{FD}(|f|\wedge 1,|f|\wedge 1)$ by $Q_{FD}(|g|\wedge 1,|g|\wedge 1)$ and then renaming $g$ by $f$ again, therefore,
    \begin{align}\label{eq:f_diff_eq}
        f(t,v) = f_0(v) + \int_0^t Q_{FD}(|f|\wedge 1,|f|\wedge 1)(t,v)\,d\tau
    \end{align}
    is satisfies for $t\in [0, T]$ and $v\in\mathbb{R}^3\setminus Z$ for some null set $Z$ independent to $t$. Finally, we check $0\leq f(t,v)\leq 1$. We first note that $f$ is absolutely continuous about $t\in [0,T]$ for $v\in \mathbb{R}^3\setminus Z'$ for some null set $Z'\supset Z$. Since $0\leq f_0(v)\leq 1$, using for example the proof of Lemma \ref{lem:fg+}, we can check $0\leq f(t,v)\leq 1$ for all $t$ and $v\in \mathbb{R}^3\setminus Z'$. Mollifying the null set $Z'$, we get $0\leq f(t,v)\leq 1$ for all $t$ and $v$ with some null set $Z''\supset Z'$ such that \eqref{eq:f_diff_eq} holds for all $t$ and $v\in\mathbb{R}^3\setminus Z''$. As $0\leq f\leq 1$, we replace $|f|\wedge 1$ in \eqref{eq:f_diff_eq} by $f$ and restore the original equation. It shows that $f(t,v)$ is a solution of the Boltzmann-Fermi-Dirac equation for $t\in [0, T]$. Since $f_n\in L^\infty([0, T], L^1_s)$ for all $s\geq 2$, applying Fatou's lemma for each fixed $t\in [0,T]$, we get $f\in L^\infty([0,T], L^1_s)$ for all $s\geq 2$.
    
    In the above, we have taken an arbitrary $T<\infty$, and the limit point $f(t,v)$ of $f_n(t,v)$ in each $C([0,T], L^1_2)$ should be unique. Therefore, we can concatenate the solution and get $f\in C([0,\infty), L^1_2)$ with $L^\infty_{loc}([0, \infty), L^1_s)$ for all $s\geq 2$.
    
    Now, we can use the original polynomial moment inequality. For $0<\gamma\leq 2$, by Lemma \ref{lem:m_s_diff}, $m_s(t)\in C^1([0,\infty))$ for all $s\geq 2$, so we can use Theorem \ref{lem:poly_L1} to conclude polynomial moment creation \eqref{eq:poly_cre} and propagation \eqref{eq:poly_pro}.
    
    For uniqueness, let $g(t,v)\in L^\infty([0,\infty), L^1_2)$ be a solution with initial function $g(0,v) = f_0(v)$. From Proposition \ref{prop:stability}, $\|f(t,v) - g(t,v)\|_{1,2} = 0$ for all $t>0$. It proves that $f(t,v)$ is the unique solution.
\end{proof}

Next, we relax the condition to $f_0\in L^1_2$.
\begin{theorem}\label{thm:exi_uni_thm_2}
    Assume the collision kernel satisfies (H1) and $0\leq \gamma\leq 2$. For $f_0\in L^1_2$ with $0\leq f_0\leq 1$, there exists a unique solution of the Boltzmann-Fermi-Dirac equation. If $\gamma>0$, then it fulfills \eqref{eq:poly_cre}. Furthermore, if $f_0\in L^1_s$, then the solution also satisfies \eqref{eq:poly_pro}.
\end{theorem}
\begin{proof}
    Again, we assume $\gamma>0$ since $\gamma = 0$ case is already proved in \cite{Lu2001353}. Let $f_{n, 0} = f_0 e^{-|v|^2/n}$. Then, there exists a unique solution $f_n(t,v)$ having initial function $f_{n,0}$ for each $n\geq 1$. Also, those solutions should satisfy \eqref{eq:poly_cre}. By Proposition \ref{prop:stability}, we have
    \begin{align*}
        \|f_n-f_m\|_{1,2}(t)\leq C_1\Phi(\|f_{n,0}-f_{m,0}\|_{1,2})e^{C_2(t+t^{1/3})},
    \end{align*}
    so $f_n$ forms a Cauchy sequence in $C([0, T],L^1_2)$ for any $T<\infty$. For fixed $T<\infty$, let $f(t,v)$ be the limit in $C([0, T],L^1_2)$. Following the arguments in Proposition \ref{prop:exi_uni_prop}, we can check that $f$ is a solution of the Boltzmann-Fermi-Dirac equation. By Fatou's lemma, $f(t,v)$ satisfies \eqref{eq:poly_cre}.

    By Proposition \ref{prop:stability} again, $f(t,v)$ is the unique solution in $C([0,T], L^1_2)$. Since it is true for all finite $T$, we eventually obtain the existence and uniqueness of the solution of the Boltzmann-Fermi-Dirac equation for the whole time.
\end{proof}

We end this section proving the entropy identity \eqref{eq:entropy_identity}.
\begin{proposition}\label{prop:entropy_identity}
    Assume the collision kernel satisfies (H1) and $0\leq \gamma\leq 2$. Let $f$ be a solution of the Boltzmann-Fermi-Dirac equation with the collision kernel $B$. Then, it satisfies the entropy identity \eqref{eq:entropy_identity}.
\end{proposition}

\begin{proof}
    It mainly follows the proof of \cite{Lu2001353}.
    Let
    \begin{align*}
        \phi_n(t,v) = -\left(f(t,v) + \frac{e^{-|v|}}{n}\right)\ln\left(f(t,v) + \frac{e^{-|v|}}{n}\right) - \left(1-f(t,v) + \frac{e^{-|v|}}{n}\right)\ln\left(1-f(t,v) + \frac{e^{-|v|}}{n}\right).
    \end{align*}
    For fixed a.e. $v$, $\phi_n(t,v)$ is a.e. differentiable about $t$ since $f$ is absolutely continuous and
    \begin{align*}
        \tilde{\phi}_n(x) = -\left(x + \frac{e^{-|v|}}{n}\right)\ln\left(x + \frac{e^{-|v|}}{n}\right) - \left(1-x + \frac{e^{-|v|}}{n}\right)\ln\left(1-x + \frac{e^{-|v|}}{n}\right)
    \end{align*}
    is Lipschitz continuous about $x$. So, we get
    \begin{align*}
        \phi_n(t,v) = \phi_n(0,v) - \int_0^t Q_{FD}(f,f)(\tau, v)\ln\left(\frac{f(\tau, v)+\frac{e^{-|v|}}{n}}{1-f(\tau, v)+\frac{e^{-|v|}}{n}}\right)\,d\tau
    \end{align*}
    for a.e. $v$. Defining
    \begin{align*}
        S_n(f)(t) = \int_{\mathbb{R}^3} \phi_n(t,v)\,dv
    \end{align*}
    and taking $v$ integral on both sides, it becomes
    \begin{align}\label{eq:S_n_prod}
        S_n(f)(t) = S_n(f)(0) - \int_{\mathbb{R}^3}\int_0^t Q_{FD}(f,f)(\tau, v)\ln\left(\frac{f(\tau, v)+\frac{e^{-|v|}}{n}}{1-f(\tau, v)+\frac{e^{-|v|}}{n}}\right)\,d\tau dv.
    \end{align}
    Our mission is to make $n\rightarrow \infty$ and obtain the entropy identity \eqref{eq:entropy_identity}. We first check the well-definedness of each term in \eqref{eq:S_n_prod}. First, as
    \begin{align*}
        &g(v)|\ln g(v)| = g(v)|\ln g(v)|\mathbf{1}_{\{g(v)\leq e^{-|v|^2}\}} + g(v)|\ln g(v)|\mathbf{1}_{\{g(v)>e^{-|v|^2}\}}\leq |v|^2g(v) + e^{-\frac{|v|^2}{2}},\\
        &(1-g(v))|\ln(1-g(v))| \leq g(v) 
    \end{align*}
    for any function $0\leq g(v)\leq 1$, we have
    \begin{align*}
        \int_{\mathbb{R}^3} |\phi_n(t,v)|\,dv\leq \int_{\mathbb{R}^3} \left((1+|v|^2)\left(f(t,v) + \frac{e^{-|v|}}{n}\right) + e^{-\frac{|v|^2}{2}}\right)\,dv\leq \|f\|_{1,2} + C_n,
    \end{align*}
    where $\sup_n C_n<\infty$. Therefore, $S_n(f)(t)$ is well-defined for all $n$ including $n=\infty$. By the dominated convergence theorem, we also get
    \begin{align}\label{eq:S_n_limit}
        \lim_{n\rightarrow\infty} S_n(f)(t) = S(f)(t)
    \end{align}
    for all $t\geq 0$.
    
    Secondly, we bound the integral of $Q_{FD}(f,f)$. As
    \begin{align*}
        \left|\ln\left(\frac{f(\tau, v)+\frac{e^{-|v|}}{n}}{1-f(\tau, v)+\frac{e^{-|v|}}{n}}\right)\right|\leq \ln(2n) + |v|,
    \end{align*}
    we write
    \begin{align*}
        &\int_{\mathbb{R}^6\times\mathbb{S}^2} B(v-v_*,\sigma)\left|f'f'_*(1-f)(1-f_*) + ff_*(1-f')(1-f'_*)\right|\ln\left(\frac{f(\tau, v)+\frac{e^{-|v|}}{n}}{1-f(\tau, v)+\frac{e^{-|v|}}{n}}\right)\,dvdv_*d\sigma\\
        &\leq \int_{\mathbb{R}^6\times\mathbb{S}^2} B(v-v_*,\sigma)(f'f'_* + ff_*)\ln\left(\frac{f(\tau, v)+\frac{e^{-|v|}}{n}}{1-f(\tau, v)+\frac{e^{-|v|}}{n}}\right)\,dvdv_*d\sigma\\
        &\leq \int_{\mathbb{R}^6\times\mathbb{S}^2} B(v-v_*,\sigma)(f'f'_* + ff_*)(\ln(2n) + |v|)\,dvdvdv_*d\sigma.
    \end{align*}
    It is bounded by
    \begin{align*}
        &\int_{\mathbb{R}^6\times\mathbb{S}^2} B(v-v_*,\sigma)(f'f'_* + ff_*)(\ln(2n) + |v|)\,dvdv_*d\sigma\\
        &\leq C_n\left(\|f\|_{1,\gamma}\|f\|_{1,1} + \|f\|_{1,1+\gamma}\|f\|_{1,0}\right),
    \end{align*}
    where $C_n$ is a constant depending on $n$. By \eqref{eq:poly_cre}, if $1+\gamma>2$, we have $\|f\|_{1,1+\gamma}\leq \max\{\frac{C}{t^{1-\frac{1}{\gamma}}}, 1\}$. Since it is integrable about $t$ in any finite interval $[0,T]$, we prove that
    \begin{align*}
        &\int_0^T \int_{\mathbb{R}^6\times\mathbb{S}^2} B(v-v_*,\sigma)\left|\left(f'f'_*(1-f)(1-f_*) + ff_*(1-f')(1-f'_*)\right)\ln\left(\frac{f(\tau, v)+\frac{e^{-|v|}}{n}}{1-f(\tau, v)+\frac{e^{-|v|}}{n}}\right)\right|\,dvdv_*d\sigma d\tau\\
        &<\infty
    \end{align*}
    for any finite $n$ and $T$. It guarantees Fubini's theorem and the change of variable, so we obtain
    \begin{align}\label{eq:Gamma_n_split}
        \begin{split}
            &\int_{\mathbb{R}^3}\int_0^t Q_{FD}(f,f)(\tau, v)\ln\left(\frac{1-f(\tau, v)+\frac{e^{-|v|}}{n}}{f(\tau, v)+\frac{e^{-|v|}}{n}}\right)\,d\tau dv\\
            &\leq \int_0^t\int_{\mathbb{R}^6\times\mathbb{S}^2} B(v-v_*,\sigma)\Gamma_n(f)\,dvdv_*d\sigma d\tau,
        \end{split}
    \end{align}
    where
    \begin{align*}
        \Gamma_n(f) &= \frac{1}{4}(f'f'_*(1-f)(1-f_*) - ff_*(1-f')(1-f'_*))\\
        &\qquad \times \ln\left(\frac{\left(f'+\frac{e^{-|v'|}}{n}\right)\left(f'_*+\frac{e^{-|v'_*|}}{n}\right)\left(1-f+\frac{e^{-|v|}}{n}\right)\left(1-f_*+\frac{e^{-|v_*|}}{n}\right)}{\left(f+\frac{e^{-|v|}}{n}\right)\left(f_*+\frac{e^{-|v_*|}}{n}\right)\left(1-f'+\frac{e^{-|v'|}}{n}\right)\left(1-f'_*+\frac{e^{-|v'_*|}}{n}\right)}\right).
    \end{align*}
    Let us split $\Gamma_n(f)$ by $\left(\Gamma_n(f)\right)^+$ and $\left(-\Gamma_n(f)\right)^+$. As
    \begin{align*}
        &\left(-(a-b)\ln \frac{c}{d}\right)^+\leq \left(a\ln \frac{d}{a}\mathbf{1}_{\{a>b\}}\right)^+ + \left(b\ln\frac{c}{b}\mathbf{1}_{\{b>a\}}\right)^+\leq c+d-a-b\\
        &\left((a-b)\ln \frac{c}{d}\right)^+\leq (a-b)\ln\frac{a}{b} + \left(a\ln \frac{c}{a}\mathbf{1}_{\{a>b\}}\right)^+ + \left(b\ln\frac{d}{b}\mathbf{1}_{\{b>a\}}\right)^+\\
        &\leq (a-b)\ln\frac{a}{b} + c+d-a-b
    \end{align*}
    for $0\leq a\leq c$ and $0\leq b\leq d$, we bound $\left(\pm \Gamma_n(f)\right)^+$ by
    \begin{align*}
        \left(\Gamma_n(f)\right)^+&\leq \Gamma(f) + \frac{1}{4}\left(f'+\frac{e^{-|v'|}}{n}\right)\left(f'_*+\frac{e^{-|v'_*|}}{n}\right)\left(1-f+\frac{e^{-|v|}}{n}\right)\left(1-f_*+\frac{e^{-|v_*|}}{n}\right) \\
        &\quad + \frac{1}{4}\left(f+\frac{e^{-|v|}}{n}\right)\left(f_*+\frac{e^{-|v_*|}}{n}\right)\left(1-f'+\frac{e^{-|v'|}}{n}\right)\left(1-f'_*+\frac{e^{-|v'_*|}}{n}\right)\\
        &\leq \Gamma(f) + \left(f'+e^{-|v'|}\right)\left(f'_*+e^{-|v'_*|}\right)+ 4\left(f+e^{-|v|}\right)\left(f_*+e^{-|v_*|}\right),\\
        \left(-\Gamma_n(f)\right)^+&\leq \left(f'+e^{-|v'|}\right)\left(f'_*+e^{-|v'_*|}\right)+ 4\left(f+e^{-|v|}\right)\left(f_*+e^{-|v_*|}\right).
    \end{align*}
    
    From \eqref{eq:S_n_prod} and \eqref{eq:Gamma_n_split}, we have
    \begin{align*}
        S_n(f)(t) = S_n(f)(0) - \int_0^t\int_{\mathbb{R}^6\times\mathbb{S}^2}B(v-v_*,\sigma) \left(\left(\Gamma_n(f)\right)^+ - \left(-\Gamma_n(f)\right)^+\right)\,dvdv_* d\sigma d\tau.
    \end{align*}
    By the above pointwise bound of $\left(-\Gamma_n(f)\right)^+$ and \eqref{eq:S_n_limit}, we obtain
    \begin{align*}
        \lim_{n\rightarrow\infty} \int_0^t\int_{\mathbb{R}^6\times\mathbb{S}^2}B(v-v_*,\sigma) \left(\Gamma_n(f)\right)^+\,dvdv_* d\sigma d\tau = S(f)(t) - S(f)(0).
    \end{align*}
    By Fatou's lemma, we also have
    \begin{align*}
         &\int_0^t\int_{\mathbb{R}^6\times\mathbb{S}^2}B(v-v_*,\sigma) \left(\Gamma(f)\right)^+\,dvdv_* d\sigma d\tau\\
         &\leq \lim_{n\rightarrow\infty} \int_0^t\int_{\mathbb{R}^6\times\mathbb{S}^2}B(v-v_*,\sigma) \left(\Gamma_n(f)\right)^+\,dvdv_* d\sigma d\tau=S(f)(t) - S(f)(0).
    \end{align*}
    It shows that $\left(\Gamma_n(f)\right)^+$ is in fact pointwisely bounded by an $L^1$ function. Therefore, using the dominated convergence theorem, we finally have the entropy identity \eqref{eq:entropy_identity}.
\end{proof}

We end this section proving Theorem \ref{thm:exi_uni_thm} and \ref{thm:L1_2_stability}.
\begin{proof}[Proof of Theorem \ref{thm:exi_uni_thm}]
    We combine Theorem \ref{thm:exi_uni_thm_2} and Proposition \ref{prop:entropy_identity}.
\end{proof}
\begin{proof}[Proof of Theorem \ref{thm:L1_2_stability}]
    It is checked in Proposition \ref{prop:stability}.
\end{proof}

\section{Propagation of a \texorpdfstring{$L^\infty$}{} Gaussian upper bound}\label{sec:L^infty_Gaussian_upper}
In this section, we establish the propagation of $L^\infty$ Gaussian upper bounds for solutions to the Boltzmann-Fermi-Dirac equation. We use a comparison argument developed in \cite{GPV2009}. This approach was later extended to the inelastic Boltzmann equation in \cite{AJL2024}.

Like the classical Boltzmann equation, we first define $Q_{FD}^+, Q_{FD}^-$ and $L_{FD}$.
\begin{definition}
    For $v \in \mathbb{R}^3$, we define
    \begin{align*}
        Q_{FD}^+ (f_1,f_2,1-f_3,1-f_4)(v) &\coloneqq  \int_{\mathbb{R}^3 \times \mathbb{S}^2} B(v-v_*,\sigma)f_1(v')f_2(v_*')(1-f_3(v))(1-f_3(v_*)) \,d\sigma dv_*,\\
        Q_{FD}^- (f_1,f_2,1-f_3,1-f_4)(v) &\coloneqq  \int_{\mathbb{R}^3 \times \mathbb{S}^2} B(v-v_*,\sigma) f_1(v)f_2(v_*)(1-f_3(v'))(1-f_4(v_*')) \,d\sigma dv_*,\\
        \intertext{and}
        L_{FD}(f_1,1-f_2,1-f_3)(v) &\coloneqq  \int_{\mathbb{R}^3 \times \mathbb{S}^2} B(v-v_*,\sigma) f_1(v_*)(1-f_2(v'))(1-f_3(v_*')) \,d\sigma dv_*.
    \end{align*}
\end{definition}
By the definition, we have
\begin{align*}
    Q_{FD}(f,f)&= Q_{FD}^+ (f,f,1-f,1-f) - Q_{FD}^- (f,f,1-f,1-f)\\
    &=Q_{FD}^+ (f,f,1-f,1-f) - fL_{FD}(f,1-f,1-f).
\end{align*} 

The next lemma states a lower bound of $L_{FD}$.
\begin{lemma} \label{lem:Lf_lower}
    We consider the collision kernel $B$ for $0 < \gamma \leq 2$ and (H1). Assume that $f \in L_2^1$ and $0\leq f\leq 1$. Then there exist constants $R>0$ and $C>0$ depending on $\|f\|_{1,0}, \|f\|_{1,2}, \gamma, C_b$, and $\varphi(\epsilon)$ such that
    \begin{align*} 
         L_{FD}(f,1-f,1-f)(v)\geq C|v|^\gamma,
         \text{\quad where \quad } |v| \geq R.
    \end{align*}
\end{lemma}
\begin{proof}
    We split $L_{FD}(f,1-f,1-f)$ into two parts by
        \begin{align*}
            &L_{FD}(f,1-f,1-f) (v) = \int_{\mathbb{R}^3 \times\mathbb{S}^2} B(|v-v_*|,\cos\theta)f(v_*) (1-f(v'))(1-f(v_*'))\, d\sigma dv_* \\
            &= \int_{\mathbb{R}^3 \times\mathbb{S}^2} B(|v-v_*|,\cos\theta)f(v_*) \, d\sigma dv_* - \int_{\mathbb{R}^3 \times\mathbb{S}^2} B(|v-v_*|,\cos\theta)f(v_*)(f(v')+f(v_*')) \, d\sigma dv_*.
        \end{align*}
    In Lemma \ref{lem:epsv}, we found constants $C_1>0$ and $C_2>0$ depending on $\gamma$ and $C_b$ such that
    \begin{align}\label{eq:app_6.2}
        \int_{\mathbb{R}^3 \times \mathbb{S}^2} B(|v-v_*|,\cos \theta)  f(v_*) (f(v')+f(v'_*))\,d\sigma dv_* \leq \frac{C_1}{\epsilon^3}\|f\|_{1,2}+ C_2 \varphi(\epsilon)(\|f\|_{1,2}+|v|^\gamma\|f\|_{1,0})
    \end{align} 
    for every $0<\epsilon <1$. 
    
    We estimate a lower bound of the first term by
    \begin{align}\label{eq:Lf>v}
        \begin{split}
            \int_{\mathbb{R}^3 \times\mathbb{S}^2} B(|v-v_*|,\cos\theta)f(v_*) \,d\sigma dv_* &= C_b\int_{\mathbb{R}^3} |v-v_*|^\gamma f(v_*) \,dv_*\\
            &\geq C_b\int_{\mathbb{R}^3} \left(\frac{1}{2}|v|^\gamma - |v_*|^\gamma\right)f(v_*) \,dv_*\\
            &\geq C_b\left(\frac{|v|^\gamma}{2}\|f\|_{1,0} - \|f\|_{1,2}\right).
        \end{split}
    \end{align}
    In the middle, we used (H1) and
    \begin{align*}
        |v-v_*|^\gamma\geq \big||v|-|v_*|\big|^\gamma \geq \frac{1}{2}|v|^\gamma - |v_*|^\gamma.
    \end{align*}
    Now, we choose $\epsilon = \epsilon_*$ in \eqref{eq:app_6.2} such that $C_2\varphi(\epsilon_*)\leq \frac{C_b\|f\|_{1,0}}{4}$. Combining \eqref{eq:app_6.2} and \eqref{eq:Lf>v}, we obtain
    \begin{align*}
        L_{FD}(f,1-f,1-f)(v)\geq C_b\frac{\|f\|_{1,0}}{4}|v|^\gamma - \left(\frac{C_1}{\epsilon_*^3} + C_2 \varphi(\epsilon_*) + C_b\right)\|f\|_{1,2}. 
    \end{align*}
    We fix a sufficiently large $R>0$ such that
    \begin{align*}
        \frac{C_b\|f\|_{1,0}}{8} R^\gamma \geq \left(\frac{C_1}{\epsilon_*^3} + C_2 \varphi(\epsilon_*) + C_b\right)\|f\|_{1,2}.
    \end{align*}
    For $|v|\geq R$, we get
    \begin{align*}
        L_{FD}(f,1-f,1-f)(v) \geq \frac{C_b\|f\|_{1,0}}{8} |v|^\gamma.
    \end{align*}
    Here, $R$ depends on $\|f\|_{1,0}^{-1}, \|f_0\|_{1,2}, \gamma, C_b$, and $\varphi(\epsilon)$.
\end{proof}

\begin{remark}
    In the case of $0<\gamma\leq 1$, in \cite{L1983}, Arkeryd proved that 
    \begin{align*}
        L_c(f)(v)\coloneqq \int_{\mathbb{R}^3 \times\mathbb{S}^2} B(v-v_*,\sigma) f(v_*)\,d\sigma dv_* \geq C(1+|v|)^\gamma 
    \end{align*}
    for some $C$ under the assumption $f\in L^1_2$ and $\int_{\mathbb{R}^3} f|\ln f|\,dv<\infty$. However, this global lower bound can not be easily adapted into the Fermi-Dirac case. For example, if we take $f = \mathbf{1}_{\{|v|\leq r\}}$ for some $r>0$, which is a saturated equilibrium, then $L(f)(v) = 0$ for $|v|\leq r$. Indeed, to make $f(v_*)\neq 0$, we need to choose $|v_*|\leq r$. As $|v|,|v_*|\leq r$, we have $|v'|\leq r$ or $|v'_*|\leq r$. It means $(1-f')(1-f'_*) = 0$ and
    \begin{align*}
        L_{FD}(f,1-f,1-f)(v) = \int_{\mathbb{R}^3 \times\mathbb{S}^2} B(|v-v_*|,\cos\theta)f(v_*) (1-f(v'))(1-f(v_*'))\, d\sigma dv_* = 0.
    \end{align*}
    We can detour this problem by adding assumption $\int_{\mathbb{R}^3} |f\ln f + (1-f)\ln (1-f)|\,dv>0$ and applying the Gaussian lower bound result. But, this method depends on the specific shape of $f$.

    To avoid this problem, the above lemma chose some large enough $R>0$ and proved a lower bound for $|v|\geq R$.
\end{remark}

In (H3), we recall $\alpha<2$ is defined by
\begin{align*}
    b(\cos\theta)\sin^\alpha \theta\leq C
\end{align*}
for some constant $C$.

Under (H3), $\varphi(\epsilon)$ is given by
\begin{align*}
    \varphi(\epsilon)&=\int_{\mathbb{S}^2} b(\cos\theta)\left(\mathbf{1}_{\{0<\theta<\epsilon\}} + \mathbf{1}_{\{\pi-\epsilon<\theta<\pi\}}\right)\,d\sigma\\
    &\leq 4\pi C\int_0^\epsilon\frac{1}{\sin^{\alpha-1} \theta}\,d\theta\leq 2^{3-\alpha}\pi C\int_0^\epsilon\frac{1}{\theta^{\alpha-1} }\,d\theta\\
    &\leq \frac{2^{3-\alpha}\pi C}{2-\alpha}\epsilon^{2-\alpha}
\end{align*}
for $0<\epsilon< 1$. Therefore, the dependency on $\varphi(\epsilon)$ can be replaced by $\alpha<2$. Also, there is an explicit upper bound of $\varpi_p$ in Lemma \ref{lem:bobylev_2} in this case; we refer to \cite{ACGM2013}. So, we can replace the dependency on $b(\cos\theta)$ in Theorem \ref{thm:L1_bound} by the dependency on $\alpha<2$.

From now on, we will follow the proof technique in \cite{GPV2009}. We first refer to a technical lemma in \cite{GPV2009}.
\begin{lemma}[Lemma 5 of \cite{GPV2009}]\label{lem:GPV2009_5}
    We consider the collision kernel \eqref{eq:B_defi} for $0 < \gamma$, (H3), and an angle restriction
    \begin{align*}
        B(|v-v_*|,\cos\theta) = B(|v-v_*|,\cos\theta)\mathbf{1}_{\{\cos\theta \geq 0\}}.
    \end{align*}
    Let $M(v) = e^{-a|v|^2}$ for $a>0$ and $\epsilon = \min\{\gamma, 2-\alpha\}>0$. Then, we have
    \begin{align*}
        Q^+_c(M, f)(v)\leq C\left\|(1+|v|^{\gamma-\epsilon})\frac{f}{M}\right\|_{L^1}(1+|v|^{\gamma-\epsilon})M(v)
    \end{align*}
    for some constant $C$ depending on $\alpha$, $\gamma$, and $a$.
\end{lemma}

Using this lemma, we can prove the following lemma.
\begin{lemma} \label{lem:Q(M)<0}
    We consider the collision kernel $B$ for $0 < \gamma \leq 2$, (H3), and an angle restriction
    \begin{align*}
        B(|v-v_*|,\cos\theta) = B(|v-v_*|,\cos\theta)\mathbf{1}_{\{\cos\theta \geq 0\}}.
    \end{align*}
    We assume $f$ satisfies $0\leq f\leq 1$ and
    \begin{align*}
        \int_{\mathbb{R}^3} f(v) e^{2a|v|^2}\; dv \leq C
    \end{align*} 
    for some constant $a>0$ and $C>0$.
    Then, for a Gaussian function $M(v) \coloneqq  e^{-a|v|^2}$, there exists $r<\infty$ such that 
    \begin{align*}
        Q_{FD}(M, f, 1-f, 1-f)\leq 0 \text{\quad for \quad} |v|\geq r.
    \end{align*}
    Here, $r$ depends on $\|f_0\|_{1,0}, \|f_0\|_{1,2}, \gamma, \alpha, C_b, a$, and $C$.
\end{lemma}
\begin{proof}
    From Lemma \ref{lem:GPV2009_5}, we get
    \begin{align*}
        Q^+_{FD}(M, f, 1-f, 1-f)(v)\leq Q^+_c(M, f)(v)\leq C_1(1+|v|^{\gamma - \epsilon}) M(v).
    \end{align*}
    
    From Lemma \ref{lem:Lf_lower}, we can find $R>0$ and $C_2>0$, which depends on $\|f\|_{1,0}, \|f\|_{1,2}, \gamma, \alpha$, and $C_b$ such that
    \begin{align*} 
        Q_{FD}^- (M, f, 1-f, 1-f)(v) = M(v)L_{FD}(f,1-f,1-f)(v)
        \geq C_2 M(v)|v|^{\gamma} \text{\quad for \quad} |v|>R.
    \end{align*}

    Since $\epsilon>0$, we can choose $r\geq R$ large enough so that
    \begin{align*}
           C_1(1+|v|^{\gamma - \epsilon}) - C_2 |v|^{\gamma} \leq 0\text{\quad for \quad} |v|\geq r.
    \end{align*}
    Thus, we obtain
    \begin{align*}
        &Q_{FD}(M, f, 1-f, 1-f)\\
        &=Q_{FD}^+(M,f,1-f,1-f)(v) -   Q_{FD}^- (M,f, 1-f,1-f)(v) \leq 0 \text{\quad for \quad} |v|\geq r.
    \end{align*}
\end{proof}

Next, we prove a technical lemma for a comparison argument, which extends Proposition 1 of \cite{GPV2009}.
\begin{lemma} \label{lem:comparison}
    Let $f : [0, \infty) \times \mathbb{R}^3 \rightarrow  [0,1]$ and $u : [0, \infty) \times \mathbb{R}^3 \rightarrow \mathbb{R}$ satisfy
    \begin{enumerate}
        \item $u(t,v), f(t,v)\in L^\infty([0, \infty), L^1_2(\mathbb{R}^3))$.
        \item $u(0,v)\leq 0$. Also, there exists $r>0$ such that $u(t,v)\leq 0$ on $|v|\leq r$ for all $t\geq 0$.
        \item $u$ and $f$ satisfy \begin{align}\label{eq:comparison_condi_1}
        u^+(t,v) \leq \int_0^t Q_{FD}(u,f, 1-f, 1-f)(\tau,v)\mathbf{1}_{\{u(\tau, v)\geq 0\}}\,d\tau \quad \text{on}\quad |v|\geq r.
    \end{align}
    \end{enumerate}
    
    Then, we obtain $u(t,v)\leq 0$ for $t>0$ and a.e. $v \in \mathbb{R}^3$.
\end{lemma}
\begin{proof}
    If $|v|\leq r$, as $u(\tau,v)\leq 0$ for all $\tau\geq 0$, the both sides of \eqref{eq:comparison_condi_1} are $0$. Therefore, \eqref{eq:comparison_condi_1} in fact holds for all $v\in\mathbb{R}^3$. Taking $v$ integration on both sides, we get
    \begin{align*}
        \int_{\mathbb{R}^3} u^+(t,v)\,dv \leq \int_0^t \int_{\mathbb{R}^3} \left(Q_{FD}^+ (u,f, 1-f, 1-f) -  Q^-_{FD}(u,f, 1-f, 1-f)\right)(\tau, v)\mathbf{1}_{\{u(\tau,v)\geq 0\}}\,dv d\tau.
    \end{align*}
    
    We regard $\mathbf{1}_{\{u(\tau,v)\geq 0\}}$ as a test function and employ symmetry \eqref{eq:bef_aft}; it is well-defined as $u,f\in L^\infty([0,\infty), L^1_2(\mathbb{R}^3))$. Then
    \begin{align*}
        &\int_{\mathbb{R}^3} u^+(t,v) \,dv\\
        &\leq \int_0^t \int_{\mathbb{R}^6 \times \mathbb{S}^2} 
        B(|v-v_*|, \cos \theta) uf_*(1-f')(1-f'_*) \left(\mathbf{1}_{\{u(\tau, v')\geq 0\}}-\mathbf{1}_{\{u(\tau, v)\geq 0\}}\right)\,d\sigma dv_*dv d\tau.
    \end{align*}
    Because $u(v)\left(\mathbf{1}_{\{u(v')\geq 0\}}-\mathbf{1}_{\{u(v)\geq 0\}}\right) \leq 0$ for any $v$ and $v'$, we deduce that $\int_{\mathbb{R}^3} u^+(t,v) \,dv\leq 0$ and $u\leq 0$ a.e. $v$.
\end{proof}

In the proof of Theorem \ref{thm:L1_bound}-(3), we will define $u(t,v) = f(t,v) - M(v)$, where $M(v)$ is a Gaussian, and apply Lemma \ref{lem:comparison}.

Now, we are ready to prove Theorem \ref{thm:L1_bound}-(3). In the proof, we apply Lemma \ref{lem:comparison}, Lemma \ref{lem:Q(M)<0}, and Theorem \ref{thm:L1_bound}-(2).

\begin{proof}[Proof of Theorem \ref{thm:L1_bound}-(3)]
    To make the proof easy, we first restrict the collision kernel by
    \begin{align*}
        B(|v-v_*|, \cos\theta) = B(|v-v_*|, \cos\theta)\mathbf{1}_{\{\cos\theta\geq 0\}}.
    \end{align*}
    Indeed, by the symmetry on $b(\cos\theta)$ and \eqref{eq:sym_sigma}, we have
    \begin{align*}
        Q_{FD}(f,f) &= \int_{\mathbb{R}^3\times\mathbb{S}^2} B(|v-v_*|, \cos\theta)\big(f(v')f(v_*')(1-f(v))(1-f(v_*))\\
        &\quad \quad -f(v)f(v_*)(1-f(v'))(1-f(v_*'))\big)\,d\sigma dv_*\\
        &=2\int_{\mathbb{R}^3\times\mathbb{S}^2} B(|v-v_*|, \cos\theta)\mathbf{1}_{\{\cos\theta\geq 0\}}\big(f(v')f(v_*')(1-f(v))(1-f(v_*))\\
        &\quad \quad -f(v)f(v_*)(1-f(v'))(1-f(v_*'))\big)\,d\sigma dv_*,
    \end{align*}
    so it makes no difference in the result.
    
    Since $f_0(v) \leq M_0(v) \coloneqq e^{-a_0|v|^2+c_0}$, there exists a constant $C_0>0$ that depends on $a_0$ and $c_0$ such that
    \begin{align*} 
        \int_{\mathbb{R}^3} f_0(v)e^{\frac{a_0}{2}|v|^2}\,dv \leq C_0.
    \end{align*} 
    By the propagation of $L^1$ exponential moments in Theorem \ref{thm:L1_bound}-(2), there exist some constants $a_1, C_1>0$ depending on $\gamma, C_b, \alpha, \|f_0\|_{1,0}\|f_0\|_{1,2}, a_0$, and $C_0$ such that 
    \begin{align*}
        \sup_{t\in[0,\infty)} \int_{\mathbb{R}^3} f(t,v) e^{a_1|v|^2} \,dv \leq C_1.
    \end{align*}
    
    We take $a = \min{\{a_0, \frac{a_1}{2}\}}$ and $M'(v) = e^{-a|v|^2}$. From Lemma \ref{lem:Q(M)<0}, there exists $r>0$ that depends on $\|f_0\|_{1,0}, \|f_0\|_{1,2}, \gamma, C_b, \alpha, a$, and $C_1$ such that
    \begin{align*}
        Q_{FD}(M', f, 1-f, 1-f)(t,v)\leq 0\text{\quad for \quad} |v|\geq r.
    \end{align*}
    For such $r$, we choose $c = \max\{c_0, ar^2\}$ and define $M(v) = e^{-a|v|^2 + c}$. We show that $M(v)$ is the desired Gaussian upper bound function by checking the conditions in Lemma \ref{lem:comparison} for $u(t,v) = f(t,v) - M(v)$. As $M(v)\in L^1_2$, $f(t,v)$ and $u(t,v)$ are in $C([0,\infty), L^1_2(\mathbb{R}^3))$. Also, we have
    \begin{align*}
        f(0,v)-M(v) &\leq  f_0(v)-M_0(v) \leq 0,\quad \text{and}\\
        f(t,v)-M(v) &\leq 1 - e^{-a|v|^2 + ar^2}\leq 0\quad \text{for $t\geq 0$ and $|v|\leq r$.}
    \end{align*}
    Therefore, it fulfills the first and second conditions of Lemma \ref{lem:comparison}.
    Since $M(v)$ is the only function of $v$, and $f$ is a solution of the Boltzmann-Fermi-Dirac equation, following the proof of Lemma \ref{lem:fg+}, we have
    \begin{align*}
        &(f(t,v) - M(v))^+ \\
        &= (f_0(v) - M(v))^+ + \int_0^t Q_{FD}(f,f, 1-f, 1-f)(\tau, v)\mathbf{1}_{\{f(\tau, v) - M(v)\geq 0\}}\,d\tau\\
        &= \int_0^t \left(Q_{FD}(f-M,f,1-f,1-f) + Q_{FD}(M,f,1-f,1-f)\right)(\tau, v)\mathbf{1}_{\{f(\tau, v) - M(v)\geq 0\}}\,d\tau.
    \end{align*}
    Since $Q_{FD}(M, f, 1-f, 1-f) = e^c Q_{FD}(M', f, 1-f, 1-f)\leq 0$ for $|v|\geq r$, we reach
    \begin{align*}
        u^+(t,v)\leq \int_0^t Q_{FD}(u,f,1-f,1-f)(\tau, v)\mathbf{1}_{\{u(\tau, v)\geq 0\}}\,d\tau,\quad \text{for $|v|\geq r$.}
    \end{align*}
    Finally, we apply Lemma \ref{lem:comparison} and complete the proof.
\end{proof}

\section{Propagation of \texorpdfstring{$L^\infty$}{} polynomial moments}\label{sec:L^infty_polynomial_upper}

In this section, we study the $L^\infty$ polynomial moments estimates for the solution of the Boltzmann-Fermi-Dirac. We adapt the classical proof scheme in \cite{L1983} to the Fermi-Dirac case. For this, we choose the collision kernel $0<\gamma\leq 1$ and $b(\cos\theta) = const$. Note that $h(\cos\theta_\omega) = 2(const)\cos\theta_\omega$ in $\omega$-representation from \eqref{eq:h_condi}.

Our proof strategy is as follows. We write the Boltzmann-Fermi-Dirac equation by
\begin{align*}
    \partial_t f + f  L_{FD}(f,1-f,1-f)= Q_{FD}^+(f,f,1-f,1-f).
\end{align*}
As in \cite{L1983}, we compute the lower bound of the $L_{FD}(f,1-f,1-f)$, which was already done in Lemma \ref{lem:Lf_lower}, and upper bound of the $Q_{FD}^+(f,f,1-f,1-f)$. In fact, since $0\leq f\leq 1$, we have $Q_{FD}^+(f,f,1-f,1-f) \leq Q_c^+(f,f)$, so its upper bound is same as the $Q_c^+(f,f)$. We will refer to some functional inequalities around $Q_c^+(f,f)$ and in \cite{L1983} and then employ these inequalities to get the result for the Fermi-Dirac case.

We list some technical functional inequalities from \cite{L1983}. For the detailed description and proof, please visit the original paper.
\begin{lemma}[Lemma 3 of \cite{L1983}]\label{lem:sup_f}
    Let $h_1(t)$ and $h_2(t)$ be $L^1_{loc}([0,\infty))$ and $h_1(t)>0$ for $t\geq 0$. If $f(t)$ is an absolutely continuous function, and it satisfies
    \begin{align*} 
        \frac{d}{dt}f + h_1 f \leq h_2
    \end{align*}
    for a.e. $t\geq 0$, then
    \begin{align*}
        f(t) \leq  \max\left\{\esssup_{0 \leq s \leq t} \frac{h_2(s)}{h_1(s)},  f(0)\right\}\quad\text{for}\quad t\geq 0.
    \end{align*}
\end{lemma}

\begin{lemma}[Lemma 6 of \cite{L1983}] \label{lem:Arkeryd}
    Suppose that 
    \begin{align*}
        s_1, s_2 \geq 0, \;\; s_2-s_1 \leq 3, \text{\quad and \quad} f \in L_{s_1}^1  \cap L_{s_2}^\infty.
    \end{align*}
    Then, for $0 < \alpha <3$ and $v \in \mathbb{R}^3$, there exists a constant $C>0$ depending on $\alpha$ such that
    \begin{align*}
        \int_{\mathbb{R}^3} f(v_1)|v-v_1|^{-\alpha} \,dv_1 \leq C(\|f\|_{1,s_1}+\|f\|_{\infty,s_2}) (1+|v|)^{-\beta},
    \end{align*} where
    \begin{align*}
        \beta = \min \left\{ \alpha, \left(1-\frac{\alpha}{3}\right)s_1 +\frac{\alpha}{3} s_2\right\}.
    \end{align*}
\end{lemma}

\begin{lemma}[Lemma 8 of \cite{L1983}] \label{lem:int_Qc_s1}
We consider the collision kernel \eqref{eq:B_defi} for $0< \gamma \leq 1$ and (H4). Assume $f \in L_{s_1}^1\cap L^\infty_0$ for some $s_1 \geq 2$. Let $E$ be an arbitrary 2D plane in $\mathbb{R}^3$, and let $v \in \mathbb{R}^3$. There exists a constant $C>0$ depending on $\gamma$, $b(\cos\theta)$, and $s_1$ such that
    \begin{align*}
        \begin{split}
            &\int_E\mathbf{1}_{\{v_1: |v_1| > |v|\}} Q_c^+(f,f)(v_1) \,dv_1 \leq C\left(\|f\|_{1,s_1} + \|f\|_{\infty, 0}\right)^2(1+|v|)^{-s_1+\gamma-1}.
        \end{split}
    \end{align*}
\end{lemma}

Now, we are ready to prove the main lemma. It bounds the integral of the higher velocity part of the solution of the Boltzmann-Fermi-Dirac equation in a $2D$ plane $E$.
\begin{lemma}\label{lem:int_E_s1s2}
    We consider the collision kernel \eqref{eq:B_defi} for $0< \gamma \leq 1$ and (H4). Let $E$ be a $2D$ plane in $\mathbb{R}^3$ and $f$ be the solution of the Boltzmann-Fermi-Dirac equation with $f_0\in L^1_{s_1}$ for some $s_1\geq 2$. Then, there exist a constant $C>0$ and $R>0$, depending on the $\|f_0\|_{1,0}, \|f_0\|_{1,s_1},\gamma, b(\cos\theta)$, and $s_1$, such that
    \begin{align*}
        &\int_E\mathbf{1}_{\{v_1: |v_1| > |v|\}}f(v_1)\, dv_1\leq C\max \left\{ \int_E\mathbf{1}_{\{v_1: |v_1| > |v|\}} f_0(v_1) \; dv_1, (1+|v|)^{-s_1-1} \right\},
    \end{align*}
  for $|v|>R$.
\end{lemma}
\begin{proof}
    We start from
    \begin{align*} 
        \partial_t f + f  L_{FD}(f,1-f,1-f)= Q_{FD}^+(f,f,1-f,1-f) \leq Q_c^+(f,f).
    \end{align*}
    Taking $\int_E$ integral on both sides, we get
    \begin{align*}
        &\partial_t \int_E \mathbf{1}_{\{v_1: |v_1| > |v|\}} f(t,v_1)\, dv_1 +  \int_E \mathbf{1}_{\{v_1: |v_1| > |v|\}}  L_{FD}(f,1-f,1-f)(t,v_1) f(t,v_1)\, dv_1 \notag \\
        &\leq \int_E \mathbf{1}_{\{v_1: |v_1| > |v|\}}  Q_c^+(f,f)(t,v_1)\, dv_1.
    \end{align*}
    We apply Lemma \ref{lem:Lf_lower} to $L_{FD}(f,1-f,1-f)$ and Lemma \ref{lem:int_Qc_s1} to $Q_c^+(f,f)$.  Then, we obtain
    \begin{align*}
        &\partial_t \int_E \mathbf{1}_{\{v_1: |v_1| > |v|\}}f(t,v_1)\, dv_1  
        +C_1(1+|v|)^{\gamma} \int_E\mathbf{1}_{\{v_1: |v_1| > |v|\}}f(t,v_1)\, dv_1\\
        &\leq C\left(\|f(t)\|_{1,s_1} + \|f(t)\|_{\infty, 0}\right)^2(1+|v|)^{-s_1+\gamma-1}
    \end{align*} 
    for $|v| > R$. Here, $C_1$ and $C_2$ depend on the constants in Lemma \ref{lem:Lf_lower} and \ref{lem:int_Qc_s1}. By Lemma \ref{lem:sup_f}, we obtain
    \begin{align*}
        &\int_E\mathbf{1}_{\{v_1: |v_1| > |v|\}}f(t,v_1)\, dv_1 \\
        &\leq \max \left\{ \int_E\mathbf{1}_{\{v_1: |v_1| > |v|\}} f_0(v_1) \; dv_1, \frac{C_2}{C_1} \sup_{0\leq \tau\leq t}\left(\|f(\tau,v)\|_{1,s_1} + \|f(\tau,v)\|_{\infty, 0}\right)^2(1+|v|)^{-s_1-1} \right\}
    \end{align*}
    for $|v|\geq R$.

    Finally, we apply the property of the solution of the Boltzmann-Fermi-Dirac equation $0\leq f\leq 1$ and $L^1_{s_1}$ propagation result \eqref{eq:poly_pro}.
\end{proof}

Suppose $f_0\in L^\infty_{s_2}$. If $s_2>3$, then we easily check $f_0\in L^1_{s'_1}$ for any $s'_1<s_2-3$. By the same reason, when $s_2>2$, we have
\begin{align*}
    \int_E f_0(v_1)\,dv_1<C(1+|v|)^{s_2-2}.
\end{align*}
for some constant $C$ depending on $s_2$. Using this observation, we can rewrite the result of Lemma \ref{lem:int_E_s1s2} as follows.
\begin{lemma}\label{lem:int_E_s1s2_2}
    We consider the collision kernel \eqref{eq:B_defi} for $0< \gamma \leq 1$ and (H4). Let $E$ be a $2D$ plane in $\mathbb{R}^3$ and $f$ be the solution of the Boltzmann-Fermi-Dirac equation with $f_0\in L^1_2\cap L^\infty_{s_2}$ for some $s_2>2$. Then, there exist a constant $C>0$ and $R>0$, depending on the $\|f_0\|_{1,0}, \|f_0\|_{1,2}, \|f_0\|_{\infty,s_2}, \gamma, b(\cos\theta)$, and $s_2$, such that
\begin{align}\label{eq:int_E_s1s2}
    &\int_E\mathbf{1}_{\{v_1: |v_1| > |v|\}}f(v_1)\, dv_1\leq C(1+|v|)^{-c}
\end{align}
for $|v|\geq R$. Here, $c$ is given by
\begin{align*}
    c = \min\left\{s_2-2, \max\{3, \bar{s}_2 - 2\}\right\},
\end{align*}
where $\bar{s}_2<s_2$.
\end{lemma}

Having the main lemma in hand, we prove the main theorem.

\begin{proof}[Proof of Theorem \ref{thm:L1_bound}-(4)]
If $|v|\leq R$ for $R$ given in Lemma \ref{lem:int_E_s1s2_2}, then we just have
\begin{align*}
    f(t,v)\leq (1+R)^{s_2}(1+|v|)^{-s_2}
\end{align*}
for any $s_2\geq 0$ since $0\leq f\leq 1$. Therefore, it is enough to assume $|v|\geq R$.

Fix $v \in \mathbb{R}^3$ such that $|v|\geq R$. We define $f_i(w), f_u(w)$ as 
    \begin{align*}  
        f_i(t,w) \coloneqq  f(t,w) \mathbf{1}_{\left\{w:|w|<  \frac{|v|}{\sqrt{2}}\right\}} \text{\quad and \quad} f_u(t,w) \coloneqq  f(t,w) \mathbf{1}_{\left\{w:|w| \geq  \frac{|v|}{\sqrt{2}}\right\}}.
    \end{align*}
    Since the post-collision velocity should satisfy $|v'|^2 + |v'_*|^2\geq |v|^2$, one of $f_i(v')$ or $f_i(v'_*)$ should be $0$ for any fixed $v$. Therefore, we get $Q_c^+(f_i, f_i)(v) = 0$. Performing the change of variable $\sigma \rightarrow -\sigma$, we get
    $Q_c^+(f_i, f_u)(t,v) = Q_c^+(f_u, f_i) (t,v)$. As a result, we write
    \begin{align*}
        Q_c^+(f, f)(t,v) = 2 Q_c^+(f_i,f_u)(t,v)+Q_c^+(f_u,f_u)(t,v) \leq 2Q_c^+(f,f_u)(t,v).
    \end{align*}
    From the Carleman representation \eqref{eq:def_cal_Q+}, we get
    \begin{align*} 
        Q_c^+(f,f_u)(t,v)&=\int_{\mathbb{R}^3}f(t,v')\frac{1}{|v-v'|^{2-\gamma}}\int_{v+E_{v'-v}} \frac{h(\cos\theta_\omega)}{\cos^\gamma  \theta_\omega}f_u(t,v'_*) \, dv_*'dv' \\
        &\leq  \sup_{\theta_\omega} \frac{h(\cos\theta_\omega)}{\cos^\gamma  \theta_\omega}
        \int_{\mathbb{R}^3}f(t,v')\frac{1}{|v-v'|^{2-\gamma}}\int_{v+E_{v'-v}} f_u(t,v'_*) \, dv_*'dv' \\
        &\leq 
        C\int_{\mathbb{R}^3}f(t,v')\frac{1}{|v-v'|^{2-\gamma}}\int_{v+E_{v'-v}} f_u(t,v'_*) \, dv_*'dv'
    \end{align*}
    for $\gamma\leq 1$.
    
    Suppose $f_0\in L_{s_1}^1 \cap L_{s_2}^{\infty}$ for some $s_1\geq 2$ and $s_2>2$. We divide into two cases $s_2>5$ and $s_2\leq 5$.\\

    \noindent (1) $s_2>5$ case. In this case, we have $f_0\in L^1_{s_1}\cap L^\infty_{s_2}$, where $s_1<s_2-3$.
    By Lemma \ref{lem:int_E_s1s2_2}, we obtain
    \begin{align}\label{eq:E_integral}
        \int_{v+E_{v'-v}} f_u(t,v'_*) \, dv_*'\leq C(1+|v|)^{-(\bar{s}_2-2)},
    \end{align}
    for any $5\leq \bar{s}_2<s_2$. At this stage, we only know $f(t,v)\in L^1_2\cap L^\infty_0$ for all $t\geq 0$. From Lemma \ref{lem:Arkeryd}, we have
    \begin{align}\label{eq:whole_integral}
        \int_{\mathbb{R}^3}f(t,v')\frac{1}{|v-v'|^{2-\gamma}}\leq C(1+|v|)^{-2 + 2\frac{2-\gamma}{3}}
    \end{align}
    Combining \eqref{eq:E_integral} and \eqref{eq:whole_integral}, we get
    \begin{align*}
        Q_c^+(f,f_u)(t,v) \leq C(1+|v|)^{-\left(\bar{s}_2 - 2\frac{2-\gamma}{3}\right)}.
    \end{align*}
    Applying this inequality and Lemma \ref{lem:Lf_lower},
    \begin{align*}
        \partial_t f(t,v)+ C_1(1+|v|)^\gamma f(t,v)&\leq \partial_t f(t,v)+ f(t,v)L_{FD}(f,1-f,1-f)(t,v)\leq Q_c^+(f,f_u)(t,v)\\
        &\leq C_2(1+|v|)^{-\left(\bar{s}_2 - 2\frac{2-\gamma}{3}\right)}.
    \end{align*} 
    By Lemma \ref{lem:sup_f}, we finally reach
    \begin{align*}
        f(t,v)\leq \max\left\{\frac{C_2}{C_1}(1+|v|)^{-\left(\bar{s}_2 - 2\frac{2-\gamma}{3}\right)-\gamma}, f_0(v)\right\}
    \end{align*}
    for all $t\geq 0$ and a.e. $v$ with $|v|\geq R$. By the choice of $\bar{s}_2$, $f(t,v)\in L^\infty_3$ for all $t$.

    Now, we use Lemma \ref{lem:Arkeryd} again for $f(t,v)\in L^1_2\cap L^\infty_2$. Then, we get
    \begin{align}\label{eq:whole_integral2}
        \int_{\mathbb{R}^3}f(t,v')\frac{1}{|v-v'|^{2-\gamma}}\leq C(1+|v|)^{-(2-\gamma)}.
    \end{align}
    Repeating the same calculation using \eqref{eq:E_integral} and \eqref{eq:whole_integral2}, we finally get 
    \begin{align*}
        \partial_t f(t,v)+ C_1(1+|v|)^\gamma f(t,v)&\leq \partial_t f(t,v)+ f(t,v)L_{FD}(f,1-f,1-f)(t,v)\leq Q_c^+(f,f_u)(t,v)\\
        &\leq C_2(1+|v|)^{-\left(\bar{s}_2 - \gamma\right)},
    \end{align*}
    so
    \begin{align*}
        f(t,v)\leq \max\left\{\frac{C_2}{C_1}(1+|v|)^{-\bar{s}_2}, f_0(v)\right\}
    \end{align*}
    for $|v|\geq R$. It proves for any $\bar{s}_2<s_2$, $\|f(t)\|_{\infty, \bar{s}_2}\leq C$ for all $t$, where $C$ depends on $\|f_0\|_{1,0}$, $\|f_0\|_{1,2}$, $\|f_0\|_{\infty, s_2}$, $\gamma$, $b(\cos\theta)$, and $s_2$.\\

    \noindent (2) $s_2\leq 5$ case. When $s_2\leq 5$, then $s_2-3\leq 2$, so
    \begin{align}\label{eq:E_integral2}
        \int_{v+E_{v'-v}} f_u(t,v'_*) \, dv_*'\leq C(1+|v|)^{-(s_2-2)}
    \end{align}
    in \eqref{eq:int_E_s1s2}. Bounding $Q^+_c(f,f_u)$ by \eqref{eq:E_integral2} and \eqref{eq:whole_integral} and repeating the same calculation in (1), we get
    \begin{align*}
        f(t,v)\leq \max\left\{\frac{C_2}{C_1}(1+|v|)^{-\left(s_2 + (2-\gamma) - 2\left(1-\frac{2-\gamma}{3}\right)\right)}, f_0(v)\right\}
    \end{align*}
    for all $t\geq 0$ and for a.e. $|v|\geq R$. If $s_2\leq s_2 + (2-\gamma) - 2\left(1-\frac{2-\gamma}{3}\right)$, then we are done. If not, we now know that $f(t,v)\in L^1_2\cap L^\infty_{s_{2,1}'}$ for all $t$, where $s_{2,1}' = s_2 + (2-\gamma) - 2\left(1-\frac{2-\gamma}{3}\right)$. We repeatedly apply Lemma \ref{lem:Arkeryd} for $f(t,v)\in L^1_2\cap L^\infty_{s_{2,1}'}$ and follow all the computations above. After the $k\geq 1$ times iteration, we have $f(t,v)\in L^1_2\cap L^\infty_{s_{2,k}'}$, where
    \begin{align*}
        s_{2,k}' &=\left(s_2 - (2-\gamma) + 2\left(1-\frac{2-\gamma}{3}\right)\right)\sum_{j=0}^{k-1}\left(\frac{2-\gamma}{3}\right)^j.
    \end{align*}
    Since $\sum_{j=0}^\infty \left(\frac{2-\gamma}{3}\right)^j = \frac{1}{1-\frac{2-\gamma}{3}}>1$, so
    \begin{align*}
        s_{2,\infty} = \left(s_2 - (2-\gamma)  + 2\left(1-\frac{2-\gamma}{3}\right)\right)\frac{1}{1-\frac{2-\gamma}{3}}\geq (s_2-(2-\gamma)) + 2>s_2.
    \end{align*}
    It proves that for any $s_2\leq 5$, there exists $k_0$ such that $s_2 \leq s_{2,k_0}$. It ends the proof.
\end{proof}

\noindent{\bf Data availability:} No data were used for the research described in the article.
\newline

\noindent{\bf Conflict of interest:} The authors declare that they have no conflict of interest.\newline

\noindent{\bf Acknowledgement} We would like to thank Professor Donghyun Lee for helpful discussions during the preparation of this work.
Gayoung An and Sungbin Park are supported by the National Research Foundation of Korea(NRF) grant funded by the Korea government(MSIT)(No.RS-2023-00212304 and No.RS-2023-00219980). 

\nocite{*}

\end{document}